\documentclass{amsart}
\usepackage{amssymb,mathtools,comment,hyperref,mathabx}
\usepackage[shortlabels]{enumitem}
\usepackage[alphabetic,initials,nobysame]{amsrefs}
\usepackage{tikz}
\usepackage{color,faktor,cancel}

\theoremstyle{plain}
\newtheorem{theorem}{Theorem}
\newtheorem{lemma}[theorem]{Lemma}

\newtheorem{prop}[theorem]{Proposition}
\newtheorem{corollary}[theorem]{Corollary}

\newtheorem{problem}[theorem]{Problem}

\theoremstyle{definition}
\newtheorem{definition}[theorem]{Definition}

\theoremstyle{remark}
\newtheorem{remark}[theorem]{Remark}
\newtheorem{example}[theorem]{Example}
\newtheorem{note}[theorem]{Note}

\numberwithin{equation}{section}
\numberwithin{theorem}{section}
\numberwithin{conjecture}{section}
\numberwithin{figure}{section}

\theoremstyle{definition}
\newtheorem{thmx}{Theorem}

\newcommand{\x}{\scalebox{1.2}{$\chi$} } 

\newcommand{\br}{\overline}
\newcommand{\R}{\mathbb R}
\newcommand{\C}{\mathbb C}
\newcommand{\D}{\mathbb D}
\newcommand{\Z}{\mathbb Z}
\newcommand{\N}{\mathbb N}
\newcommand{\Q}{\mathbb Q}

\DeclareMathOperator{\dist}{\textup{\text{dist}}}
\DeclareMathOperator{\diam}{\textup{\text{diam}}}

\DeclareMathOperator{\inter}{\textup{\text{int}}}
\DeclareMathOperator{\Int}{\textup{\text{Int}}}
\DeclareMathOperator{\Ext}{\textup{\text{Ext}}}

\DeclareMathOperator{\length}{\textup{\text{length}}}

\DeclareMathOperator{\area}{\textup{Area}}
\DeclareMathOperator{\loc}{\textup{loc}}

\DeclareMathOperator{\Deg}{deg}

\makeatletter
\DeclareFontFamily{U}{tipa}{}
\DeclareFontShape{U}{tipa}{m}{n}{<->tipa10}{}
\newcommand{\arc@char}{{\usefont{U}{tipa}{m}{n}\symbol{62}}}%

\newcommand{\arc}[1]{\mathpalette\arc@arc{#1}}

\newcommand{\arc@arc}[2]{%
  \sbox0{$\m@th#1#2$}%
  \vbox{
    \hbox{\resizebox{\wd0}{\height}{\arc@char}}
    \nointerlineskip
    \box0
  }%
}
\makeatother

\def\note#1
{\marginpar
{\nt $\leftarrow$
\par
\hfuzz=20pt \hbadness=9000 \hyphenpenalty=-100 \exhyphenpenalty=-100
\pretolerance=-1 \tolerance=9999 \doublehyphendemerits=-100000
\finalhyphendemerits=-100000 \baselineskip=6pt
#1}\hfuzz=1pt}

\makeindex

\begin{document}

\title[David extensions, welding, mating, and removability]{David extension of circle homeomorphisms, welding, mating, and removability}

\author[M. Lyubich]{Mikhail Lyubich}
\address{Institute for Mathematical Sciences, Stony Brook University, Stony Brook, NY 11794, USA.}
\email{mlyubich@stonybrook.edu}
\thanks{M.~Lyubich was supported by NSF grants DMS-1600519 and 1901357, and by a Fellowship from the 
Hagler Institute for Advanced Study.}

\author[S. Merenkov]{Sergei Merenkov}
\address{Department of Mathematics, City College of New York and CUNY Graduate Center, New York, NY 10031, USA}
\email{smerenkov@ccny.cuny.edu}
\thanks{S.~Merenkov was supported by NSF grants DMS-1800180 and 2247364.}

\author[S. Mukherjee]{Sabyasachi Mukherjee}
\address{School of Mathematics, Tata Institute of Fundamental Research, 1 Homi Bhabha Road, Mumbai 400005, India}
\email{sabya@math.tifr.res.in}
\thanks{S.~Mukherjee was supported by the Department of Atomic Energy, Government of India, under project no.12-R\&D-TFR-5.01-0500, an endowment of the Infosys Foundation, and SERB research project grants SRG/2020/000018 and MTR/2022/000248.}

\author[D. Ntalampekos]{Dimitrios Ntalampekos}
\address{Department of Mathematics, Aristotle University of Thessaloniki, Thessaloniki, 54152, Greece}
\email{dntalam@math.auth.gr}
\thanks{D.~Ntalampekos was supported by NSF grants DMS-2000096 and 2246485, and by the Simons Foundation.}

\date{\today}


\begin{abstract}
  We provide a David extension result for circle homeomorphisms conjugating two dynamical systems
  such that parabolic periodic points go to parabolic periodic points,
  but hyperbolic points can go to parabolics as well.
  We use this result, in particular,
  to prove the existence of a new class of welding homeomorphisms,
 to establish an explicit dynamical connection between critically fixed anti-rational maps and kissing reflection groups,
to show conformal removability of the Julia sets of geometrically finite polynomials and of the limit sets of necklace reflection groups,
to produce matings of anti-polynomials and necklace reflection groups,
and to give a new proof of the existence of Suffridge polynomials (extremal points in certain spaces of univalent maps). 
\end{abstract}

\maketitle

\setcounter{tocdepth}{2}
\tableofcontents

\section{Introduction}\label{S:Intro}
The nicest and best understood classes of dynamical systems are the ones with uniform hyperbolicity. In rational dynamics, it means that a map admits a conformal metric in a neighborhood of its Julia set with respect to which it is uniformly expanding. The next simplest class of rational maps are \emph{parabolic} maps, which have parabolic cycles but no critical points on their Julia sets.  Due to the existence of parabolic cycles, such maps are not uniformly hyperbolic, but still enjoy a weaker property of expansiveness on their Julia sets. In the parallel field of Kleinian groups, the concept of hyperbolicity translates into compactness of a certain manifold (or orbifold) with boundary (the ``convex core''),
and Kleinian groups satisfying this condition are called \emph{convex co-compact}. The analogue of parabolic rational maps in this world are given by \emph{geometrically finite} Kleinian groups. Such groups may also have parabolic fixed points. Also, recently a new area of Schwarz reflection dynamics has been developed which combines features of both rational maps and Kleinian groups in a single dynamical plane, and the notion of geometric finiteness has a natural meaning in this setting as well (see \cites{LLMM18a,LLMM19}).

In general, many topological, dynamical, and analytic properties of hyperbolic dynamical systems are shared by their parabolic counterparts. However, in the absence of uniform expansion, one often needs to develop more elaborate techniques to study parabolic systems. The current paper grew out of concrete problems in conformal dynamics that necessitated direct passage from hyperbolic conformal dynamical systems to parabolic ones in a sufficiently regular manner.
Such a surgery procedure first appeared in the work of Ha{\"i}ssinsky in the context of complex polynomials \cites{Hai98,Hai00, BF14}. It was shown there that hyperbolic polynomial Julia sets can be turned into parabolic Julia sets via \emph{David} homeomorphisms, which are generalizations of quasiconformal homeomorphisms. In the present work, we carry this philosophy further by developing new analytic tools, namely,
a David Extension Theorem for dynamical circle homeomorphisms
(based upon results of J.~Chen, Z.~Chen and He \cite{CCH96}, and Zakeri \cite{Zak08})
and a surgery technique using the David Integrability Theorem,
that gives rise to a unified approach to turn hyperbolic (anti-)rational maps to parabolic (an\-ti-)rational maps, Kleinian reflection groups, and \emph{Schwarz reflection maps} that are matings of anti-polynomials and reflection groups.

At the technical heart of the paper lies a theorem that guarantees that a general class of dynamically defined circle homeomorphisms can be extended as David homeomorphisms of the unit disk. More precisely, let $f$ and $g$ be two expansive covering maps of the unit circle that have the same orientation, have equivalent Markov partitions, are analytic on the defining intervals of the corresponding Markov partitions, and admit piecewise conformal extensions satisfying certain natural conditions; see conditions \eqref{condition:uv} and~\eqref{condition:holomorphic} below. Let $h$ be an orientation-preserving homeomorphism of the circle that conjugates $f$ to $g$. Then, our main result, Theorem~\ref{theorem:extension_generalization}, states that $h$ has a David extension to the unit disk,
provided it takes parabolic periodic points of $f$ to parabolic periodic points of $g$,
while hyperbolic (i.e., repelling) periodic points of $f$ can go  to either hyperbolic or to parabolic periodic points of $g$.
The following result is a special case of Theorem~\ref{theorem:extension_generalization}.

\begin{thmx}[Blaschke product--circle reflections]\label{david_intro_version}
Consider ordered points $a_{0}=a_{d+1}, a_1,\dots,a_d$ on the circle $\mathbb S^1$, where $d\geq 2$. For $k\in \{0,\dots,d\}$ let $f|_{\arc{(a_k,a_{k+1})}}$ be the reflection along the circle $C_k$ that is orthogonal to the unit circle at the points $a_k$ and $a_{k+1}$. Moreover, let $B$ be an anti-holomorphic Blaschke product of degree $d$ with an attracting fixed point in $\D$. Then there exists a homeomorphism $h\colon \mathbb S^1\to \mathbb S^1$ that conjugates $B|_{\mathbb S^1}$ to $f$, and has a David extension in $\D$.
\end{thmx}

A particular case of the above theorem, where $B(z)=\overline{z}^2$ and the points $a_0,a_1,a_2$ are the third roots of unity, was proved in a recent work~\cite[Lemma~9.4]{LLMM19}. The proof given in \cite{LLMM19} relies upon arithmetic properties of the conjugacy $h$. In contrast, the proof of Theorem~\ref{david_intro_version} is purely dynamical, and can be applied to a broad range of situations. The piecewise analytic (rather than globally analytic) nature of the map $f$ and the matching of hyperbolic periodic points of $B$ to parabolic points of $f$ pose the main complications in the proof that cannot rely on the standard quasiconformal machinery.

A combination of Theorem~\ref{david_intro_version} together with a surgery technique devised in Section~\ref{david_surgery_sec} permit us to construct conformal dynamical systems by replacing the attracting
dynamics of an anti-rational map on an invariant Fatou component by piecewise defined circular reflections. Implementing this strategy, we are able to convert anti-rational maps to Kleinian reflection groups and to Schwarz reflection maps in a dynamically natural way. A precursor to the birth of this general machinery is a concrete construction carried out in \cite{LLMM19}, where the Julia set of a cubic critically fixed anti-rational map (an anti-rational map with all critical points fixed) was shown to be homeomorphic to the classical Apollonian gasket (the limit set of a Kleinian reflection group) and to the limit set of the Schwarz reflection map associated with a deltoid and an inscribed circle. More generally, a dynamical correspondence between Apollonian-like gaskets (associated to triangulations of the Riemann sphere) and a class of critically fixed anti-rational maps was set up in the same paper. A further generalization appeared in \cite{LLM20}, where a bijective correspondence between kissing reflection groups (groups generated by reflections in the circles of finite circle packings) with connected limit sets and critically fixed anti-rational maps was established using Thurston's topological characterization of rational functions. Using the surgery technique designed in the current paper, one can transform each critically fixed anti-rational map into a kissing reflection group such that the Julia set of the former is homeomorphic to the limit set of the latter under a global David homeomorphism. The next theorem presents a direct construction of a map from the class of critically fixed anti-rational maps to the class of kissing reflection groups.

\begin{thmx}[From anti-rational maps to kissing groups]\label{bijection_via_david__intro_version}
Let $R$ be a critically fixed anti-rational map. Then, there exists a kissing reflection group $\Gamma$ such that the Julia set $\mathcal{J}(R)$ is homeomorphic to the limit set $\Lambda(\Gamma)$ via a dynamically natural David homeomorphism of $\widehat{\C}$.
\end{thmx}

\noindent (See Theorem~\ref{bijection_via_david_thm} for a precise version of Theorem~\ref{bijection_via_david__intro_version}.)
\vspace{2mm}

As another interesting application of the main extension Theorem~\ref{theorem:extension_generalization},
we produce a new class of {\em welding} homeomorphisms. A welding homeomorphism is a homeomorphism of the circle that arises as the composition of a conformal map from the unit disk onto the interior region of a Jordan curve
with a conformal map from the exterior of this Jordan curve onto the exterior of the unit disk.
Previously known examples of welding homeomorphisms include quasisymmetric homeomorphisms \cites{LV60,Pfl60}, circle maps that extend to David homeomorphisms of the sphere \cite{Dav88}, weakenings of quasisymmetric homeomorphisms that are sufficiently regular \cites{Ham91}, and some quite ``wild'' homeomorphisms \cite{Bis07}. We prove here that a circle homeomorphism that conjugates two expansive piecewise analytic maps
(under conditions \eqref{condition:uv} and \eqref{condition:holomorphic}),
each of whose periodic points is either \emph{symmetrically} parabolic or hyperbolic, is a welding homeomorphism (see Sections \ref{section:markov_partitions} and \ref{welding_sec} for the definitions). These homeomorphisms are composed of David homeomorphisms and their inverse; see Theorem~\ref{theorem:welding}.
\vspace{2mm}

Let us now elaborate on the background and history of the various problems that motivated the development of the David Extension Theorem~\ref{theorem:extension_generalization}. Along the way, we will also give informal statements of some of the main results.

Conformal removability (see Subsection~\ref{david_removability_subsec}) of certain Julia sets was known from the works of  P.~Jones \cite{Jon91} and P.~Jones--S.~Smirnov~\cite{JS00}. In~\cite{Jon91} it is shown, in particular, that the boundaries of John domains are conformally removable. As a consequence, Jones obtained that the Julia sets of subhyperbolic polynomials are conformally removable. In~\cite{JS00}, the authors extended these removability results to the boundaries of H\"older domains, and, in particular to the Julia sets of Collet--Eckmann polynomials. However, these techniques do not yield removability of polynomial Julia sets with parabolic periodic points as these Julia sets have cusps. 

A Kleinian reflection group is called a \emph{necklace} group if it is generated by reflections in the circles of a circle packing with $2$-connected, outerplanar contact graph (see Remark~\ref{outerplanar_2_conn_rem} for the definitions and Figure~\ref{necklace_group_fig} for an example). Similarly to Julia sets of parabolic polynomials, limit sets of necklace groups have intricate geometric structure, see, e.g., Figure~\ref{necklace_group_fig}; indeed, limit sets of necklace groups also have cusps. Nevertheless, we are able to use our main result, Theorem~\ref{theorem:extension_generalization}, to show that connected Julia sets of geometrically finite polynomials as well as limit sets of necklace groups are conformally removable.

\begin{thmx}[Removability of limit and Julia sets]\label{removability_intro_version}
\noindent\begin{itemize}
\item Limit sets of necklace reflection groups are conformally removable.

\item Connected Julia sets of geometrically finite polynomials are conformally removable.
\end{itemize}
\end{thmx}

Moreover, we produce an example of a conformally removable Julia set, called the {\em pine tree}, which is a Jordan curve with both inward and outward cusps; see Figure \ref{figure:two_sided_cusps} and Example~\ref{pine_example}.
Examples of this kind seem to be rare as the techniques of \cite{Jon91} and \cite{JS00} do not apply directly to such sets. If, instead, a Jordan curve has only inward cusps (in a precise sense), then the region bounded by the curve is a John domain so \cite{Jon91} would imply the conformal removability of the curve. 

\begin{figure}
\centering
\includegraphics[scale=0.5]{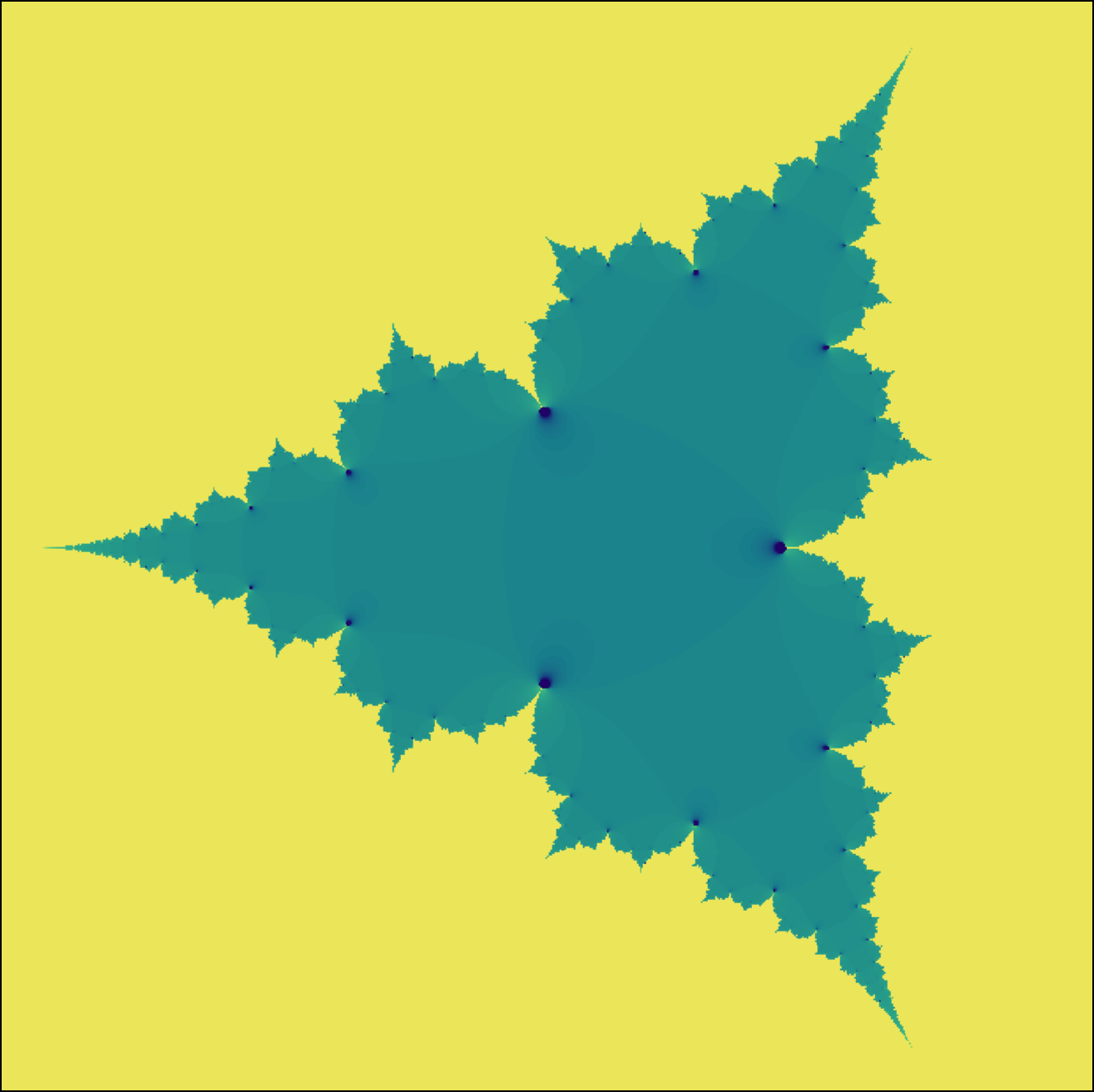}
\caption{The pine tree Julia set. The Julia set of $R(z)=\frac{4z^3+8-3(1-\sqrt{3})}{(1-\sqrt{3})z^3+8+4\sqrt{3}}$ is a conformally removable curve with both inward and outward cusps.}\label{figure:two_sided_cusps}
\end{figure}

The proof of Theorem~\ref{removability_intro_version} is given in Section~\ref{conf_rem_julia_limit_sec}, more precisely in Theorems~\ref{limit_removable_thm} and~\ref{T:ConfRemove}, respectively. In fact, the proofs of the two removability statements have a common philosophy. We topologically realize the limit set of a necklace group or the Julia set of a geometrically finite polynomial as a (sub-)hyperbolic polynomial Julia set (which are known to be $W^{1,1}$-removable), and then construct a global David homeomorphism that carries this (sub-)hyperbolic polynomial Julia set to the limit
or Julia set in question. This, combined with Theorem~\ref{theorem:w11_removable}, which states that the image of a
$W^{1,1}$-removable compact set under a global David homeomorphism is conformally removable, leads to the desired conclusion. In fact, we are unaware of an intrinsic proof of Theorem~\ref{removability_intro_version}. It is also worth mentioning that these removability results contrast with recent works by the fourth author~\cite{Nta19, Nta20b}, where he shows that the standard Sierpi\'nski gasket as well as all planar Sierpi\'nski carpets are not removable for conformal maps.
The classical Apollonian gasket (which is the limit set of a Kleinian reflection group) contains a homeomorphic copy of the Sierpi\'nski gasket. This suggests that the Apollonian gasket is non-removable as well, but it still remains an open question.

We now turn our attention to the next application of our David extension and surgery theory. The concept of mating of two conformal dynamical systems goes back to
{\em Klein's Combination Theorem},
and was brought to a new level in the seminal work of Bers on the \emph{simultaneous uniformization} of two Riemann surfaces,
which allows one to mate two Fuchsian groups to obtain a quasi-Fuchsian group \cite{Ber60}. A significantly more difficult theorem of W.~Thurston, called the
{\em Double Limit Theorem},
provides a sufficient condition to `mate' two projective measured laminations (or, two groups on the boundary of the corresponding Teichm{\"u}ller space) to obtain a Kleinian group \cite{Thu98,Ota01}. Later on, Douady and Hubbard introduced the notion of mating of two polynomials to produce a rational map, which can be seen as a philosophical parallel to the combination theorems in the world of rational dynamics \cite{Dou83}. 
Each of the aforementioned mating constructions attempts to combine two similar conformal dynamical systems to produce a richer conformal dynamical system. Often it is not hard to mate two conformal dynamical systems topologically, but uniformizing the topological mating (i.e., endowing it with a complex structure) lies at the core of the problem.

In~\cite{BP94} (see also~\cite[\S 7.8]{BF14}), S.~Bullett and C.~Penrose used iterated algebraic correspondences
to introduce a notion of mating of the modular group ${\rm PSL}(2,\Z)$ and certain quadratic polynomials. Also, in~\cite{BH00} S.~Bullett and W.~Harvey used quasiconformal surgery to construct a holomorphic correspondence between any quadratic polynomial and certain representations of the free product $C_2*C_3$ of cyclic groups (of orders 2 and 3) in ${\rm PSL}(2,\C)$.

Recently, novel perspectives and techniques were introduced to bind together the actions of two different types of conformal dynamical systems; namely, an anti-rational map and a Kleinian reflection group, in a single conformal dynamical system \cite{LLMM18a,LLMM18b,LLMM19,LMM20}. In particular, various examples of matings of Nielsen maps of Kleinian reflection groups and anti-rational maps were discovered, and these matings were realized as Schwarz reflection maps associated with suitable quadrature domains (see Section~\ref{mating_sec} for precise definitions). In this paper we introduce a unified framework for these `hybrid' matings, prove a general theorem that ensures the existence of matings for a large class of Kleinian reflection groups and anti-polynomials, and illustrate the mating phenomena with a few explicit examples. The following theorem can be thought of as an analogue of the Thurston Realization Theorem for matings of anti-polynomials and necklace groups.

\begin{figure}[ht!]
\includegraphics[width=0.6\linewidth]{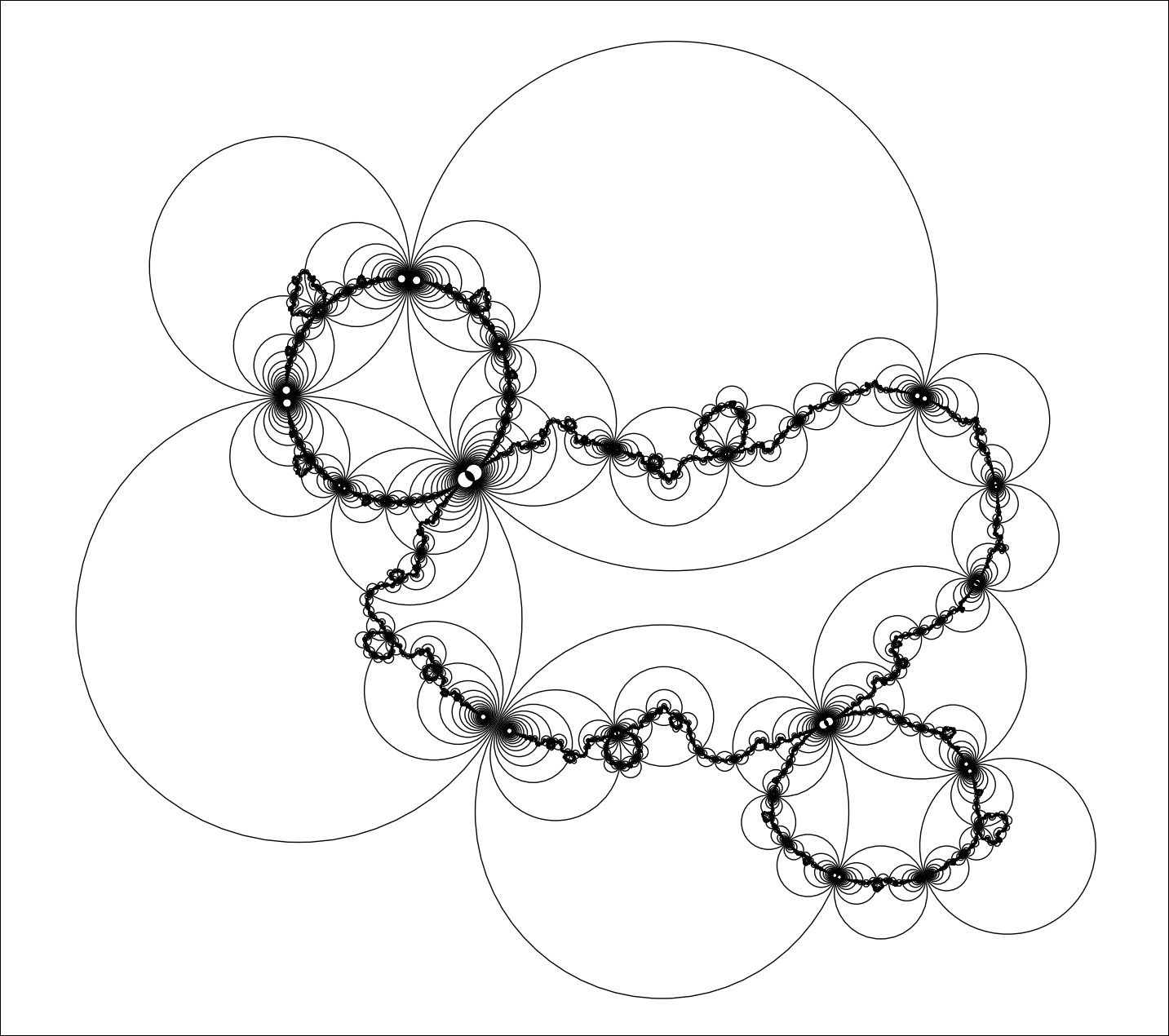}
\caption{The limit set of a necklace group on the boundary of the Bers slice of an ideal $9$-gon reflection group.}
\label{necklace_group_fig}
\end{figure}

\begin{thmx}[Mating reflection groups with anti-polynomials]\label{mating_intro_version}
A postcritically finite, hyperbolic anti-polynomial of degree $d$ and a necklace group of rank $d+1$ are conformally mateable if and only if they are topologically mateable, and a conformal mating (when it exists) is a piecewise Schwarz reflection map.
\end{thmx}

All previously known examples of conformal matings of
  postcritically finite anti-polynomials with necklace groups, namely,
  matings of quadratic postcritically finite anti-polynomials with the
  ideal triangle reflection group \cite{LLMM18a,LLMM18b}, and matings
  of $\overline{z}^d$ with all necklace groups of rank $d+1$
  \cite{LMM20}, are easily recovered from
  Theorem~\ref{mating_intro_version} (see
  Subsection~\ref{known_examples_subsec}). To illustrate the power of
  this theorem, we construct first examples of conformal matings of
  anti-polynomials and necklace groups, neither of which has a Jordan
  curve Julia/limit set (see Subsections~\ref{example_1_subsec} and~\ref{example_2_subsec}).

Since the topological mating of a necklace group $\Gamma$ and a postcritically finite anti-polynomial $P$ is not defined on the whole $2$-sphere, one needs to build a novel realization theory to prove Theorem~\ref{mating_intro_version}. As mentioned before, the principal difficulty in this theory arises from the mismatch between the quantitative behavior of the Julia dynamics of a hyperbolic anti-polynomial and the limit set dynamics of the Nielsen map of a reflection group. Indeed, the former has only hyperbolic fixed points on its Julia set, while the latter has parabolic fixed points on its limit set, which renders purely quasiconformal tools inapplicable
to this setting. This problem can be tackled in two steps. One can use
{\em W.~Thurston's topological characterization theorem}
to construct an anti-rational map $R$ that is a conformal mating of $P$ and another anti-polynomial $P_\Gamma$ such that the Julia dynamics of $P_\Gamma$ is topologically conjugate to the limit set dynamics of the Nielsen map of $\Gamma$. The existence of such an anti-polynomial $P_\Gamma$ follows from
\cite{LMM20} or \cite{LLM20} (making use of Poirier's realization of Hubbard trees by anti-polynomials \cite{Poi10,Poi13}), while conformal mateability of $P$ and $P_\Gamma$ follows from a general mateability criterion proved in \cite{LLM20}.
Subsequently, we apply Theorem~\ref{david_intro_version} and our David surgery tools to conformally glue the  reflection group dynamics into suitable invariant Fatou components of $R$. This produces the desired conformal mating of $\Gamma$ and $P$, which turns out to be the Schwarz reflection map associated with finitely many disjoint quadrature domains. The proof of Theorem~\ref{mating_intro_version} is carried out in Section~\ref{mating_sec}; we refer the reader to Theorem~\ref{poly_group_mating_thm} for a precise statement.
The mating construction is illustrated with various explicit examples in Section~\ref{mating_examples_sec}; see Figures~\ref{mating_1_fig},~\ref{ellipse_disk_dyn_fig},~\ref{mating_2_fig}, and~\ref{mating_3_fig}.

While the Extension Theorem~\ref{david_intro_version} is sufficient for the purpose of mating anti-polynomials with necklace reflection groups, the general extension Theorem~\ref{theorem:extension_generalization} yields a mating result for partially defined conformal dynamical systems on the closed unit disk. More precisely, we show that any two expansive piecewise analytic covering maps of the unit circle admitting conformal extensions satisfying conditions \eqref{condition:uv} and \eqref{condition:holomorphic} can be mated conformally, as long as each of their periodic points is {symmetrically} parabolic or hyperbolic.
(We allow ourselves to match parabolic and hyperbolic points in an arbitrary way).
This is stated as Theorem \ref{theorem:mating_general}.

The final application of our hyperbolic-parabolic surgery theory concerns certain extremal problems in spaces of univalent functions. It is well-known that extremal points of the classically studied space $\Sigma$ of (suitably normalized) schlicht functions on the exterior disk $\D^*:=\widehat{\C}\setminus\overline{\D}$ are the so-called \emph{full mappings}; i.e., $f$ is an extremal point of $\Sigma$ if and only if $\C\setminus f(\D^*)$ has zero area \cite[\S 9.6]{Dur83}. These extremal points play an important role in the study of coefficient bounds for the space $\Sigma$
(the analogue of the Bieberbach Conjecture/De Brange's Theorem is still open for class $\Sigma$), but the purely transcendental nature of these extremal maps add to their complexity. For each $d\geq 2$, the compact subspace $\Sigma_d^*\subset\Sigma$ consists of maps in $\Sigma$ that extend as degree $d+1$ rational maps of $\widehat{\C}$ with the maximal number of critical points on $\mathbb{S}^1$. The union of these subspaces in dense in $\Sigma$. The extremal points $f$ of $\Sigma_d^*$ correspond to maps with the maximal number of singular points on $f(\mathbb{S}^1)$: $d+1$ cusps and $d-2$ double points. This allows one to describe the `shape' of extremal maps of $\Sigma_d^*$ in terms of certain combinatorial trees with angle data, which we call \emph{bi-angled trees}. Using the tools developed in this paper, we give a new proof of the following result that recently appeared in~\cite{LMM19}.

\begin{thmx}[Extremal points for schlicht functions]\label{extremal_points_intro_version}
Extremal points of $\Sigma_d^*$ are classified by bi-angled trees with $d-1$ vertices.
\end{thmx}

We refer the readers to Theorem~\ref{existence_suffridge_thm} for a precise formulation of Theorem~\ref{extremal_points_intro_version}. Let us also remark that while the problem of classifying the extremal points of $\Sigma_d^*$ is a non-dynamical one, the proof of Theorem~\ref{extremal_points_intro_version} makes essential use of the dynamics of Schwarz reflection maps associated with the members of $\Sigma_d^*$. With this perspective, the classification problem becomes a problem of realization and rigidity of certain Schwarz reflection dynamical systems.
In fact, the desired Schwarz reflection maps are realized using David surgery techniques developed in the paper.
On the other hand, rigidity of such Schwarz reflection maps is a consequence of conformal removability of their limit sets.

\bigskip

\noindent\textbf{Acknowledgements.} The authors would like to thank Malik Younsi for posing the question on removability of limit sets of Kleinian reflection groups, and for a motivating discussion. The second and third authors thank the Institute for Mathematical Sciences, and Simons Center for Geometry and Physics at Stony Brook University for their hospitality.

\bigskip

\section{Preliminaries on David homeomorphisms}

\subsection{David Integrability Theorem}\label{section:david}

Recall that if $U\subset {\C}$ is an open set and $p\in [1,\infty]$, a measurable function $f\colon U \to \C$ lies in the Sobolev space $W^{1,p}_{\loc}(U)$ if $f\in L^p_{\loc}(U)$ and $f$ has distributional derivatives lying in $L^p_{\loc}(U)$. We say that a function $f\colon U\to \C$ lies in $W^{1,p}(U)$ if $f\in W^{1,p}_{\loc}(U)$ and the quantity 
$$\|f\|_{W^{1,p}(U)}\coloneqq \|f\|_{L^p(U)} + \|Df\|_{L^p(U)}$$
is finite. Here, $Df$ denotes the differential matrix of $f$ and $\|Df\|_{L^p(U)}$ denotes the $L^p$ norm of the operator norm of $Df$. Moreover, the $L^p$ norm here is taken with respect to the Lebesgue measure of $\C$. 

Consider sets $U_i$, $i\in \{1,2\}$, and suppose that each of them is an open subset of either $\C$ or $\widehat{\C}$; we make this distinction because we are going to use a different metric and measure in each case. We say that a measurable function $f\colon U_1\to U_2$ lies in $W^{1,p}_{\loc}(U_1\to U_2)$ if it lies in $W^{1,p}_{\loc}(U_1)$ in local coordinates. Consider the derivative $D_{(d_1,d_2)}f$ of $f$, regarded as a map between the Riemannian manifolds $(U_1,d_1,\mu_1)$ and $(U_2,d_2,\mu_2)$; here $d_i,\mu_i$ are the Euclidean metric and measure if the set $U_i$ is a subset of the plane and $d_i,\mu_i$ are the spherical metric and measure, if $U_i$ is a subset of $\widehat{\C}$. Note that in planar coordinates the derivative $D_{(d_1,d_2)}f$ is just a multiple of the Euclidean differential $Df$. We say that $f\in W^{1,p}(U_1\to U_2)$ if $d_2(f,0)\in L^p(U_1;\mu_1) $ and $D_{(d_1,d_2)}f \in L^p( U_1;\mu_1)$. We note that if the point at $\infty$ does not lie in the closure of any of the sets $U_1$ and $U_2$ then the space $W^{1,p}(U_1\to U_2)$ coincides as a set with $W^{1,p}(U_1)$, since the Euclidean and spherical metrics of $U_1$ and $U_2$ are comparable. In all cases we will be using the simplified notation $W^{1,p}_{\loc}(U_1)$ and $W^{1,p}(U_1)$, whenever it does not lead to a confusion.

An orientation-preserving homeomorphism $H\colon U\to V$ between domains in the Riemann sphere $\widehat{\C}$ is called a \textit{David homeomorphism} if it lies in the Sobolev class $W^{1,1}_{\loc}(U)$ and there exist constants $C,\alpha,K_0>0$ with
\begin{align}\label{definition:david1}
\sigma(\{z\in U: K_H(z)\geq K\}) \leq Ce^{-\alpha K}, \quad K\geq K_0.
\end{align}
Here $\sigma$ is the spherical measure and $K_H$ is the distortion function of $H$, given by
\begin{align*}
K_H(z)=\frac{1+|\mu_H|}{1-|\mu_H|}, 
\end{align*}
where
\begin{align*}
\mu_H= \frac{\partial H/ \partial \br z}{\partial H/\partial z}
\end{align*}
is the Beltrami coefficient of $H$. By condition \eqref{definition:david1}, $K_H$ is finite a.e. and $\mu_H$ takes values in $\D$ a.e. Condition \eqref{definition:david1} is equivalent to 
\begin{align}\label{definition:david2}
\sigma(\{z\in U: |\mu_H(z)|\geq 1-\varepsilon\}) \leq C'e^{-\alpha'/\varepsilon}, \quad \varepsilon\leq \varepsilon_0,
\end{align}
where $C',\alpha',\varepsilon_0$ depend only on $C,\alpha,K_0$. Moreover, another condition equivalent to \eqref{definition:david1} is the exponential integrability of $K_H$:
\begin{align*}
\int_U \exp(p K_H) \,d\sigma <\infty
\end{align*}
for some $p>0$, where $p$ and $\alpha$ are related to each other. Hence, David homeomorphisms are also called homeomorphisms of \textit{exponentially integrable distortion}. The theory of such mappings has been developed to a great extent over the past decades. We direct the reader to \cite[Chapter 20]{AIM09} for more background.  

We list some properties of David homeomorphisms. Let $H\colon U\to V$ be such a map. First, $H$ and $H^{-1}$ are \textit{absolutely continuous in measure}: 
\begin{align*}
\sigma (H(E))=0 \quad \textrm{if and only if} \quad \sigma(E)=0
\end{align*}
for all measurable sets $E\subset U$; see \cite[Theorem 20.4.21]{AIM09}. Second, $H$ satisfies the change of coordinates formula
\begin{align}\label{david_change_of_variables}
\int_{V} g \, d\sigma =\int_{U} g\circ H\cdot  J_H^{\sigma}\, d\sigma
\end{align} 
for all non-negative Borel measurable functions $g$ on $V$, where $J_H^{\sigma}$ is the spherical Jacobian of $H$. This result is due to Federer \cite[Theorem 3.2.5, p.~244]{Fed69} and holds in much more general settings. In particular, \eqref{david_change_of_variables} implies that $J_H^{\sigma}\neq 0$ a.e.

The main result in the theory of David homeomorphisms is the following integrability theorem. If $U$ is an open subset of $\widehat {\C}$ and $\mu \colon U \to \D$ is a measurable function such that $(1+|\mu|)/(1-|\mu|)$ is exponentially integrable in $U$, then we say that $\mu$ is a \textit{David coefficient in $U$} or just a \textit{David coefficient} if $U$ is implicitly understood.
\begin{theorem}[David Integrability Theorem; see {\cite{Dav88}}, {\cite[Theorem~20.6.2, p.~578]{AIM09}}]\label{theorem:integrability_david}
Let $\mu\colon \widehat{\C} \to \D$ be a David coefficient.  Then there exists a homeomorphism $H\colon \widehat{\C} \to \widehat{\C}$ of class $W^{1,1}(\widehat \C)$ that solves the Beltrami equation
\begin{align*}
\frac{\partial H}{\partial \br z}= \mu \frac{\partial H}{\partial z}.
\end{align*}
Moreover, $H$ is unique up to postcomposition with M\"obius transformations.
\end{theorem}

The David Integrability Theorem is a generalization of the Measurable Riemann Mapping Theorem \cite[Theorem 5.3.4, p.~170]{AIM09}, which states that if $\|\mu\|_\infty<1$, then there exists a quasiconformal homeomorphism  $H\colon \widehat{\C} \to \widehat{\C}$ (thus, of class $ W^{1,2}(\widehat \C)$) that solves the Beltrami equation
\begin{align*}
\frac{\partial H}{\partial \br z}= \mu \frac{\partial H}{\partial z}.
\end{align*}

The uniqueness part in Theorem \ref{theorem:integrability_david} also holds locally and we have the following factorization theorem.

\begin{theorem}[{\cite[Theorem~20.4.19, p.~565]{AIM09}}]\label{theorem:stoilow}
Let $\Omega\subset \widehat{\C}$ be an open set and $f,g\colon \Omega\to \widehat{\C}$ be David embeddings with
\begin{align*}
\mu_f=\mu_g
\end{align*}
almost everywhere. Then $f\circ g^{-1}$ is a conformal map on $g(\Omega)$.  
\end{theorem}

We note that, unlike the theory of quasiconformal mappings, the inverse of a David homeomorphism is not necessarily a David map. The following simple example is given in \cite[p.~123]{Zak04}. The homeomorphism $\phi\colon\D\to \D$ defined by 
$$\phi(re^{i\theta})= \frac{1}{\log(1/r)+1} e^{i\theta}$$
is a David map but its inverse is not a David map. Indeed $K_\phi(re^{i\theta}) \simeq \log(1/r)$, which is exponentially integrable near $0$, but $K_{\phi^{-1}}(re^{i\theta})\simeq \log(1/r)^2$, which is not exponentially integrable near $0$.

\bigskip 

\subsection{David extensions of circle homeomorphisms}\label{section:preliminaries_extension}
We will need an extension result for David maps, which is a generalization of the well-known extension of Beurling and Ahlfors \cite{BA56}. Let $h\colon \mathbb S^1\to \mathbb S^1$ be an orientation-preserving homeomorphism. We define the \textit{distortion function} of $h$ to be 
\begin{align*}
\rho_h(z,t)= \max\left\{  \frac{|h(e^{2\pi i t}z)-h(z)| }{ |h(e^{-2\pi i t}z)-h(z)| } ,  \frac{|h(e^{-2\pi i t}z)-h(z)| }{ |h(e^{2\pi i t}z)-h(z)| }\right\},
\end{align*}
where $z\in \mathbb S^1$ and $0<t<1/2$. We define the \textit{scalewise distortion function} of $h$ to be  
\begin{align*}
\rho_h(t)= \max_{z\in \mathbb{S}^1}\rho_h(z,t),  
\end{align*}
where $0<t<1/2$. If $\rho_h(t)$ is bounded above, then the function $h$ is a quasisymmetry and has a quasiconformal extension on $\D$, by the theorem of Beurling and Ahlfors. Zakeri observed in \cite[Theorem 3.1]{Zak08}, by applying a result of \cite{CCH96}, that there is a growth condition on $\rho_h(t)$ that is sufficient for a homeomorphism $h$ of the circle to have a David extension in the disk. See also the recent work \cite{KN22} for a stronger extension result.

\begin{theorem}[{\cite[Theorem 3]{CCH96},\cite[Theorem 3.1]{Zak08}}]\label{theorem:extension_david}
Let $h\colon \mathbb S^1\to \mathbb S^1$ be an orientation-preserving homeomorphism and suppose that
\begin{align*}
\rho_h(t) = O(\log(1/t))\quad \textrm{as} \quad t\to 0.
\end{align*}
Then $h$ has an extension to a David homeomorphism $\widetilde h\colon \D\to \D$. 
\end{theorem}

In fact, Zakeri discusses the extensions of homeomorphisms $H$ of $\R$ that commute with the function $z\mapsto z+1$. These homeomorphisms arise as lifts of circle homeomorphisms $h$ under the universal covering map $z\mapsto e^{2\pi i z}$. Then he simply takes the Beurling--Ahlfors extension $\widetilde H$ of $H$ to the upper half-plane and observes that \cite[Theorem 3]{CCH96} implies that under the condition of Theorem \ref{theorem:extension_david}, translated to the homeomorphism $H$ of $\R$, the map $\widetilde H$ is a David map of the upper half-plane. Finally, the extension $\widetilde H$ descends to a David extension $\widetilde h$ of $h$ in $\D$, as it is pointed out in \cite[p.~243]{Zak08}.

\begin{remark}\label{remark:symmetric_extension}
If $h\colon \mathbb S^1\to \mathbb S^1$ is an orientation-preserving homeomorphism that is real-symmetric, i.e., $h(\br z)=\br{h(z)}$ for $z\in \mathbb S^1$, then the David extension of $h$ that is given by Theorem \ref{theorem:extension_david} has the same property. Indeed, let $H$ denote the lift of $h$ to the real line, under the universal cover $z\mapsto e^{2\pi i z}$. The real-symmetry of $h$ is equivalent to the condition $H(1-z)=1-H(z)$ for all $z\in \R$. Let $\widetilde H$ be the Beurling-Ahlfors extension of $H$. 
The Beurling-Ahlfors extension operator is equivariant under precomposition and postcomposition with linear maps of the form $z\mapsto az+b$, $a>0$, $b\in \R$. That is, if $T_1$ and $T_2$ are such linear maps, then the Beurling-Ahlfors extension of $T_2\circ H\circ T_1$ is the map $T_2\circ \widetilde H \circ T_1$; see \cite[p.~246]{Zak08}. In our case, this implies that $\widetilde H(1-z)=1-\widetilde H(z)$ for all $z$ in the upper half-plane. Equivalently, the extension $\widetilde h$ of $h$ to the disk $\D$ is real-symmetric, as claimed. 
\end{remark}

\bigskip

\subsection{Composition with quasiconformal maps}\label{section:composition}
We also discuss the invariance of the David property under quasiconformal maps. Recall that a \textit{quasidisk} is the image of the unit disk under a quasiconformal map of $\widehat{\C}$. A domain $\Omega\subset \widehat{\C}$ is a \textit{John domain} if for each base point $z_0\in \Omega$ there exists a constant $c>0$ such that for each point $z_1\in \Omega$ there exists a simple path $\gamma$ joining $z_0$ to $z_1$ in $\Omega$ with the property that for each point $z$ on the path $\gamma$ we have 
\begin{align*}
\dist_{\sigma}(z, \partial \Omega) \geq c \length_{\sigma}(\gamma|_{[z,z_1]}),
\end{align*}
where $\gamma|_{[z,z_1]}$ denotes the subpath of $\gamma$ whose endpoints are $z$ and $z_1$. Here we are using the spherical distance and spherical length. Intuitively, John domains can have inward cusps but not outward cusps and they are a generalization of a quasidisks, which cannot have any cusps at all.

\begin{prop}\label{prop:david_qc_invariance}
Let $f\colon U\to V$ be a David homeomorphism between open sets $U,V\subset \widehat {\C}$.
\begin{enumerate}[\upshape(i)]
\item If $g\colon V \to \widehat{\C}$ is a quasiconformal embedding then $g\circ f$ is a David map.
\item If $W\subset {\widehat{\C}}$ is an open set and $g\colon W \to U$ is a quasiconformal homeomorphism that extends to a quasiconformal homeomorphism of an open neighborhood of $\br W$ onto an open neighborhood of $\br U$, then $f\circ g$ is a David map.
\item If $W\subset \widehat{\C}$ is an open set, and $g\colon W\to U$ is a non-constant quasiregular map that extends to a quasiregular map in an open neighborhood of $\br W$, then the function $$\mu_{f\circ g}=  \frac{\partial (f\circ g)/ \partial \br z}{\partial (f\circ g)/\partial z} $$
is a David coefficient in $W$. 
\item If $U$ is a quasidisk, $W$ is a John domain, and $g\colon W\to U$ is a quasiconformal homeomorphism, then $f\circ g$ is a David map. 
\end{enumerate}
\end{prop}

Note that (ii) and (iii) imply that if $\mu$ is a David coefficient and $g$ is a quasiconformal or quasiregular map, then the pullback $g^*\mu$ is also a David coefficient. The proposition actually requires the existence of a David map $f$ with $\mu_f=\mu$; however, there is always such a map $f$ given by Theorem \ref{theorem:integrability_david}.

The proof relies partially on \cite[Theorem 5.13]{HK14}, which concerns the composition of Sobolev functions with Sobolev homeomorphisms. We quote below some special instances of the theorem that we will use.

\begin{theorem}[{\cite[Theorem 5.13]{HK14}}]\label{theorem:Hencl_Koskela}
Let $\Omega_1,\Omega_2$ be open subsets of $\C$ and let $F\colon \Omega_1\to \Omega_2$ be a homeomorphism with $F\in W^{1,1}_{\loc}(\Omega_1)$ and $J_F\neq 0$ a.e. We set $K_F= \|DF\|^2/J_F$ a.e. in $\Omega_1$.
\begin{enumerate}[\upshape(i)]
\item If $K_F\in L^{1}_{\loc}(\Omega_1)$ and  $g\in W^{1,2}_{\loc}(\Omega_2)$, then $g\circ F\in W^{1,1}_{\loc}(\Omega_1)$.
\item If $K_F \cdot \|DF\|^{-\varepsilon}\in L_{\loc}^{1/(1-\varepsilon)}(\Omega_1)$ for some $\varepsilon\in (0,1)$ and $g\in W^{1,2-\varepsilon}_{\loc}(\Omega_2)$, then $g\circ F \in W^{1,1}_{\loc}(\Omega_1)$.
\end{enumerate} 
Moreover, we have the usual chain rule:
\begin{align*}
D(g\circ F)(z)=Dg(F(z))\circ   DF(z) 
\end{align*}
for a.e.\ $z\in \Omega_1$.
\end{theorem}

We note that $K_F$ here agrees with the definition of the distortion function given in the beginning of Section \ref{section:david}. Recall that $\|Df\|$ denotes the operator norm of the differential $Df$ of $f$ and the $L^p$ spaces are with respect to the Lebesgue measure of $\C$. 

\begin{proof}[Proof of Proposition \ref{prop:david_qc_invariance}]
(i) First, we reduce the assertion to the case that $U$ and $V$ are planar sets. If $U=V=\widehat{\C}$, we precompose $f$ and we postcompose $g$ with suitable isometries of $\widehat{\C}$ so that $f$ and $g$ fix the point at $\infty$. Thus, we may assume that $f\colon \C\to \C$ is a David map and $g\colon \C \to \C$ is a quasiconformal map. If $U,V$ are strict subsets of $\widehat{\C}$, by precomposing and postcomposing $f$ and $g$ with suitable isometries of $\widehat{\C}$, we can assume that $U,V\subset \C$.

We note that if $f\colon U\to V$ is a David homeomorphism, then $\exp(pK_f)\in L^1_{\loc}(U)$, so $K_f\in L^1_{\loc}(U)$. By Theorem \ref{theorem:Hencl_Koskela} (i), we obtain that if $g\in W^{1,2}_{\loc}(V)$, then $g\circ f \in W^{1,1}_{\loc}(U)$ and
$$D(g\circ f)(z)= Dg(f(z)) \circ Df(z)$$
for a.e.\ $z\in U$. 

Suppose that $g$ is $K$-quasiconformal as in the statement, so that $g\in W^{1,2}_{\loc}(V)$ and  
\begin{align*}
\|Dg(w)\|^2\leq K J_g(w)
\end{align*}
for a.e.\ $w\in V$. Note that $Df(z)$ exists for a.e.\ $z\in U$ and $Dg(w)$ exists for a.e.\ $w\in V$. Since $f^{-1}$ is absolutely continuous, it follows that the set of $z\in U$ such that $Dg(f(z))$ does not exist has measure zero. Hence, for a.e.\ $z\in U$ we have
\begin{align*}
\|D(g\circ f) (z)\|^2& = \|Dg(f(z))\circ Df(z)\|^2 \leq \|Dg(f(z))\|^2 \|Df(z)\|^2 \\
& \leq K K_f(z) J_g(f(z))J_f(z) = KK_f(z) J_{g\circ f}(z). 
\end{align*}
It follows that $K_{g\circ f} \leq KK_f$ a.e.\ and thus $K_{g\circ f}$ is exponentially integrable, as desired.

\bigskip

\noindent
(ii) As in (i), by suitable compositions with isometries of $\widehat{\C}$, we may assume that $U,V,W$ are planar sets. In this case, if $g$ is $K$-quasiconformal, then
$$\frac{\|Dg\|^{2-\varepsilon}}{J_g} \leq  K \|Dg\|^{-\varepsilon}\leq K J_g^{-\varepsilon/2}.$$
It is known that $J_g^{-\delta}$ is locally integrable for a small $\delta>0$; see e.g.\ \cite[Theorem 13.4.2, p.~345]{AIM09}. Hence, $\|Dg\|^{2-\varepsilon}/J_g\in L^{1/(1-\varepsilon)}_{\loc}(W)$. On the other hand, if $f$ is a David homeomorphism, then $f\in W^{1,2-\varepsilon}_{\loc}(U)$ for all $\varepsilon>0$ by \cite[(20.69), p.~557]{AIM09}. By Theorem \ref{theorem:Hencl_Koskela} (ii), it follows that $f\circ g\in W^{1,1}_{\loc}(W)$.

Next, we wish to show that 
\begin{align*}
\int_W \exp(p K_{f\circ g}) \, d\sigma<\infty 
\end{align*}
for some $p>0$. Arguing as in (i), we have $K_{f\circ g} \leq KK_f(g(w))$ for a.e.\ $w\in W$. By changing coordinates, we have
\begin{align*}
\int_W \exp(p K_{f\circ g}) \, d\sigma & \leq \int_W \exp(p K \cdot K_f\circ g)\, d\sigma \\
&= \int_W \exp(p K \cdot K_f(g(w))) J_g^{\sigma}(w) J_{g^{-1}}^{\sigma}( g(w)) \, d\sigma(w)\\
&=\int_U \exp(p K \cdot K_f(z))  J_{g^{-1}}^{\sigma}(z) \, d\sigma(z)\\
&\leq \left(\int_U \exp(qp K \cdot K_f(z)) \, d\sigma(z) \right)^{1/q}  \left(\int_U J_{g^{-1}}^{\sigma}(z)^{q'}\, d\sigma(z) \right)^{1/q'},
\end{align*}
where $1/q +1/q'=1$. The second factor is finite if $q'>1$ is sufficiently close to $1$, since $g^{-1}$ is quasiconformal in a neighborhood of $\br U$ by assumption; see \cite[Theorem 13.4.2, p.~345]{AIM09}. The first factor is also finite for a sufficiently small $p>0$, since $f$ is a David homeomorphism. 

\bigskip
\noindent
(iii) We show, first, that the Beltrami coefficient $$\mu_{f\circ g}=  \frac{\partial (f\circ g)/ \partial \br z}{\partial (f\circ g)/\partial z} $$ is defined a.e. in $W$ and takes values in $\D$. If $w\in W$ is not a critical point of $g$, then there exists a neighborhood $O$ of $w$ in which $g$ is quasiconformal. By (ii) we conclude that $f\circ g$ is a David map in $O$. Therefore, $\mu_{f\circ g}$ is defined in $O$ and takes values in $\D$. Since $g$ is quasiregular and non-constant, it has at most countably many critical points in $W$. It follows that $\mu_{f\circ g}$ and $K_{f\circ g}=(1+|\mu_{f\circ g}|)/({ 1- |\mu_{f\circ g}|})$ are defined a.e. 

Next, we will show that $K_{f\circ g}$ is exponentially integrable. First, we reduce to the case that $g$ is holomorphic. Suppose that $g$ is $K$-quasiregular in a neighborhood $Z$ of $\br W$. By the measurable Riemann mapping theorem, there exists a $K$-quasiconformal homeomorphism $\widetilde g$ of $\widehat{\C}$ such that $\mu_{\widetilde g}=\mu_g \cdot \x_Z$. Moreover, the map $h=g\circ (\widetilde g)^{-1}$ is holomorphic in $\widetilde Z=\widetilde g(Z)$ and we set $\widetilde W=\widetilde g(W)$. We note that $K_{f\circ g}=K_{f\circ h\circ \widetilde g}\leq KK_{f\circ h}(\widetilde g(w))$ as in (ii). The exact same computation of (ii) implies that since 
$$\int_{\widetilde W} J_{(\widetilde g)^{-1}}^{\sigma}(z)^{q'} \, d\sigma(z)<\infty$$
for some $q'>1$ by the quasiconformality of $\widetilde g$, it suffices to show that
$$\int_{\widetilde W} \exp(p K_{f\circ h}) \, d\sigma <\infty$$
for some $p>0$. Thus, we have reduced to the case that $g\colon W\to U$ is holomorphic and has a holomorphic extension to a neighborhood $Z$ of $\br W$.

Let $w_0\in Z$ and let $O(w_0)\subset \br{O(w_0)}\subset Z$ be an open neighborhood of $w_0$. As in (ii) we have
\begin{align*}
\int_{O(w_0)\cap W} \exp(p K_{f\circ g}) \, d\sigma  &\leq  \int_{O(w_0)\cap W} \exp(p  K_f(g(w))) J_g^{\sigma}(w) J_{g}^{\sigma}(w)^{-1} \, d\sigma(w)\\
&\leq \left(\int_{O(w_0)\cap W} \exp(qp  K_f(g(w)))J_g^{\sigma}(w) \, d\sigma(w) \right)^{1/q} \\
&\qquad \qquad \cdot \left(\int_{O(w_0)} J_{g}^{\sigma}(w)^{1-q'}\, d\sigma(w) \right)^{1/q'}\\
&=\left(\int_{g(O(w_0))\cap U} \exp(qp  K_f(z)) \, d\sigma(z) \right)^{1/q}\\
&\qquad\qquad \cdot  \left(\int_{O(w_0)} J_{g}^{\sigma}(w)^{1-q'}\, d\sigma(w) \right)^{1/q'},
\end{align*}
where $1/q+1/q'=1$. For each choice of $q'$ there exists a $p>0$ so that the first factor is finite, since $f$ is a David map on $U$. We claim that the second factor is finite if the neighborhood $O(w_0)$ of $w_0$ is sufficiently small and $q'$ is sufficiently close to $1$.

By using an isometry of $\widehat{\C}$ we assume that $w_0=0$ and we set $O=O(w_0)$. If $O$ is a sufficiently small neighborhood of the origin then the spherical Jacobian and the spherical measure are comparable to the Euclidean ones. Therefore, it suffices to show that 
\begin{align*}
\int_{O} J_g (w)^{-\delta}dw <\infty 
\end{align*}
for a sufficiently small $\delta>0$ and a sufficiently small neighborhood $O$ of the origin. Suppose that $g$ is $m$-to-$1$ at the origin for some $m\geq 1$. Then there exists a neighborhood $O$ of the origin and a conformal map $\phi$ on $O$ with $\phi(0)=0$ such that $g(w)=\phi(w)^m$ for all $w\in O$. We have $J_g=|g'|^2= m^2 |\phi|^{2m-2} |\phi'|^2$. By shrinking the neighborhood $O$, we may have $|\phi'|\simeq 1$ and $|\phi(w)|\simeq |w|$ for all $w\in O$. Hence $J_g(w)\simeq |w|^{2m-2}$ for $w\in O$. Now, we have
\begin{align*}
\int_O J_g(w)^{-\delta} dw \simeq \int_O |w|^{-\delta(2m-2)} dw, 
\end{align*}
which is finite for all small $\delta>0$, as desired. 

Summarizing, we have proved that each point $w_0\in Z$ has a neighborhood $O(w_0)$ in $Z$ such that 
\begin{align*}
\int_{O(w_0)\cap W} \exp(p K_{f\circ g}) \, d\sigma <\infty.
\end{align*} 
Since $\br W$ is a compact subset of $Z$, we can cover it by finitely many neighborhoods $O(w_0)$ of points $w_0 \in \br W$. We conclude that
 \begin{align*}
\int_{W} \exp(p K_{f\circ g}) \, d\sigma <\infty,
\end{align*} 
as desired.

\bigskip
\noindent
(iv) As in (ii), it follows that $f\circ g\in W^{1,1}_{\loc}(W)$. In order to obtain the exponential integrability of $K_{f\circ g}$, we only have to justify that 
\begin{align}\label{prop:uniform}
\int_U J_{g^{-1}}^{\sigma}(z)^{q'}\, d\sigma(z) <\infty
\end{align}
for some $q'>1$, close to $1$. 

In order to prove this, we need the notion of a \textit{uniform domain}. An open set $\Omega\subset \C$ is a uniform domain if there exists a constant $c>0$ such that for any two points $a,b\in \Omega$ there exists a continuum $E\subset \Omega$  that connects them with $\diam(E)\leq c|a-b|$ and the set
$$\bigcup_{x\in E} B(x,c^{-1}\min\{|x-a|,|x-b|\}),$$
which is called a $c$-cigar, is contained in $\Omega$. A bounded uniform domain is also a John domain. Other uniform domains include quasidisks and annuli bounded by two quasicircles. We direct the reader to the work of Martio and V\"ais\"al\"a \cite{MV88} and the references therein for more background on uniform and John domains.

The desired condition \eqref{prop:uniform} follows from a result of Martio and V\"ais\"al\"a \cite[Theorem 2.16]{MV88}, which implies that a quasiconformal map from a \textit{bounded} {uniform domain} onto a John domain has Jacobian lying in $L^{1+\varepsilon}$ for some small $\varepsilon>0$. The quoted result does not apply immediately to $g^{-1}$, since the quasidisk $U$, which is a uniform domain, is not necessarily bounded as a subset of the plane.   

In order to apply the result, we first precompose and postcompose $g$ with suitable isometries of $\widehat{\C}$ so that the point at $\infty$ lies in $W$ and $U$, and $g$ fixes $\infty$. Next, we remove from $W$ a closed ball $B\subset W$ that contains $\infty$. The set $W\setminus B$ is still a John domain, with a constant that is possibly different from the constant of $W$. We note that $W\setminus B$ is a John domain even if we use the Euclidean rather than the spherical metric in the definition of a John domain. Moreover, $U\setminus g(B)$ is a uniform domain since it is an annulus bounded by two quasicircles. Now the map $g^{-1}\colon U\setminus g(B) \to W\setminus B$ satisfies the assumptions of \cite[Theorem 2.16]{MV88}, so $J_{g^{-1}} \in L^{1+\varepsilon} (U\setminus g(B))$ for some $\varepsilon>0$. This implies that 
\begin{align*}
\int_{U\setminus g(B)} J_{g^{-1}}^{\sigma}(z)^{1+\varepsilon}\, d\sigma(z) <\infty
\end{align*}
since the Euclidean and spherical measures are comparable in the bounded set $U\setminus g(B)$. Finally, we also have
\begin{align*}
\int_{g(B)} J_{g^{-1}}^{\sigma}(z)^{1+\varepsilon}\, d\sigma(z) <\infty
\end{align*}
for a possibly smaller $\varepsilon>0$ by the local regularity of the Jacobian of a quasiconformal map; see  \cite[Theorem 13.4.2, p.~345]{AIM09}.
\end{proof}

\bigskip

\subsection{David maps and removability}\label{david_removability_subsec}

A compact set $E\subset \widehat{\C}$ is \textit{conformally removable} if every homeomorphism $f\colon \widehat{\C}\to \widehat{\C}$ that is conformal on $\widehat{\C} \setminus E$ is a M\"obius transformation. We also say that $E$ is \textit{locally} conformally removable if for every open set $U\subset \widehat{\C}$, \textit{not necessarily containing $E$}, every homeomorphic embedding $f\colon U\to \widehat{\C}$ that is conformal on $U\setminus E$ is conformal on $U$. It is an open problem whether every removable set is locally removable. 

While quasiconformal maps preserve the quality of conformal removability, it is not clear whether David homeomorphisms do so. However, they interact with a stronger notion of removability; namely, \textit{removability for the Sobolev space $W^{1,1}$}. A compact set $E\subset \widehat{\C}$ is removable for $W^{1,1}$ functions if every continuous function $f\colon \widehat{\C}\to \R$ that lies in $W^{1,1}(\widehat{\C}\setminus E)$ lies actually in $W^{1,1}(\widehat{\C})$. Equivalently, $E$ is removable for $W^{1,1}$ functions if for every open set $U$, \textit{not necessarily containing $E$}, every continuous function $f\colon U \to \R$ that lies in $W^{1,1}(U\setminus E)$ lies in $W^{1,1}(U)$ \cite[Introduction]{Nta20a}. To see the necessity, if $f$ is continuous in $U$ and lies in $W^{1,1}(U\setminus E)$ and $\phi$ is a smooth function compactly supported in $U$, then we may consider a smooth function $\psi$ that is compactly supported in $U$ and is identically equal to $1$ on the support of $\phi$. Since $f\cdot \psi$ is continuous in $\widehat \C$ and lies in $W^{1,1}(\widehat \C\setminus E)$, we have $f\cdot \psi\in W^{1,1}(\widehat \C)$ by the removability of $E$. Working in local coordinates that contain the support of $\psi$, we see that for $i=1,2$, 
$$\int \partial_i f \cdot \phi=\int \partial_i(f\psi) \cdot \phi= - \int f \psi \partial_i \phi= - \int f \cdot \partial_i\phi.$$ This shows that $f\in W^{1,1}(U)$, as desired. Hence, in the case of removability for the Sobolev space $W^{1,1}$ the notions of removability and local removability agree.

\begin{remark}Similarly one can define sets that are removable for $W^{1,p}$ functions, where $p\geq 1$. There are examples of sets that are removable for some values of $p$ and non-removable for other values. One example is the Sierpi\'nski gasket, which is removable if and only if $p>2$ \cites{Nta19,Nta20a}. For further examples see \cite{KRZ17}*{Lemma 4.4}. Such examples are rare in the literature though, because of the technicalities of the constructions. We note that $W^{1,1}$-removable sets are also $W^{1,2}$-removable, and $W^{1,2}$-removable sets have measure zero and are also conformally removable (by Weyl's lemma). It seems plausible that there exists a set that is $W^{1,2}$-removable (and thus conformally removable) and not $W^{1,1}$-removable, but we have not been able to locate such an example in the literature.
\end{remark}

\begin{theorem}[Conformal removability]\label{theorem:w11_removable}
Suppose that $E\subset \widehat{\C}$ is a compact set that is removable for $W^{1,1}$ functions and $f\colon \widehat{\C}\to \widehat{\C}$ is a David homeomorphism. Then $f(E)$ is locally conformally removable.  
\end{theorem}

We will prove this using the following auxiliary result.
\begin{lemma}[David removability]\label{lemma:david_removable}
Suppose that $E\subset \widehat{\C}$ is a compact set that is removable for $W^{1,1}$ functions, $U\subset \widehat{\C}$ is an open set, and $g\colon U \to \widehat{\C}$ is a homeomorphic embedding that is a David  map in $U\setminus E$. Then $g$ is a David map in $U$.
\end{lemma}

Note that the set $E$ is not necessarily contained in $U$.

\begin{proof}
By precomposing and postcomposing $g$ with suitable isometries of $\widehat{\C}$, we may assume that $E$ and $g(E)$  do not contain the point at $\infty$ and $g(\infty)=\infty$. Since $g$ is a David map on $U\setminus E$, by definition, we have
$$\int_{U\setminus E}\exp(pK_g)\, d\sigma <\infty $$
for some $p>0$. By \cite[Theorem 1.3]{Nta19}, the set $E$ must have measure zero, since it is removable for $W^{1,1}$ functions.  Hence, we have
$$\int_{U}\exp(pK_g)\, d\sigma <\infty.$$
It suffices to show that $g\in W^{1,1}_{\loc}(U)$.

We claim that if $V$ is a bounded open subset of $U$, then $\|Dg\|\in L^1(V\setminus E)$.  We have 
\begin{align*}
\|Dg(z)\|= K_g(z)^{1/2} J_g(z)^{1/2}
\end{align*} 
for a.e.\ $z\in V\setminus E$. Note that $K_g\in L^1(V\setminus E)$ since $\exp(pK_g)\in L^1(V)$. Moreover, $J_g\in L^1(V\setminus E)$ since by the change of coordinates formula \eqref{david_change_of_variables} (expressed in Euclidean coordinates) we have
\begin{align*}
\int_{V\setminus E} J_g(z) \, dz = \area(g(V\setminus E))<\infty.
\end{align*}
Here we used the normalization $g(\infty)=\infty$. It follows that $\|Dg\|\in L^1(V\setminus E)$ as desired.

Summarizing, if $V$ is a bounded open subset of $U$ containing $U\cap E$, then the map $g$ lies in $W^{1,1}(V\setminus E)$ and is continuous on $V$. Since $E$ is removable for the space $W^{1,1}$ by assumption, we have $g\in W^{1,1}(V)$. We conclude that $g\in W^{1,1}_{\loc}(U)$, as desired.
\end{proof}

\begin{remark}In the proof we used \cite[Theorem 1.3]{Nta19}, which asserts that if a compact set $E$ is removable for $W^{1,p}$ functions, where $1\leq p<\infty$, then it must have measure zero. Without invoking this result, we note that the non-removability of sets of positive measure for conformal maps (which can be proved using the Measurable Riemann Mapping Theorem) does not imply in a straightforward way the non-removability of sets of positive measure for $W^{1,2}$ and $W^{1,1}$ functions. 
\end{remark}

\begin{proof}[Proof of Theorem \ref{theorem:w11_removable}]
Let $U\subset \widehat{\C}$ be an open set and $h \colon U\to \widehat{\C}$ be a homeomorphic embedding that is conformal in $U\setminus f(E)$. Our goal is to show that $h$ is conformal on $U$. Consider the map $g=h\circ f$.  Since $f$ is a David map and $h$ is conformal on $U\setminus f(E)$, by Proposition \ref{prop:david_qc_invariance} (i) it follows that $g$ is a David map on $f^{-1}(U)\setminus E$. By Lemma \ref{lemma:david_removable} we conclude that $g$ is a David map on $f^{-1}(U)$.

Note $\mu_g=\mu_f$ on $f^{-1}(U)\setminus E$ because $h$ is conformal. Since $E$ is removable for $W^{1,1}$ functions, it must have measure zero by  \cite[Theorem 1.3]{Nta19}. Therefore, $\mu_{g}=\mu_f$ a.e.  Finally, the factorization Theorem \ref{theorem:stoilow} implies that $h=g\circ f^{-1}$ is conformal, as desired.
\end{proof}

Boundaries of John domains are removable for $W^{1,1}$ functions. This was proved by Jones and Smirnov  \cite[Theorem 4]{JS00}. In fact, something stronger is true:
\begin{theorem}\label{theorem:john_union_removable}
Let $\{\Omega_{i}\}_{i\in I}$ be a collection of finitely many, disjoint John domains in $\widehat{\C}$. Then $\bigcup_{i\in I} \partial \Omega_i$ is removable for $W^{1,1}$ functions.
\end{theorem}
A similar statement is discussed without proof in \cite[p.~265]{JS00}. We give a proof based on the following result from \cite[Proposition 5.3]{Nta20a}. We state a slightly modified version that is proved in the same way and is sufficient for our purposes.
\begin{prop}\label{prop:john_absolute_continuity}
Let $\Omega\subset \widehat\C$ be a John domain and let $f\colon \br \Omega \to \R$ be a continuous function lying in $W^{1,1}(\Omega)$. Then 
$$m_1(f(L\cap \partial \Omega)) =0$$
for a.e.\ line $L$ parallel to a fixed direction.  
\end{prop}
Here $m_1$ denotes the $1$-dimensional Lebesgue measure and the lines $L$ are considered as subsets of $\widehat{\C}$ through stereographic projection.

\begin{proof}[Proof of Theorem \ref{theorem:john_union_removable}]
Let $f\colon \widehat{\C}\to \R$ be a continuous function lying in the space $W^{1,1}(\widehat{\C} \setminus \bigcup_{i\in I} \partial \Omega_i)$. By a known characterization of Sobolev spaces \cite[Theorem 2.1.4]{Zie89}, for each $i\in I$, $f|_L$ is absolutely continuous on $L\cap { \Omega_i}$ for almost every line $L$, in the sense that $f$ maps sets of linear measure zero to sets of linear measure zero. Proposition \ref{prop:john_absolute_continuity} implies that for each $i\in I$, $f|_L$ is absolutely continuous on $L\cap \br{ \Omega_i}$ for almost every line $L$. Hence, $f$ is absolutely continuous on almost every line. It follows that $f\in W^{1,1}(\widehat{\C})$, since boundaries of John domains have measure zero. 
\end{proof}

By combining \cite[Theorem 4]{JS00}  with Theorem \ref{theorem:w11_removable}, we have the following result.

\begin{theorem}\label{theorem:john_removable}
Suppose that $\Omega \subset {\widehat{\C}}$ is a John domain and $f\colon \widehat{\C}\to \widehat{\C}$ is a David homeomorphism. Then, $f(\partial \Omega)$ is conformally removable. 
\end{theorem}

\bigskip

\section{Expansive covering maps of the circle}\label{section:expansive}
Let $f,g\colon \mathbb{S}^1\to \mathbb{S}^1$ be covering maps of degree $d\geq 2$, having the same orientation. Then {under some expansion assumptions} there exists an orientation-preserving homeomorphism $h\colon \mathbb S^1\to \mathbb S^1$ that conjugates the map $f$ to $g$. Our goal, under suitable assumptions on the covering maps $f$ and $g$, is to extend  the conjugating homeomorphism $h$ to a David self-map of the unit disk. This is the content of Section \ref{section:main_theorem}.  The conditions on $f$ and $g$ will be in terms of a Markov partition associated to $f$ and $g$. In particular, we will give conditions in the special, but very interesting case that one of the maps is $z\mapsto z^d$ or $z\mapsto \br z^d$. In this section we discuss the notion of an expansive covering map of the circle and the notion of a Markov partition associated to such a map. 

If $a,b\in \mathbb{S}^1$, we denote by $\arc{[a,b]}$ and $\arc{(a,b)}$ the closed and open arcs, respectively, from $a$ to $b$ in the positive orientation. The arc $\arc{(b,a)}$, for example, is the complementary arc of $\arc{[a,b]}$. We also denote the arc $\arc{(a,b)}$ by $\inter{\arc{[a,b]}} $. We say that two non-overlapping arcs $I,J\subset \mathbb{S}^1$ are \textit{adjacent} if they share an endpoint. 

\begin{definition}\label{definition:markov_partition}
A \textit{Markov partition} associated to a covering map $f\colon \mathbb{S}^1\to \mathbb{S}^1$ is a covering of the unit circle by closed arcs $A_k=\arc{[a_k, a_{k+1}]}$, $k\in \{0,\dots, r\}$, $r\geq 1$, that have disjoint interiors and satisfy the following conditions.
\begin{enumerate}[(i)]
\item The map $f_k=f|_{\inter{A_k}}$ is injective for $k\in \{0,\dots,r\}$.
\item If $f(\inter{A_k})\cap \inter{A_j}\neq \emptyset$ for some $k,j\in \{0,\dots,r\}$, then $\inter{A_j}\subset f(\inter{A_k})$. 
\item The set $\{a_0,\dots,a_r\}$ is invariant under $f$.
\end{enumerate}    
We denote the above Markov partition by $\mathcal P(f;\{a_0,\dots,a_r\})$. 
\end{definition}

Note that by definition the points $a_0,\dots,a_r$ are ordered in the positive orientation if $r\geq 2$; if $r=1$, there is no natural order. Moreover, (ii) and (iii) are equivalent under condition (i). 

We provide some more definitions. Let $f\colon \mathbb{S}^1\to \mathbb{S}^1$ be a covering map and consider a Markov partition $\mathcal P=\mathcal P(f;\{a_0,\dots,a_r\})$. We can associate a matrix $B=(b_{kj})_{k,j=0}^r$ to $\mathcal P$ so that $b_{kj}=1$ if $f_k(A_k)\supset A_j$ and $b_{kj}=0$ otherwise. In the case $b_{kj}=1$ we define $A_{kj}=f_k^{-1}(A_j)$. If $w=(j_1,\dots,j_n)\in \{0,\dots,r\}^n$, $n\in \N$, $k\in \{0,\dots,r\}$, and once $A_w$ has been defined, we define $A_{kw}=f_k^{-1}(A_w)$ whenever $b_{kj_1}=1$. A \textit{word} $w=(j_1,\dots,j_n)\in \{0,\dots,r\}^n$, $n\in \N$, is \textit{admissible (for the Markov partition $\mathcal P$)} if $b_{j_1j_2} =\dots=b_{j_{n-1}j_n}=1$. We also define $A_w=\emptyset$ if $w$ is not admissible. The \textit{length} of a word $w=(j_1,\dots,j_n)\in \{0,\dots,r\}^n$ is defined to be $|w|=n$. It follows from properties (i) and (ii) that for each $n\in \N$ the arcs $A_w$, where $|w|=n$, have disjoint interiors and their union is equal to $\mathbb{S}^1$. Inductively, we have $A_{wj}\subset A_w$ for all admissible words $w$ and $j\in \{0,\dots,r\}$. If $A_{wj}$ is non-empty, we say that $A_{wj}$ is a \textit{child} of $A_w$ and $A_w$ is the \textit{parent} of $A_{wj}$. Thus, $A_w$ has at most $r+1$ children. We direct the reader to \cite[Section 19.14]{Lyu20} for more background on Markov partitions.

\bigskip

\begin{definition}\label{definition:expansive}
A continuous map $f\colon \mathbb{S}^1\to \mathbb{S}^1$ is called \textit{expansive} if there exists a constant $\delta>0$ such that for any $a,b\in \mathbb{S}^1$ with $a\neq b$ we have $|f^{\circ n}(a)-f^{\circ n}(b)|>\delta$ for some $n\in \N\cup\{0\}$.
\end{definition}

We now list some important properties of expansive maps of $\mathbb S^1$. Let $f\colon \mathbb{S}^1\to \mathbb{S}^1$ be an expansive covering map.
\begin{enumerate}[($E$\upshape1)]
\item\label{lemma:expansive_isolated}
For all $n\in \N$ the map $f^{\circ n}$ has finitely many fixed points.
\item\label{corollary:diameters}
Let  $\mathcal P(f;\{a_0,\dots,a_r\})$ be a Markov partition. Then
\begin{align*}
\lim_{n\to\infty}\max\{\diam{A_w}: |w|=n, \, w\,\, \textup{admissible} \} =0.
\end{align*} 
\item\label{theorem:expansive_conjugate}Suppose that the degree of $f$ is $d\geq 2$. Then there exists an orientation-preserving homeomorphism $h\colon \mathbb{S}^1\to \mathbb{S}^1$ that conjugates $f$ to either the map $g(z)=\br z^d$ or the map $g(z)=z^d$, depending on whether $f$ is orien\-tation-reversing or orientation-preserving, respectively. Moreover, $h$ is unique up to rotation by a $(d+1)$-st root of unity if $g(z)=\br z^d$ or a $(d-1)$-st root of unity if $g(z)=z^d$.
\end{enumerate}

Property \ref{corollary:diameters} can be proved easily using \cite[Theorem 3.6.1, p.~143]{PU11}. Property \ref{theorem:expansive_conjugate} is a consequence of \ref{corollary:diameters}. A more general statement for expansive self-maps of a compact manifold can be found in \cite[Property $(2')$, p.~99]{CR80}. A refined version of property \ref{theorem:expansive_conjugate} that we will need for our considerations is the following lemma. Its proof is straightforward, based on property \ref{corollary:diameters}, and is omitted.

\begin{lemma}\label{lemma:expansive_conjugate_fg}
Let $f,g\colon \mathbb S^1\to \mathbb S^1 $ be expansive covering maps of the same orientation, and $\mathcal P(f;\{a_0,\dots,a_r\})$, $\mathcal P(g;\{b_0,\dots,b_r\})$  be  Markov partitions. Consider the map $h\colon \{a_0,\dots,a_r\} \to \{b_0,\dots,b_r\}$ defined by $h(a_k)=b_k$ for $k\in \{0,\dots,r\}$ and suppose that $h$ conjugates the map $f$ to $g$ on the set $\{a_0,\dots,a_r\}$, i.e., 
\begin{align*}
h(f(a_k))=g(b_k)
\end{align*}
for $k\in \{0,\dots,r\}$. Then $h$ has an extension to an orientation-preserving homeomorphism of $\mathbb S^1$ that conjugates $f$ to $g$ on $\mathbb S^1$.
\end{lemma}

\begin{remark}
The conjugacy assumption on the map $h$ defined in the statement of Lemma~\ref{lemma:expansive_conjugate_fg} implies that $f$ and $g$ have the same degree.
\end{remark}

Suppose that $a$ is a \textit{periodic point} of a covering map $f\colon \mathbb S^1\to \mathbb S^1$. That is, there exists a minimal $n\in \N$ such that $f^{\circ n}(a)=a$. The number $n$ is called the \textit{period} of $a$. We note that if $f^{\circ n}$ is orientation-preserving, then it maps an arc of the form $\arc{(z_1,a)}$ to an arc of the form $\arc{(z_2,a)}$. We say that $f^{\circ n}$ is the \textit{first orientation-preserving return map to the periodic point $a$} and $n$ is the \textit{orientation-preserving period of $a$}. If $f^{\circ n}$ is orientation-reversing, then it maps  an arc of the form $\arc{(z_1,a)}$ to an arc of the form $\arc{(a,z_2)}$. In the latter case, $f^{\circ 2n}$ is orientation-preserving, it maps an arc of the form $\arc{(z_1,a)}$ to an arc of the form $\arc{(z_3,a)}$ fixing $a$, and $2n$ is the smallest integer with that property. In this case $f^{\circ 2n}$ is the \textit{first orientation-preserving return map to the periodic point $a$} and $2n$ is the \textit{orientation-preserving period of $a$}. We denote by $f_a$ the first orientation-preserving return map to $a$. In what follows we suppress the term ``orientation-preserving'' and the term ``first return map'' always refers to the first orientation-preserving return map. 

We define the \textit{one-sided multipliers} $\lambda(a^+), \lambda(a^-)$ of $f$ at $a$ to be the one-sided derivatives of the first return maps, if they exist:
\begin{align*}
\lambda(a^+)&= f_a'(a^+)= \lim_{\substack{z\to a\\ \tiny{z\in \arc{(a,z_0)}}}} \frac{f_a(z)-a}{z-a}=\lim_{z\to a^+} \frac{f_a(z)-a}{z-a} \quad\textrm{and}\\
\lambda(a^-)&=f_a'(a^-)=  \lim_{\substack{z\to a\\ \tiny{z\in \arc{(z_0,a)}}}} \frac{f_a(z)-a}{z-a}=\lim_{z\to a^-} \frac{f_a(z)-a}{z-a},
\end{align*} 
where $z_0\neq a$ is any point on $\mathbb{S}^1$. Observe that $\lambda(a^+),\lambda(a^-)$ are always non-negative real numbers, since $f_a$ maps the circle to itself with positive orientation. 

We extend this definition to preperiodic points. If $a$ is a preperiodic point of $f$, then there exists a minimal $m\in \N$ such that $f^{\circ m}(a)$ is periodic. We define 
$$\lambda(a^{\pm})= \lambda (f^{\circ m}(a)^{\pm})$$
if $f^{\circ m}$ is orientation-preserving and
$$\lambda(a^{\pm})= \lambda (f^{\circ m}(a)^{\mp})$$
if $f^{\circ m}$ is orientation-reversing. We also define the orientation-preserving period of the preperiodic point $a$ to be equal to the orientation-preserving period of the periodic point $f^{\circ m}(a)$.

If $f$ is expansive, we obtain some extra information about the one-sided multipliers.

\begin{enumerate}[($E$\upshape 4)]
\item\label{lemma:expansive_multiplier}
Suppose that $a$ is a periodic point of $f$. If the one-sided multiplier $\lambda(a^\pm)$ exists, then $\lambda(a^\pm)\geq 1$.
\end{enumerate}
Indeed, if the conclusion of the statement were not true, then some orbits would be attracted to the periodic point $a$ and this would contradict the expansivity of $f$.

Next, we give two general classes of expansive maps of the circle. We denote by $\D^*$ the complement of the closed unit disk $\overline{\D}$ in the Riemann sphere $\widehat{\C}$. For a Euclidean circle $C$ in the plane, the bounded complementary component of $C$ will be denoted by $\Int{C}$. This is not to be confused with the interior of an arc $I$ of $\mathbb{S}^1$, which is denoted by $\inter{I}$. The distinction will be clear from the context. 

\begin{example}[Blaschke products]\label{example:expansive_blaschke}
Consider a Blaschke product $$B(z)=e^{i\theta}\prod_{i=1}^d \frac{z-c_i}{1-\overline{c_i}z},$$ where $\theta\in\R$, $d\geq 2$, and $c_1,\cdots, c_d\in\D$. Then, $B\colon \D\to\D$ and $B\colon \D^*\to\D^*$ are branched coverings of degree $d$, and $B:\mathbb{S}^1\to\mathbb{S}^1$ is a degree $d$ covering. 

Suppose further that $B$ has a parabolic fixed point at $1\in\mathbb{S}^1$. By the theory of parabolic fixed points \cite[\S 10]{Mil06}, there exists a basin of attraction of the point $1$, i.e., an open set of points whose iterates under $B$ converge to $1$. Since the action of $B$ in $\D$ is conjugate under the map $z\mapsto 1/\br z$ to the action of $B$ in $\D^*$, it follows that the basin of attraction is symmetric with respect to the unit circle, and in particular it intersects both $\D$ and $\D^*$. Then, by the Denjoy--Wolff Theorem \cite[Theorem 5.4]{Mil06}, we conclude that the successive iterates $B^{\circ n}$ converge, uniformly on compact subsets of $\D$ (respectively, on $\D^*$), to the constant function $z\mapsto 1$. Note that if $B''(1)=0$, then there are at least two attracting directions to the parabolic fixed point $1$, and hence at least two immediate basins of attraction of $1$. Hence, if $B''(1)=0$, then the Fatou set of $B$ must have at least two components, so its Julia set must be equal to $\mathbb{S}^1$. We remark that if $B''(1)\neq 0$, then the Julia set of $B$ would be a Cantor set contained in $\mathbb S^1$.

According to the Riemann-Hurwitz formula, $B$ has $(d-1)$ critical points (counted with multiplicity) in each of $\D$ and $\D^*$. Thus, $B$ has no critical point on the Julia set. By \cite[Theorem~4]{DU91}, it follows that if $B''(1)=0$, then the map $B\colon \mathbb{S}^1\to\mathbb{S}^1$ is expansive.

A particular example of such a Blaschke product is
$$
B(z)=\frac{(d+1)z^d+(d-1)}{(d-1)z^d+(d+1)},\ d\geq 2.
$$
The point 1 is a parabolic fixed point of $B$ and $B''(1)=0$.
\end{example}

\begin{example}[Circle reflections]\label{example:expansive_reflection}
Consider ordered points $a_0=a_{d+1},a_1,\dots$, $a_{d}$, $d\geq 2$, on the circle $\mathbb{S}^1$. Let $f_k$ be the reflection along the circle $C_k$ that is orthogonal to the unit circle at the points $a_k$ and $a_{k+1}$. Set $A_k=\arc{[a_k, a_{k+1}]}\subset\mathbb{S}^1$, $k\in \{0,\dots,d\}$. We will assume that the points $a_0,\dots,a_{d}$ satisfy the following condition.
\begin{align}
\tag{$\star$} \length(\arc{(a_k,a_{k+1})})<\pi \quad \textrm{for} \quad k\in \{0,\dots,d\}.
\end{align}
We now define a map $f\colon \mathbb{S}^1\to\mathbb{S}^1$ by $z\mapsto f_k(z)$ for  $z\in A_k$. Since $f(A_k)=\bigcup_{k'\neq k} A_{k'}$, it follows that $f$ is an orientation-reversing covering map of degree $d$ with fixed point set $\{a_0,\cdots, a_{d}\}$. Moreover, condition~($\star$) guarantees that $\vert f'\vert\geq 1$ on $\mathbb{S}^1$ with equality precisely at the fixed points $a_k$. It will follow from Lemma~\ref{lem:expansive_map} below that $f$ is an expansive map.
\end{example}

\begin{example}[Hybrid]\label{example:expansive_hybrid}
Let $d\in\N,\ d\geq2$, and $a_j=e^{\frac{2\pi j}{d+1}},\ j\in \{0,1,\dots, d\}$, be the fixed points of the map $z\mapsto\bar z^d$. We select some adjacent pairs 
$$
\{a_{j_1},a_{j_1+1}\},\dots, \{a_{j_m},a_{j_m+1}\}, 
$$
where the indices are taken modulo $d+1$. For each selected pair $\{a_{j_k},a_{j_k+1}\},\ k\in \{1,\dots,m\}$, let $C_k$ be the circle that is orthogonal to the unit circle at the points $a_{j_k}$ and $a_{j_k+1}$. If $A_j=\arc{[a_{j}, a_{j+1}]}\subset\mathbb{S}^1,\ j\in \{0,1,\dots, d\}$, we can define a map $f\colon \mathbb{S}^1\to\mathbb{S}^1$ by setting it to be the reflection on $A_{j_k},\ k\in \{1,\dots,m\}$, in the circle $C_k$, and on the remaining arcs $A_j$ to be the map $z\mapsto\bar z^d$. The map $f$ is an orientation-reversing covering map of degree $d$ whose fixed point set is $\{a_0,\cdots, a_{d}\}$. Moreover, we have that the map $f$ is piecewise $C^1$, and its set of non-differentiability is contained in $\{a_{j_1}, a_{j_1+1},\dots, a_{j_k}, a_{j_k+1}\}$. The one-sided derivatives of $f$ exist at each point of non-differentiability and the map $f$ is $C^1$ on each closed arc $A_j,\ j\in \{0,1,\dots, d\}$. Moreover, $\vert f'\vert\geq 1$ on $A_j,\ j\in \{0,1,\dots, d\}$, with equality only at the endpoints of $A_{j_k},\ k\in \{1,\dots, m\}$. It will follow from Lemma~\ref{lem:expansive_map} below that $f$ is an expansive map.
\end{example}

\begin{lemma}\label{lem:expansive_map}
Let $f:\mathbb{S}^1\to\mathbb{S}^1$ be a piecewise $C^1$ covering map that has one-sided derivatives at points of non-differentiability, it is $C^1$ on the closure of each complementary arc of the non-differentiability set, and such that
\begin{enumerate}[\upshape(i)]
\item  $f$ has finitely many fixed points, and each non-differentiability point is fixed,

\item $\vert f'\vert\geq 1$ on $\mathbb{S}^1$, where at points of non-differentiability we assume that the magnitudes of both one-sided derivatives are $\geq1$, and

\item if $\zeta\in \mathbb S^1$ is not a fixed point of $f$, then $|f'(\zeta)|>1$.
\end{enumerate}
Then, $f$ is expansive.
\end{lemma}

\begin{proof}
We only provide a sketch of the proof, since it is elementary. We denote by $d\colon \mathbb S^1\times \mathbb S^1\to [0,\pi]$ the length metric in $\mathbb S^1$ and by $X$ the finite set of fixed points of $f$. First, using the assumption (iii), one can show that there exists $\delta>0$ such that if $z_0\in\mathbb{S}^1$ and $d(f^{\circ n}(z_0),\zeta)<\delta$ for all $n\geq 0$, and some $\zeta\in X$, then $z_0=\zeta$. In other words, an orbit of a point $z_0$ that is not a fixed point cannot stay for all times near a fixed point. Now, consider two points $z_1,z_2\in \mathbb S^1$ such that $d(f^{\circ n}(z_1),f^{\circ n}(z_2))$ remains small, say less than $\eta$, for all $n\geq 0$. There are two main cases. 

If $f^{\circ n_0}(z_1)$ is a fixed point of $f$ for some $n_0\geq 0$, then $f^{\circ n}(z_2)$ remains near a fixed point of $f$ for all times $n\geq n_0$. By the previous, we must have $f^{\circ n_0}(z_1)=f^{\circ n_0}(z_2)$. Without loss of generality $n_0\geq 1$. By assumption, $$d(f^{\circ (n_0-1)}(z_1),f^{\circ (n_0-1)}(z_2))<\eta.$$ If $\eta$ is chosen suitably, then by (ii) $f$ is injective on an arc  containing $f^{\circ (n_0-1)}(z_1)$ and $f^{\circ (n_0-1)}(z_2)$. Hence, $f^{\circ (n_0-1)}(z_1)=f^{\circ (n_0-1)}(z_2)$. Inductively, $z_1=z_2$. The same conclusion holds if $f^{\circ n_0}(z_2)$ is a fixed point for some $n_0$. 

The other case is that $f^{\circ n}(z_1)$ and $f^{\circ n}(z_2)$ are not fixed points for all $n\geq 0.$ Then, they actually have to stay away from the $\delta$-neighborhood of the fixed point set $X$ infinitely often, by the first paragraph. Property (iii) implies that $d(f^{\circ (n+1)}(z_1),f^{\circ (n+1)}(z_2))\geq \lambda d(f^{\circ n}(z_1),f^{\circ n}(z_2))$ for some definite factor $\lambda>1$, whenever $f^{\circ n}(z_1)$ and $f^{\circ n}(z_2)$ are not in the $\delta$-neighborhood of $X$. If either $f^{\circ n}(z_1)$ or $f^{\circ n}(z_2)$ is $\delta$-close to $X$ then we still have $d(f^{\circ (n+1)}(z_1),f^{\circ (n+1)}(z_2))\geq  d(f^{\circ n}(z_1),f^{\circ n}(z_2))$ by (ii). Since the first alternative occurs infinitely often, we obtain a contradiction, unless $z_1=z_2$.
\end{proof}
 
\begin{remark}\label{rem:C2_1}
For piecewise $C^2$ maps, assumptions of Lemma~\ref{lem:expansive_map} can probably be
relaxed. Namely, it should be sufficient to assume that all periodic points are topologically repelling on both sides (compare \cite{Man85}).
Moreover, if the map has negative Schwarzian derivative (for piecewise $C^3$ maps)
this assumption is satisfied automatically  by Singer's Theorem (see \cite{CE80}),
except at points of non-differentiability,
and the generalization of the lemma to this setting looks straightforward. 
\end{remark}

We end this section with a known result for hyperbolic Blaschke products that we will use later.
\begin{lemma}\label{lemma:blaschke_qs}
Let $B$ be a holomorphic (resp.\ anti-holomorphic) Blaschke product of degree $d\geq 2$ that has an attracting fixed point in $\D$. Then there exists a quasisymmetric map $h\colon \mathbb S^1\to \mathbb S^1$ that conjugates the map $B$ to the map $f(z)=z^d$ (resp.\  $f(z)=\br z^d$).
\end{lemma}

Here, an anti-holomorphic Blaschke product is the complex conjugate of a holomorphic Blaschke product.

\begin{proof}
The proof is based on the fact that a conjugacy $h$ between two expanding maps of $\mathbb S^1$ that are of class $C^{2}$ is quasisymmetric; see \cite[Proposition 19.64]{Lyu20}. 

Since the Blaschke product $B$ has an attracting fixed point in $\D$, it follows that $B$ is hyperbolic, so $B^{\circ m}$ is expanding for some $m\in \N$ (see \cite[Lemma 2.1]{CG93}). By property \ref{theorem:expansive_conjugate}, there exists an orientation-preserving conjugacy $h$ between $B$ and $f$. Note that $h$ also conjugates $B^{\circ m}$ and $f^{\circ m}$. By the above fact, the conjugacy between $B^{\circ m}$ and $f^{\circ m}$ is quasisymmetric.
\end{proof}

\bigskip 

\section{David extensions of dynamical homeomorphisms of the circle}\label{section:main_theorem}
Let $f,g\colon \mathbb{S}^1\to \mathbb{S}^1$ be covering maps of degree $d\geq 2$ having the same orientation and consider Markov partitions $\mathcal P(f;\{a_0,\dots,a_r\})$ and $\mathcal P(g;\{b_0,\dots,b_r\})$. Let $h\colon \{a_0,\dots,a_r\} \to \{b_0,\dots,b_r\}$ be the map defined by $h(a_k)=b_k$ for $k\in \{0,\dots,r\}$ and suppose that $h$ conjugates the map $f$ to $g$ on the set $\{a_0,\dots,a_r\}$. If the maps $f$ and $g$ are expansive, then we know from Lemma \ref{lemma:expansive_conjugate_fg} that $h$ extends to an orientation-preserving homeomorphism $h\colon \mathbb{S}^1\to \mathbb{S}^1$ that conjugates the map $f$ to $g$. That is, $h( f(z) )=g( h(z))$ for all $z\in \mathbb{S}^1$. In this section, under some assumptions, we will extend the map $h$ to a homeomorphism of the unit disk $\D$ that is a quasiconformal or a David map.  Our main extension result is Theorem \ref{theorem:extension_generalization} and its proof will be given in Subsection \ref{section:distortion}. 

To illustrate the technical difficulties of the proof of Theorem \ref{theorem:extension_generalization}, we consider the following example. Let $C_1,\dots,C_4$ be circles that bound disjoint disks, are orthogonal to the unit circle, and such that $C_i$ is tangent to $C_{i+1}$ (here the index $i$ is taken mod $4$). Define $f$ to be a degree $3$ covering map of $\mathbb S^1$ that is piecewise the reflection on the circles $C_i$, $i\in \{1,\dots,4\}$. Similarly, define another covering map $g$, corresponding to some other circles $C_1',\dots,C_4'$ with the same properties. The maps $f$ and $g$ are conjugate to each other and to $z\mapsto \br z^3$. The conjugacy between $f$ and $g$ can be constructed implicitly as follows. Consider a quasiconformal map $h$ from the ideal $4$-gon defined by the circles $C_i$ and contained in $\D$ onto the ideal $4$-gon defined by the circles $C_i'$, such that $h$ preserves the cusps. Then $h$ can be extended by reflections to a quasiconformal map of $\D$. It follows from the theory of quasiconformal maps that $h$ extends to a homeomorphism of $\mathbb S^1$ that is quasisymmetric and conjugates $f$ to $g$. Note that $h$ takes the parabolic fixed points of $f$ to the parabolic fixed points of $g$. However, if one wishes to study the conjugacy between $\br z^3$ and $f$, then such a reflection argument cannot be implemented. Instead, we study directly the conjugating homeomorphism $h\colon \mathbb S^1\to \mathbb S^1$, and prove certain bounds for its scalewise distortion (a quantity defined in Subsection \ref{section:preliminaries_extension}). These bounds imply, by Theorem \ref{theorem:extension_david}, that there exists a David extension of $h$ in $\D$. The bounds for $h$ are proved by studying carefully the local behavior of $f$ and $z\mapsto \br z^3$ near parabolic and hyperbolic points and then use dynamics to spread around the estimates to all points of the circle.

Before formulating Theorem \ref{theorem:extension_generalization}, we describe the assumptions on the maps $f,g$ and the Markov partitions $\mathcal P(f;\{a_0,\dots,a_r\})$, $\mathcal P(g;\{b_0,\dots,b_r\})$. We will give the definitions using the notation associated to the map $f$.

\bigskip

\subsection{Markov partitions of dynamical covering maps}\label{section:markov_partitions}
Define $A_k=\arc{[a_k,a_{k+1}]}$ for $k\in \{0,\dots,r\}$ and recall that $f_k=f|_{\inter{A_k}}$ is injective by the definition of a Markov partition. We assume, that $f_k$ is analytic and that there exist open neighborhoods $U_k$ of $\inter{A_k}$ and $V_k$ of $f_k(\inter{A_k})$ in the plane such that $f_k$ has a conformal extension from $U_k$ onto $V_k$. We still denote the extension by $f_k$. We impose the condition that 
\begin{align}\label{condition:uv}
\bigcup_{\substack{0\leq j\leq r\\(k,j) \,\, \textrm{admissible}}}U_j \subset V_k
\end{align} 
for all $k\in \{0,\dots,r\}$. We also require that 
\begin{align}\label{condition:holomorphic}
\textrm{$f_k$ extends holomorphically to neighborhoods of $a_k$ and $a_{k+1}$}
\end{align}
for each $k\in \{0,\dots,r\}$.

\begin{example}[Power map]\label{example:power_map}
Let $f\colon \mathbb S^1 \to \mathbb S^1$ be the map $z\mapsto z^d$ or $z\mapsto \br z^d$.  Then for any Markov partition $\mathcal P(f;\{a_0,\dots,a_r\})$ both conditions \eqref{condition:uv} and \eqref{condition:holomorphic} are satisfied. Indeed, if $f(z)=z^d$ on $\mathbb S^1$, then $f$ has an analytic extension to $\C$. If $f(z)=\br z^d$ on $\mathbb S^1$, observe that $f(z)=1/z^d$, so $f$ has an analytic extension to $\C\setminus \{0\}$. We denote the extension by $f$. Condition \eqref{condition:holomorphic} holds trivially. For $k\in \{0,\dots,r\}$ let $U_k$ an open sector with vertex $0$ and angle subtended by the arc $A_k$.  Since $f$ is injective on $\inter{A_k}$ by the definition of a Markov partition, it follows that $f$ is conformal on $U_k$. We set $V_k=f(U_k)$. If $A_j\subset f(A_k)$, then $U_j\subset V_k$. This shows that \eqref{condition:uv} is satisfied. 
\end{example}

\begin{example}[Blaschke product]\label{example:blaschke}
The map
$$
B(z)=\frac{(d+1)z^d+(d-1)}{(d-1)z^d+(d+1)}
$$ 
has a Markov partition that satisfies both conditions~\eqref{condition:uv} and \eqref{condition:holomorphic}. Such a Markov partition is $\mathcal P(B;\{a_0, \dots, a_{2d-1}\})$, where $a_0, \dots, a_{2d-1}$ are $d$-th roots of $1$ and $-1$ with $a_0=1$. Here, in condition~\eqref{condition:uv}, the set $U_k$, $k\in \{0,1,\dots 2d-1\}$, is an open sector with vertex 0 and angle $\pi/d$, whose boundary contains the points $a_k$ and $a_{k+1}$, indices taken modulo $2d$; the sets $V_k$ are the upper half-plane for even $k$ and the lower half-plane for odd $k$. Condition~\eqref{condition:holomorphic} is trivially satisfied.
\end{example}

\begin{example}[Circle reflections]\label{example:reflection_uv}
Consider ordered points $a_0,\dots,a_d$ on the circle $\mathbb{S}^1$ such that $\length(\arc{(a_k,a_{k+1})})<\pi$ for $ k\in \{0,\dots,d\}$. If $f_k$ is the reflection along the circle $C_k$ that is orthogonal to the unit circle at the points $a_k$ and $a_{k+1}$, then the conformal extension of $f_k$ is \textit{not} the reflection along $C_k$ (which is anti-conformal). Instead, it is the reflection along $C_k$, composed with a reflection in the unit circle, so that the resulting map is conformal. (Note that reflections along circles that intersect orthogonally commute.) The map $f_k$ is a M\"obius transformation, so \eqref{condition:holomorphic} is trivially satisfied. 

Suppose now that  $f_l$ is the reflection along $C_l$ for all $l\in \{0,\dots,d\}$. In this way, we define a covering map $f$ of $\mathbb{S}^1$ of degree $d$, as in Example \ref{example:expansive_reflection}. Note that the circles $C_l$ bound disjoint disks and the points $a_0,\dots,a_d$ are fixed points of $f$. We fix $k\in \{0,\dots,d\}$. The map $f_k$ maps $\Int{C_k}$ conformally to its exterior (recall that $f_k$ denotes the conformal extension of $f_k|_{\arc{(a_k,a_{k+1})}}$ and not the reflection itself). We consider a large open ball $B_k$ containing all circles $C_l$, $l\in \{0,\dots,d\}$, and we define $V_k$ to be the intersection of the exterior of $C_k$ with $B_k$. It is important that $V_k$ is planar and $\infty\notin V_k$. Now, we simply define $U_k=f_k^{-1}(V_k)$. It is immediate now that \eqref{condition:uv} is satisfied.  
\end{example}

\begin{example}[Hybrid]\label{example:hybrid_uv}
Consider the hybrid map $f$ from Example \ref{example:expansive_hybrid}. This map satisfies both conditions~\eqref{condition:uv} and \eqref{condition:holomorphic}. Indeed, reflections in circles orthogonal to the unit circle can be treated similarly to the previous Example~\ref{example:reflection_uv}. For the map $z\mapsto\bar z^d$, the conformal extension is $z\mapsto 1/z^d$. If $A_k$ is an arc on which the map $f$ is $z\mapsto\bar z^d$, the set $U_k$ in condition~\eqref{condition:uv} is the smallest open sector with vertex 0 that contains ${\rm int}A_k$, intersected with an annulus $\{1/R<|z|<R\}$  with sufficiently large $R>1$.
\end{example}

\begin{definition}\label{definition:parabolic_hyperbolic}
Let $a\in \{a_0,\dots,a_r\}$. We say that $a$ is \textit{parabolic on the right} (resp., \textit{on the left}) if  $\lambda(a^{+})=1$ (resp., $\lambda(a^{-})=1$) and $a$ is an isolated fixed point of $f_a$. Likewise, $a$ is \textit{hyperbolic on the right} (resp., \textit{on the left}) if $\lambda(a^{+})>1$ (resp., $\lambda(a^{-})>1$). 

For convenience, we say that $a^+$ is \textit{parabolic} (resp., \textit{hyperbolic}) if $a$ is parabolic (resp., hyperbolic) on the right. 
Similarly, we say that $a^-$ is \textit{parabolic} (resp., \textit{hyperbolic}) if $a$ is parabolic (resp., hyperbolic) on the left. 
\end{definition}

We  observe that if $f$ is expansive, then by property \ref{lemma:expansive_multiplier} and property \ref{lemma:expansive_isolated}
\begin{align}\label{condition:multipliers}
\textrm{each of}\quad a^+,a^- \quad \textrm{is parabolic or hyperbolic}  
\end{align}
for all $a\in \{a_0,\dots,a_r\}$. Equivalently, $\lambda(a^{\pm})\geq 1$ and $f_a$ is not the identity map {for all} $a\in \{a_0,\dots,a_r\}$.

In the case of parabolic points there is a further distinction. If $a=a_k$, $k\in \{0,\dots,r\}$, and $a^+$ is parabolic, then condition \eqref{condition:holomorphic} implies that the first return map $f_a$ has a holomorphic extension valid in a complex neighborhood of $\arc{[a,z_0]}$ for some point $z_0\in \mathbb S^1$. We denote the extension by $f_a^+$. The map $f_a^+$ is not the identity map, so there exists a Taylor expansion
\begin{align*}
f_a^+(z)&=z+c(z-a)^{N+1} +O((z-a)^{N+2})
\end{align*}
where $c\neq 0$ and $N$ is a non-negative integer. The number $N+1$ is called the \textit{multiplicity} of $f_a^{+}$ at the parabolic point $a^+$, and it is invariant under conjugation. Abusing terminology, we also say that $N+1$ is the multiplicity of $a^+$. We denote the multiplicity by $N(a^{+})+1$. Similarly, we define the multiplicity $N(a^-)+1$ of $a^-$, in the case that $a^-$ is parabolic.

From the theory of parabolic fixed points (see e.g.\ \cite[\S 10]{Mil06}) there are $N(a^+)$ attracting directions and $N(a^+)$ repelling directions for the holomorphic germ $f_a^+$. Note that $f_a^+$ maps the arc $\arc{[a,z_0]}$ into $\mathbb S^1$, since it is an extension of $f_a$. This implies that either all points of $\arc{[a,z_0]}$ are attracted to the point $a$ under iteration of $f_a$, or they are repelled away from $a$. We say that $\arc{[a,z_0]}$ defines an attracting or repelling direction for $f_a$, respectively. Similarly, $\arc{[z_0,a]}$ defines either an attracting or a repelling direction for $f_a$, which has a holomorphic extension $f_a^-$ in a neighborhood of $\arc{[z_0,a]}$. A trivial consequence of the expansivity of $f$ is that
\begin{align}\label{condition:repelling}
\begin{aligned}
&\textrm{if $a^+$ is parabolic then $\arc{[a,z_0]}$ defines a repelling direction for $f_a$ and}\\
&\textrm{if $a^-$ is parabolic then $\arc{[z_0,a]}$ defines a repelling direction for $f_a$.}
\end{aligned}
\end{align}

\begin{definition}\label{definition:parabolic_hyperbolic_symmetric}
Let $a\in \{a_0,\dots,a_r\}$. We say that $a$ is symmetrically parabolic if $a^{+}$ and $a^-$ are parabolic with $N(a^+)=N(a^-)$. In this case we denote this common number by $N(a)$. We say that $a$ is symmetrically hyperbolic if $a^+$ and $a^-$ are hyperbolic with $\lambda(a^+)=\lambda(a^-)$. In this case we denote by $\lambda(a)$ the common multiplier.
\end{definition}

\begin{remark}\label{remark:parabolic}
If $f$ is \textit{orientation-reversing} and $a\in \{a_0,\dots,a_r\}$ is a periodic point with \textit{odd} period, then it is automatic  from condition \eqref{condition:holomorphic} and from condition \eqref{condition:multipliers} that $a$ is symmetrically hyperbolic or parabolic. Indeed, suppose that the period of $a$ is an odd number $n\in \N$. The first return map $f_a$ is $f^{\circ 2n}$ in this case. Since $f^{\circ n}$ is orientation-reversing, by the chain rule we have: 
\begin{align*}
\lambda(a^+)= (f^{\circ 2n})'(a^+)=(f^{\circ n})' ( a^-) \cdot (f^{\circ n})'(a^+) =(f^{\circ 2n})'(a^-)= \lambda(a^-).
\end{align*}
Moreover, for points $z$ lying in an arc $\arc{[a,z_1]}$ we have
\begin{align*}
f^{\circ n}\circ f_a^+(z) = f_a^- \circ f^{\circ n}(z).
\end{align*}
By condition \eqref{condition:holomorphic}, $f^{\circ n}$ has a holomorphic extension to a complex neighborhood $U$ of a possibly smaller arc $\arc{[a,z_1]}$ that we denote by $(f^{\circ n})^+$. The uniqueness of analytic maps implies that 
\begin{align*}
(f^{\circ n})^+\circ f_a^+(z) = f_a^- \circ (f^{\circ n})^+(z)
\end{align*}
for all $z\in U$, assuming that $U$ is sufficiently small so that the holomorphic maps involved are defined on $U$. The map $(f^{\circ n})^+$ is injective near $a$ since its derivative at $a$ is non-zero by condition \eqref{condition:multipliers}. It follows that $(f^{\circ n})^+$ conjugates $f_a^+$ to $f_a^-$ near $a$. Therefore $a^+$ is parabolic (resp.\ hyperbolic) if and only if $a^-$ is parabolic (resp.\ hyperbolic) and the multipliers $\lambda(a^+),\lambda(a^-)$ are necessarily equal to each other by the conjugation. Moreover, if $a^+$ is parabolic, then the multiplicity of $a^+$ is a conjugation invariant, so it is equal to the multiplicity of $a^-$. Our claim follows.
\end{remark}

\begin{remark}
It is not hard to construct expansive, real-analytic, circle coverings having parabolic fixed points with arbitrarily many attracting directions in $\D$ (see \cite[\S 4]{Lom15} for constructions of such maps as external classes of so-called parabolic-like maps).
\end{remark}

\begin{example}\label{example:reflection_multipliers}
The Blaschke product $B(z)=(2z^3+1)/(z^3+2)$, which is a special case of Example~\ref{example:blaschke}, has two fixed points on $\mathbb{S}^1$, which are $1$ and $-1$. The point $1$ is symmetrically parabolic, while the point $-1$ is symmetrically hyperbolic with multiplier $9$.

If $f(z)=\br z^d$, which is equal to $1/z^d$ on $\mathbb{S}^1$, and $\omega_k$, $k\in \{0,\dots,d\}$, are the fixed points of $f$, then $f'(\omega_k)=-d$ for $k=0,\dots,d$. Since $f$ is orientation-reversing, for each fixed point the first return map in this case is $f^{\circ 2}$. Thus all multipliers of the first return map at the points $\omega_k$ are equal to $d^2$. It follows that all fixed points of $f$ are symmetrically hyperbolic.

If, instead, $a_0,\dots,a_d$, $d\geq 2$, are ordered points on $\mathbb{S}^1$ and for each $k\in \{0,\dots,d\}$ the map $f_k=f|_{{\arc{(a_k,a_{k+1})}}}$ is the reflection along the circle that is orthogonal to $\mathbb{S}^1$ at $a_k$ and $a_{k+1}$, then $f'(a_k^+)=f'(a_{k+1}^-)=-1$. Indeed, $f_k$ is conjugate to the reflection  of the top half of the unit circle along the real line, i.e., $f_k$ is conjugate to $\br z$. Restricted to the circle, this map is equal to the map $z\mapsto 1/z$, so the derivative at the fixed points $\pm 1$ is $-1$. For all fixed points the first return map is again $f^{\circ 2}$, so all multipliers are equal to $1$. Moreover, since $d\geq 2$, the map $f^{\circ 2}$ cannot be the identity map on any arc $\arc{(a_k,a_{k+1})}$. Therefore, all points $a_k$, $k\in \{0,\dots,d\}$, are symmetrically parabolic, in view of Remark \ref{remark:parabolic}.  

In the hybrid case of Example \ref{example:expansive_hybrid}, if $a_k,\ k\in \{0,1,\dots, d\}$, is a fixed point such that on one side of it the map $f$ is the reflection in the circle orthogonal to the unit circle and on the other side it is $z\mapsto \bar z^d$, then $a_k$ is symmetrically hyperbolic. Indeed, it is easy to see that the first return map is $f^{\circ 2}$ and the multiplier is $d\geq 2$. 
\end{example}

We now state the main extension theorem.

\begin{theorem}\label{theorem:extension_generalization}Let $f,g\colon \mathbb{S}^1\to \mathbb{S}^1$ be expansive covering maps with the same orientation and $\mathcal P(f;\{a_0,\dots,a_r\})$, $\mathcal P(g;\{b_0,\dots,b_r\})$ be Markov partitions satisfying conditions \eqref{condition:uv} and \eqref{condition:holomorphic}.  Suppose that the map $h\colon \{a_0,\dots,a_r\} \to \{b_0,\dots,b_r\}$ defined by $h(a_k)=b_k$, $k\in \{0,\dots,r\}$, conjugates $f$ to $g$ on the set $\{a_0,\dots,a_r\}$ and assume that for each periodic point $a\in \{a_0,\dots,a_r\}$ of $f$ and for $b=h(a)$ one of the following alternatives occurs.
\vspace{1em}

\begin{enumerate}[\upshape {\scriptsize(\textbf{H/P${\rightarrow}$H/P})}, leftmargin=6em] 
\item\label{HH} There exists $\mu>0$ such that if $a^{\pm}$ is parabolic then $b^{\pm}$ is parabolic with 
$\mu^{-1} N(a^{\pm})=N(b^{\pm})$, and if $a^{\pm}$ is hyperbolic then  $b^{\pm}$ is hyperbolic with  $\lambda(a^{\pm})^{\mu}=\lambda(b^{\pm})$.
\end{enumerate}

\begin{enumerate}[\upshape{\scriptsize(\textbf{H$\to$P})}, leftmargin=6em]
\item\label{HP} $a$ is symmetrically hyperbolic and $b$ is symmetrically parabolic.
\end{enumerate}

\vspace{1em}

\noindent
Then the map $h$ extends to a homeomorphism $\widetilde h$ of $\br \D$ such that $\widetilde h|_{\mathbb S^1}$ conjugates $f$ to $g$ and $\widetilde h|_{\D}$ is a David map. Moreover, if the alternative \ref{HP} does not occur, then $\widetilde h|_{\D}$ is a quasiconformal map and $\widetilde h|_{\mathbb S^1}$ is a quasisymmetry.
\end{theorem}

\begin{remark}We note that the alternative \ref{HH} allows $a^+$ (resp.\ $b^+$) to be hyperbolic and $a^{-}$ (resp.\ $b^-$) to be parabolic and vice versa. In fact, \ref{HH} covers the following cases:
\begin{itemize}
\item $a^-$ is hyperbolic and $a^+$ is hyperbolic
\item $a^-$ is hyperbolic and $a^+$ is parabolic
\item $a^-$ is parabolic and $a^+$ is hyperbolic
\item $a^-$ is parabolic and $a^+$ is parabolic
\end{itemize} 
The only restriction is that the multiplicities and multipliers of the points $a^\pm$ and $b^\pm$ have to be related by the same number $\mu$, which depends only on the point $a$.
\end{remark}

\begin{remark}\label{rem:C2_2}
It is conceivable  that by means of the smooth distortion techniques (see \cite{dMvS93})
Theorem~\ref{theorem:extension_generalization} can be extended to the piecewise $C^2$ setting
described in Remark~\ref{rem:C2_1}.
Moreover, this looks straightforward under the negative Schwarzian derivative assumption. 
\end{remark}

The essential assumption in Theorem~\ref{theorem:extension_generalization} is that parabolic periodic points of $f$ must be mapped to parabolic points of $g$. On the other hand, hyperbolic points can be mapped to either hyperbolic or parabolic points. We remark that not all points need to be checked, but only the boundary points $\{a_0,\dots,a_r\}$ of the Markov pieces. The theorem holds if one strengthens the two alternatives to one of the following symmetric alternatives, in which the left and right multipliers and multiplicities are the same:
\begin{enumerate}[\upshape ({{S}}-I)]
\item\label{S_HH} $a$ is symmetrically hyperbolic and $b$ is symmetrically hyperbolic.
\item\label{S_PP} $a$ is symmetrically parabolic and $b$ is symmetrically parabolic.
\item\label{S_HP} $a$ is symmetrically hyperbolic and $b$ is symmetrically parabolic.
\end{enumerate}

A special case of Theorem \ref{theorem:extension_generalization} is the following statement.

\begin{theorem}[Power map]\label{theorem:extension_special_case}Let $f\colon \mathbb{S}^1\to \mathbb{S}^1$ be an expansive covering map of degree $d\geq 2$ and let $\mathcal P(f;\{a_0,\dots,a_r\})$ be a Markov partition satisfying conditions \eqref{condition:uv} and \eqref{condition:holomorphic}, and with the property that $a_k$ is either symmetrically hyperbolic or symmetrically parabolic for each $k\in \{0,\dots,r\}$. Then there exists an orientation-preserving homeomorphism $h\colon \mathbb{S}^1\to \mathbb{S}^1$ that conjugates the map $z\mapsto z^d$ or $z\mapsto \br z^d$ to $f$ and has a David extension in $\D$.
\end{theorem}
\begin{proof}
Using property \ref{theorem:expansive_conjugate}, we obtain an orientation-preserving homeomorphism $h\colon \mathbb S^1 \to \mathbb S^1$ that conjugates the map $z\mapsto z^d$ or $z\mapsto \br z^d$ to $f$. Consider the induced Markov partition $\{h^{-1}(a_0),\dots,h^{-1}(a_r)\}$. By Example \ref{example:power_map}, this Markov partition satisfies \eqref{condition:uv} and \eqref{condition:holomorphic}. We now apply Theorem \ref{theorem:extension_generalization} to the map $h$.
\end{proof}

As an application of Theorem \ref{theorem:extension_special_case} we have the following theorem.

\begin{theorem}[Blaschke product--Circle reflections]\label{theorem:reflections}
Consider ordered points $a_{0}=a_{d+1}, a_1,\dots,a_d$ on the circle $\mathbb S^1$, where $d\geq 2$. For $k\in \{0,\dots,d\}$ let $f|_{\arc{(a_k,a_{k+1})}}$ be the reflection along the circle $C_k$ that is orthogonal to the unit circle at the points $a_k$ and $a_{k+1}$. Moreover, let $B$ be an anti-holomorphic Blaschke product of degree $d$ with an attracting fixed point in $\D$. Then there exists a homeomorphism $h\colon \mathbb S^1\to \mathbb S^1$ that conjugates $B|_{\mathbb S^1}$ to $f$ and has a David extension in $\D$.
\end{theorem}

\begin{proof}[Proof of Theorem~\ref{theorem:reflections}]
We first treat the special case $B(z)=\br z^d$. Suppose first that we have $\length(\arc{(a_k,a_{k+1}}) )<\pi$ for all $k\in \{0,\dots,d\}$. In Example \ref{example:expansive_reflection} we proved that $f$ is expansive. Consider the Markov partition $\mathcal P(f; \{a_0,\dots,a_d\})$. In Example \ref{example:reflection_uv} we verified that conditions \eqref{condition:uv} and \eqref{condition:holomorphic} hold. Finally, in Example \ref{example:reflection_multipliers} we proved that if $d\geq 2$, then the fixed points $a_k$, $k\in \{0,\dots,d\}$, are all symmetrically parabolic. Therefore, Theorem \ref{theorem:extension_special_case} gives the desired conclusion. 

Suppose now that $\length(\arc{(a_i,a_{i+1}}) )>\pi$ for some $i\in \{0,\dots,d\}$. Note that this can only be the case for one value of $i$. Then there exists a M\"obius transformation $M$ of $\D$ such that $\length(\arc{(M(a_k),M(a_{k+1})}) )<\pi$ for all $k\in \{0,\dots,d\}$. Consider the map $F=M\circ f \circ M^{-1}$. Then $F|_{\arc{(M(a_k),M(a_{k+1}))}}$ is the reflection along the circle $M(C_k)$ that is orthogonal to the unit circle at the points $M(a_k)$ and $M(a_{k+1})$. By the previous case, the map $F$ has a David extension in $\D$, that we still denote by $F$. By Proposition \ref{prop:david_qc_invariance} (i) and (ii) with $U=V=W=\D$, the map $M^{-1}\circ F\circ M$ is a David map that extends $f$, as desired.

Next, let $B$ be a general anti-holomorphic Blaschke product of degree $d$ with an attracting fixed point in $\D$. By Lemma \ref{lemma:blaschke_qs} there exists a quasisymmetric map $h_1\colon \mathbb S^1\to \mathbb S^1$ that conjugates $B$ to $z\mapsto \br z^d$. We extend this map to a quasiconformal map of $\D$ that we still denote by $h_1$. By the previous, there exists a homeomorphism $h_2\colon \mathbb S^1\to \mathbb S^1$ that conjugates $z\mapsto \br z^d$ to $f$ and has a David extension in $\D$. We also denote the extension by $h_2$. Then $h_2\circ h_1$ conjugates on $\mathbb S^1$ the map $B$ to $f$. Moreover, by Proposition \ref{prop:david_qc_invariance} (ii) $h_2\circ h_1$ is David map on $\D$.   
\end{proof}

Another special case of Theorem \ref{theorem:extension_generalization} is the following theorem.

\begin{theorem}[Hybrid--Circle reflections]\label{theorem:hybridtoreflections}
Let $a_{0}=a_{d+1}, a_1,\dots,a_d$ be fixed points on the circle $\mathbb S^1$ of the map $z\mapsto \bar z^d$, where $d\geq 2$, and $f$ be a hybrid map as in Example~\ref{example:expansive_hybrid}. Moreover, for $k\in \{0,\dots,d\}$ let $g|_{\arc{(a_k,a_{k+1})}}$ be the reflection along the circle $C_k$ that is orthogonal to the unit circle at the points $a_k$ and $a_{k+1}$. Then there exists a homeomorphism $h\colon \mathbb S^1\to \mathbb S^1$ that conjugates $f$ to $g$ and has a David extension in $\D$.
\end{theorem}
\begin{proof}
In Examples \ref{example:expansive_hybrid} and \ref{example:expansive_reflection} we proved that $f$ and $g$ are expansive. The Markov partitions $\mathcal P(f;\{a_0,\dots,a_d\})$ and $\mathcal P(g;\{a_0,\dots,a_r\})$ satisfy \eqref{condition:uv} and \eqref{condition:holomorphic} by Examples \ref{example:hybrid_uv} and \ref{example:reflection_uv}. Finally, in Example \ref{example:reflection_multipliers} we saw that each point $a_k$ is symmetrically hyperbolic or parabolic for $f$ and symmetrically parabolic for $g$. Therefore, by Theorem \ref{theorem:extension_generalization}, the identity map on the set $\{a_0,\dots,a_d\}$ extends to a homeomorphism $h\colon \br \D \to \br \D$ that conjugates $f $ to $g$ on $\mathbb S^1$ and is a David map in $\D$.
\end{proof}

\bigskip

\subsection{Distortion estimates and proof of Theorem \ref{theorem:extension_generalization}}\label{section:distortion}
We will need some preparation in order to prove Theorem \ref{theorem:extension_generalization}. We will formulate most of the statements in this subsection using the notation associated to the map $f$ and the Markov partition $\mathcal P(f;\{a_0,\dots,a_r\})$. The map $f$ is assumed, throughout, to be an expansive covering map of $\mathbb S^1$, satisfying conditions \eqref{condition:uv} and \eqref{condition:holomorphic}. As already remarked, conditions  \eqref{condition:multipliers} and \eqref{condition:repelling} follow automatically from expansivity, so they will be assumed. The analogous results hold, of course, for the map $g$, which satisfies the same assumptions. Moreover, the expansivity of $f$ and $g$ and Lemma \ref{lemma:expansive_conjugate_fg} imply that the map $h\colon \{a_0,\dots,a_r\}\to \{b_0,\dots,b_r\}$ from the statement of Theorem \ref{theorem:extension_generalization} extends to an orientation-preserving homeomorphism of $\mathbb S^1$ that conjugates $f$ to $g$. It is also implicitly assumed that if $a^{\pm}$ is parabolic, then $b^{\pm}=h(a^{\pm})$ is also parabolic, as in the assumptions of Theorem \ref{theorem:extension_generalization}. 

Using condition \eqref{condition:uv}, we can define $U_{kj}=f_k^{-1}(U_j)$ for $k,j\in\{0,\dots,r\}$, whenever $(k,j)$ is admissible; recall the definitions given after Definition \ref{definition:markov_partition}. Note that $U_{kj}\subset U_k$ and $f_k$ maps conformally $U_{kj}$ onto $U_j$. Inductively, for each admissible word $w$ we can find open regions $U_w$ with the following properties:
\begin{enumerate}[(i)]
\item $U_{wj} \subset U_w$, if $(w,j)$ is admissible, and 
\item $f_k$ maps conformally $U_{kw}$ onto $U_w$, if $(k,w)$ is admissible. 
\end{enumerate}
If $w=(k_1,\dots, k_n)$ is admissible, we define $f_w=f_{k_n}\circ\dots \circ f_{k_1}$ on the set $\bigcup \{U_{wj}: 0\leq j\leq r,\, (w,j)\,\, \textrm{admissible}\}$. It follows that $f_w$ maps conformally $U_{wj}$ onto $U_j$. Observe that for each admissible word $w$ the open arc $\inter{A_w}$ is contained in $U_w$ and is a preimage of one of the arcs $\inter{A_k}$, $k\in \{0,\dots,r\}$, under some iterate of $f$. Note that $f_w|_{{\inter{A_{wj}}}}$, where $(w,j)$ is admissible, is just a restriction of the $n$-th iterate $f^{\circ n}$.

We let $F_n$ be the preimages of $F_1= \{a_0,\dots,a_r\}$ under $n-1$ iterations of $f$ and $F_{0}=\emptyset$. Observe that $F_{n}\supset F_{n-1}$ for each $n\in \N$. Indeed, $F_1\supset f(F_1)$ by condition (iii) in Definition \ref{definition:markov_partition}, hence $F_2=f^{-1}(F_1)\supset f^{-1}(f(F_1)) \supset F_1$, so the conclusion follows by induction. We define the \textit{level} of a point $c\in \bigcup_{n\geq 1}F_n$ to be the unique $n\in \N$ such that $c\in F_n\setminus F_{n-1}$. 

We say that a finite collection of non-overlapping arcs $I_1,\dots,I_m\subset \mathbb{S}^1$, $m\in \N$, consists of \textit{consecutive} arcs if every arc $I_i$ shares at least one endpoint with another arc $I_j$, $j\neq i$. If all arcs in the family are open, after renumbering, we may assume that $I_i=\arc{(x_{i-1},x_i)}$ for some points $x_j\in \mathbb{S}^1$, $j=0,1,\dots,m$.

For each $n\in \N$, the set $\mathbb S^1\setminus F_n$ consists of consecutive open arcs. For practical purposes we will use the terminology \textit{complementary arcs of $F_n$} to indicate the family of the \textit{closures} of the components of $\mathbb S^1 \setminus F_n.$ Hence, all complementary arcs of $F_n$ are closed arcs. Note that the complementary arcs of $F_n$ are the arcs $A_w$, where $w$ is an admissible word with $|w|=n$; see the comments after Definition \ref{definition:markov_partition}.

Let $X,Y$ be metric spaces. We say that a homeomorphism $\phi\colon X\to Y$ has \textit{bounded relative distortion} if there exists $M\geq 1$ such that for any sets $A,B\subset X$ we have
\begin{align*}
M^{-1} \frac{\diam{A}}{\diam{B}}\leq \frac{\diam{\phi(A)}}{\diam{\phi(B)}} \leq M \frac{\diam{A}}{\diam{B}}.
\end{align*} 
In this case we say that $\phi$ has {relative distortion bounded by $M$}. Note that this implies that $\phi$ is an $\eta$-quasisymmetry with $\eta(t)=Mt$. The inverse $\phi^{-1}$ of a homeomorphism $\phi$ of bounded relative distortion also has bounded relative distortion. In what follows, if we say that a map $\phi$ has bounded relative distortion, it is implicitly understood that $\phi$ is a homeomorphism. 

An important consequence of expansivity is the following lemma. Recall the definition of the orientation-preserving period of a periodic point $a\in \{a_0,\dots,a_r\}$, given in Section \ref{section:expansive}.

\begin{lemma}\label{lemma:pre_expansivity}
Let $a\in F_1$ be a periodic point with orientation-preserving period equal to $q$, $n\in \N$, and $A$ be a complementary arc of $F_n$ having the point $a$ as an endpoint. Then $A$ contains at least two and at most $(r+1)^q$ complementary arcs of $F_{n+q}$. 
\end{lemma}
\begin{proof}
We first treat the upper bound. By the definition of a Markov partition, if $A$ is a complementary arc of $F_n$, then $A$ contains at most $r+1$ complementary arcs of $F_{n+1}$, since $A$ has at most $r+1$ children; see the comments after Definition \ref{definition:markov_partition}. Therefore, $A$ contains at most $(r+1)^q$ complementary arcs of $F_{n+q}$.

Next, we prove the lower bound. There exists an admissible word $(j_1,\dots,j_n)$ such that $A_{j_1}\supset A_{j_1j_2}\supset \cdots\supset A_{j_1\dots j_n}=A$. Consider the corresponding regions $U_{j_1}\supset U_{j_1j_2} \supset\cdots \supset U_{j_1\dots j_n}$. 

We argue by contradiction. Suppose that the arc  $A$  does not contain two complementary arcs of $F_{n+q}$. Then $A$ is itself a complementary arc of $F_{n+q}$. It follows that there exist $j_{n+1},\dots,j_{n+q}\in \{0,\dots,r\}$ such that $A=A_{j_1\dots j_n}=\dots= A_{j_1\dots j_{n+q}}$. By applying $f^{\circ (n-1)}$, we see that $A_{j_n}= A_{j_{n}j_{n+1}}=\dots=A_{j_{n}\dots j_{n+q}}\eqqcolon A'$ and that $A'$ is  a complementary arc of $F_1$.  We define $a'=f^{\circ (n-1)}(a)$, and note that $a'$ has orientation-preserving period equal to $q$, since it is contained in the orbit of the point $a$. 

By the definition of the regions $U_w$, where $w$ is an admissible word, it follows that the map $f^{\circ q}$ maps $U_{j_{n}\dots j_{n+q}}$ conformally onto $U_{j_{n+q}}$ and $A_{j_n\dots j_{n+q}}$ onto $A_{j_{n+q}}$. However, the orientation-preserving period of $a'$, which is an endpoint of $A_{j_n\dots j_{n+q}}$ is $q$. This implies that $A_{j_n\dots j_{n+q}}\subset A_{j_{n+q}}$, and therefore $j_n=j_{n+q}$ and $A_{j_n\dots j_{n+q}}=A_{j_n}= A_{j_{n+q}}$. It follows that $U_{j_{n}\dots j_{n+q}}\subset U_{j_{n+q}}$ and that the orientation-preserving period of the other endpoint $b'$ of $A'$ is a multiple of $q$.

Summarizing, the map $F= (f^{\circ  q}|_{U_{j_{n+q}}})^{-1}$ maps conformally $U\coloneqq U_{j_{n+q}}$ onto $V\coloneqq U_{j_{n}\dots j_{n+q}}$, which is a subset of $U$. Moreover, $F$ has two fixed points $a'$ and $b'$ in $\br U$ and $F$ extends holomorphically to a neighborhood of $a'$ and $b'$, by condition \eqref{condition:holomorphic}. Each of $a'$ and $b'$ is either attracting or parabolic, by condition \eqref{condition:multipliers}. 

We note that $U$ is necessarily a hyperbolic subset of $\widehat{\C}$. Indeed, otherwise, the conformal map $F$ would have to be a loxodromic M\"obius transformation with only two fixed points, one of which is attracting and the other is repelling. This is a contradiction. Since $U\supset V$ and $U$ is hyperbolic, it follows by Montel's Theorem that the sequence of conformal maps $F^{\circ m}$, $m\in \N$, has a subsequence that converges locally uniformly to a holomorphic map $F_{\infty}$ on $U$. 

If $a'$ is an attracting fixed point, then all points near $a'$ are attracted under iterates of $F$ to $a'$. If $a'$ is a parabolic fixed point, then by condition \eqref{condition:repelling} the arc $A'$ defines an attracting direction for $F$, so all points of $A'$ near $a'$ are attracted to $a'$ under iteration of $F$. It follows that $F_{\infty}$ is identically the constant function $z\mapsto a'$. However, the same conclusions hold for the fixed point $b'$, so $F_\infty$ has to be the function $z\mapsto b'$. We have arrived at a contradiction, because $a'\neq b'$.   
\end{proof}

\bigskip

\subsubsection{Local estimates near hyperbolic and parabolic points}
We are first going to establish local estimates near hyperbolic preperiodic points. Recall that all points of $F_1=\{a_0,\dots,a_r\}$ are preperiodic by condition (iii) in Definition \ref{definition:markov_partition}. Before proceeding to the next lemma, recall the definition of the one-sided multipliers $\lambda(a^{\pm})$ and the orientation-preserving period of a preperiodic point $a\in \{a_0,\dots,a_r\}$, given in Section \ref{section:expansive}.

\begin{lemma}[Hyperbolic estimates]\label{lemma:adjacenthyperbolic}
Suppose that $a\in F_1$, $a^+$ is hyperbolic, and let $q\in \N$ be the orientation-preserving period of $a$. For each $p\in \N$, $z_0\in \mathbb S^1$ with $z_0\neq a$, and for all sufficiently large $N_0\in \N$ there exists a constant $L\geq 1$ such that the following is true. If $I_1,\dots,I_p\subset \arc{[a,z_0]}$ are consecutive complementary arcs of $F_n$, $n\geq 1$, and $a$ is an endpoint of $I_1$, then for each $i\in \{1,\dots,p\}$ we have
\begin{align}\label{inequality:adjacenthyperbolic}
L^{-1}\lambda(a^+)^{-n/q}&\leq \diam {I_i} \leq L\lambda(a^+)^{-n/q}.
\end{align}
Moreover, if $n\geq N_0$, the map $f^{\circ(n-N_0)}$ has relative distortion bounded by $L$ on $\bigcup_{i=1}^p I_i$. The analogous statements hold if $a^-$ is hyperbolic and $I_1,\dots,I_p \subset \arc{[z_0,a]}$.
\end{lemma}

Intuitively, the last statement of the lemma says that using dynamics one can blow up with bounded distortion complementary intervals of $F_n$ near a hyperbolic point to complementary intervals of $F_{N_0}$.

\begin{proof}
For any $N_1\in \N$ there exists a constant $L\geq 1$ so that \eqref{inequality:adjacenthyperbolic} holds for $n\leq N_1$. This is trivial since there is only a finite number of complementary arcs of $F_n$, $n\leq N_1$. Hence, we will find bounds only when $n$ is sufficiently large.

Suppose first that $a$ is a periodic point. Consider the first return map $f_a=f^{\circ q}$.  By condition \eqref{condition:holomorphic} $f_a$ has a holomorphic extension valid in a complex neighborhood of $\arc{[a,z_1]} \subset \arc{[a,z_0]}$ for some point $z_1\in \mathbb S^1$, $z_1\neq a$. We denote this extension by $f_a^+$. Since $a^+$ is hyperbolic, we have $\lambda(a^{+})>1$. 

By the K{\oe}nigs linearization theorem (see \cite[Theorem 8.2]{Mil06}), there exists a conformal map $\phi$ defined in a neighborhood $U$ of $\arc{[a,z_1]}$ such that $\phi(a)=0$ and
\begin{align}\label{koenigs_conjugation}
\phi (f_a^{+}(z)) =\lambda(a^{+}) \phi(z)
\end{align}
for all $z\in U$. Later, we are going to shrink the arc $\arc{[a,z_1]}$ and the neighborhood $U$ appropriately.

We let $N_0\in \N$ be an integer such that there are $p$ consecutive complementary arcs of $F_{N_0}$ that are contained in $\arc{[a,z_1]}$ and one of which has $a$ as an endpoint. The existence of $N_0$ follows from the expansivity of $f$ and property \ref{corollary:diameters}. For each $j\in \{1,\dots,q-1\}$ we also consider the $p$ consecutive complementary arcs of $F_{N_0+j}$ that are contained in $\arc{[a,z_1]}$ and one of which has $a$ as an endpoint. In total, we fixed $qp$ arcs $J_l$, $l\in \{1,\dots,qp\}$. Suppose that $n\geq N_0$.

Consider the arcs $I_1,\dots,I_p$ as in the statement of the lemma, which are complementary arcs of $F_n$. For each $i\in \{1,\dots,p\}$ we have $I_i\subset \bigcup_{l=1}^{qp} J_l\subset \arc{[a,z_1]}$. Hence, $(f_a^{+})^{\circ m}(I_i)=J_l$ for some $l\in \{1,\dots,qp\}$, where $m=\lfloor (n-N_0)/q \rfloor$. Note that $\phi$ is bi-Lipschitz on $J_l$ and on $I_i$ with a uniform constant, since all arcs are contained in a fixed compact subset of $U$. In combination with the conjugation relation \eqref{koenigs_conjugation}, we obtain
\begin{align*}
\diam{J_l}&=\diam {(f_a^{+})^{\circ m}(I_i)} \simeq \diam{\phi((f_a^{+})^{\circ m}(I_i))}\\
&\simeq \lambda(a^+)^m \diam{\phi(I_i)}\simeq \lambda(a^+)^{(n-N_0)/q}\diam{I_i}.
\end{align*}
It follows that we have
\begin{align*}
\diam{I_i}\simeq \lambda(a^+) ^{-n/q},
\end{align*}
since the arcs $J_l$, $l\in \{1,\dots,qp\}$, and the integer $N_0$ are fixed. 

Now, if $I \subset \bigcup_{i=1}^p I_i$ is any arc, then the preceding argument shows that 
\begin{align*}
\diam{f^{\circ qm}(I)}=\diam{(f_a^+)^{\circ m}(I)} \simeq  \lambda(a^+)^{n/q}\diam{I}.
\end{align*}
Since $n\geq N_0$, there exists $k\in \{0,1,\dots,q-1\}$ such that $n-N_0=qm+k$. By condition \eqref{condition:holomorphic}, the maps $f,f^{\circ 2},\dots,f^{\circ (q-1)}$ have analytic extensions in a neighborhood of the arc $\arc{[a,z_1]}$, after possibly shrinking the arc. Moreover, the derivatives of these extensions are non-zero at $a$ since $\lambda(a^+)\neq 0$. It follows that $f,f^{\circ 2},\dots, f^{\circ (q-1)}$ are bi-Lipschitz in $\arc{[a,z_1]}$. Since $f^{\circ qm}(I) \subset \bigcup_{l=1}^{qp} J_l\subset \arc{[a,z_1]}$, we have
\begin{align*}
\diam{f^{\circ (n-N_0)}(I)}= \diam{f^{\circ k}(f^{\circ qm}(I))}  \simeq \diam{f^{\circ qm}(I)} \simeq  \lambda(a^+)^{n/q}\diam{I}.
\end{align*}
Therefore, if $I,J \subset \bigcup_{i=1}^p I_i$, we have
\begin{align*}
\frac{\diam{f^{\circ (n-N_0)}(I)}}{\diam{f^{\circ (n-N_0)}(J)}} \simeq \frac{\diam{I}}{\diam{J}},
\end{align*}
showing that $f^{\circ(n-N_0)}$ has bounded relative distortion. This completes the proof in the case that $a$ is periodic. Note that the constants in all above inequalities depend on $N_0$.

If $a\in F_1$ is a preperiodic point and $a^+$ is hyperbolic, then there exists a smallest $k\in \{0,\dots,r\}$ such that $f^{\circ k}(a)$ is periodic and $f^{\circ k}(a^+)$ is hyperbolic. Observe that the map $f^{\circ k}$ is bi-Lipschitz in an arc $\arc{[a,z_1]}$. The map $f^{\circ k}$ takes the complementary arcs $I_1,\dots,I_p\subset \arc{[a,z_1]}$ of $F_n$ to complementary arcs of $F_{n-k}$, as long as $n\geq r+1\geq k+1$. It follows by the previous case of periodic points that
\begin{align*}
\diam{I_i} \simeq \diam{f^{\circ k}(I_i)}\simeq \lambda( f^{\circ k}(a^+))^{-(n-k)/q} \simeq \lambda(a^+) ^{-n/q}.
\end{align*}
This also holds trivially for $n<r+1$, since there are only finitely many complementary intervals of $F_n$ for $n<r+1$. 

Let $N_0$ be a sufficiently large integer corresponding to the relative distortion bounds near the periodic point $f^{\circ k}(a)$. The bounds of the relative distortion for the point $a$ follow in the same way, if one observes that for any $I\subset \bigcup_{i=1}^p I_i$ we have
\begin{align*}
\diam {f^{\circ (n-N_0)}(I)} \simeq \diam{f^{\circ (n+k-N_0)}(I)} = \diam{f^{\circ (n-N_0)} (f^{\circ k}(I))}
\end{align*}
and then uses the relative distortion bounds near the periodic point $f^{\circ k}(a)$.
\end{proof}

We prove analogous estimates near parabolic points.

\begin{lemma}[Parabolic estimates]\label{lemma:adjacentparabolic}
Suppose that $a\in F_1$ and $a^+$ is parabolic. For each $p\in \N$, $z_0\in \mathbb S^1$ with $z_0\neq a$, and for all sufficiently large $N_0\in \N$ there exists a constant $L\geq 1$ such that the following is true. If $I_1,\dots,I_p\subset \arc{[a,z_0]}$ are consecutive complementary arcs of $F_n$, $n\geq 1$, and $a$ is an endpoint of $I_1$, then we have
\begin{align}\label{inequality:parabolic_1}
{L^{-1}}{n^{-1/{N(a^{+})}}}&\leq \diam {I_1} \leq {L}{n^{-1/{N(a^{+})}}}
\end{align}
and 
\begin{align}\label{inequality:parabolic_2}
{L^{-1}}{n^{-1/{N(a^{+}) -1}}}&\leq \diam {I_i} \leq {L}{n^{-1/{N(a^{+})-1}}}
\end{align}
for $i\in \{2,\dots,p\}$. Moreover, if $n\geq N_0$, the map $f^{\circ(n-N_0)}$ has relative distortion bounded by $L$ on $\bigcup_{i=2}^p I_i$. The analogous statements hold if $a^-$ is parabolic and $I_1,\dots,I_p \subset \arc{[z_0,a]}$.
\end{lemma}
Recall that $N(a^{\pm})+1$ is the multiplicity of a parabolic point $a^{\pm}$. Also note that we only obtain relative distortion bounds for the arcs $I_2,\dots,I_p$, but not for the arc $I_1$. This contrasts the hyperbolic case in Lemma \ref{lemma:adjacenthyperbolic}, in which we also had relative distortion bounds for $I_1$.

\begin{proof}
As in the proof of Lemma \ref{lemma:adjacenthyperbolic}, we will only find bounds for sufficiently large $n$. Also, the case of preperiodic points is treated exactly in the same way, so we will only focus on periodic points here. Suppose that $a$ is periodic and $a^+$ is parabolic.

We consider the map $f_{a}^+$, which is a holomorphic extension of $f_a$ in a complex neighborhood of some arc $\arc{[a,z_1]}\subset \arc{[a,z_0]}$. Since $(f_a^+)'(a)=1$ and $a$ is an isolated fixed point of $f_a^+$ by condition \eqref{condition:multipliers}, it follows that near $a$ we have 
\begin{align}\label{parabolic_taylor}
f_a^+(z)=a+ (z-a) +c(z-a)^{N(a^+)+1}+O((z-a)^{N(a^+)+2}),
\end{align}
where $c\neq 0$ and $N(a^+)+1$ is the multiplicity of $a^+$. By \cite[Lemma 10.1]{Mil06} we deduce that there exists a constant $M\geq 1$ such that if a backwards orbit $w_0\mapsto w_1\mapsto w_2\mapsto\cdots$ under $ (f_a^+)^{-1}$ converges to $a$, then for all sufficiently large $m\in \N$ we have 
\begin{align}\label{parabolic_estimates}
M^{-1 }\frac{1}{m^{1/N(a^+)}}\leq |w_m-a| \leq M  \frac{1}{m^{1/N(a^+)}}.
\end{align}

Recall that by condition \eqref{condition:repelling} the arc $\arc{[a,z_1]}$ defines a {repelling direction} of the parabolic point $a$. Equivalently, the inverse orbit of the point $w\in \arc{[a,z_1]}$ under $(f_a^{+})^{-1}$ converges to $a$. Therefore, \eqref{parabolic_estimates} holds for all inverse orbits of points in $\arc{[a,z_1]}$.

Let $N_0\in \N$ be an integer such that there are $p$ consecutive complementary arcs of $F_{N_0}$ that are contained in $\arc{[a,z_1]}$ and one of which has $a$ as an endpoint. As in Lemma \ref{lemma:adjacenthyperbolic}, the existence of $N_0$ follows from the expansivity of $f$. Let $n\in\N,\ n\geq N_0$. Suppose that $I_1=\arc{[a,w_n]}$ is a complementary arc of $F_n$ as in the statement of the lemma. If $n-N_0=qm+k,\ k\in \{0,\dots,q-1\},\ m\in \N\cup \{0\}$,  from \eqref{parabolic_estimates} we have
$$ |w_n-a|\simeq  \frac{1}{((n-N_0)/q)^{1/N(a^+)}} \simeq \frac{1}{n^{1/N(a^+)}}.$$
Note that there are only $q$ possible backward orbits $w_n$, $n\in \N$,  under $f_a^+$, where $w_n$ is an endpoint of a complementary arc $\arc{[a,w_n]}$ of $F_n$, and hence we may have \eqref{parabolic_estimates} for all $n\in \N$. This already proves the first inequality in the statement of the lemma.

The second inequality is more subtle and follows from the relative distortion bounds on $\bigcup_{i=2}^p I_i$ that we claim below.
\vspace{1em}

\noindent
\textbf{Claim.} Let $s\in \N$ and suppose that $K_1,\dots,K_s\subset \arc{[a,z_1]}$ are consecutive complementary arcs of $F_n$ and $a$ is an endpoint of $K_1$. If $N_0$ is sufficiently large, then for $n\geq N_0$ the map $f^{\circ(n-N_0)}$ has bounded relative distortion on $\bigcup_{i=2}^s K_i$ with constants independent of $n$.

\vspace{1em}

We postpone the proof of the claim for the moment. The claim implies that 
\begin{align*}
\frac{\diam{I_i}}{\diam{I_j}}  \simeq \frac{\diam{f^{\circ (n-N_0)}(I_i)}}{\diam{f^{\circ (n-N_0)}(I_j)}}
\end{align*}
for $i,j\in \{2,\dots,p\}$. The arcs ${f^{\circ (n-N_0)}(I_i)},{f^{\circ (n-N_0)}(I_j)}$ are complementary arcs of $F_{N_0}$, so the ratio of their diameters is comparable to $1$, with constants depending on $N_0$. Therefore, $\diam{I_i}\simeq \diam{I_j}$ for $i,j\in \{2,\dots,p\}$.

Let $I_2=\arc{[w_n,b]},\ w_n,b\in F_n$, and let $w_{n-1}=f_a^{+}(w_n)$. Consider the arc $\arc{[a,w_{n-1}]}$, which is a complementary arc of $F_{n-q}$. By Lemma \ref{lemma:pre_expansivity} the arc $\arc{[a,w_{n-1}]}$ contains at least two complementary arcs of $F_n$, one of which is $\arc{[a,w_n]}$. Therefore,   $I_2$ is contained in  $\arc{[w_n,w_{n-1}]}$, which contains at most $(r+1)^q-1$ complementary arcs of $F_n$, by Lemma \ref{lemma:pre_expansivity}.  Since $\arc{[w_n,w_{n-1}]}$ contains a uniformly bounded number of complementary arcs of $F_n$, including $I_2$, we have $\diam {I_2}\simeq \diam {\arc{[w_n,w_{n-1}]}}$ by the Claim. It follows that $\diam{I_i}\simeq  \diam {\arc{[w_n,w_{n-1}]}}$ for $i\in \{2,\dots,p\}$. If $n$ is sufficiently large, then $\diam {\arc{[w_n,w_{n-1}]}}=|w_n-w_{n-1}|$. Therefore, by the Taylor expansion of $f_a^+$ at $a$ in \eqref{parabolic_taylor} and by \eqref{parabolic_estimates}, we have
\begin{align*}
\diam{I_i} \simeq |w_n-w_{n-1}|= |w_n-f_a^+(w_n)| \simeq |w_n-a|^{N(a^+)+1}\simeq \frac{1}{n^{1/N(a^+)+1}}
\end{align*}
for $i\in \{2,\dots,p\}$.
\end{proof}

\begin{proof}[Proof of Claim]
Consider the map $f_a^+$ that is holomorphic in a neighborhood of its parabolic fixed point $a$. For this map there exists a \textit{repelling petal} $\mathcal P$, which is an open set in the plane, that contains a neighborhood of $a$ inside the arc $\arc{(a,z_1]}$; this is because all points of $\arc{(a,z_1]}$ near $a$ are attracted to $a$ under iterations of $(f_a^+)^{-1}$ by condition \eqref{condition:repelling}. 

By the parabolic linearization theorem \cite[Theorem 10.9]{Mil06} there exists a conformal embedding $ \alpha \colon \mathcal P\to \C$, unique up to a translation of $\C$, such that
\begin{align*}
\alpha(f_a^{+}(z))=1+\alpha(z)
\end{align*}
for all $z\in \mathcal P\cap (f_a^+)^{-1}(\mathcal P)$.  Moreover, the domain $\alpha(\mathcal P)$ contains a left half-plane $\mathbb H$. By shrinking $\arc{(a,z_1]}$, we assume that the above conjugation holds in $\arc{(a,z_1]}$ and that $\alpha(\arc{(a,z_1]})\subset \mathbb H$. Since $f_a^+$ maps $\arc{(a,z_1]}\subset \mathbb{S}^1$ to an arc of the circle $\mathbb{S}^1$, by symmetry, we have that $\alpha(\arc{(a,z_1]})$ is an interval on the negative real axis, contained in $\mathbb H$. In particular, the distance between $\alpha(\arc{(a,z_1]})$ and $\partial \mathbb H$ is positive. 

We let $N_0\in \N$ be an integer such that there are $s$ consecutive complementary arcs of $F_{N_0}$ that are contained in $\arc{[a,z_1]}$ and one of which has $a$ as an endpoint. The existence of $N_0$ follows from the expansivity of $f$ and property \ref{corollary:diameters}. For each $j\in \{1,\dots,q-1\}$ we also consider the $s$ consecutive complementary arcs of $F_{N_0+j}$ that are contained in $\arc{[a,z_1]}$ and one of which has $a$ as an endpoint. We discard from these collections all arcs that have $a$ as an endpoint. In total, we have $q(s-1)$ fixed arcs $J_l$, $l\in \{1,\dots,q(s-1)\}$. Suppose that $n\geq N_0$. 

Consider the arcs $K_1,\dots,K_s$ as in the statement of the claim, which are complementary arcs of $F_n$. For each $i\in \{2,\dots,s\}$ we have $K_i\subset \arc{(a,z_1]}$. Hence, $(f_a^{+})^{\circ m}(K_i)=J_l$ for some $l\in \{1,\dots,q(s-1)\}$, where $m=\lfloor (n-N_0)/q \rfloor$. Consider the corresponding arcs $K_i'=\alpha(K_i)$ and $J_l'=\alpha(J_l)$. We then have $J_l'=K_i'+m$ by the conjugation. 

Note that $\bigcup_{l=1}^{q(s-1)} J_l'$ is contained in a fixed ball $B(x,R)$ with $B(x,tR) \subset \mathbb H$ for some fixed $t>1$, since the distance between $\bigcup_{l=1}^{q(s-1)} J_l'$ and $\partial \mathbb H$ is positive and we have discarded all unbounded arcs. Using Koebe's distortion theorem, we see that $\alpha^{-1}$ has bounded relative distortion on $\bigcup_{l=1}^{q(s-1)} J_l'$. Now observe that $\bigcup_{i=2}^s K_i'$ is contained in $B(x-m,R)$ and $B(x-m,tR)$ is contained in $\mathbb H$. Koebe's distortion theorem implies in this case that $\alpha^{-1}$ has bounded relative distortion on $\bigcup_{i=2}^s K_i'$. Since the inverse of a map of bounded relative distortion has the same property, we conclude that $\alpha$ has bounded relative distortion on $\bigcup_{i=2}^s K_i$.
 
By the conjugation, the map $(f_a^+)^{\circ m}$, restricted to $\bigcup_{i=2}^s K_i$, is the composition of $\alpha^{-1}$, the translation to the right by $m$, and $\alpha$. Since all maps involved have bounded relative distortion, we conclude that $(f_a^+)^{\circ m}=f^{\circ qm}$ has bounded relative distortion on $\bigcup_{i=2}^s K_i$.

Since $n\geq N_0$, there exists $k\in \{0,1,\dots,q-1\}$ such that $n-N_0=qm+k$. By condition \eqref{condition:holomorphic}, the maps $f,f^{\circ 2},\dots,f^{\circ (q-1)}$ have analytic extensions in a neighborhood of the arc $\arc{[a,z_1]}$, after possibly shrinking the arc. Since $\lambda(a^+)\neq 0$, after possibly shrinking the arc again, the derivatives of these extensions are non-zero. It follows that $f,f^{\circ 2},\dots, f^{\circ (q-1)}$ are bi-Lipschitz in $\arc{[a,z_1]}$. Since $f^{\circ qm}( \bigcup_{i=2}^s K_i)$ is contained in $\arc{[a,z_1]}$, we conclude that $f^{\circ (n-N_0)}=f^{\circ k} \circ  f^{\circ qm}$ has bounded relative distortion on $\bigcup_{i=2}^s K_i$.
\end{proof}

\begin{corollary}\label{corollary:hyperbolic_parabolic}
Suppose that $a\in F_1$. For each $p\in \N$, $z_0\in \mathbb S^1$ with $z_0\neq a$, there exists a constant $L\geq 1$ such that the following is true. If $I_1,\dots,I_p\subset \arc{[a,z_0]}$ are consecutive complementary arcs of $F_n$, $n\geq 1$, and $a$ is an endpoint of $I_1$, then we have
\begin{align*}
\diam{I_1}\geq L^{-1} \diam{I_i} \quad \textrm{and}\quad 
L^{-1}\diam{I_j} \leq \diam{I_i} \leq L \diam{I_j}.
\end{align*}
for $i,j\in \{2,\dots,p\}$.
\end{corollary}

So far, we have established diameter bounds for \textit{dynamical} arcs, i.e., complementary arcs of $F_n$. In the next very technical lemma we prove estimates for the diameter of a \textit{non-dynamical} arc $I$, located near a point $a\in F_1$. Recall that $h\colon \mathbb S^1\to \mathbb S^1$ is an orientation-preserving homeomorphism that conjugates $f$ to $g$, by Lemma \ref{lemma:expansive_conjugate_fg}.

\begin{lemma}[One-sided estimates]\label{lemma:one_sided_estimates}
Suppose that $a\in F_1$, $q$ is the orientation-preserving period of $a$, and $b=h(a)$. For each $p\in \N$ with $p\geq 2$, $z_0\in \mathbb S^1$ with $z_0\neq a$, and $M\geq 1$ there exists $L\geq 1$ such that the following is true. If $I_1,\dots,I_p\subset \arc{[a,z_0]}$ are consecutive complementary arcs of $F_n$, $n\geq 1$, and $a$ is an endpoint of $I_1$, then for each closed arc $I\subset \mathbb S^1$ with 
\begin{align*}
I\subset \bigcup_{i=1}^p I_i \quad \textrm{and} \quad \diam{I} \geq M^{-1} \diam{I_2}
\end{align*}
the following alternatives occur. We let $k\in \N\cup \{0\}$ be the smallest integer such that a  complementary arc of $F_{n+k}$ not having $a$ as an endpoint intersects  $I$ and $l\in \N\cup \{0\}$ be the smallest integer, if there exists one, such that there exists a complementary arc of $F_{n+k+l}$ contained in $\arc{[a,z_0]}$ having $a$ as an endpoint and not intersecting $I$. If no such $l$ exists, it is set to be $\infty$.
\begin{itemize}
\item If $a^+$ is parabolic, then	
\begin{align*}
L^{-1} (n+k)^{-\alpha-1}\leq \frac{\diam{I}}{\min\{l+1,n+k\}}\leq L (n+k)^{-\alpha-1},
\end{align*}
where $\alpha=1/N(a^+)$. If, in addition, $b^+$ is parabolic, the same estimates hold for the arc $h(I)$, if we replace $\alpha$ by $\beta=1/N(b^+)$. Namely, 
\begin{align*}
L^{-1} (n+k)^{-\beta-1}\leq \frac{\diam{h(I)}}{\min\{l+1,n+k\}}\leq L (n+k)^{-\beta-1}.
\end{align*}
In particular, if $a^+$ and $b^+$ are both parabolic, then 
\begin{align*}
L^{-2} (n+k)^{\alpha-\beta}\leq \frac{\diam{h(I)}}{\diam{I}} \leq L^2 (n+k)^{\alpha-\beta}.
\end{align*}

\item If $a^+$ is hyperbolic, then
$$L^{-1}\lambda(a^+)^{-n/q}\leq \diam{I} \leq L\lambda(a^+)^{-n/q}$$ 
and $k\leq L$. If, in addition, $b^+$ is hyperbolic, then
$$L^{-1}\lambda(b^+)^{-n/q}\leq \diam{h(I)} \leq L\lambda(b^+)^{-n/q}$$
and if $b^+$ is parabolic, then
$$L^{-1} (n+k)^{-\beta-1}\leq \frac{\diam{h(I)}}{\min\{l+1,n+k\}}\leq L (n+k)^{-\beta-1},$$
where $\beta=1/N(b^+)$.
\end{itemize}
The corresponding estimates hold for $a^-$ and $b^-$.
\end{lemma}

\begin{proof}
First, we note that it suffices to prove the statements for sufficiently large $n$. Indeed, if $n<N_0$ for some $N_0\in \N$, then $\diam{I_i}\simeq 1$ and $\diam{h(I_i)} \simeq 1$ for $i\in \{1,\dots,p\}$, as there are only finitely many complementary arcs of $F_{n}$, $n<N_0$. Since $\diam{I}\gtrsim \diam{I_2}$, we have $\diam{I}\simeq 1$. By the continuity of the homeomorphism $h$, we have $\diam{h(I)}\simeq 1$. This shows that the estimates in the case that $a^+$ or $b^+$ is hyperbolic hold with constants depending on $N_0$. Moreover, if $k_0$ is a sufficiently large integer, then the complementary arcs of $F_{n+k_0}$ are small enough by property \ref{corollary:diameters}, so that the arc $I$, whose diameter is comparable to $1$, is not contained in any complementary arc of $F_{n+k_0}$. Hence, $I$ intersects at least two complementary arcs of $F_{n+k_0}$ and $k\leq k_0$, where $k$ is the integer that is defined in the statement of the lemma. Since $k$ and $n$ are bounded, we also obtain the desired estimates in case $a^+$ or $b^+$ is a parabolic point.

Another reduction is to assume that $a$ is a periodic point of $f$. Otherwise, if $a$ is strictly preperiodic, one can obtain the desired estimate using the fact that an iterate of $f$ that maps $a$ to a periodic point is bi-Lipschitz on $\bigcup_{i=1}^p I_i$, if $n$ is sufficiently large, as in the proof of Lemma \ref{lemma:adjacenthyperbolic}.
We will split the proof into the two basic cases.

\bigskip
\noindent
\textbf{Case 1.} $a^+$ is parabolic. We let $k\in \N\cup \{0\}$ be the smallest integer such that a  complementary arc of $F_{n+k}$ not having $a$ as an endpoint intersects  $I$. We will reduce to the case that $k=0$. If $k\geq 1$, there exists a complementary arc $J$ of $F_{n+k-1}$ that has $a$ as an endpoint and contains $I$. By passing to the next level, we obtain consecutive complementary arcs $J_1,\dots,J_{p'}$ of $F_{n+k}$, which are the children of $J$, so $p'\leq r+1$ by Lemma \ref{lemma:pre_expansivity}, such that $I$ is not contained in $J_1$. By assumption, we have $\diam{I}\gtrsim \diam{I_2}$. Inequality~\eqref{inequality:parabolic_2} from Lemma \ref{lemma:adjacentparabolic} implies that $\diam{I_2}\simeq n^{-\alpha-1}$ and $\diam{J_2}\simeq (n+k)^{-\alpha-1}$. Since $n^{-\alpha-1} \geq (n+k)^{-\alpha-1}$, it follows that $\diam{I}\gtrsim \diam{J_2}$. Therefore, we have reduced to the case that 
\begin{align}\label{lemma:one_sided_estimates:reduction}
I\subset \bigcup_{i=1}^p I_i,\quad \diam{I}\gtrsim \diam{I_2},\quad  I\cap \bigcup_{i=2}^p I_i\neq \emptyset, \quad \textrm{and}\quad  k=0.
\end{align}

We let $l\in \N\cup\{0\}$ be the smallest integer such that there exists a complementary arc of $F_{n+l}$ that has $a$ as an endpoint and does not intersect $I$. If such an integer does not exist, we set $l=\infty$. In the latter case we have $I\supset I_1$, since $I\cap \bigcup_{i=2}^pI_i\neq \emptyset$. It follows from \eqref{inequality:parabolic_1} and \eqref{inequality:parabolic_2} in Lemma \ref{lemma:adjacentparabolic} that
\begin{align*}
n^{-\alpha} &\simeq \diam{I_1} \lesssim \diam{I}\\
 &\lesssim \diam{I_1} +\sum_{i=2}^p\diam{I_i} \lesssim n^{-\alpha} +(p-1)n^{-\alpha-1} \simeq n^{-\alpha}.
\end{align*} 
Hence $\diam{I}\simeq n^{-\alpha-1} \cdot \min\{\infty, n\} $. Similarly, if $l=0$, then $I_1\cap I=\emptyset $ and $I\subset \bigcup _{i=2}^p I_i$, so
$$n^{-\alpha-1} \simeq \diam{I_2}\lesssim \diam{I} \lesssim \sum_{i=2}^p\diam{I_i} \simeq n^{-\alpha-1}.$$
Hence, $\diam{I}\simeq n^{-\alpha -1} \cdot \min \{1,n\}$.

Suppose now that $0<l<\infty$. Recall that $q$ is the orientation-preserving period of $a$. Consider $m\in \N$ such that $l= q(m-1)+s$, $s\in \{1,\dots,q\}$. Note that $q(m-1)< l\leq qm$. Set $J_0=I_1$ and let $J_1\subset J_0$ be the complementary arc of $F_{n+q}$ that has $a$ as an endpoint. Then, by Lemma \ref{lemma:pre_expansivity}, $J_0\setminus J_1$ contains at least one complementary arc of $F_{n+q}$ and is contained in the union of at most $(r+1)^q-1$ complementary arcs of $F_{n+q}$. It follows from \eqref{inequality:parabolic_2} in Lemma \ref{lemma:adjacentparabolic} that $\diam(J_0\setminus J_1)\simeq (n+q)^{-\alpha-1}$. Inductively, we let $J_j\subset J_{j-1}$ be the complementary arc of $F_{n+jq}$ that has $a$ as an endpoint. Again from Lemma \ref{lemma:adjacentparabolic} we have $\diam(J_{j-1}\setminus J_j)\simeq (n+jq)^{-\alpha-1}$. We note that 
\begin{align*}
\bigcup_{j=1}^{m-1}(J_{j-1}\setminus J_{j}) \subset I \subset  \bigcup_{j=1}^{m}(J_{j-1}\setminus J_{j}) \bigcup \bigcup_{i=2}^p I_i.
\end{align*}
Since $\diam{I}\geq M^{-1} \diam{I_2}$, we have 
\begin{align*}
\sum_{j=1}^{m-1} \diam(J_{j-1}\setminus J_j)+ \diam{I_2}  \lesssim \diam{I} \leq \sum_{j=1}^m \diam(J_{j-1}\setminus J_j)  + \sum_{i=2}^p \diam {I_i}.
\end{align*}
It follows that 
\begin{align*}
\diam{I} &\simeq \sum_{j=0}^m (n+jq)^{-\alpha-1} \\
&\simeq \int_n^{n+l+1} x^{-\alpha-1}\, dx \simeq n^{-\alpha}-(n+l+1)^{-\alpha}\simeq n^{-\alpha}(1- (1+(l+1)/n)^{-\alpha})\\
&\simeq  n^{-\alpha} \cdot  \begin{cases}(l+1)/n, & \textup{if} \quad l+1\leq n \\ 1, & \textup{if} \quad l+1>n \end{cases}.
\end{align*}
The conclusion now follows.

Next, we show that if $b^+$ is parabolic, it satisfies the same estimate. Note that by \eqref{lemma:one_sided_estimates:reduction} we have $h(I)\subset \bigcup_{i=1}^p h(I_i)$ and $h(I)\cap \bigcup_{i=2}^p h(I_i)\neq \emptyset$. If we had that 
\begin{align}\label{lemma:one_sided_estimates:parabolic}
\diam{h(I)}\gtrsim \diam {h(I_2)},
\end{align}
then we would have the desired estimate (for $k=0$) by following the above argument. Hence, our goal will be to establish inequality \eqref{lemma:one_sided_estimates:parabolic}. 

We consider the family of complementary arcs of $F_{n+q}$ that are contained in $I_i$, $i\in \{1,\dots,p\}$. We denote these arcs by $K_i$, $i\in \{1,\dots,p''\}$. Note that each one of $I_i$ contains at most $(r+1)^q$ such arcs by Lemma \ref{lemma:pre_expansivity}, so $p''\leq p(r+1)^q$. Moreover, by Lemma \ref{lemma:pre_expansivity} $I_1$ contains at least two complementary arcs of $F_{n+q}$. Let $K_1\subset I_1$ be the complementary arc of $F_{n+q}$ that has $a$ as an endpoint and $K_2\subset I_1$ be its adjacent arc. Lemma \ref{lemma:adjacentparabolic} implies that 
$$\diam{I_2}\simeq  n^{-\alpha-1} \simeq (n+q)^{-\alpha-1}\simeq \diam{K_2}\simeq \diam{K_i}$$
for all $i\in \{2,\dots,p''\}$. If $I\cap K_1\neq \emptyset$, then we necessarily have $K_2\subset I$, since $I\cap \bigcup_{i=2}^p I_i\neq \emptyset$. Otherwise, if $I\cap K_1=\emptyset$, we have $I\subset \bigcup_{i=2}^{p''} K_i$.  Therefore, in both of the above cases, since $\diam{I}\gtrsim \diam{K_2}$, we have 
$$\diam{K} \simeq \diam{K_2}, \quad \textrm{where}\quad K=I\cap \bigcup_{i=2}^{p''} K_i \subset \bigcup_{i=2}^{p''} K_i.$$
 
The importance of these conditions is that on $\bigcup_{i=2}^{p''}K_i$ and on $\bigcup_{i=2}^{p''} h(K_i)$ we have relative distortion bounds for the maps $f^{\circ (n-N_0)}$ and $g^{\circ (n-N_0)}$, respectively, by Lemma \ref{lemma:adjacentparabolic}. These distortion bounds are not available on $\bigcup_{i=1}^{p''} K_i$ and $\bigcup_{i=1}^{p''} h(K_i)$. By Lemma \ref{lemma:adjacentparabolic}, for all sufficiently large $N_0\in \N$, if $n\geq N_0$, we have
$$\diam{f^{\circ (n-N_0)} (K)}\simeq \diam{f^{\circ (n-N_0)} (K_2)}.$$
Since $f^{\circ (n-N_0)}(K_2)$ is a complementary arc of $N_0+q$, we may assume that 
$$\diam{f^{\circ (n-N_0)} (K)}\simeq \diam{f^{\circ (n-N_0)} (K_2)}\simeq 1$$
with constants depending on $N_0$. Using the continuity of $h$ and the conjugation, we obtain
$$\diam{g^{\circ (n-N_0)} (h(K))}\simeq \diam{g^{\circ (n-N_0)} h((K_2))}\simeq 1.$$
Using again the distortion bounds from Lemma \ref{lemma:adjacentparabolic} and enlarging $N_0$ if necessary, we conclude that $$\diam{h(K)} \simeq \diam{h(K_2)}.$$ Since $I\supset K$, we have $\diam {h(I)} \geq  \diam{h(K_2)}$. Finally, by Lemma \ref{lemma:adjacentparabolic}, we have $$\diam{h(K_2)} \simeq (n+q)^{-\beta-1} \simeq n^{-\beta-1} \simeq \diam{h(I_2)}.$$ Altogether, we have $\diam{h(I)}\gtrsim \diam {h(I_2)}$, i.e., \eqref{lemma:one_sided_estimates:parabolic} holds, assuming that $n\geq N_0$. This completes the proof of Case 1.

\bigskip

\noindent
\textbf{Case 2.} $a^+$ is hyperbolic. We have $\diam{I_i} \simeq \lambda(a^+)^{-n/q}$ for $i\in \{1,\dots,p\}$ by Lemma \ref{lemma:adjacenthyperbolic}. Since 
$$M^{-1}\diam{I_2} \leq  \diam{I} \leq \sum_{i=1}^{p}\diam{I_i},$$
it follows that $\diam{I} \simeq \lambda(a^+)^{-n/q}$ with constants depending on $p$ and $M$.

Let $k\in \N\cup \{0\}$ be the smallest integer such that a  complementary arc of $F_{n+k}$ not having $a$ as an endpoint intersects  $I$. If $k\geq 1$, there exists a complementary arc $J$ of $F_{n+k-1}$ that has $a$ as an endpoint and contains $I$. By passing to a further level, we obtain consecutive complementary arcs $J_1,\dots,J_{p'}$ of $F_{n+k}$, which are the children of $J$, so $p'\leq r+1$, such that $I$ is not contained in $J_1$ and $I\subset \bigcup_{i=1}^{p'}J_i$. Inequality \eqref{inequality:adjacenthyperbolic} from Lemma \ref{lemma:adjacenthyperbolic} implies that $\diam{J_i}\simeq \lambda(a^+)^{-(n+k)/q}$, so $\diam{I}\lesssim \lambda(a^+)^{-(n+k)/q}$. However, $\diam{I}\simeq \lambda(a^+)^{-n/q}$. Therefore, $k\leq C$ for some uniform constant $C>0$.

Suppose that $b^+$ is also hyperbolic. By Lemma \ref{lemma:adjacenthyperbolic}, for all sufficiently large $N_0\in \N$, if $n\geq N_0$ then $f^{\circ (n-N_0)}$ has bounded distortion on $\bigcup_{i=1}^pI_i$. It follows that 
\begin{align*}
\diam{f^{\circ (n-N_0)} (I)} &\gtrsim  \diam{f^{\circ (n-N_0)} (I_2)}.
\end{align*}
Note that $f^{\circ (n-N_0)} (I_2)$ is a complementary arc of $F_{N_0}$. If $N_0$ is fixed, then we may assume that $\diam{f^{\circ (n-N_0)} (I_2)}\simeq 1$. By continuity, this implies that 
\begin{align*}
\diam{h(f^{\circ (n-N_0)} (I))}\simeq 1,
\end{align*}
or equivalently
\begin{align*}
\diam{g^{\circ (n-N_0)} (h(I))} \simeq 1 \simeq \diam {g^{\circ (n-N_0)}(h(I_2))}.
\end{align*}
Since $b^+$ is hyperbolic, using Lemma \ref{lemma:adjacenthyperbolic} and enlarging $N_0$, we may also have that $g^{\circ (n-N_0)}$ has bounded distortion on $\bigcup_{i=1}^p h(I_i)$. Therefore, we obtain 
\begin{align*}
\diam{h(I)} \simeq \diam {h(I_2)}.
\end{align*}
By \eqref{inequality:adjacenthyperbolic} we have 
$$\diam{h(I_2)} \simeq \lambda(b^+)^{-n/q}.$$
Hence, $\diam{h(I)}\simeq \lambda(b^+)^{-n/q}$, as desired. 

Suppose now that $b^+$ is parabolic. In this case, we cannot apply distortion estimates near $b^+$ so we will first pass to some further subdivisions. Consider the consecutive complementary arcs $J_1,\dots,J_{p'}$ of $F_{n+k}$, such that $I$ is not contained in $J_1$ and $I\subset \bigcup_{i=1}^{p'}J_i$. Since $k$ is uniformly bounded, we have $\diam{J_2}\simeq \lambda(a^+)^{-n/q}\simeq \diam{I}$ by \eqref{inequality:adjacenthyperbolic} and $\diam{h(J_2)} \simeq n^{-\beta-1} \simeq \diam{h(I_2)}$ by \eqref{inequality:parabolic_2}. If we replace $n$ with $n+k$ and the arcs $I_i$ with the arcs $J_i$, we have $I\subset \bigcup_{i=1}^pI_i$, $\diam{I}\simeq \diam{I_2}$, and $I\cap \bigcup_{i=2}^p I_i\neq \emptyset$. By arguing as in Case 1 (see \eqref{lemma:one_sided_estimates:parabolic}), it suffices to prove that
\begin{align*}
\diam{h(I)} \gtrsim \diam{h(I_2)}.
\end{align*}

We can now proceed as in Case 1. We consider the family of complementary arcs of $F_{n+q}$ that are contained in $I_i$, $i\in \{1,\dots,p\}$. We denote these arcs by $K_i$, $i\in \{1,\dots,p''\}$, where $p''\leq p(r+1)^q$. Let $K_1\subset I_1$ be the complementary arc of $F_{n+q}$ that has $a$ as an endpoint and $K_2\subset I_1$ be its adjacent arc. By Lemma \ref{lemma:adjacenthyperbolic}, we have $\diam{I_2}\simeq \diam{K_2}\simeq \diam{K_i}$ for all $i\in \{1,\dots,p''\}$. Since $\diam{I}\simeq \diam{I_2}\simeq \diam{K_2}$, we have 
$$\diam{K} \simeq \diam{K_2}, \quad \textrm{where}\quad K=I\cap \bigcup_{i=2}^{p''} K_i \subset \bigcup_{i=2}^{p''} K_i.$$
The argument now continues exactly as in Case 1, using distortion bounds to blow up the arcs $K_2$ and $K$ to arcs of large diameter. On $\bigcup_{i=1}^{p''} K_i$ the distortion bounds come from Lemma \ref{lemma:adjacenthyperbolic} and on $\bigcup_{i=2}^{p''} h(K_i)$ we use Lemma \ref{lemma:adjacentparabolic}. The proof is complete.
\end{proof}

\bigskip

\subsubsection{Conformal elevator and completion of proof of Theorem \ref{theorem:extension_generalization}}
Our first goal here is to start with an arbitrary arc $I\subset \mathbb S^1$ and map it conformally and with bounded relative distortion, by applying a suitable iterate of $f$, to an arc $I'$ that is located ``near" a point $a\in F_1$. This procedure is referred to as the \textit{conformal elevator} and is described more precisely in Lemma \ref{lemma:blowup_simultaneous}. Once the arc $I$ is blown up to the arc $I'$ that is near $a\in F_1$, then one can apply the diameter estimates from Lemma \ref{lemma:one_sided_estimates}. However, there is a basic dichotomy. Either $I'$ contains the point $a$, or $I'$ lies only on one side of $a$. Each of these cases is treated separately in Lemma \ref{lemma:one_sided_distortion} and Lemma \ref{lemma:two_sided_distortion}, respectively. Finally, using the latter two lemmas, we conclude the proof of the main Theorem \ref{theorem:extension_generalization}.

\begin{lemma}[Conformal elevator]\label{lemma:blowup_simultaneous}
There exists $M\geq 1$ such that for any non-degenerate closed arc $I\subset \mathbb S^1$ there exist $n\in \N$, $m\in \N\cup \{0\}$, $a\in F_1$, and $z_0\in \mathbb S^1$ with $z_0\neq a$ such that one of the following alternatives holds.
\begin{enumerate}[\upshape (i)]
\item\label{lemma:blowup:two_sides} The arc $I'=f^{\circ m}(I)$ contains the point $a\in F_1$.
\item\label{lemma:blowup:one_side} There exist consecutive complementary arcs $I_1,\dots,I_p$, $p\leq M$, of $F_{n+m}$ with 
\begin{align*}
I\subset \bigcup_{i=1}^p I_i \quad \textrm{and} \quad I\cap \bigcup_{i=2}^p I_i \neq \emptyset
\end{align*}
such that the arcs $I_1'=f^{\circ m}(I_1),\dots,I_p'=f^{\circ m}(I_p)$ are consecutive complementary arcs of $F_n$ that contain the arc $I'=f^{\circ m}(I)$, the point $a$ is an endpoint of $I_1'$, 
\begin{align*}
&\bigcup_{i=1}^p I_i'\subset \arc{[a,z_0]} \quad \textrm{or}\quad  \bigcup_{i=1}^p I_i'\subset \arc{[z_0,a]}, \quad \textrm{and}\\
&\diam{I'}\geq M^{-1}\diam{I_2'}.
\end{align*}
\end{enumerate}
Moreover, in both cases $f^{\circ m}$ and $g^{\circ m}$ have relative distortion bounded by $M$ on $I$ and $h(I)$, respectively, and
\begin{align*}
\diam{I'}\geq M^{-1} \diam{I}.
\end{align*} 
\end{lemma}

\begin{proof}
Recall that $A_k,\ k\in \{0,\dots,r\}$, are the complementary arcs of $F_1$. For each $k\in \{0,\dots,r\}$ we define $\widetilde{A_k}\subset\subset  A_k$ to be a closed arc with the property that each of its endpoints is separated from the corresponding endpoint of $A_k$ by one arc of $F_q$. Here, $q\in \N$ is chosen so that $\widetilde{A_k}$ contains all points of $F_2$ that lie in $\inter{A_k}$ and moreover each point of $F_2$ in $\widetilde{A_k}$ is separated from the endpoints of $\widetilde{A_k}$ by at least one complementary arc of $F_q$. The existence of such a $q$ follows from property \ref{corollary:diameters}. Indeed, as $q\to \infty$, the endpoints of the arc $\widetilde{A_k}$ converge to the endpoints of the arc $A_k$, since the diameters of the arcs $A_w$, $|w|=q$, tend to $0$. Hence, we can achieve that $\widetilde{A_k}$ contains in its interior all points of $F_2$ that are contained in $\inter{A_k}$. Moreover, any point of $z\in F_2\cap \widetilde{A_k}$ has positive distance from the endpoints of $\widetilde{A_k}$. Using again the fact that the diameters of $A_w$, $|w|=q$, tend to $0$ as $q\to\infty$, we may achieve that there are arbitrarily many consecutive arcs $A_w$, $|w|=q$, separating $z$ from the endpoints of $\widetilde{A_k}$.  

If $I$ contains a point of $F_1$, there is nothing to show, since we are already in alternative \ref{lemma:blowup:two_sides} of the lemma. Hence, we assume that $I$ does not intersect $F_1$. Then, by property \ref{corollary:diameters} there exists a largest integer $l\geq 1$ such that the arc $I$ is contained a nested sequence of arcs $A_{j_1}\supset A_{j_1j_2}\supset \cdots\supset A_{j_1\dots j_l}\eqqcolon A$. By the choice of $l$, the arc $I$ must contain a point $c\in F_{l+1}$. Consider the corresponding regions $U_{j_1}\supset U_{j_1j_2} \supset\cdots \supset U_{j_1\dots j_l} \eqqcolon U$ and the map $\phi=f^{\circ (l-1)}=f_{j_1\dots j_{l-1}}$, which maps $U$ conformally onto $U_{j_l}$. We denote by $\widetilde A$ the preimage of $\widetilde{A_{j_l}}$ under $\phi$ and observe that $\phi(c)\in F_2$. The basic dichotomy arises from whether $I\subset \widetilde A$ or not. 

\bigskip

\noindent
\textbf{Case 1.} $I\subset \widetilde A$. Then $\phi(I)\subset \widetilde{A_{j_l}}$. Since $\widetilde A\subset\subset A\subset U$, by applying Koebe's distortion theorem to the conformal map $\phi^{-1}|_{\widetilde{A_{j_l}}}$, we see that $\phi$ has uniformly bounded relative distortion on $I$. Moreover, since $\phi(I)\subset \widetilde {A_{j_l}}\subset\subset  U_{j_l}$, we can apply again Koebe's distortion theorem to $f|_{U_{j_l}}$ and conclude that $f\circ \phi=f^{\circ l}$ has bounded relative distortion on $I$. The point $a= f^{\circ l}(c)$ lies in $F_1$, so we have arrived to  alternative \ref{lemma:blowup:two_sides}, with $m=l$. Using the conjugation between $f$ and $g$, we can perform the same combinatorial analysis and show that $g^{\circ m}$ has bounded relative distortion on $h(I)$. 

\bigskip

\noindent
\textbf{Case 2.} $I$ is not contained in $\widetilde A$.  By the choice of $q$, there exists at least one complementary arc of $F_q$ that separates $\phi(c)\in F_2$ from the endpoints of $\widetilde{A_{j_l}}$. Since $\phi(I)$ is not contained in $\widetilde{A_{j_l}}$, we conclude that there exists a complementary arc of $F_q$ that is contained in $\phi(I)\cap \widetilde{A_{j_l}}$.

Suppose that $A=\arc{[z_1,z_2]}$ and without loss of generality $z_1$ has the smallest level among the two endpoints, equal to $m+1\in \N$. We note that $m+1\leq l$, since $A$ is a complementary arc of $F_l$. The arc $A_{j_l}=\phi(A)$ consists of a uniformly bounded number of complementary arcs of $F_{q}$. Hence, $A$ is the union of a uniformly bounded number of consecutive complementary arcs of $F_{q+l-1}$ that we denote by $I_1,\dots,I_p$. We number them so that $I_1$ has $z_1$ as an endpoint. Note that by the previous discussion, there exists $i_0\in \{2,\dots,p\}$ such that $I_{i_0}\subset I$. In particular, $I\cap \bigcup_{i=2}^p I_i\neq \emptyset$. We set $n=q+l-1-m \geq q$ and we claim that alternative \ref{lemma:blowup:one_side} holds with the defined $m,n$ and the arcs $I_i$. 

If $m=0$, then the level of $z_1$ is $1$, thus $a=z_1\in F_1$, and there is nothing to be proved. So we assume that $m\geq 1$.
Consider the arc $A_{j_1\dots j_{m}}$, which is a complementary arc of $F_m$ and its endpoints have level at most $m$. It follows that $A_{j_1\dots j_m}$ contains the arc $A$ and the point $z_1$ in its interior. Let $\psi= f^{\circ (m-1)}$ and note that $\psi$ maps conformally $U_{j_1 \dots j_m}$ onto $U_{j_m}$. Since $\psi(z_1)\in F_2$ and $\psi(A)$ is a complementary arc of $F_{l-m+1}$, the other endpoint $\psi(z_2)$ of $\psi(A)$ has level $k$, where $2\leq k\leq l-m+1$. Note also that $\psi(z_1)\in F_2\cap \inter{A_{j_m}}$, so $\psi(z_1)\in \widetilde {A_{j_m}}$, by the definition of $ \widetilde {A_{j_m}}$.

If $k\leq q$, then $\psi(z_2)$ cannot be contained in the interior of a complementary arc of $F_q$ that separates the endpoints of $\widetilde {A_{j_m}}$ from the endpoints of $A_{j_m}$. Therefore, $\psi(z_2)\in \widetilde {A_{j_m}}$ and $\psi(A)$ is contained in $\widetilde {A_{j_m}}$. 

If $k\geq q+1$, then $l-m+1\geq q+1$, so the complementary arc $\psi(A)$ of $F_{l-m+1}$ cannot contain in its interior a point of level $q$ or less. Thus, the arc $\psi(A)$ cannot intersect the interior of the complementary arcs of $F_q$ that separate the endpoints of $\widetilde {A_{j_m}}$ from the endpoints of $A_{j_m}$. In this case, we also have that $\psi(A)\subset \widetilde {A_{j_m}}$.

It follows  as in Case 1 that $f\circ \psi=f^{\circ m}$ has bounded relative distortion on $\bigcup_{i=1}^p I_i$, which is contained in $\psi^{-1}(\widetilde {A_{ j_m}})$. The same conclusion holds for $g$, using the same combinatorial analysis. Moreover, since $I_{i_0} \subset I$, we have
\begin{align*}
\diam{f^{\circ m}(I)} \geq \diam{f^{\circ m}(I_{i_0})}.
\end{align*}
The point $a=f^{\circ m}(z_1)$ lies in $F_1$, and the arc $f^{\circ m}(A)=\bigcup_{i=1}^p f^{\circ m}(I_i)$ is contained in $\arc{[a,z_0]}$ or in $\arc{[z_0,a]}$ for some $z_0\neq a$. Since $i_0\neq 1$, by  Corollary \ref{corollary:hyperbolic_parabolic} we conclude that 
\begin{align*}
\diam {f^{\circ m}(I_{i_0})}\simeq \diam{f^{\circ m}(I_2)}.
\end{align*}
This proves the desired inequality
\begin{align*}
\diam{f^{\circ m}(I)} \gtrsim \diam {f^{\circ m}(I_{2})}.
\end{align*}

\bigskip
Finally, we prove the last statement of the lemma. In both cases, by the relative distortion bounds of $f^{\circ m}$ we have
\begin{align*}
\frac{\diam{f^{\circ m}(I)}}{\diam{I}} \simeq \frac{\diam{\widetilde {A_{j_m}}}}{\diam{\widetilde A}}.
\end{align*}
Since $\widetilde A\subset U\subset U_{j_1}$, we have
\begin{align*}
\frac{\diam{f^{\circ m}(I)}}{\diam{I}} \gtrsim \frac{\diam{\widetilde {A_{j_m}}}}{\diam{U_{j_1}}}.
\end{align*}
Let $M_0= \max\{ \diam{\widetilde{A_{k_1}}}/\diam{U_{k_2}}:k_1,k_2 \in \{0,\dots,r\}\}$. Then we have
$$\diam{f^{\circ m}(I)} \gtrsim M_0 \diam{I},$$
as desired. 
\end{proof}

\begin{lemma}[One-sided distortion]\label{lemma:one_sided_distortion}
Let $I,J\subset \mathbb S^1$ be adjacent closed arcs each of which has length $t\in (0,1/2)$. Suppose that alternative \refeq{lemma:blowup:one_side} of Lemma \refeq{lemma:blowup_simultaneous} occurs for the arc $I\cup J$ and consider points $a\in F_1$ and $b=h(a)$ as in Lemma \refeq{lemma:blowup_simultaneous}.
\begin{itemize}
\item If $a^{\pm}$ is hyperbolic and $b^{\pm}$ is hyperbolic, 
then 
\begin{align*}
\frac{\diam{h(I)}}{\diam{h(J)}} \simeq 1.
\end{align*}
\item If $a^{\pm}$ is parabolic and $b^{\pm}$ is parabolic, then 
\begin{align*}
\frac{\diam{h(I)}}{\diam{h(J)}} \simeq 1.
\end{align*}
\item If $a^{\pm}$ is hyperbolic and $b^{\pm}$ is parabolic, then
\begin{align*}
\max \left\{\frac{\diam{h(J)}}{\diam{h(I)}}, \frac{\diam{h(I)}}{\diam{h(J)}}\right\}  \lesssim \log(1/t).
\end{align*}
\end{itemize}  
\end{lemma}
\begin{proof}
Let $m,n,p$ be as in Lemma \ref{lemma:blowup_simultaneous} \ref{lemma:blowup:one_side} and set $I'=f^{\circ m}(I)$ and $J'=f^{\circ m}(J)$. Moreover, let $I_1,\dots,I_p$ be complementary arcs of $F_{n+m}$, $n\geq 1$, such that $I\cup J\subset \bigcup_{i=1}^p I_i$, $(I\cup J)\cap \bigcup_{i=2}^p I_i\neq \emptyset$, and set $I_i'=f^{\circ m}(I_i)$ for $i\in \{1,\dots,p\}$. Since $f^{\circ m}$ has bounded distortion on $I\cup J$, we have $\diam{I'}\simeq \diam{J'}$. Moreover, $\diam{I'}\gtrsim \diam {I_2'}$ and $\diam{J'}\gtrsim \diam {I_2'}$. We will work in the proof with $a^+$ and $b^+$, so $\bigcup_{i=1}^p I_i\subset \arc{[a,z_0]}$. We have $I'\cap \bigcup_{i=2}^p I_i' \neq \emptyset$ or $J'\cap \bigcup_{i=2}^p I_i' \neq \emptyset$. If $J'\cap \bigcup_{i=2}^p I_i' \neq \emptyset$ and $I'$ is not contained in the arc between $J'$ and the point $a$, then $I'$ also has this intersection property. Hence, by reversing the roles of $I'$ and $J'$ if necessary, we assume that $I'$ is closer to the point $a$ than $J'$, i.e., $I'$ is contained in the arc between $J'$ and $a$, and that $J'\cap \bigcup_{i=2}^p I_i' \neq \emptyset$. We note that since  $g^{\circ m}$ has bounded relative distortion on $h(I)\cup h(J)$ it suffices to derive the conclusions of the lemma for the arcs $h(I')$ and $h(J')$.

First, suppose that $a^+$ and $b^+$ are hyperbolic. Since $\diam{I'}\simeq \diam{J'}\gtrsim \diam{I_2'}$, we can apply Lemma \ref{lemma:one_sided_estimates} to each of the arcs $I',J'$. We obtain 
$$\diam{I'}\simeq \diam{J'}\simeq  \lambda(a^+)^{-n/q},$$
where $q$ is the orientation-preserving period of $a$. Since $b^+$ is hyperbolic, it follows from the same lemma that 
\begin{align*}
\diam{h(I')}\simeq \diam{h(J')} \simeq \lambda(b^+)^{-n/q}.
\end{align*}

Next, suppose that $a^+$ and $b^+$ are parabolic. We will apply again Lemma \ref{lemma:one_sided_estimates}. We let $k_1$ be the smallest integer such that a complementary arc of $F_{n+k_1}$ not having $a$ as an endpoint intersects $I'$. Similarly, we define $k_2$, corresponding to $J'$. Since $J'\cap \bigcup_{i=2}^p I_i'\neq \emptyset$, we conclude that $k_2=0$.  Moreover, let $l_1$ be the smallest integer such that there exists a complementary arc of $F_{n+k_1+l_1}$ having $a$ as an endpoint and not intersecting $I'$, and $l_2$ be the corresponding integer for $J'$. Since $I'$ and $J'$ are adjacent and $I'$ is between $J'$ and $a$, it follows that $\vert k_1-l_2\vert\leq 1$.  By Lemma \ref{lemma:one_sided_estimates} we have, since $\vert k_1-l_2\vert\leq 1$ and $k_2=0$,
\begin{align*}
n^{-\alpha-1} \min\{k_1+1,n\} \simeq \diam{J'} \simeq \diam{I'} \simeq (n+k_1)^{-\alpha-1}(\min\{l_1+1,n+k_1\}).
\end{align*}
If $k_1+1\geq n$, then we obtain
\begin{align*}
n^{-\alpha} \simeq  (k_1+1)^{-\alpha-1} (\min\{l_1+1,n+k_1\}) \lesssim (k_1+1)^{-\alpha-1}(k_1+n)\lesssim (k_1+1)^{-\alpha}.
\end{align*}
It follows that $n\gtrsim k_1+1$. So, in any case $0\leq k_1\lesssim n$. Therefore, from Lemma \ref{lemma:one_sided_estimates}  we conclude that
\begin{align*}
\frac{\diam{h(I')}}{\diam{I'}} \simeq n^{\alpha-\beta} \simeq \frac{\diam{h(J')}}{\diam{J'}},
\end{align*}
which implies that $\diam {h(I')}\simeq\diam {h(J')}$. 

Finally, suppose that $a^+$ is hyperbolic, but $b^+$ is parabolic. First, note that $\diam{I'} \simeq \diam{J'}\simeq  \lambda(a^+)^{-n/q}$ by Lemma \ref{lemma:one_sided_estimates}.  According to the very last inequality in Lemma \ref{lemma:blowup_simultaneous} \ref{lemma:blowup:one_side}, we have $\diam {I'}\gtrsim \diam{I}=t$. Hence, $\lambda(a^+)^{-n/q}\gtrsim t$. This implies that $\log(1/t) \gtrsim n$, since $t<1/2$ by assumption. Since $b^+$ is parabolic, by Lemma \ref{lemma:one_sided_estimates} we have, as in the previous paragraph, that there exist $k_1,k_2,l_1,l_2\in \N\cup \{0,\infty\}$ with $k_2=0$ and  $\vert k_1-l_2\vert\leq 1$ such that
\begin{align*}
\diam{h(I')}&\simeq (n+k_1)^{-\beta-1} (\min\{l_1+1,n+k_1\}) \quad \textrm{and}\\
\diam{h(J')}&\simeq n^{-\beta-1} \min\{k_1+1,n\}.
\end{align*}
Moreover, $k_1$ is uniformly bounded by Lemma \ref{lemma:one_sided_estimates}, so
\begin{align*}
\diam{h(I')}\simeq n^{-\beta-1} \min\{l_1+1,n\} \quad \textrm{and} \quad \diam{h(J')}\simeq n^{-\beta-1} \min \{k_1+1,n\}.
\end{align*}
It follows that 
\begin{align*}
\max \left\{ \frac{\diam{h(J')}}{\diam{h(I')}}, \frac{\diam{h(I')}}{\diam{h(J')}} \right\} \lesssim n\lesssim \log(1/t).
\end{align*}
The proof is complete.
\end{proof}

\begin{lemma}[Two-sided distortion]\label{lemma:two_sided_distortion}
Let $I,J\subset \mathbb S^1$ be adjacent arcs each of which has length $t\in (0,1/2)$. Suppose that alternative \textup{(i)} of Lemma \refeq{lemma:blowup_simultaneous} occurs for the arc $I\cup J$ and consider points $a\in F_1$ and $b=h(a)$ as in Lemma \refeq{lemma:blowup_simultaneous}.  Then we have 
\begin{align*}
\frac{\diam{h(I)}}{\diam{h(J)}} \simeq 1.
\end{align*}
under condition \ref{HH} for the points $a,b$, and 
\begin{align*}
\max \left\{\frac{\diam{h(J)}}{\diam{h(I)}}, \frac{\diam{h(I)}}{\diam{h(J)}}\right\}  \lesssim \log(1/t).
\end{align*}
under condition \ref{HP} for the points $a,b$.
\end{lemma}

\begin{proof}
We let $I'=f^{\circ m}(I)$ and $J'=f^{\circ m}(J)$, where $m$ is as in Lemma \refeq{lemma:blowup_simultaneous}. We have $\diam{I'}\simeq \diam{J'}$ by the relative distortion bounds of $f^{\circ m}$. We note that since  $g^{\circ m}$ has bounded relative distortion on $h(I)\cup h(J)$ it suffices to derive the conclusions of the lemma for the arcs $h(I')$ and $h(J')$.

Without loss of generality, suppose that $J'\subset \arc{[a,z_0]}$ and $I'= K^+\cup K^-$, where $K^+ \subset \arc{[a,z_0]}$ and $K^- \subset \arc{[z_0,a]}$. We may assume that $I',J'$ are sufficiently small arcs, so that $K^+\cup J'$ is contained in a complementary arc of $F_1$ and $K^-$ is contained in a complementary arc of $F_1$. Indeed, if $I',J'$ have diameters comparable to $1$, then by continuity, $h(I'),h(J')$ also have diameters comparable to $1$ and, thus, to each other.

We let $n-1\in \N$ be the largest integer such that there exists a complementary arc $J_1''$ of $F_{n-1}$ having $a$ as an endpoint and containing $J'$. Then, there exists $p\leq r+1$ and consecutive complementary arcs $J_1',J_2',\dots,J_p'\subset \arc{[a,z_0]}$ of $F_n$, which are the children of $J_1''$, such that $J'\subset \bigcup_{i=1}^p J_i'$ and $J'\cap \bigcup_{i=2}^p J_i'\neq \emptyset$. We note that $\diam{I'\cup J'} \geq \diam{J_1'}$, so $\diam{J'}\gtrsim \diam{J_1'}$.  In view of Corollary \ref{corollary:hyperbolic_parabolic}, we also have $\diam{J'}\gtrsim \diam{J_2'}$. Similarly, we can also find consecutive complementary arcs $K_1^-,\dots,K_{p_1}^-$ of $F_{n_1}$ such that $K^-\subset \bigcup_{i=1}^{p_1} K_i^-$, $K^-\cap \bigcup_{i=2}^{p_1}K_i^-\neq \emptyset$, and $\diam{K^-}\gtrsim \diam{K_2^-}$ and consecutive complementary arcs $K_1^+,\dots,K_{p_2}^+$ of $F_{n_2}$ such that $K^+\subset \bigcup_{i=1}^{p_2} K_i^+$, $K^+\cap \bigcup_{i=2}^{p_2}K_i^+\neq \emptyset$. and $\diam{K^+}\gtrsim \diam{K_2^+}$ for some $n_1,n_2\in \N$ and $p_1,p_2\leq r+1$.  We are exactly in the setting of  Lemma \ref{lemma:one_sided_estimates}. 

We first suppose that  the alternative \ref{HH} holds for the points $a^{\pm},b^{\pm}$. That is, there exists $\mu>0$ such that if $a^{\pm}$ is parabolic, then $b^{\pm}$ is necessarily parabolic with $\mu^{-1} N(a^{\pm})=N(b^{\pm})$ and if $a^{\pm}$ is hyperbolic, then $b^{\pm}$ is hyperbolic with $\lambda(a^{\pm})^{\mu}=\lambda(b^{\pm})$. Our goal is to prove that in all of these cases we have
\begin{align*}
\diam h(J') &\simeq (\diam J')^{\mu} \quad \textrm{and}\\
\diam h(K^{\pm}) &\simeq (\diam K^{\pm})^{\mu}.
\end{align*}
These imply that
\begin{align*}
\diam h(I')&\simeq \diam h(K^+)+\diam h(K^-) \simeq (\diam K^+)^\mu +(\diam K^-)^\mu \\
&\simeq (\diam I')^{\mu} \simeq (\diam J')^{\mu} \simeq \diam h(J'),
\end{align*}
which is the desired conclusion.

Suppose that $a^+$ and $b^+$ are hyperbolic. Then by Lemma \ref{lemma:one_sided_estimates} we have 
\begin{align*}
\diam{J'} \simeq \lambda(a^+)^{-n/q} \quad &\textrm{and} \quad \diam{h(J')} \simeq \lambda(b^+)^{-n/q}\simeq \lambda(a^+)^{-\mu n/q}, \quad \textrm{and}\\
\diam{K^+} \simeq \lambda(a^+)^{-n_2/q}\quad &\textrm{and} \quad  \diam{h(K^+)} \simeq \lambda(b^+)^{-n_2/q} \simeq \lambda(a^+)^{-\mu n_2/q}. 
\end{align*}
It follows that 
\begin{align*}
\diam h(J') &\simeq (\diam J')^{\mu} \quad \textrm{and}\\
\diam h(K^{+}) &\simeq (\diam K^{+})^{\mu}.
\end{align*}
If $a^-$ and $b^-$ are hyperbolic, then with the same argument we have
\begin{align*}
\diam h(K^{-}) \simeq (\diam K^{-})^{\mu}.
\end{align*}

Next, assume that $a^+$ and $b^+$ are parabolic. Consider $k,l\in \N \cup \{0,\infty\}$ as in Lemma \ref{lemma:one_sided_estimates} corresponding to $J'$ and and $k_2,l_2$ corresponding to $K^+$. Since $J'\cap \bigcup_{i=2}^p J_i' \neq \emptyset$, we have $k=0$. Similarly, $k_2=0$. Moreover, since $a$ is an endpoint of $K^+$, we have $l_2=\infty$. Therefore, if we set $\alpha^+=1/N(a^+)$ and $\beta^+=1/N(b^+)$, we have
\begin{align*}
\diam{J'} \simeq n^{-\alpha^+-1}\min\{l+1,n\} \quad &\textrm{and}\quad \diam{h(J')} \simeq n^{-\beta^+-1}\min \{l+1,n\}, \quad \textrm{and}\\
\diam{K^+} \simeq n_2^{-\alpha^+} \quad &\textrm{and}\quad \diam{h(K^+)} \simeq n_2^{-\beta^+}.
\end{align*}
We have  $\diam{J'} \simeq \diam(I')\gtrsim \diam(K^+)$. Therefore,
\begin{align}\label{lemma:two_sided_distortion:parabolic11}
n_2^{-\alpha^+}\lesssim  n^{-\alpha^+-1}\min\{l+1,n\}
\end{align}
Note that by the relative position of $K^+$ and $J'$, and the definition of $l$, there exists a complementary arc of $F_{n+l}$ having $a$ as an endpoint and being contained in $K^+$. This, combined with \eqref{inequality:parabolic_1} in Lemma \ref{lemma:adjacentparabolic}, gives $(n+l)^{-\alpha^+} \lesssim n_2^{-\alpha^+}$. This inequality and \eqref{lemma:two_sided_distortion:parabolic11} imply that $l+1 \gtrsim n$. Therefore,
\begin{align*}
\diam J'\simeq n^{-\alpha^+} \quad \textrm{and}\quad \diam h(J') \simeq n^{-\beta^+} \simeq n^{-\mu \alpha^+}.
\end{align*}
It follows that 
\begin{align*}
\diam h(J') &\simeq (\diam J')^{\mu} \quad \textrm{and}\\
\diam h(K^{+}) &\simeq (\diam K^{+})^{\mu}.
\end{align*}
If $a^-$ and $b^-$ are parabolic, with the same argument we have
\begin{align*}
\diam h(K^{-}) \simeq (\diam K^{-})^{\mu}.
\end{align*}

Finally, we treat the alternative \ref{HP}. Suppose that $a^+$ and $a^-$ are hyperbolic with the same multiplier
$\lambda(a)=\lambda(a^+)=\lambda(a^-)$ and $b$ is parabolic. We can apply Lemma \ref{lemma:one_sided_estimates} and obtain
\begin{align*}
\diam{J'}\simeq \lambda(a^+)^{-n/q} \quad &\textrm{and}\quad  \diam{h(J')} \simeq n^{-\beta-1} \min \{l+1,n\}, \\
\diam{K^+} \simeq \lambda(a^+)^{-n_2/q} \quad &\textrm{and}\quad \diam{h(K^+)} \simeq n_2^{-\beta}, \quad \textrm{and}\\
\diam{K^-} \simeq \lambda(a^-)^{-n_1/q} \quad &\textrm{and}\quad \diam{h(K^-)} \simeq n_1^{-\beta}, 
\end{align*}
where $\beta=1/N(b)$. Without loss of generality we suppose that $\diam{K^+}\leq \diam{K^-}$, which implies that $\diam{K^-}\simeq \diam{J'}$. From these we deduce that $n_2\gtrsim n_1$ and $n_1\simeq n$. The first condition implies that 
\begin{align*}
\diam{h(I')}\simeq \diam{h(K^-)}. 
\end{align*}
and the latter condition implies that 
\begin{align*}
\diam{h(K^-)} \simeq n^{-\beta}
\end{align*}
Altogether, we have
\begin{align*}
\frac{\diam{h(J')}}{\diam{h(I')}} \simeq n^{-1} \min \{l+1,n\}.
\end{align*}
This implies that 
\begin{align*}
\max \left\{\frac{\diam{h(J')}}{\diam{h(I')}}, \frac{\diam{h(I')}}{\diam{h(J')}}\right\}  \lesssim n.
\end{align*}
Finally, the inequality $\diam{J'}\gtrsim \diam{J}=t$ from Lemma \ref{lemma:blowup_simultaneous} implies that  $n\lesssim \log(1/t)$.
\end{proof}

\begin{proof}[Proof of Theorem \ref{theorem:extension_generalization}]
Let $I,J\subset \mathbb S^1$ be adjacent closed arcs, each of which has length $t\in (0,1/2)$. Under the condition \ref{HH}, we obtain 
\begin{align*}
\diam{h(I)}\simeq \diam{h(J)}
\end{align*}
by Lemma \ref{lemma:one_sided_distortion} and Lemma \ref{lemma:two_sided_distortion}. Therefore, 
\begin{align*}
\rho_h(t) \simeq 1.
\end{align*}
This implies that the map $h\colon \mathbb S^1\to \mathbb S^1$ is quasisymmetric and has a quasiconformal extension in $\D$.

If condition \ref{HP} is also allowed for some periodic points $a\in F_1$, then by the same distortion lemmas we obtain instead
\begin{align*}
\rho_h(t) \lesssim \log(1/t).
\end{align*}
By Theorem \ref{theorem:extension_david} we have that $h$ has a David extension in $\D$.
\end{proof}

\bigskip

\section{David welding}\label{welding_sec}

A homeomorphism $h\colon \mathbb S^1 \to \mathbb S^1$ is a \textit{welding homeomorphism} if there exists a Jordan curve $J$ and conformal homeomorphisms $H_1$ from $\D$ onto the interior of $J$ and $H_2$ from $\widehat{\C}\setminus  \br \D$ onto the exterior of $J$ so that $h=\widetilde{H_2}^{-1}\circ \widetilde{H_1}$, where $\widetilde{H_1}$ and $\widetilde{H_2}$ are the homeomorphic extensions of $H_1$ and $H_2$ to the closures of $\D$ and $\widehat{\C}\setminus \br \D$, respectively. The Jordan curve $J$ is called a \textit{welding curve} that corresponds to $h$. We note that in general the curve $J$ is not unique up to M\"obius transformations. However, if there is a welding curve $J$ corresponding to the welding homeomorphism $h$ that is conformally removable, then $J$ is unique up to M\"obius transformations.

The goal of this section is to prove the existence of a new class of welding homeomorphisms.

\begin{theorem}\label{theorem:welding}
Let $f,g\colon \mathbb{S}^1\to \mathbb{S}^1$ be expansive covering maps of the same degree and the same orientation, and $\mathcal P(f;\{a_0,\dots,a_r\})$, $\mathcal P(g;\{b_0,\dots,b_s\})$ be Markov partitions satisfying conditions \eqref{condition:uv} and \eqref{condition:holomorphic}. Assume that each periodic point $a\in \{a_0,\dots,a_r\}$ of $f$ and each periodic point $b\in \{b_0,\dots,b_s\}$ of $g$ is either symmetrically hyperbolic or symmetrically parabolic. Then any conjugating homeomorphism $h\colon \mathbb S^1\to \mathbb S^1$ between $f$ and $g$ is a welding homeomorphism and the corresponding welding curve is unique up to M\"obius transformations.  
\end{theorem}

Theorem~\ref{theorem:welding} will be derived from the proof of the following mateability result for piecewise analytic circle coverings of $\mathbb{S}^1$. This mateability theorem is a direct implication of Theorem \ref{theorem:extension_generalization}.

\begin{theorem}[Mating piecewise analytic circle maps]\label{theorem:mating_general}
Let $f,g\colon \mathbb{S}^1\to \mathbb{S}^1$ be expansive covering maps of the same degree and the same orientation, and let $\mathcal P(f;\{a_0,\dots,a_r\})$, $\mathcal P(g;\{b_0,\dots,b_s\})$ be Markov partitions satisfying conditions \eqref{condition:uv} and \eqref{condition:holomorphic} where each periodic point $a\in \{a_0,\dots,a_r\}$ of $f$ and each periodic point $b\in \{b_0,\dots,b_s\}$ of $g$ is either symmetrically hyperbolic or symmetrically parabolic. Further, let $h:\mathbb{S}^1\to\mathbb{S}^1$ be a topological conjugacy between $f$ and $g$. 
Then $f$ and $g$ are conformally mateable, so that for each $x\in\mathbb{S}^1$ the point $x$ is mated with the point $y=h(x)$.
\end{theorem}

Recall that by condition \eqref{condition:uv}, $f$ (respectively, $g$) has conformal extensions in open neighborhoods of the arcs $\arc{(a_k,a_{k+1})}$, $k\in \{0,\dots,r\}$ (respectively, $\arc{(b_i,b_{i+1})}$, $i\in\{0,\cdots, s\}$).  The \textit{conformal mateability} of $f$ and $g$ in this theorem means that there exist a Jordan curve $J$, a partition of $J$ into open arcs $J_0,\cdots,J_l$, and a map $R$ that is analytic in an open neighborhood $W_m$ of $J_m$, $m\in \{0,\dots,l\}$, such that $R$ is conformally conjugate to $f$ on $W_m\cap \Int J $ and conformally conjugate to $g$ on $W_m\cap \Ext J$; here $\Int J$ and $\Ext J$ denote the interior and exterior open regions of the Jordan curve $J$, respectively. The map $R$ need not be defined in the entire sphere, although it might extend analytically to open sets that are larger than $W_m$. However, the proof given below shows that if $f$ and $g$ are restrictions of Blaschke products, then the map $R$ that realizes the mating is analytic everywhere and hence it is rational.
\begin{figure}[ht!]
\begin{tikzpicture}
\node[anchor=south west,inner sep=0] at (0,0) {\includegraphics[width=0.88\textwidth]{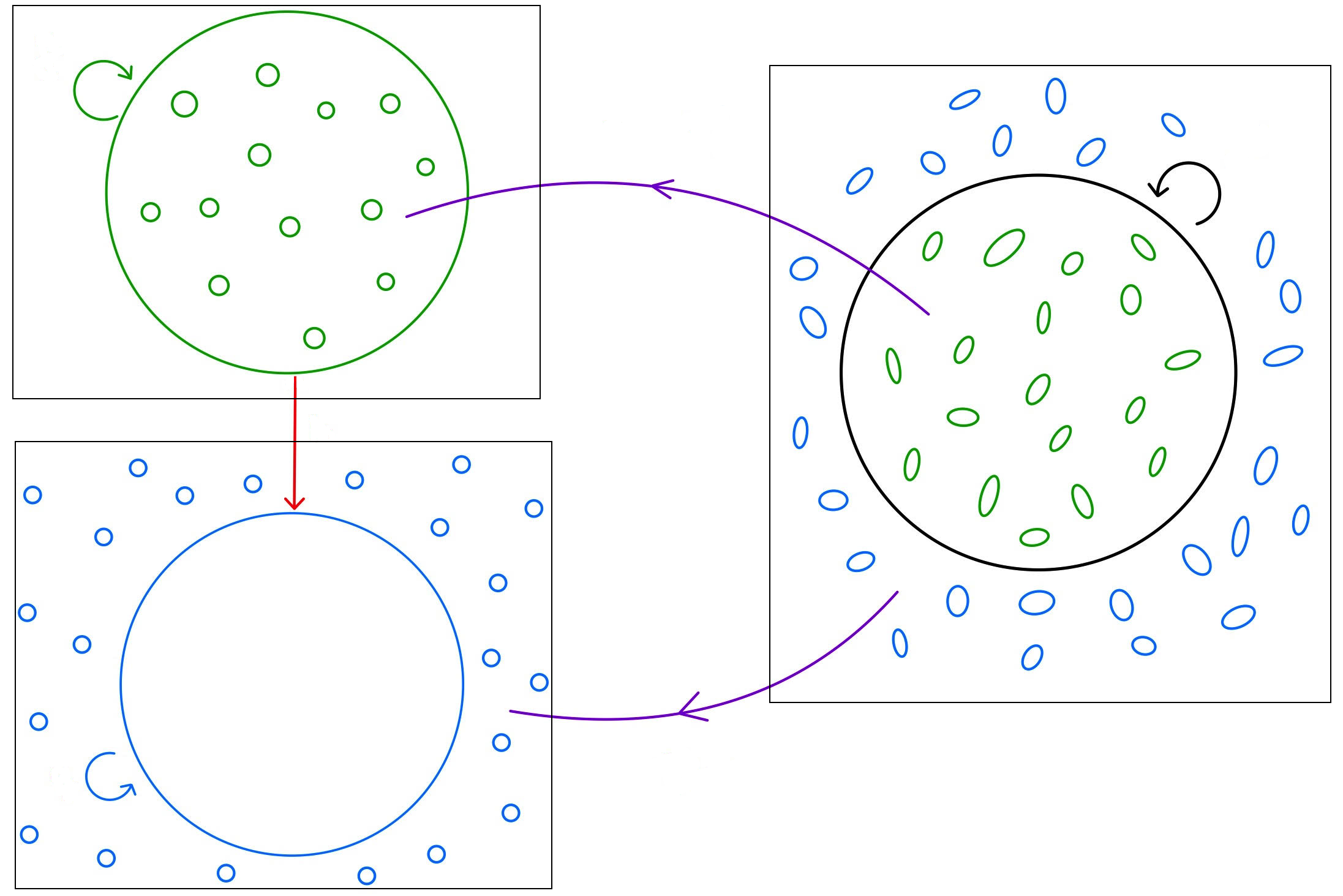}};
\node at (0.5,1) {$g$};
\node at (0.4,6.72) {$f$};
\node at (2.2,3.96) {$h$};
\node at (4.2,0.25) {$\D^*$};
\node at (2.4,5) {$\D$};
\node at (6,2) {$\widetilde{h}_2$};
\node at (6,1.2) {\textrm{David}};
\node at (5.6,5.5) {$h_1$};
\node at (5.6,6.2) {\textrm{David}};
\node at (10.32,6) {$P$};
\draw[->,line width=0.6pt] (8.64,1.8)->(8.64,0.6);
\node at (8.32,1.05) {$H$};
\node at (9.7,1.28) {\textrm{David}};
\node at (9.8,0.9) {\textrm{straightening}};
\node at (8.64,0.25) {\begin{large}$R$\end{large}};
\end{tikzpicture}
\caption{Illustrated is the proof of Theorem~\ref{theorem:mating_general}.}
\label{mating_fig}
\end{figure}
\begin{proof}
Let $P(z)$ be the map $z\mapsto z^d$ or $z\mapsto \br z^d$, depending on whether $f$ and $g$ are orientation-preserving or orientation-reversing, respectively. By Theorem \ref{theorem:extension_special_case} there exist orientation-preserving homeomorphisms $h_i\colon \mathbb S^1\to \mathbb S^1$, $i\in \{1,2\}$, that conjugate $P$ to $f$ and $g$, respectively, and have David extensions in $\D$. Note that $h_2$ and $h\circ h_1$ conjugate $P$ to $g$. By the uniqueness part in property \ref{theorem:expansive_conjugate}, we may precompose $h_2$ with a rotation so that it agrees with $h\circ h_1$. This will guarantee that each $x\in\mathbb{S}^1$ is mated with $y=h(x)$ through the construction below.

We define a Beltrami coefficient $\mu$ in the sphere as follows. In $\D$ we let $\mu$ be the pullback of the standard complex structure under the David homeomorphism $h_1$. In $\widehat{\C}\setminus \br{\D}$ we let $\mu$ the pullback of the standard complex structure under map $\widetilde h_2=(1/\br z)\circ h_2\circ (1/\br z)$. By Proposition \ref{prop:david_qc_invariance} (i) and (ii) (which also hold in the orientation-reversing setting), this map is a David homeomorphism. Therefore, $\mu$ is a David coefficient on $\widehat{\C}$. 

By the David Integrability Theorem \ref{theorem:integrability_david} there exists a David homeomorphism $H$ of $\widehat{\C}$ with $\mu_H=\mu$. Consider the map
\begin{align*}
R= \begin{cases}   H \circ h_1^{-1}\circ f\circ h_1\circ H^{-1},  & \textrm{in}\,\, H(\D)\\
 H \circ \widetilde{h_2}^{-1}\circ g\circ \widetilde h_2\circ H^{-1}, &\textrm{in}\,\, H(\widehat{\C}\setminus \D). 
\end{cases}
\end{align*}
The two definitions agree on $H(\mathbb S^1)$. Since $f$ and $g$ are not necessarily globally defined in $\D$, we have to further restrict the domain of $R$ to small neighborhoods of  suitable Jordan arcs $J_0,\cdots,J_l$ that form a partition of $\mathbb{S}^1$. In general, these arcs are obtained as a common refinement of the partitions $\{H( \arc{h_1^{-1}(a_k), h_1^{-1}(a_{k+1})}):\ k\in \{0,\dots,r\}\}$ and $\{H( \arc{\widetilde h_2^{-1}(b_i), \widetilde h_2^{-1}(b_{i+1})}):\ i\in \{0,\dots,s\}\}$. We claim that each $J_m$ has an open neighborhood $W_m$ in which $R$ is analytic. Moreover, we claim that $R$ is conformally conjugate to $f$ in $W_m\cap  H(\D)$ and conformally conjugate to $g$ in $W_m\cap H(\widehat{\C}\setminus \br \D)$.

By Theorem \ref{theorem:stoilow}, the map $h_1\circ H^{-1}$ is conformal in $H(\D)$ and the map $\widetilde h_2\circ H^{-1}$ is conformal in $H(\widehat{\C}\setminus \br \D)$. This proves the claims regarding the conformal conjugacy. 

By \eqref{condition:uv} $f$ extends conformally to a neighborhood of each arc $\arc{(a_k,a_{k+1})}$ and $g$ extends conformally to a neighborhood of each arc $\arc{(b_i,b_{i+1})}$. This implies that $R$ extends to a homeomorphism in a neighborhood $W_m$ of $J_m$ that is conformal in $W_m\setminus H(\mathbb S^1)$. Since $\mathbb S^1$ is removable for $W^{1,1}$ functions (e.g., it bounds a John domain), we conclude from Theorem \ref{theorem:w11_removable} that $H(\mathbb S^1)$ is locally conformally removable.  This implies that $R$ is conformal in $W_m$.
\end{proof}

\begin{proof}[Proof of Theorem~\ref{theorem:welding}]
From the proof of Theorem~\ref{theorem:mating_general}, we see that the map $h_1\circ H^{-1}$ is conformal in $H(\D)$ and the map $\widetilde h_2\circ H^{-1}$ is conformal in $H(\widehat{\C}\setminus \br \D)$. Therefore, 
\begin{align*}
h=\widetilde h_2\circ h_1^{-1}=  (\widetilde h_2\circ H^{-1}) \circ (h_1\circ H^{-1})^{-1}.
\end{align*}
This implies that the conjugating homeomorphism $h$ between $f$ and $g$ is a welding homeomorphism. Moreover, the welding curve $H(\mathbb S^1)$ of $h$ is conformally removable by Theorem \ref{theorem:w11_removable}, since it is the image of the unit circle under a David homeomorphism. 
\end{proof}

\begin{example}\label{pine_example}
For an illustration, consider the Blaschke product
\begin{align*}
B(z)=\frac{2z^3+1}{z^3+2},
\end{align*}
which is expansive on $\mathbb S^1$ by Example \ref{example:expansive_blaschke}. Moreover, there exists a Markov partition $\mathcal P(B;\{a_0,\dots,a_5\})$ satisfying \eqref{condition:uv} and \eqref{condition:holomorphic} by Example \ref{example:blaschke}, where $a_0,\dots, a_5$ are the $3$-rd roots of $1$ and $-1$. The point $1$ is a symmetrically parabolic fixed point of $B$ and $-1$ is a symmetrically hyperbolic fixed point. Finally, we define a bijective map $h\colon \{a_0,\dots,a_5\} \to \{a_0,\dots, a_5\}$ that preserves the orientation of $\mathbb S^1$ and maps $1$ to $-1$ and $-1$ to $1$. It is easy to see that $h$ conjugates $B$ to itself on the set $\{a_0,\dots,a_5\}$, and hence extends to a homeomorphism of $\mathbb{S}^1$ that conjugates $B$ to itself. By Theorem \ref{theorem:mating_general}, we can mate the Blaschke product $B$ with itself so that the point $1$ is mated with $-1$ and the point $-1$ is mated with the point $1$. In this case, since $B$ is analytic in the entire disk, the  mating is realized by a rational $R$; see the comments after the statement of Theorem \ref{theorem:mating_general}. One can in fact compute this rational map, after doing some normalizations and obtain the formula
\begin{align*}
R(z)=\frac{4z^3+8-3(1-\sqrt{3})}{(1-\sqrt{3})z^3+8+4\sqrt{3}}.
\end{align*}
The Julia set of $R$ is shown in Figure \ref{figure:two_sided_cusps} (the pine tree Julia set). It is a Jordan curve with both inward and outward cusps. These cusps arise from the mating of the parabolic point $1$ with the hyperbolic point $-1$. Note that this is a conformally removable Jordan curve, since it is the image of the unit circle under a David homeomorphism.  
\end{example}

We will now show that under certain regularity assumptions, the mating of parabolic points with hyperbolic points of two circle maps as in Theorem \ref{theorem:mating_general} produces~cusps.

Let $U\subset \widehat{\C}$ be a Jordan domain and let $x_0\in\partial U$. We say that $U$ has an \emph{untwisted inward cusp at $x_0$} if there exists $\theta_0\in\R/2\pi\Z$ and for each $\epsilon\in(0,\pi)$ there exists $\delta>0$ such that $\{x_0+re^{i(\theta_0+\theta)}: 0<r<\delta,\ \vert\theta\vert<\pi-\epsilon\}\subset U$. In that case we say that the region $\widehat{\C}\setminus \br U$ \emph{has an untwisted outward cusp at $x_0$} and that the curve $\partial U$ \emph{has an untwisted cusp at $x_0$}. 

\begin{prop}\label{para_hyp_cusp_prop}
Using the notation of Theorems~\ref{theorem:welding} and~\ref{theorem:mating_general}, suppose that $f, g$ admit (anti-)conformal extensions in open neighborhoods of $1, h(1)$, respectively.

\begin{enumerate}[\upshape(1)]
	\item If $1$ is a hyperbolic fixed point of $f$ and $h(1)$ is a parabolic fixed point of $g$ with an attracting direction in $\D$, then the region $H( \D^*)$ has an untwisted inward cusp at $H(h_1^{-1}(1))$.
	\item If both $1, h(1)$ are parabolic fixed points of $f, g$, respectively, then the welding curve $H(\mathbb{S}^1)$ does not have an untwisted cusp at $H(h_1^{-1}(1))$.
\end{enumerate}

\end{prop}
\begin{proof}
The assumption that $f, g$ admit (anti-)conformal extensions in open neighborhoods of $1, h(1)$ guarantees that the mating $R$ has a (anti-)conformal extension in a neighborhood of $x_0:=H(h_1^{-1}(1))$.

(1) In this case, the restriction of $R$ in a small neighborhood of $x_0$ is a parabolic germ of multiplicity $N\geq 2$ having all of its $N-1$ attracting directions in $H(\D^*)$. By the classical theory of parabolic germs, each of the corresponding attracting petals can be well-approximated at $x_0$ by a sector of angle $2\pi/(N-1)$. Hence, the domain $H(\D^*)$ can be well-approximated at $x_0$ by a sector of angle $2\pi$; i.e., the region $H(\D^*)$ has an untwisted inward cusp at $x_0$.

(2) In this case, the restriction of $R$ in a small neighborhood of $x_0$ is a parabolic germ of multiplicity $N\geq 2$ having attracting directions on both sides of $H(\mathbb{S}^1)$; note that due to the assumption that $f$ and $g$ are expansive, the directions tangent to $\mathbb S^1$ are repelling. As each of these attracting petals can be well-approximated at $x_0$ by a sector of angle $2\pi/(N-1)$, it follows that $H(\mathbb{S}^1)$ does not have an untwisted cusp at $x_0$. 
\end{proof}

A variant of the above result, which covers an important class of matings of circle maps, is proved in Subsection~\ref{cauli_triangle_mating_subsec}.
In general, we pose the following problem.

\begin{problem}
Let $f, g$ be as in Theorems~\ref{theorem:welding} and~\ref{theorem:mating_general}.
Does the mating of a hyperbolic fixed point of $f$ and a parabolic fixed point of $g$ necessarily produce an untwisted cusp (or a more general cusp in the sense of \cite{SS87}) on the resulting welding curve?
\end{problem}

\bigskip

\section{Reflection groups and Schwarz reflection maps}\label{kissing_group_sschwarz_sec}

In this section, we will collect some preliminaries on Kleinian reflection groups and Schwarz reflection maps associated with quadrature domains. 

\subsection{Kleinian reflection groups}

We denote by $\textrm{Aut}^\pm(\widehat{\C})$ be the group of all M{\"o}bius and anti-M{\"o}bius automorphisms of $\widehat{\C}$, correspondingly. 

\begin{definition}\label{Kleinian_reflection_group} 
A discrete subgroup $\Gamma$ of $\textrm{Aut}^\pm(\widehat{\C})$ is called a \emph{Kleinian reflection group} if $\Gamma$ is generated by reflections in finitely many Euclidean circles.
\end{definition}

\begin{remark}\label{3d_discrete_rem}
1) For a Euclidean circle $C$, consider the upper hemisphere $S\subset\mathbb{H}^3:=\{(x,y,t)\in\R^3: t>0\}$ such that $\partial S\cap\partial\mathbb{H}^3= C$; i.e., $C$ bounds the upper hemisphere $S$. Note that the anti-M{\"o}bius reflection $r$ with respect to $C$ extends naturally to reflection in $S$, and defines an orientation-reversing isometry of $\mathbb{H}^3$. Hence, a Kleinian reflection $\Gamma$ group can be thought of as a $3$-dimensional hyperbolic reflection group. 

2) Since $\Gamma$ is discrete, by \cite[Part~II, Chapter~5, Proposition~1.4]{VS93}, we can and will always choose its generators to be reflections in Euclidean circles $C_1,\cdots,C_{d+1}$ such that for each $i$, the closure of the bounded component of $\mathbb{C}\setminus C_i$ does not contain any other $C_j$.
\end{remark}

\begin{definition}\label{limit_regular_def} Let $\Gamma$ be a Kleinian reflection group. The \emph{domain of discontinuity} of $\Gamma$, denoted $\Omega(\Gamma)$, is the maximal open subset of $\widehat{\C}$ on which $\Gamma$ acts properly discontinuously (equivalently, the transformations in $\Gamma$ form a normal family). The \emph{limit set} of $\Gamma$, denoted by $\Lambda(\Gamma)$, is defined by $\Lambda(\Gamma):=\widehat{\mathbb{C}}\setminus\Omega(\Gamma)$. 
\end{definition}

Recall that for a Euclidean circle $C$, the bounded complementary component of $C$ is denoted by $\Int{C}$. A circle packing is a connected collection of oriented circles in $\C$ with disjoint interiors (where the interior is determined by the orientation). Up to a M{\"o}bius map, we can always assume that no circle of the circle packing contains $\infty$ in its interior; i.e., the interior of each circle $C$ of the circle packing can be assumed to be the bounded complementary component $\Int{C}$. 

\begin{definition}\label{kissing_group_def}
A \emph{kissing reflection group} is a group generated by reflections in the circles of a finite circle packing (with at least three circles).
\end{definition}

Combinatorially, a circle packing can be described by its \emph{contact graph}, where we associate a vertex to each circle, and connect two vertices by an edge if and only if the two associated circles intersect. By the Circle Packing Theorem, every connected, simple, planar graph is the contact graph of some circle packing. According to \cite[Proposition~3.4]{LLM20}, the limit set of a kissing reflection group is connected if and only if the contact graph of the underlying circle packing is $2$-connected (i.e., the contact graph remains connected if any vertex is deleted). 

Let $\Gamma$ be a kissing reflection group generated by reflections in the circles $C_1$, $\cdots$, $C_{d+1}$. Set $$\mathcal{F}_\Gamma:=\displaystyle\widehat{\C}\setminus\left(\bigcup_{i=1}^{d+1}\Int{C_i}\bigcup_{j\neq k} (C_j \cap C_k)  \right).$$

\begin{prop}\label{fund_dom_prop}
Let $\Gamma$ be a kissing reflection group. Then $\mathcal{F}_\Gamma$ is a fundamental domain for the action of $\Gamma$ on $\Omega(\Gamma)$.
\end{prop}
\begin{proof}
Let $\mathcal{P}_\Gamma$ be the convex hyperbolic polyhedron (in $\mathbb{H}^3$) whose relative boundary in $\mathbb{H}^3$ is the union of the hyperplanes $S_i$ bounded by the circles $C_i$ (see Remark~\ref{3d_discrete_rem}). Then, by \cite[Part~II, Chapter~5, Theorem~1.2]{VS93}, $\mathcal{P}_\Gamma$ is a fundamental domain for the action of $\Gamma$ on $\mathbb{H}^3$. It now follows that $\mathcal{F}_\Gamma=\overline{\mathcal{P}_\Gamma}\cap\Omega(\Gamma)$ (where the closure is taken in $\Omega(\Gamma)\cup\mathbb{H}^3$) is a fundamental domain for the action of $\Gamma$ on $\Omega(\Gamma)$ \cite[\S 3.5]{Mar16}. 
\end{proof}

To a kissing reflection group $\Gamma$, we can associate a piecewise anti-M{\"o}bius reflection map $\rho_\Gamma$ that will play an important role in the paper.

\begin{definition}\label{reflection_map} Let $\Gamma$ be a kissing reflection group generated by reflections $(r_i)_{i=1}^{d+1}$ in circles $(C_i)_{i=1}^{d+1}$. We define the associated \emph{Nielsen map} $\rho_\Gamma$ by: \begin{align} \rho_{\Gamma} : \bigcup_{i=1}^{d+1} \overline{\Int{C_i}} \rightarrow \widehat{\mathbb{C}} \nonumber \hspace{16mm} \\ \hspace{10mm} z\longmapsto r_i(z) \textrm{        if } z \in \overline{\Int{C_i}}.  \nonumber\end{align}
\end{definition}

\begin{definition}\label{necklace_group} 
Let $\Gamma$ be a kissing reflection group generated by reflections in the circles of a finite circle packing $C_1,\cdots, C_{d+1}$. We say that $\Gamma$ is a \emph{necklace group} if 
\begin{enumerate}\upshape
\item each circle $C_i$ is tangent to $C_{i+1}$ (with $i+1$ taken mod $(d+1)$), and

\item the boundary of the unbounded component of $\mathbb{C}\setminus\bigcup_iC_i$ intersects each $C_i$. 
\end{enumerate} 
\end{definition}

\begin{remark}\label{outerplanar_2_conn_rem}
Necklace groups can be characterized as Kleinian reflection groups generated by reflections in the circles of a finite circle packing whose contact graph is \emph{$2$-connected} and \emph{outerplanar}; i.e., the contact graph remains connected if any vertex is deleted, and has a face containing all the vertices on its boundary. Necklace groups can be compared with complex polynomials with connected Julia set. Indeed, the connected Julia set of a polynomial is the boundary of a simply connected, completely invariant Fatou component, namely, the basin of attraction of infinity. Similarly, as we will shortly see, the limit set of a necklace group is the boundary of a simply connected, invariant component of its domain of discontinuity (compare \cite[Theorem~1.2]{LLM20}).
\end{remark}

We conclude this subsection with the definition of the regular ideal polygon reflection group, which will play a central role in the rest of the paper.

\begin{definition}\label{ideal_group} Consider the Euclidean circles $\mathbf{C}_1,\cdots, \mathbf{C}_{d+1}$ where $\mathbf{C}_j$ intersects $\mathbb{S}^1$ at right angles at the roots of unity $\exp{(\frac{2\pi i\cdot(j-1)}{d+1})}$, $\exp{(\frac{2\pi i\cdot j}{d+1})}$. By \cite[Part~II, Chapter~5, Theorem~1.2]{VS93}, the group generated by reflections in these circles is discrete. Therefore, it is also a necklace group, and we denote it by $\pmb{\Gamma}_{d+1}$. Moreover, denoting the reflection map in the circle $\mathbf{C}_j$ by  $\rho_j$, we have the following presentation
$$
\pmb{\Gamma}_{d+1}=\langle \rho_1,\cdots, \rho_{d+1}: \rho_1^2=\cdots=\rho_{d+1}^2=1\rangle.
$$
\end{definition}

To ease notations, we will denote the Nielsen map of $\pmb{\Gamma}_{d+1}$ by $\pmb{\rho}_d$. Note that $\pmb{\rho}_d$ restricts to a degree $d$ orientation-reversing covering of $\mathbb{S}^1$.

\bigskip

\subsection{Quadrature domains and Schwarz reflections}\label{schwarz_def}

\begin{definition}[Schwarz Function]
Let $\Omega\subsetneq\widehat{\C}$ be a domain such that $\infty\notin\partial\Omega$ and $\inter{\overline{\Omega}}=\Omega$. A \emph{Schwarz function} of $\Omega$ is a meromorphic extension of $\overline{z}\vert_{\partial\Omega}$ to all of $\Omega$. More precisely, a continuous function $S:\overline{\Omega}\to\widehat{\C}$ of $\Omega$ is called a Schwarz function of $\Omega$ if it satisfies the following two properties:
\begin{enumerate}\upshape
\item $S$ is meromorphic on $\Omega$,

\item $S(z)=\overline{z}$ on $\partial \Omega$.
\end{enumerate}
\end{definition}

It is easy to see from the definition that a Schwarz function of a domain (if it exists) is unique. 

\begin{definition}[Quadrature Domains]
A domain $\Omega\subsetneq\widehat{\C}$ with $\infty\notin\partial\Omega$ and $\inter{\overline{\Omega}}=\Omega$ is called a \emph{quadrature domain} if $\Omega$ admits a Schwarz function.
\end{definition}

Therefore, for a quadrature domain $\Omega$, the map $\sigma:=\overline{S}:\overline{\Omega}\to\widehat{\C}$ is an anti-meromorphic extension of the Schwarz reflection map with respect to $\partial \Omega$ (the reflection map fixes $\partial\Omega$ pointwise). We will call $\sigma$ the \emph{Schwarz reflection map of} $\Omega$.

\begin{figure}
\begin{tikzpicture}
\node at (0.9,3.1) {\begin{Large}$\overline{\D}$\end{Large}};
\node at (3.6,3.1) {\begin{Large}$\overline{\Omega}$\end{Large}};
\node at (0.9,1) {\begin{Large}$\widehat{\C}\setminus\D$\end{Large}};
\node at (3.6,1) {\begin{Large}$\widehat{\C}$\end{Large}};
\draw[->,line width=1pt] (1.28,3.1)--(3.3,3.1);
\draw[->,line width=1pt] (3.6,2.8)-|(3.6,1.32);
\draw[->,line width=1pt] (0.9,2.8)-|(0.9,1.32);
\draw[->,line width=1pt] (1.48,1)--(3.3,1);
\node at (2.4,3.5) {\begin{large}$f$\end{large}};
\node at (2.4,0.7) {\begin{large}$f$\end{large}};
\node at (0.45,2) {\begin{large}$1/\overline{z}$\end{large}};
\node at (4,2) {\begin{Large}$\sigma$\end{Large}};
\end{tikzpicture}
\caption{The rational map $f$ semiconjugates the reflection map $1/\overline{z}$ of $\D$ to the Schwarz reflection map 
$\sigma$ of $\Omega$ .}
\label{comm_diag_schwarz}
\end{figure}
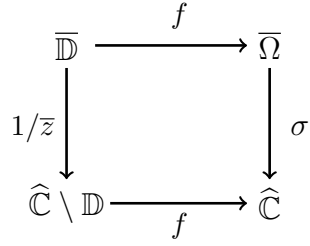

Simply connected quadrature domains are of particular interest, and these admit a simple characterization (see \cite[Theorem~1]{AS76}).

\begin{prop}[Simply Connected Quadrature Domains]\label{simp_conn_quad}
A simply connected domain $\Omega\subsetneq\widehat{\C}$ with $\infty\notin\partial\Omega$ and $\inter{\overline{\Omega}}=\Omega$ is a quadrature domain if and only if the Riemann uniformization $f:\mathbb{D}\to\Omega$ extends to a rational map on $\widehat{\C}$. 

In this case, the Schwarz reflection map $\sigma$ of $\Omega$ is given by $f\circ(1/\overline{z})\circ(f\vert_{\D})^{-1}$. Moreover, if the degree of the rational map $f$ is $d$, then $\sigma:\sigma^{-1}(\Omega)\to\Omega$ is a (branched) covering of degree $(d-1)$, and $\sigma:\sigma^{-1}(\inter{\Omega^c})\to\inter{\Omega}^c$ is a (branched) covering of degree $d$.
\end{prop}

\begin{remark}\label{prelim_quad}
If $\Omega$ is a simply connected quadrature domain with associated Schwarz reflection map $\sigma$, and $M$ is a M{\"o}bius transformation, then $M(\Omega)$ is also a quadrature domain with Schwarz reflection map $M\circ\sigma\circ M^{-1}$.
\end{remark}

\bigskip

\section{A general David surgery}\label{david_surgery_sec}

This section is devoted to the proof of a technical lemma, that will serve as a motor for much of the rest of the paper. The result is sufficiently general to allow us to construct various kinds of conformal dynamical systems by replacing the attracting dynamics of an anti-rational map on suitable invariant Fatou components by Nielsen maps of ideal polygon groups. 

We denote the Julia set of a rational or an anti-rational map $R$ by $\mathcal{J}(R)$. Recall that a rational map is called \emph{subhyperbolic} if every critical orbit is either finite or converges to an attracting periodic orbit.

\begin{lemma}\label{david_surgery_lemma}
Let $R$ be a subhyperbolic anti-rational map with connected Julia set, and let $U_i,\ i=1,\cdots, k$, be invariant Fatou components of $R$ such that $R\vert_{\partial U_i}$ have degree $d_i$. Then, there is a global David surgery that replaces the dynamics of $R$ on each $U_i$ by the dynamics of $\pmb{\rho}_{d_i}\vert_{\D}$, transferred to $U_i$ via a Riemann map.
\end{lemma}
\noindent More precisely, we will show that there exists a global David homeomorphism $\Psi$, and an anti-analytic map $F$, defined on a subset of $\widehat{\C}$, such that $F\vert_{\Psi(U_i)}$ is conformally conjugate to $\pmb{\rho}_{d_i}\vert_{\D}$, and $F$ is conformally conjugate by $\Psi$ to $R$ outside the grand orbit of $\bigcup_{i=1}^k\Psi(U_i)$.
\begin{proof}
For $i\in\{1,2,\dots, k\}$, let $\phi_i:\D\to U_i$ be a Riemann map. Note that since $R$ is subhyperbolic, each $U_i$ is the immediate basin of attraction of an attracting fixed point. Hence, the Riemann map $\phi_i$ conjugates a degree $d_i$ anti-Blaschke product $B_i$, with an attracting fixed point in $\D$, to $R$. In particular, $B_i$ is expanding (with respect to some conformal metric) on its Julia set $\mathbb{S}^1$. Since connected, subhyperbolic Julia sets are locally connected \cite[Theorem~19.7]{Mil06}, we have a continuous extension $\phi_i:\mathbb{S}^1\to \partial U_i$ semiconjugating $B_i$ to $R$.

By Theorem~\ref{theorem:reflections}, there is a topological conjugacy $H_i$ between $B_i\vert_{\mathbb{S}^1}$ and $\pmb{\rho}_{d_i}\vert_{\mathbb{S}^1}$ that continuously extends to $H_i:\D\to\D$ as a David homeomorphism.

We now define a new map $\widetilde{R}$ by modifying the map $R$ as follows:
$$
\widetilde{R}:=
\begin{cases}
\phi_i\circ\left( H_i^{-1}\circ \pmb{\rho}_{d_i}\circ H_i\right)\circ\phi_i^{-1},\quad {\rm in}\ U_i\setminus \inter{T_{i}},\quad i=1,2,\dots, k,\\
R,\quad {\rm in\ } \widehat{\mathbb C}\setminus\bigcup_{i=1}^k U_i,
\end{cases}
$$
where $T_{i}=\phi_i\left( H_i^{-1}\left(T^0(\pmb{\Gamma}_{d_i+1})\right)\right)$, and $T^0(\pmb{\Gamma}_{d_i+1})$ is a regular ideal $(d_i+1)$-gon in $\D$. The fact that $H_i^{-1}\circ \pmb{\rho}_{d_i}\circ H_i\equiv B_i$ on $\mathbb{S}^1$ combined with the semiconjugacy relation $\phi_i\circ B_i\equiv R\circ\phi_i$ on $\mathbb{S}^1$ imply that the piecewise definitions of the map $\widetilde{R}$ match continuously. Thus, $\widetilde{R}$ is a continuous orientation-reversing map of $\widehat{\mathbb C}\setminus\bigcup_{i=1}^k\inter{T_{i}}$ onto $\widehat{\mathbb C}$.
If $\mu_i$ is the pullback to $U_i,\ i\in\{1,2,\cdots, k\}$, of the standard complex structure on $\mathbb D$ by the map $H_i\circ\phi_i^{-1}$, then we have $(\widetilde{R}\vert_{U_i\setminus T_{i}})^*(\mu_i)=\mu_i$. 

We now use the dynamics of $R$ to spread the complex structure out to all the iterated preimage components of $U_i,\ i\in\{1,2,\cdots, k\}$, under all the iterates of $R$. 
Everywhere else we use the standard complex structure, i.e., the zero Beltrami coefficient. This way we obtain an $\widetilde{R}$-invariant measurable complex structure $\mu$ on $\widehat{\C}$.

We will now argue that $\mu$ is a David coefficient on $\widehat{\C}$; i.e., it satisfies condition~\eqref{definition:david2}. We say that an iterated preimage component $U'$ of $U_i$ ($i\in\{1,\cdots, k\}$) has \emph{rank} $r=r(U')$ if $r$ is the smallest non-negative integer $j$ such that $R^{\circ j}(U')=U_i$. By subhyperbolicity of $R$, the closures of at most finitely many iterated preimage components of $U_i$ may intersect the critical orbits of $R$. Let $m$ be the maximum of the ranks of the iterated preimage components of $U_i$ ($i\in\{1,\cdots, k\}$) whose closures intersect the critical orbits of $R$. We denote the iterated preimage components of $U_i$ ($i\in\{1,\cdots, k\}$) of rank at most $(m+1)$ by $V_1,\cdots, V_l$.

Since $R$ is subhyperbolic and $\mathcal{J}(R)$ is connected, each $U_i$ is a John domain \cite[\S 7, Theorem~3.1]{CG93}. By part (iv) of Proposition~\ref{prop:david_qc_invariance}, the map $H_{i}\circ\phi_{i}^{-1}:U_i\to\D$ is a David homeomorphism, and hence, $\mu$ is a David coefficient on $U_i$, for $i\in\{1,\cdots, k\}$. Moreover, since $\mu$ is defined on a preimage component $U'$ of $U_i$ by pulling back $\mu\vert_{U_i}$ by $R^{\circ r(U')}$, it follows that
$$
\mu\vert_{U'}=\left(H_{i}\circ\phi_{i}^{-1}\circ R^{\circ r(U')}\right)^\ast(\mu_0\vert_{\D}),
$$
where $\mu_0$ is the trivial Beltrami coefficient. By part (iii) of Proposition~\ref{prop:david_qc_invariance}, $\mu$ satisfies condition~\eqref{definition:david2} on $U'$. It now follows that $\mu$ satisfies condition~\eqref{definition:david2} on $\bigcup_{j=1}^l V_j$ with some constants $M,\alpha, \varepsilon_0>0$.

It remains to check the David condition on the union of the rest of the preimage components of $U_i$. If $U'$ is an iterated preimage component of $U_i$ ($i=1,2,\dots, k$) of rank larger than $(m+1)$, then it lands on a rank $(m+1)$ component in time $n(U')=(r(U')-m-1)$ univalently. In other words, $n(U')$ is the \emph{first landing time} of $U'$ in $\bigcup_{j=1}^l V_j$. By our choice of $m$, there is a neighborhood of the closures of the rank $(m+1)$ components that is disjoint from the postcritical set of $R$. Hence, by the Koebe distortion theorem, $R^{\circ n(U')}\circ \lambda_{U'}$ is an $L$-bi-Lipschitz map between $\frac{1}{\diam{U'}}U'$ and $V_j$, for some absolute constant $L\geq1$ and a rank $(m+1)$ component $V_j$, where $\lambda_{U'}(z)=\diam{U'}\cdot z$ is a scaling map. This implies that, given any $0<\varepsilon\leq\varepsilon_0$,  
$$
\sigma\left(\{z\in U'\colon|\mu(z)|\geq 1-\varepsilon\}\right)\leq L^2(\diam{U'})^2\sigma\left(\{z\in V_j\colon|\mu(z)|\geq 1-\varepsilon\}\right).
$$  
Moreover, since all the Fatou components $U'$ are uniform John domains \cite[Proposition~10]{Mih11}, it follows from \cite[p. 444]{Nta18} that there exists a constant $C>0$ such that
$$
(\diam{U'})^2\leq C\sigma\left(U'\right),
$$
for all preimage components $U'$ of $V_j$.

Therefore,
\begin{equation}\notag
\begin{split}
\sigma&\left(\{z\in\widehat{\mathbb C}\colon|\mu(z)|\geq 1-\varepsilon\}\right)\\&=\sum_{r(U')>m+1}\sigma\left(\{z\in{U'}\colon|\mu(z)|\geq 1-\varepsilon\}\right)+\sum_{j=1}^l \sigma\left(\{z\in{V_j}\colon|\mu(z)|\geq 1-\varepsilon\}\right)\\
&\leq \left(L^2 C\left(\sum_{U'}\sigma\left(U'\right)\right)+1\right)\cdot\left(\sum_{j=1}^l\sigma\left(\{z\in V_j\colon|\mu(z)|\geq 1-\varepsilon\}\right)\right)\\
&\leq \left(L^2Cl\sigma\left(\widehat{\mathbb C}\right)+l\right)Me^{-\alpha/\varepsilon},\quad \varepsilon\leq\varepsilon_0.
\end{split}
\end{equation}

Theorem~\ref{theorem:integrability_david} then gives us an orientation preserving homeomorphism $\Psi$ of $\widehat{\mathbb C}$ such that the pullback of the standard complex structure under $\Psi$ is equal to $\mu$.

We now proceed to show that the map $F:=\Psi\circ \widetilde{R}\circ \Psi^{-1}$ is anti-analytic. This is the desired map that replaces the dynamics of $R$ on each $U_i,\ i=1,2,\dots,k$, with the dynamics of the Nielsen map $\pmb{\rho}_{d_i}$. 

We will show that if $V$ is an open set in $\widehat{\C}\setminus\bigcup_{i=1}^k\inter{T_i}$ on which $\widetilde{R}$ restricts as a homeomorphism, then $F$ is anti-analytic on $\Psi(V)$. This will imply that $F$ is anti-analytic away from its critical points. Hence, by the Riemann removability theorem and the continuity of $F$, we can conclude that $F$ is anti-analytic on the interior of its domain of definition.

To this end, let $V$ be an open set in $\widehat{\C}\setminus\bigcup_{i=1}^k\inter{T_i}$ such that $\widetilde{R}\vert_V$ is a homeomorphism. Since $F\circ \Psi=\Psi\circ \widetilde{R}$, it is now enough to show that $\Psi\circ \widetilde{R}$ is an orientation-reversing David homeomorphism on $V$. Indeed, if both $\Psi$ and $\Psi\circ \widetilde{R}$ are David homeomorphisms on $V$ with opposite orientation, since both of them integrate $\mu$, we can apply Theorem~\ref{theorem:stoilow} to obtain that $F$ is anti-analytic on $\Psi(V)$. 

Since each $U_i$ is a John domain, Theorem~\ref{theorem:john_union_removable} tells us that $\bigcup_{i=1}^k\partial U_i$ is $W^{1,1}$--removable. Therefore, by Lemma~\ref{lemma:david_removable}, it now suffices to show that $\Psi\circ\widetilde{R}$ is an orientation-reversing David homeomorphism on $V\setminus\bigcup_{i=1}^k\partial U_i$.

Note that by construction, $\widetilde{R}$ is anti-conformal on $V\setminus\bigcup_{i=1}^k\overline{U_i}$. Hence, by part (ii) of Proposition~\ref{prop:david_qc_invariance}, $\Psi\circ\widetilde{R}$ is an orientation-reversing David homeomorphism on $V\setminus\bigcup_{i=1}^k\overline{U_i}$.

Now let us look at $V\cap\left(\bigcup_{i=1}^k U_i\right)$. Let us set $V_i:=V\cap U_i$, for $i\in\{1,\cdots, k\}$. We can factorize $\Psi\circ\widetilde{R}$ on $V_i$ as follows: 
\begin{equation}
\Psi\circ \widetilde{R}=(\Psi\circ \phi_i\circ H_i^{-1})\circ(\pmb{\rho}_{d_i}\circ H_i\circ \phi_i^{-1}).
\label{functional_factorization}
\end{equation}
Since $\phi_i^{-1}$ and $\pmb{\rho}_{d_i}$ are conformal and anti-conformal (respectively), the composition $\pmb{\rho}_{d_i}\circ H_i\circ \phi_i^{-1}$ is an orientation-reversing David homeomorphism by parts (i) and (ii) of Proposition~\ref{prop:david_qc_invariance}. By the same proposition, the map $H_i\circ\phi_i^{-1}$ is also a David homeomorphism. Moreover, both the maps $\Psi$ and $H_i\circ\phi_i^{-1}$ integrate $\mu$. Therefore, by Theorem~\ref{theorem:stoilow}, we conclude that $\Psi\circ \phi_i \circ H_i^{-1}$ is anti-conformal. By part (i) of Proposition~\ref{prop:david_qc_invariance} and Relation~\eqref{functional_factorization}, we now have that $\Psi\circ\widetilde{R}$ an orientation-reversing David homeomorphism on each $V_i$, and hence on $V\cap\left(\bigcup_{i=1}^k U_i\right)$. This completes the proof of the fact that $\Psi\circ\widetilde{R}$ is an orientation-reversing David homeomorphism on $V\setminus\bigcup_{i=1}^k\partial U_i$.
\end{proof}

\bigskip

\section{David surgery from anti-rational maps to kissing reflection groups}\label{anti_rat_group_surgery_sec}

An anti-rational map is called \emph{critically fixed} if all of its critical points are fixed. Each fixed critical point $c$ of a critically fixed anti-rational map $R$ lies in an invariant Fatou component $U_c$. If the local degree of $R$ at $c$ is $k$, then $R\vert_{U_c}$ is conformally conjugate to $\overline{z}^k\vert_{\D}$. In particular, there are $k+1$ fixed internal rays in $U_c$. The \emph{Tischler graph} of $R$ is defined as the union of all fixed internal rays in the invariant Fatou components. Clearly, this graph contains the postcritical set of $R$. It can be seen as a natural generalization of Hubbard trees of postcritically finite polynomials to the context of critically fixed maps. According to \cite[Lemma~4.9]{LLM20}, the planar dual of the Tischler graph of a critically fixed anti-rational map is simple and $2$-connected.

The goal of this section is to show that there is a global David surgery that turns a degree $d$ critically fixed anti-rational map $R$ into a kissing reflection group $\Gamma$ of rank $d+1$. 

\begin{theorem}\label{bijection_via_david_thm}
Suppose that the Tischler graph of a critically fixed anti-rational map $R$ is dual to the contact graph of the circle packing defining a kissing reflection group $\Gamma$. Then there exists a David homeomorphism of $\widehat{\C}$ that conjugates $R \vert_{\mathcal{J}(R)}$ to $\rho_\Gamma\vert_{\Lambda(\Gamma)}$.
\end{theorem}
\begin{proof}
Let $\Gamma$ be generated by reflections in the circles $C_1,\cdots, C_{d+1}$. We denote the connected components of the fundamental domain $\mathcal{F}_\Gamma$ by $\mathcal{P}_1,\cdots,\mathcal{P}_k$. Note that each $\mathcal{P}_i$ is a topological $(d_i+1)$-gon with vertices removed, for some $d_i\geq 2$, $i\in\{1,\cdots, k\}$. Thus, the faces of the contact graph of the underlying circle packing are topological $(d_i+1)$-gons, $i\in\{1,\cdots, k\}$. Since the Tischler graph of $R$ is dual to this contact graph, it follows that $R$ has $k$ invariant Fatou components $U_1,\cdots, U_k$ such that $R\vert_{\partial U_i}$ has degree $d_i$.

Applying Lemma~\ref{david_surgery_lemma} to the Fatou components $U_1,\cdots, U_k$, we obtain a global David homeomorphism $\Psi$ and an anti-meromorphic map $\sigma$ on a subset of $\widehat{\C}$ such that $\sigma$ is conformally conjugate to $\pmb{\rho}_{d_i}\vert_{\D}$ on $\Psi(U_i)$, $i\in\{1,\cdots, k\}$, and topologically conjugate to $R\vert_{\mathcal{J}(R)}$ on $\Psi(\mathcal{J}(R))$ (the conjugacy $\Psi$ maps the dynamical plane of $R$ to that of $\sigma$). It follows from the proof of Lemma~\ref{david_surgery_lemma} that the domain of definition of $\sigma$ is $\Psi\left(\widehat{\C}\setminus\bigcup_{i=1}^k\inter T_i\right)$, where each $T_i\subset U_i$ is a vertices-removed topological $(d_i+1)$-gon with vertices at the landing points of the $(d_i+1)$ fixed internal rays of $R$ in $U_i$. It follows from the construction that the domain of definition of $\sigma$ is naturally homeomorphic to $\displaystyle\bigcup_{i=1}^{d+1}\overline{\Int{C_i}}$. We denote these $(d+1)$ Jordan domains by $\Omega_1,\cdots, \Omega_{d+1}$ such that $\Omega_j$ corresponds to $\Int{C_j}$ under the canonical homeomorphism between the domain of definition of $\sigma$ and $\displaystyle\bigcup_{i=1}^{d+1}\overline{\Int{C_i}}$. In particular, the domain of definition of $\sigma$ is connected, and its interior is the union of $(d+1)$ disjoint Jordan domains. Moreover, $\sigma$ fixes the boundary of each of these domains pointwise. Thus, by definition, each of these Jordan domains is a quadrature domain, and $\sigma$ restricts to each such domain as its Schwarz reflection map.

By \cite[Corollary~4.7]{LLM20}, the map $R$ carries each face of its Tischler graph onto the closure of
its complement as an orientation reversing homeomorphism. Since $\Psi$ conjugates $R\vert_{\mathcal{J}(R)}$ to $\sigma\vert_{\Psi(\mathcal{J}(R))}$, it follows that $\sigma$ carries $\Psi(\mathcal{J}(R))\cap\Omega_j$ univalently to $\Psi(\mathcal{J}(R))\setminus\overline{\Omega_j}$.

By Proposition~\ref{simp_conn_quad}, there exist rational maps $\phi_j,\ j\in\{1,\cdots, d+1\}$, such that $\phi_j:\D\to\Omega_j$ is a conformal isomorphism. Furthermore, by the same proposition, the map $\sigma:\sigma^{-1}(\Omega_j)\to\Omega_j$ is a (branched) cover of degree $(\deg{\phi_j}-1)$. Since $\Psi(\mathcal{J}(R))$ is completely invariant under $\sigma$, we conclude by the previous paragraph that $\deg{\phi_j}=1,\ j\in\{1,\cdots, d+1\}$. Thus, each $\phi_j$ is a M{\"o}bius map. It follows that each $\Omega_j$ is a round disk, and $\sigma\vert_{\overline{\Omega_j}}$ is simply the reflection map in the boundary circle $C_j':=\partial\Omega_j$. Our labeling of the domains $\Omega_j$ also guarantees that the contact graph of the circle packing $\{C_1',\cdots, C_{d+1}'\}$ is planar isomorphic to that of the circle packing $\{C_1,\cdots, C_{d+1}\}$. 

Let $\Gamma'$ be the group generated by reflections in the circles $C_i'$. The above discussion implies that $\sigma\equiv\rho_{\Gamma'}$, and the David homeomorphism $\Psi$ conjugates $R\vert_{\mathcal{J}(R)}$ to $\rho_{\Gamma'}\vert_{\Lambda(\Gamma')}$.

By general rigidity results for geometrically finite Kleinian groups \cite[Theorem~4.2]{Tuk85} (cf. \cite{Mos68, Pra73}), there exists a quasiconformal homeomorphism $\Phi$ of the sphere that conjugates $\Gamma$ to $\Gamma'$ by conjugating the reflection in $C_j$ to the reflection in $C_j'$, $j\in\{1,\cdots,d+1\}$. In particular, such a $\Phi$ conjugates $\rho_{\Gamma}\vert_{\Lambda(\Gamma)}$ to $\rho_{\Gamma'}\vert_{\Lambda(\Gamma')}$. In light of Proposition~\ref{prop:david_qc_invariance} (i), we conclude that $\Phi^{-1}\circ\Psi$ is the desired global David homeomorphism conjugating $R\vert_{\mathcal{J}(R)}$ to $\rho_\Gamma\vert_{\Lambda(\Gamma)}$.
\end{proof}

Note that the Tischler graph of a critically fixed anti-polynomial has a vertex at $\infty$, and this vertex lies on the boundary of each face of the Tischler graph. Hence, the planar dual of the Tischler graph of a critically fixed anti-polynomial has a face containing all the vertices on its boundary; i.e., the planar dual is outerplanar. By \cite[Theorem~1.2]{LLM20}, the contact graph of the circle packing defining a kissing reflection group $\Gamma$ is $2$-connected and outerplanar if and only if $\Gamma$ is a necklace group. Thus, if the Tischler graph of a critically fixed anti-polynomial $P$ is dual to the contact graph of the circle packing defining a kissing reflection group $\Gamma$, then $\Gamma$ is a necklace group. With these observations on mind, we now record a useful corollary of Theorem~\ref{bijection_via_david_thm}.

\begin{corollary}\label{david_necklace_cor}
Suppose that the Tischler graph of a critically fixed anti-polynomial $P$ is dual to the contact graph of the circle packing defining a necklace reflection group $\Gamma$. Then there exists a David homeomorphism of $\widehat{\C}$ that conjugates $P\vert_{\mathcal{J}(R)}$ to $\rho_\Gamma\vert_{\Lambda(\Gamma)}$.
\end{corollary}

\bigskip

\subsection{Connections with known results}\label{kissing_group_crit_fixed_subsec}

Let us now mention connections of Theorem~\ref{bijection_via_david_thm} with some recent works on critically fixed anti-rational maps and kissing reflection groups.

In \cite[\S 10]{LLMM19}, the existence of the David homeomorphism of Theorem~\ref{bijection_via_david_thm} was proved
in the special case where $\Gamma$ is the classical Apollonian gasket reflection group
and $R$ is a cubic critically fixed anti-rational map (with four distinct critical points).
(More generally, it was done for a class of Apollonian-like gaskets.)

Critically fixed anti-rational maps were classified in terms of their Tischler graphs in \cite{Gey20} and \cite{LLM20}. Furthermore, in \cite[Theorem~1.1]{LLM20}, a dynamically natural bijection between the following classes of objects was established:
\bigskip

\begin{itemize}
\item $\{$Critically fixed anti-rational maps of degree $d$ up to M{\"o}bius conjugacy$\}$,\\

\item $\{$Kissing reflection groups of rank $d+1$ with connected limit set up to QC conjugacy$\}$.
\end{itemize}

The bijection is such that the contact graph of the circle packing defining a kissing reflection group $\Gamma$ is dual to the Tischler graph of the corresponding anti-rational map $R$. Moreover, if $R$ and $\Gamma$ correspond to each other under this bijection, then the dynamics of $R$ on the Julia set $\mathcal{J}(R)$ is topologically conjugate to the action of the Nielsen map $\rho_\Gamma$ on the limit set $\Lambda(\Gamma)$.
Particular occasions of this bijection
had been earlier exhibited in the aforementioned case of Apollonian-like gaskets
(corresponding to triangulations) \cite{LLMM19}
and in the case of the correspondence between
critically fixed anti-polynomials and necklace groups \cite[Theorem~B]{LMM20} (also compare \cite[Theorem~1.2]{LLM20}).

Note that all periodic points of a hyperbolic anti-rational map on its Julia set are repelling, while the Nielsen map of a kissing reflection group has parabolic fixed points on the limit set. Hence, the aforementioned conjugacies cannot be quasisymmetric; i.e., they cannot admit a quasiconformal extensions to the Riemann sphere. Theorem~\ref{bijection_via_david_thm} shows that these conjugacies extend as David homeomorphisms of the sphere.

Combining Theorem~\ref{bijection_via_david_thm} with the bijection results mentioned above, we have the following result.

\begin{corollary}\label{bijection_via_david_cor}
Let $\Gamma$ be a kissing reflection group of rank $d+1$. Then, there exists a degree $d$ critically fixed anti-rational map $R$ such that $R\vert_{\mathcal{J}(R)}$ is conjugate to $\rho_\Gamma\vert_{\Lambda(\Gamma)}$ by a David homeomorphism of $\widehat{\C}$.

Moreover, if $\Gamma$ is a necklace group, then $R$ can be chosen as a critically fixed anti-polynomial.
\end{corollary}

\bigskip

\section{Conformal removability of Julia and limit sets}\label{conf_rem_julia_limit_sec}

\subsection{Conformal removability of limit sets of necklace reflection groups}\label{conf_rem_limit_subsec}

An application of Theorem~\ref{theorem:john_removable} and Theorem~\ref{bijection_via_david_thm} is the following conformal removability result for limit sets of necklace reflection groups (see Figures~\ref{necklace_group_fig} and~\ref{kleinian_cusp} for examples of such limit sets).

\begin{theorem}\label{limit_removable_thm}
Let $\Gamma$ be a necklace group. Then, the limit set $\Lambda(\Gamma)$ is conformally removable.
\end{theorem}
\begin{proof}
By Corollary~\ref{bijection_via_david_cor}, there exist a degree $d$ critically fixed anti-polynomial $P$ and a David homeomorphism $\Psi$ of $\widehat{\C}$ such that $\Psi(\mathcal{J}(P))=\Lambda(\Gamma)$. Since the basin of infinity of a hyperbolic polynomial is a John domain, it follows by Theorem~\ref{theorem:john_removable} that $\Lambda(\Gamma)$ is conformally removable.
\end{proof}

\bigskip

\subsection{Conformal removability of geometrically finite polynomial Julia sets}\label{conf_rem_julia_subsec}

We recall that a rational map $R$ is said to be \textit{geometrically finite} if its postcritical set intersects the Julia set $\mathcal J(R)$ of $R$ in a finite set, or equivalently, if every critical point of $R$ is either preperiodic, or attracted to an attracting or parabolic cycle. The main result of this section is the following.

\begin{theorem}\label{T:ConfRemove}
Let $P$ be a geometrically finite polynomial of degree $d\geq2$ with connected Julia set $\mathcal J(P)$. Then $\mathcal J(P)$ is conformally removable.
\end{theorem}

We remark that special cases of Theorem~\ref{T:ConfRemove} were already known in the 90s. Indeed, conformal removability of connected Julia sets of subhyperbolic polynomials was proved by Jones \cite{Jon91}. In fact, it was shown there that in the subhyperbolic case, the basin of infinity is a John domain, which implies conformal removability of the Julia set. On the other hand, when a polynomial has parabolic cycles, the basin of infinity is no longer a John domain. However, for the polynomial $P(z)=z^2+\frac14$, the whole Julia set is the boundary of the immediate basin of attraction of a parabolic fixed point. In this case, it was shown in \cite{CJY94} that the parabolic basin is John, and hence the corresponding Julia set is conformally removable. 

The proof of Theorem~\ref{T:ConfRemove} is indirect. We first realize the connected Julia set of a geometrically finite polynomial $P$ as that of a subhyperbolic polynomial $Q$, whose Julia set is known to be $W^{1,1}$-removable. We then use our David surgery technique to construct a global David homeomorphism that carries the Julia set of $Q$ onto the Julia set of $P$. The proof is completed by invoking Theorem~\ref{theorem:john_removable}.

We note that given a geometrically finite polynomial $P$ with connected Julia set, the existence of a subhyperbolic polynomial $Q$ with topologically conjugate Julia dynamics and a David conjugacy between the Julia sets of $Q$ and $P$ was proved earlier in \cite[Corollary~3.10]{HT04}. However, to illustrate the scope of the machinery that we have developed, we include independent and self-contained proofs of these facts. We also point out that we use arguments due to Kiwi (using the Thurston Realization Theorem) to demonstrate the existence of such a polynomial $Q$, while the proof of this fact given in \cite{Hai00} involves perturbing $P$ to a polynomial-like map and then applying the straightening theorem (for polynomial-like maps) to upgrade it to a polynomial $Q$. Further, we construct a David conjugacy between the Julia dynamics of $Q$ and $P$ using our David Extension Theorem and a surgery technique using Riemann maps of Fatou components (Lemma~\ref{david_surgery_lemma}), whereas the same result was proved in \cite{Hai98} using more local arguments.
 
For the first step, we need to recall some basic facts from polynomial dynamics. Let $\mathcal{P}_d$ be the space of all monic, centered, holomorphic polynomials of degree $d$; i.e.,
$$
\mathcal{P}_d:=\{P(z)=z^d+a_{d-2}z^{d-2}+\cdots+a_1z+a_0:\ a_0,\cdots, a_{d-2}\in\C\}.
$$ 
Two distinct members of $\mathcal{P}_d$ are affinely conjugate if and only if they are conjugate via rotation by a $(d-1)$-st root of unity. The filled Julia set and the basin of infinity of a polynomial $P$ are denoted by $\mathcal{K}(P)$ and $\mathcal{B}_\infty(P)$ respectively. 

If $P$ is a monic, centered, polynomial of degree $d$ such that $\mathcal{J}(P)$ is connected, then there exists a conformal map $\phi_P: \mathbb{D}^*\rightarrow \mathcal{B}_\infty(P)$ that conjugates $z^d\vert_{\mathbb{D}^*}$ to $P\vert_{\mathcal{B}_\infty(P)}$, and satisfies $\phi_P'(\infty)=1$. We will call $\phi_P$ the \emph{B\"ottcher coordinate} for $P$. Furthermore, if $\partial \mathcal{K}(P)=\mathcal{J}(P)$ is locally connected, then $\phi_P$ extends to a semiconjugacy between $z^{d}\vert_{\mathbb{S}^1}$ and $P\vert_{\mathcal{J}(P)}$. The external dynamical ray of $P$ at angle $\theta$ (the image of the radial line at angle $\theta$ under $\phi_P$) is denoted by $R_P(\theta)$. 

Our immediate goal is to associate a postcritically finite polynomial (in $\mathcal{P}_d$) to each geometrically finite member of $\mathcal{P}_d$ (with connected Julia set). To this end, we first recall the notion of \emph{polynomial laminations}.

\begin{definition}\label{lamination_def} 
Let $p\in\mathcal{P}_d$ be such that $\mathcal{J}(P)$ is connected and locally connected. We define $\lambda(P)$ to be he smallest equivalence relation on $\R/\Z$ that identifies $s,t\in\R/\Z$ whenever the external dynamical rays $R_P(s)$ and $R_P(t)$ land at a common point on $\mathcal{J}(P)$. We call $\lambda(P)$ the \emph{lamination} of $P$.
\end{definition}

\begin{prop}\label{geom_finite_laminations_prop}
Let $P\in\mathcal{P}_d$ be geometrically finite with connected $\mathcal{J}(P)$. Then, $\lambda(P)$ satisfies the following properties, where the map $m_d:\R/\Z\to\R/\Z$ is given by $\theta\mapsto d\theta$.
\begin{enumerate}\upshape
\item $\lambda(P)$ is closed in $\R/\Z\times\R/\Z$.

\item Each equivalence class $A$ of $\lambda(P)$ is a finite subset of $\R/\Z$.

\item If $A$ is a $\lambda(P)$-equivalence class, then $m_d(A)$ is also a $\lambda(P)$-equivalence class.

\item If $A$ is a $\lambda(P)$-equivalence class, then $A\mapsto m_d(A)$ is \emph{consecutive preserving}\footnote{This notion, which is slightly more general than that of being cyclic order preserving, is required to handle the case when $A$ is formed by the arguments of the external rays landing at a critical point of $P$.}; i.e., for every connected component $(s,t)$ of $\R/\Z\setminus A$, we have that $(ds,dt)$ is a connected component of $\R/\Z\setminus m_d(A)$.

\item $\lambda(P)$-equivalence classes are pairwise \emph{unlinked}; i.e., if $A$ and $B$ are two distinct equivalence classes of $\lambda(P)$, then there exist disjoint intervals $I_A, I_B\subset\R/\Z$ such that $A\subset I_A$ and $B\subset I_B$.

\item $\faktor{(\R/\Z)}{\lambda(P)}\cong\mathcal{J}(P)$, and $\phi_P$ descends to a topological conjugacy between $$m_d:\faktor{(\R/\Z)}{\lambda(P)}\to\faktor{(\R/\Z)}{\lambda(P)}\ \textrm{and}\ P: \mathcal{J}(P)\to\mathcal{J}(P).$$

\item If $\gamma\subset\faktor{(\R/\Z)}{\lambda(P)}$ is a periodic simple closed curve, then the corresponding return map is not a homeomorphism.

\item If $c$ is a critical point of $P$ on the Julia set $\mathcal{J}(P)$, then $\phi_P^{-1}(c)\subset\Q/\Z$, and $\phi_P^{-1}(c)\mapsto m_d(\phi_P^{-1}(c))$ has some degree $\delta>1$. Every other equivalence class of $\lambda(P)$ maps injectively onto its image equivalence class under $m_d$.
\end{enumerate}
\end{prop}
\begin{proof}
First note that by \cite[Expos{\'e}~X, \S 1, Theorem~1]{DH07} (also see \cite[Theorem~A]{TY96}), the Julia set $\mathcal{J}(P)$ is locally connected, and hence $\lambda(P)$ can be defined following Definition~\ref{lamination_def}. The first five properties now follow from \cite[Theorem~1]{Kiw04}. Moreover, local connectivity of $\mathcal{J}(P)$ implies that the B{\"o}ttcher coordinate of $P$ extends continuously to a surjection $\phi_P:\R/\Z\to\mathcal{J}(P)$ that semiconjugates $m_d$ to $P$. By definition, the equivalence classes of $\lambda(P)$ are precisely the fibers of this map $\phi_P$, from which property (6) follows. Property (7) follows from the fact that $P$ has no Siegel disk. Finally, the last property is a consequence of geometric finiteness of $P$; more precisely, of the fact that each critical point of $P$ on $\mathcal{J}(P)$ is strictly pre-periodic.
\end{proof}

Following \cite{Kiw04}, we say that an equivalence relation $\lambda$ on $\R/\Z$ is a \emph{$\R$eal lamination} (not to be confused with rational laminations; see \cite[\S 1]{Kiw01}) if it satisfies conditions (1)--(5). Moreover, if a $\R$eal lamination $\lambda$ satisfies condition (7), it is called a \emph{$\R$eal lamination with no rotation curves}. An equivalence class $A$ of $\lambda$ is called a \emph{Julia critical element} if the degree of the map $m_d:A\to m_d(A)$ is greater than one. Finally, a $\R$eal lamination $\lambda$ is called \emph{postcritically finite} if every Julia critical element of $\lambda$ is contained in $\Q/\Z$. Using this terminology, Proposition~\ref{geom_finite_laminations_prop} can be restated as follows.

\begin{corollary}\label{geom_finite_laminations_prop_1}
Let $P\in\mathcal{P}_d$ be geometrically finite with connected $\mathcal{J}(P)$. Then, $\lambda(P)$ is a postcritically finite $\R$eal lamination (with no rotation curves).
\end{corollary}

The following folklore result can be easily derived from Kiwi's theory of polynomial laminations, though it was not explicitly stated in \cite{Kiw04}.

\begin{theorem}\label{realizing_laminations_thm}
For a postcritically finite $\R$eal lamination $\lambda$ (with no rotation curves), there exists a monic, centered, postcritically finite polynomial $Q$ with $\lambda(Q)=\lambda$. In particular, $\mathcal{J}(Q)\cong\faktor{(\R/\Z)}{\lambda}$.
\end{theorem}
\begin{proof}
To prove this result, one needs to work with a more general definition of laminations (than the one given in Definition~\ref{lamination_def}) involving prime end impressions. For some $Q\in\mathcal{P}_d$ with connected Julia set, the lamination $\lambda(Q)$ is defined as the smallest equivalence relation on $\R/\Z$ that identifies $s,t\in\R/\Z$ whenever\qquad $\textrm{Imp}(s)\cap\textrm{Imp}(t)\neq\emptyset$ (where $\textrm{Imp}(t)$ stands for the impression of an angle $t\in\R/\Z$ on the Julia set $\mathcal{J}(Q)$, see \cite[Definition~2.3]{Kiw04}). Clearly, if $\mathcal{J}(Q)$ is locally connected, then every impression is a singleton, and this definition agrees with the one given in Definition~\ref{lamination_def}.

Since $\lambda$ is a $\R$eal lamination with no rotation curves, \cite[Lemma~6.34, Theorem~1]{Kiw04} guarantees the existence of a monic, centered polynomial $Q$ with connected Julia set such that each cycle of $Q$ is either repelling or superattracting, each bounded Fatou component of $Q$ contains a unique element in a critical grand orbit, and $\lambda(Q)=\lambda$.  In particular, every critical point of $Q$ in its Fatou set is eventually periodic. (At this point, it is unclear whether $\mathcal{J}(Q)$ is locally connected, which is why the more general definition of $\lambda(Q)$ is necessary here).

Now, it follows from postcritical finiteness of $\lambda$ that each critical point of $Q$ on its Julia set $\mathcal{J}(Q)$ lies in the impression of some rational (i.e., preperiodic under $m_d$) angle. By \cite[Corollary~1.2]{Kiw04}, the impression of a preperiodic angle is a singleton, and hence, every critical point of $Q$ on $\mathcal{J}(Q)$  is strictly preperiodic. Therefore, $Q$ is postcritically finite. In particular, $\mathcal{J}(Q)$ is locally connected, and the B{\"o}ttcher coordinate of $Q$ extends continuously to $\R/\Z$ to yield a topological conjugacy between $$m_d:\faktor{(\R/\Z)}{\lambda(Q)}\to\faktor{(\R/\Z)}{\lambda(Q)}\ \textrm{and}\ Q: \mathcal{J}(Q)\to\mathcal{J}(Q).$$ This completes the proof.
\end{proof}

Finally, combining Property (6) of Proposition~\ref{geom_finite_laminations_prop} with Corollary~\ref{geom_finite_laminations_prop_1} and Theorem~\ref{realizing_laminations_thm}, we obtain the following result that allows us to associate a postcritically finite polynomial (in $\mathcal{P}_d$) to each geometrically finite member of $\mathcal{P}_d$ (with connected Julia set).

\begin{corollary}\label{geom_finite_to_postcrit_finite}
Let $P\in\mathcal{P}_d$ be geometrically finite with connected $\mathcal{J}(P)$. Then, there exists a postcritically finite polynomial $Q\in\mathcal{P}_d$ such that $\lambda(Q)=\lambda(P)$. In particular, $Q\vert_{\mathcal{J}(Q)}$ is topologically conjugate to $P\vert_{\mathcal{J}(P)}$.
\end{corollary}

\begin{remark}
1) Invariant laminations of complex polynomials were first defined in a slightly different language by Thurston (see \cite[\S II.4]{Thu09}). 

2) Theorem~\ref{realizing_laminations_thm} is an instance of realization results in complex dynamics. Realization of postcritically finite maps with prescribed combinatorics was first proved by Thurston (see \cite{DH93}), which was used by Hubbard, Bielefeld, and Fisher to classify critically pre-periodic polynomials using co-landing structure of suitable external dynamical rays \cite{BFH}, and by Poirier to classify postcritically finite polynomials in terms of their Hubbard trees \cite{Poi10}. An alternative way of constructing critically pre-periodic polynomials with prescribed laminations is to study the connectedness locus of $\mathcal{P}_d$ from `outside'; i.e., to approximate such polynomials from the \emph{shift locus} (see \cite{Kiw05}).
\end{remark}

We are now prepared to prove the main theorem of this subsection.

\begin{proof}[Proof of Theorem~\ref{T:ConfRemove}]
By conjugating $P$ with an appropriate conformal linear map, we may assume that $P\in\mathcal{P}_d$. Furthermore, replacing $P$ by a suitable iterate, we can assume that each periodic Fatou component of $P$ is fixed.  

The first step of the proof is to reduce the map $P$ to its \emph{rigid model}. Roughly speaking, the rigid model of $P$ is a geometrically finite (but not necessarily postcritically finite) polynomial of the same degree, whose Julia dynamics is topologically conjugate to that of $P$, and whose Fatou critical points have the simplest possible dynamics in a suitable sense. To make this precise, let us first note that the Blaschke product
$$
B_k(z)=\frac{(k+1)z^k+(k-1)}{(k-1)z^k+(k+1)},\ k\geq 2,
$$
has a parabolic fixed point at $1$, and $k-1$ distinct repelling fixed points on $\mathbb{S}^1$. Moreover, $B_k$ has a unique critical point in $\D$; namely, at the origin.

By \cite[Proposition~6.9]{McM}, one can perform quasiconformal surgeries on the Fatou set of $P$ to produce a degree $d$ geometrically finite polynomial $\widehat{P}$, called the rigid model of $P$, such that the following properties hold true.

\begin{enumerate}
\item There exists a global quasiconformal map $\psi$ conjugating $P\vert_{\mathcal{J}(P)}$ to $\widehat{P}\vert_{\mathcal{J}(\widehat{P})}$; in particular, each periodic Fatou component of $\widehat{P}$ is fixed.

\item\label{basin_rigid} The restriction of $\widehat{P}$ to each fixed Fatou component is conformally conjugate to $z^k\vert_{\D}$ (if the corresponding Fatou component of $P$ is the immediate basin of attraction of an attracting fixed point) or $B_k\vert_{\D}$ (if the corresponding Fatou component of $P$ is an immediate basin of attraction of a parabolic fixed point), where $k\geq 2$ is the degree of the restriction of $P$ to the corresponding Fatou component.

\item\label{whole_fatou_rigid} If $c$ is a critical point of $\widehat{P}$ contained in a Fatou component, then there is a least $n\geq 0$ such that $\widehat{P}^{\circ n}(c)$ lies in a fixed Fatou component $U$. Moreover, $\widehat{P}^{\circ n}(c)$ is the unique critical point of $\widehat{P}$ in $U$. 
\end{enumerate}
In particular, each Fatou component of $\widehat{P}$, except for the immediate basins of the parabolic fixed points, contains a unique element in a critical grand orbit.

Since the property of being conformally removable is preserved under global quasiconformal maps, it now suffices to prove that $\mathcal{J}(\widehat{P})$ is conformally removable.

By Corollary~\ref{geom_finite_to_postcrit_finite}, there exists a postcritically finite polynomial $Q\in\mathcal{P}_d$ such that $Q\vert_{\mathcal{J}(Q)}$ is topologically conjugate to $\widehat{P}\vert_{\mathcal{J}(\widehat{P})}$. By construction, each periodic Fatou component of $Q$ is fixed, and the restriction of $Q$ to each fixed Fatou components is conformally conjugate to $z^k\vert_{\D}$, for some $k\geq 2$. Also note that the topological conjugacy between the Julia sets of $Q$ and $\widehat{P}$ induces a bijection between their fixed Fatou components.

Let $V_1,\cdots, V_l$ be the fixed Fatou components of $Q$ such that the corresponding Fatou components of $\widehat{P}$ are parabolic. Recall from Example~\ref{example:blaschke} that for any $k\geq 2$, the expansive covering map $B_k\vert_{\mathbb{S}^1}$ admits a Markov partition satisfying conditions~\eqref{condition:uv} and~\eqref{condition:holomorphic}. Hence, by Theorem~\ref{theorem:extension_special_case}, there exists a topological conjugacy between $z^k\vert_{\mathbb{S}^1}$ and $B_k\vert_{\mathbb{S}^1}$ that extends as a David homeomorphism of $\D$. One can now repeat the arguments of Lemma~\ref{david_surgery_lemma} to replace the power map dynamics on the Fatou component $V_i$ by the dynamics of $B_{k_i}$, where $k_i$ is the degree of $Q$ restricted to the component $V_i$, for all $i\in\{1,\cdots, l\}$. This yields a geometrically finite polynomial $\widetilde{Q}$ and a global David homeomorphism $H$ such that $H$ carries $\mathcal{J}(Q)$ onto $\mathcal{J}(\widetilde{Q})$, and conformally conjugates $Q$ to $\widetilde{Q}$ outside the grand orbit of $\cup_{i=1}^l V_i$. On the other hand, the restriction of $\widetilde{Q}$ on $H(V_i)$ is conformally conjugate to $B_{k_i}\vert_{\D}$. By construction, the map $\widetilde{Q}$ satisfies condition~\eqref{basin_rigid}. Thanks to \cite[Proposition~6.9]{McM}, possibly after performing quasiconformal surgeries on the Fatou set of $\widetilde{Q}$, we can further assume that $\widetilde{Q}$ satisfies condition~\eqref{whole_fatou_rigid}. By Proposition~\ref{prop:david_qc_invariance} (i), we still have that  $H(\mathcal{J}(Q))=\mathcal{J}(\widetilde{Q})$, where $H$ is a global David homeomorphism.

It is now easy to see from the above construction of $\widetilde{Q}$ that $\widetilde{Q}\vert_{\mathcal{J}(\widetilde{Q})}$ is topologically conjugate to $\widehat{P}\vert_{\mathcal{J}(\widehat{P})}$, and this conjugacy extends to a global topological conjugacy between $\widetilde{Q}$ and $\widehat{P}$ such that the conjugacy is conformal on the Fatou set of $\widetilde{Q}$ (compare \cite[Proposition~6.10]{McM}). According to~\cite{Jon91}, the basin of infinity $\mathcal{B}_\infty(Q)$ of the postcritically finite (hence, subhyperbolic) polynomial $Q$ is a John domain. Hence, Theorem~\ref{theorem:john_removable} combined with the fact that $\partial\mathcal{B}_\infty(Q)=\mathcal{J}(Q)$ yield that $H(\partial \mathcal{B}_\infty(Q))=\mathcal{J}(\widetilde{Q})$ is conformally removable. 
It follows that $\widetilde{Q}$ is M{\"o}bius conjugate to $\widehat{P}$, and $\mathcal{J}(\widehat{P})$ is a M{\"o}bius image of $\mathcal{J}(\widetilde{Q})$. Hence, $\mathcal{J}(\widehat{P})$ is also conformally removable.
\end{proof}

\bigskip

\section{Mating reflection groups with anti-polynomials: existence theorem}\label{mating_sec}

The goal of this section is to apply our result on David extensions of circle homeomorphisms to the theory of mating in conformal dynamics.

\subsection{Necklace groups and Bers slice}\label{necklace_sec}

For the purposes of mating necklace groups with anti-polynomials, it will be important to work with groups equipped with labeled generators (compare Remark~\ref{why_representation}). Moreover, the conformal conjugacy class of the action of a necklace group on the unbounded component of its domain of discontinuity will play no special role in the mating theory. Hence, we may freeze the `external class' of the necklace groups under consideration. We now formalize this by defining a space of representations of the necklace group $\pmb{\Gamma}_{d+1}$ (see Definition~\ref{reflection_map}).

\begin{definition} Let $\Gamma$ be a discrete subgroup of $\textrm{Aut}^\pm(\widehat{\C})$. An isomorphism 
$$
\xi:\pmb{\Gamma}_{d+1}\rightarrow \Gamma
$$
is said to be \emph{weakly type-preserving}, or \emph{w.t.p.}, if 
\begin{enumerate}\upshape
\item $\xi(g)$ is orientation-preserving if and only if $g$ is orientation-preserving, and 

\item $\xi(g)\in\Gamma$ is a parabolic M{\"o}bius map for each parabolic M{\"o}bius map $g\in \pmb{\Gamma}_{d+1}$.
\end{enumerate}
\end{definition}

\begin{definition}\label{D_def} We define 

\begin{align*}
\mathcal{D}(\pmb{\Gamma}_{d+1}):= \lbrace   \xi :  \pmb{\Gamma}_{d+1} \to \Gamma\vert\ \Gamma \textrm{ is a discrete subgroup of } \textrm{Aut}^\pm(\widehat{\C}),\\\textrm{ and } \xi \textrm{ is a w.t.p. isomorphism} \rbrace.
\end{align*}

\noindent We endow $\mathcal{D}(\pmb{\Gamma}_{d+1})$ with the topology of \emph{algebraic convergence}: we say that a sequence $(\xi_n)_{n=1}^{\infty}\subset\mathcal{D}(\pmb{\Gamma}_{d+1})$ converges to $\xi\in\mathcal{D}(\pmb{\Gamma}_{d+1})$ if $\xi_n(\rho_i)\to\xi(\rho_i)$ coefficient-wise (as $n\rightarrow\infty$) for $i\in\{1,\cdots,d+1\}$.
\end{definition}

\begin{remark}\label{rep_var_def_rem} Let $\xi\in\mathcal{D}(\pmb{\Gamma}_{d+1})$. Since for each $i\in \Z/(d+1)\Z$, the M{\"o}bius map $\rho_i\circ\rho_{i+1}$ is parabolic (this follows from the fact that each $\mathbf{C}_i$ is tangent to $\mathbf{C}_{i+1}$), the w.t.p. condition implies that $\xi(\rho_i)\circ\xi(\rho_{i+1})$ is also parabolic. As each $\xi(\rho_i)$ is an anti-conformal involution, it follows that $\xi(\rho_i)$ is M{\"o}bius conjugate to the circular reflection $z\mapsto 1/\overline{z}$ or the antipodal map $z\mapsto -1/\overline{z}$. A straightforward computation shows that the composition of $-1/\overline{z}$ with either the reflection or the antipodal map with respect to any circle has two distinct fixed points in $\widehat{\C}$, and hence not parabolic. Therefore, it follows that no $\xi(\rho_i)$ is M{\"o}bius conjugate to the antipodal map $-1/\overline{z}$. Hence, each $\xi(\rho_i)$ must be the reflection in some Euclidean circle $C_i$. Thus, $\Gamma=\xi(\pmb{\Gamma}_{d+1})$ is generated by reflections in the circles $C_1, \cdots, C_{d+1}$. The fact that $\xi(\rho_i)\circ\xi(\rho_{i+1})$ is parabolic now translates to the condition that each $C_i$ is tangent to $C_{i+1}$ (for $i\in \Z/(d+1)\Z$). However, new tangencies among the circles $C_i$ may arise. Moreover, that $\xi$ is an isomorphism rules out non-tangential intersection between circles $C_i$, $C_j$ (indeed, a non-tangential intersection between $C_i$ and $C_j$ would introduce a new relation between $\xi(\rho_i)$ and $\xi(\rho_j)$, compare \cite[Part~II, Chapter~5, \S 1.1]{VS93}). Therefore, $\Gamma=\xi(\pmb{\Gamma}_{d+1})$ is a Kleinian reflection group satisfying properties (1) and (2) of necklace groups.
\end{remark}

Recall that $\D^*=\widehat{\C}\setminus\overline{\D}$.

\begin{definition}\label{Bers_slice}
Let $\textrm{Bel}_{\pmb{\Gamma}_{d+1}}$ denote those Beltrami coefficients $\mu$ invariant under $\pmb{\Gamma}_{d+1}$, satisfying $\mu=0$ a.e. on $\D^*$. For a Beltrami coefficient $\mu$, let $\tau_\mu:\mathbb{C}\rightarrow\mathbb{C}$ denote the quasiconformal integrating map of $\mu$ normalized so that $\tau_\mu(z)=z+O(1/z)$ as $z\rightarrow\infty$. The \emph{Bers slice} of $\Gamma$ is defined as \[ \beta(\pmb{\Gamma}_{d+1}) := \{\xi\in\mathcal{D}(\pmb{\Gamma}_{d+1}):\ \exists\ \mu \in  \textrm{Bel}_{\pmb{\Gamma}_{d+1}}\ \textrm{such that}\ \xi(g)= \tau_\mu\circ g\circ \tau_\mu^{-1}\ \forall\ g\in\pmb{\Gamma}_{d+1} \}. \]
\end{definition}

\begin{remark} There is a natural free $\textrm{PSL}_2(\mathbb{C})$-action on $\mathcal{D}(\pmb{\Gamma}_{d+1})$ given by conjugation, and so it is natural to consider the space $\textrm{AH}(\pmb{\Gamma}_{d+1}):=\mathcal{D}(\pmb{\Gamma}_{d+1})/\textrm{PSL}_2(\mathbb{C})$. The following definition of the Bers slice, where no normalization for $\tau_\mu$ is specified, is more aligned with the classical Kleinian group literature: 
\begin{align}\label{alternative_bers} 
\{\xi\in\textrm{AH}(\pmb{\Gamma}_{d+1}):\ \exists\ \mu \in  \textrm{Bel}_{\pmb{\Gamma}_{d+1}}\ \textrm{such that}\ \xi(g)= \tau_\mu\circ g\circ \tau_\mu^{-1}\ \forall\ g\in\pmb{\Gamma}_{d+1} \} \tag{$\star$}.
\end{align}
Our Definition~\ref{Bers_slice} of $\beta(\pmb{\Gamma}_{d+1})$ is simply a canonical choice of representative from each equivalence class of (\ref{alternative_bers}), and will be more appropriate for the present work. 
\end{remark}

\begin{prop}\cite[Proposition~11]{LMM20}\label{compactify_prop}
The Bers slice $\beta(\pmb{\Gamma}_{d+1})$ is pre-compact in $\mathcal{D}(\pmb{\Gamma}_{d+1})$, and for each $\xi\in\overline{\beta(\pmb{\Gamma}_{d+1})}$, the group $\xi(\pmb{\Gamma}_{d+1})$ is a necklace group.
\end{prop}

\begin{definition}\label{compatify_def} 
We refer to $\overline{\beta(\pmb{\Gamma}_{d+1})} \subset \mathcal{D}(\pmb{\Gamma}_{d+1})$ as the \emph{Bers compactification} of the Bers slice $\beta(\pmb{\Gamma}_{d+1})$. 
\end{definition}

\begin{remark}\label{rep_group_identify_rem} 
We will often identify $\xi\in\overline{\beta(\pmb{\Gamma}_{d+1})}$ with the group $\Gamma:=\xi(\pmb{\Gamma}_{d+1})$, and simply write $\Gamma\in\overline{\beta(\pmb{\Gamma}_{d+1})}$, but always with the understanding of an associated representation $\xi:\pmb{\Gamma}_{d+1} \rightarrow \Gamma$. Since $\xi$ is completely determined by its action on the generators $\rho_1,\cdots, \rho_{d+1}$ of $\pmb{\Gamma}_{d+1}$, this is equivalent to remembering the `labeled' circle packing $C_1,\cdots,C_{d+1}$, where $\xi(\rho_i)$ is the reflection in the circle $C_i$, for $i\in\{1,\cdots,d+1\}$. 
\end{remark}

For $\Gamma\in\overline{\beta(\pmb{\Gamma}_{d+1})}$, the unbounded component of the domain of discontinuity $\Omega(\Gamma)$ is denoted by $\Omega_\infty(\Gamma)$. We set $\mathcal{K}(\Gamma):=\mathbb{C}\setminus\Omega_\infty(\Gamma)$. We further denote the union of all bounded components of the fundamental domain $\mathcal{F}_\Gamma$ by $T^0(\Gamma)$, and the unique unbounded component of $\mathcal{F}_\Gamma$ by $\Pi^0(\Gamma)$. Finally, we set $\Pi(\Gamma):=\overline{\Pi^0(\Gamma)}$.

Note that a Kleinian group is said to be \emph{geometrically finite} if its action on $\mathbb{H}^3$ admits a fundamental polyhedron with finitely many faces (see \cite[\S 3.6]{Mar16} for many equivalent definitions).

\begin{prop}\cite[Proposition~15]{LMM20}\label{inv_comp_limit_boundary}
Let $\Gamma\in\overline{\beta(\pmb{\Gamma}_{d+1})}$. Then the following hold true.
\begin{enumerate}
\upshape
\item $\Omega_\infty(\Gamma)$ is simply connected, and $\Gamma$-invariant.

\item $\partial\Omega_\infty(\Gamma)=\Lambda(\Gamma)$; in particular, $\inter{\mathcal{K}(\Gamma)}=\Omega(\Gamma)\setminus\Omega_\infty(\Gamma)$.
 
\item $\Lambda(\Gamma)$ is connected.

\item All bounded components of $\Omega(\Gamma)$ are Jordan domains.
\end{enumerate}
\end{prop}

\begin{remark}
We note that since the groups on the Bers boundary $\partial\beta(\pmb{\Gamma}_{d+1})$ are reflection groups, they are all geometrically finite. Thus, the boundary of the Bers slice of an ideal polygon reflection group is considerably simpler than Bers slices of Fuchsian groups. Indeed, the boundary of the Bers slice of a Fuchsian group (except for the thrice punctured sphere group) necessarily contains degenerate groups; i.e., groups that are not geometrically finite (see \cite{Ber70}).
\end{remark}

\begin{figure}[ht!]
\begin{tikzpicture}
 \node[anchor=south west,inner sep=0] at (0,0) {\includegraphics[width=0.48\textwidth]{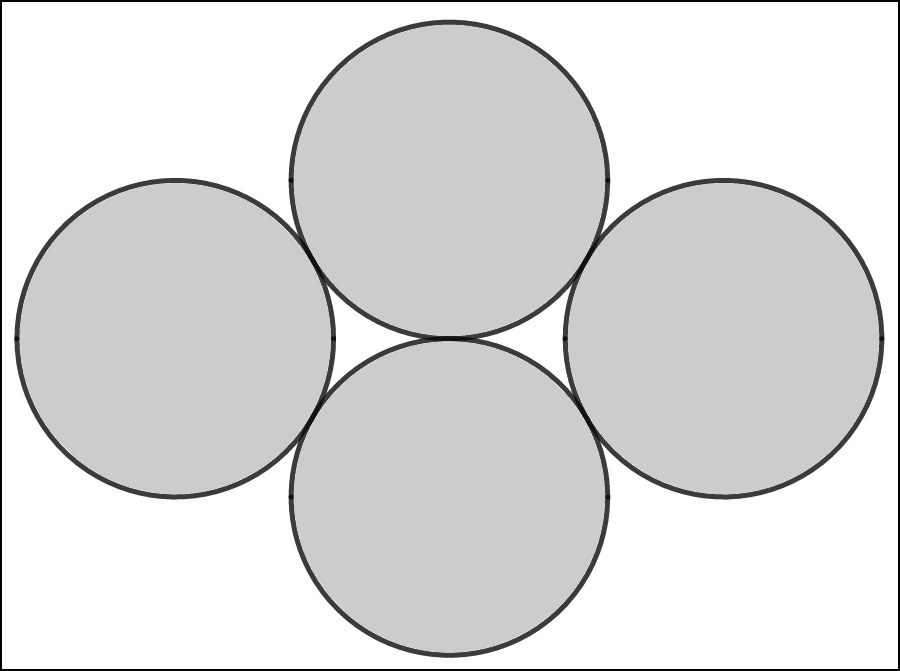}};
\node[anchor=south west,inner sep=0] at (6.4,0) {\includegraphics[width=0.48\textwidth]{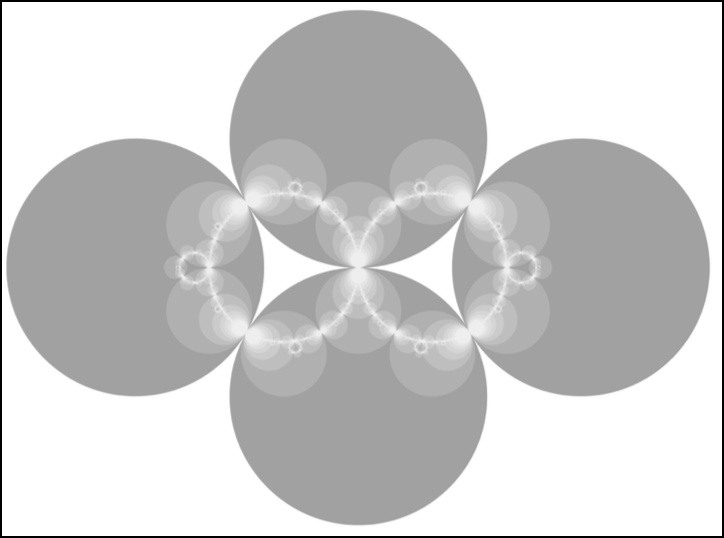}};
\node at (5,2.2) {$\Int{C_4}$};
\node at (3,3.5) {$\Int{C_1}$};
\node at (1,2.2) {$\Int{C_2}$};
\node at (3,1) {$\Int{C_3}$};
\node at (1.2,4) {$\mathcal{F}$};
\node at (3.64,2.25) {\begin{tiny}$\mathcal{F}$\end{tiny}};
\node at (2.45,2.25) {\tiny$\mathcal{F}$};
\node at (7.8,4) {$\Omega_\infty(\Gamma)$};
\end{tikzpicture}
\caption{Left: The circles $C_{i}$ generate a necklace group $\Gamma$. The Nielsen map $\rho_\Gamma$ is defined piecewise on the union of the closed disks $\overline{\Int{C_i}}$. The fundamental domain $\mathcal{F}=\mathcal{F}_\Gamma$ (for the action of $\Gamma$ on $\Omega(\Gamma)$) is the complement of these open disks with the singular boundary points removed. The connected components of $\mathcal{F}$ are marked. Right: The unbounded component of the domain of discontinuity $\Omega(\Gamma)$ is $\Omega_\infty(\Gamma)$. The boundary of $\Omega_\infty(\Gamma)$ is the limit set $\Lambda(\Gamma)$. Every point in $\Omega(\Gamma)$ escapes to $\mathcal{F}$ under iterates of $\rho_\Gamma$. The dynamics of $\rho_\Gamma$ on the limit set $\Lambda(\Gamma)$ is topologically conjugate to the Julia dynamics of the cubic anti-polynomial depicted in Figure~\ref{fat_ant}.}
\label{kleinian_cusp}
\end{figure}
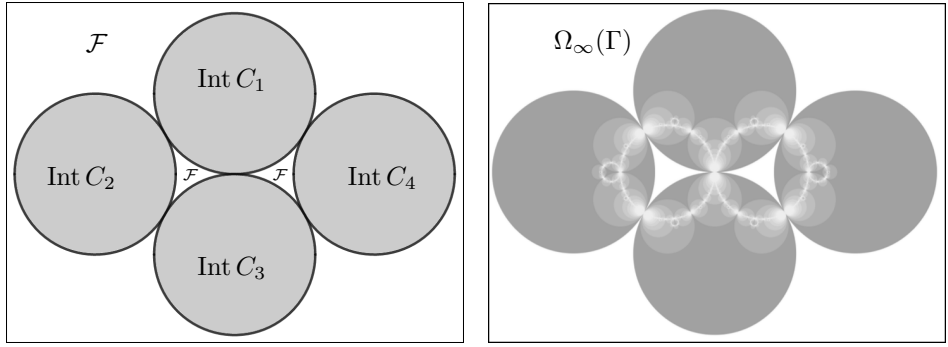

\begin{prop}\cite[Proposition~16]{LMM20}\label{orbit_equiv_prop}
Let $\Gamma\in\overline{\beta(\pmb{\Gamma}_{d+1})}$. The map $\rho_\Gamma$ is orbit equivalent to $\Gamma$ on $\widehat{\C}$; i.e., for any two points $z,w\in\widehat{\mathbb{C}}$, there exists $g\in\Gamma$ with $g(z)=w$ if and only if there exist non-negative integers $n_1, n_2$ such that $\rho_\Gamma^{\circ n_1}(z)=\rho_\Gamma^{\circ n_2}(w)$.
\end{prop}

\begin{remark}
Note that although $\rho_\Gamma$ is not defined on $\inter{\mathcal{F}_\Gamma}$, the expression $\rho_\Gamma^{\circ n}(z)$ makes sense for $z\in\inter{\mathcal{F}_\Gamma}$ when $n=0$ (with the convention that $\rho_\Gamma^{\circ 0}=\textrm{id}$).
\end{remark}

\begin{prop}\label{regular_fund_pre_image_prop} 
Let $\Gamma\in\overline{\beta(\pmb{\Gamma}_{d+1})}$. Then, the following hold true.
\begin{enumerate}\upshape
\item $\overline{T^0(\Gamma)}$ is connected.

\item $\displaystyle\mathcal{K}(\Gamma)=\overline{\bigcup_{n\geq0}\rho_\Gamma^{-n}(T^0(\Gamma))}$.

\item $\Lambda(\Gamma)$ is completely invariant under $\rho_\Gamma$.

\item If $\mathcal{U}$ is a connected component of $\inter{\mathcal{K}(\Gamma)}$ containing a component of $T^0(\Gamma)$, then $\rho_\Gamma\vert_{\partial\mathcal{U}}$ is topologically conjugate to $\pmb{\rho}_{d'}\vert_{\mathbb{S}^1}$, for some $d'\geq 2$. 
\end{enumerate}
\end{prop}
\begin{proof}
1) This follows from the fact that $\Gamma$ is generated by reflections in the circles of a finite circle packing with $2$-connected and outerplanar contact graph.

2) Recall that $\mathcal{F}_\Gamma=T^0(\Gamma)\sqcup\Pi^0(\Gamma)$. Since $\Omega_\infty(\Gamma)\supset \Pi^0(\Gamma)$ is an invariant component of $\Omega(\Gamma)$, it follows from Proposition~\ref{fund_dom_prop} that
$$
\Omega(\Gamma)\setminus\Omega_\infty(\Gamma)=\Gamma\left( T^0(\Gamma)\right).
$$
Proposition~\ref{orbit_equiv_prop} now implies that
$$
\Omega(\Gamma)\setminus\Omega_\infty(\Gamma)=\bigcup_{n\geq0}\rho_\Gamma^{-n}(T^0(\Gamma)).
$$
As $\Omega(\Gamma)\setminus\Omega_\infty(\Gamma)$ is also $\Gamma$-invariant, we have that 
$$
\partial\left(\Omega(\Gamma)\setminus\Omega_\infty(\Gamma)\right)=\Lambda(\Gamma).
$$
Hence, 
$$
\mathcal{K}(\Gamma)=\Lambda(\Gamma)\sqcup\left(\Omega(\Gamma)\setminus\Omega_\infty(\Gamma)\right)=\overline{\bigcup_{n\geq0}\rho_\Gamma^{-n}(T^0(\Gamma))}.
$$

3) Complete invariance of $\Lambda(\Gamma)$ under $\rho_\Gamma$ follows from $\Gamma$-invariance of the limit set.

4) If $\mathcal{U}$ is a connected component of $\inter{\mathcal{K}(\Gamma)}$ containing a component of $T^0(\Gamma)$, then $\mathcal{U}\cap T^0(\Gamma)$ is a topological $(d'+1)$-gon (with vertices removed), for some $d'\geq 2$. Choose a quasiconformal homeomorphism $\kappa:T^0(\pmb{\Gamma}_{d'+1})\to\mathcal{U}\cap T^0(\Gamma)$ preserving the vertices. By iterated Schwarz reflections (and quasiconformal removability of analytic arcs), we obtain a quasiconformal homeomorphism $\kappa:\D\to\mathcal{U}$ conjugating $\pmb{\rho}_{d'}$ to $\rho_\Gamma$. Since $\mathcal{U}$ is a Jordan domain (by Proposition~\ref{inv_comp_limit_boundary}), the quasiconformal homeomorphism $\kappa$ extends continuously to a topological conjugacy between $\pmb{\rho}_{d'}\vert_{\mathbb{S}^1}$ and $\rho_\Gamma\vert_{\partial\mathcal{U}}$.
\end{proof}

\begin{prop}\cite[Proposition~22]{LMM20}\label{group_lamination_prop}
Let $\Gamma \in \overline{\beta(\pmb{\Gamma}_{d+1})}$. There exists a conformal map $\phi_\Gamma: \mathbb{D}^* \rightarrow \Omega_{\infty}(\Gamma)$ such that
\begin{align}\label{group_conjugacy}  
\pmb{\rho}_d(z) = \phi_\Gamma^{-1} \circ \rho_{\Gamma} \circ \phi_\Gamma(z) \textrm{, for } z\in  \mathbb{D}^*\setminus \inter{\Pi(\pmb{\Gamma}_{d+1})}.   
\end{align}  The map $\phi_\Gamma$ extends continuously to a semiconjugacy $\phi_\Gamma: \mathbb{S}^1 \rightarrow \Lambda(\Gamma)$ between $\pmb{\rho}_d\vert_{\mathbb{S}^1}$ and $\rho_{\Gamma}\vert_{\Lambda(\Gamma)}$, and for each $i$, sends the cusp of $\partial \Pi(\pmb{\Gamma}_{d+1})$ at $\mathbf{C}_i\cap\mathbf{C}_{i+1}$ to the cusp of $\partial \Pi(\Gamma)$ at $C_i\cap C_{i+1}$. 
\end{prop}

The Nielsen map $\pmb{\rho}_d$ of the group $\pmb{\Gamma}_{d+1}$, restricted to the limit set $\mathbb{S}^1$, admits the Markov partition 
$$
\mathcal{P}(\pmb{\rho}_d;\{1,e^{\frac{2\pi i}{d+1}}\cdots,e^{\frac{2\pi id}{d+1}}\}). 
$$
Note that the expanding map 
$$
z\mapsto\overline{z}^{d}:\mathbb{S}^1\to\mathbb{S}^1,
$$ 
or equivalently, 
$$
m_{-d}:\R/\Z\to\R/\Z,\ \theta\mapsto-d\theta
$$ 
also admits the same Markov partition with the same transition matrix (identifying $\mathbb{S}^1$ with $\R/\Z$). By Lemma~\ref{lemma:expansive_conjugate_fg}, there exists a homeomorphism $$\mathcal{E}_{d}:\mathbb{S}^1\to\mathbb{S}^1$$ conjugating $\pmb{\rho}_d$ to $\overline{z}^{d}$ (or $m_{-d}$) with $\mathcal{E}_{d}(1)=1$.

\bigskip

\subsection{Conformal mating}\label{conformal_mating_sec} We will now define the notion of conformal mating of the Nielsen map of a necklace group and an anti-polynomial. Our definitions follow the classical definition of conformal matings of two (anti-)polynomials, which we recall below (we refer the readers to \cite{PM12} for a more extensive discussion on conformal mating of polynomials).

Recall that the Julia set of an anti-rational map $R$ is denoted by $\mathcal{J}(R)$, and the filled Julia set and the basin of infinity of an anti-polynomial $P$ are denoted by $\mathcal{K}(P)$ and $\mathcal{B}_\infty(P)$ respectively.

Let $P$ be a monic, centered, anti-polynomial of degree $d$ such that $\mathcal{J}(P)$ is connected and locally connected. Denote by $\phi_P: \mathbb{D}^*\rightarrow \mathcal{B}_\infty(P)$ the B\"ottcher coordinate for $P$ such that $\phi_P'(\infty)=1$. We note that since $\partial \mathcal{K}(P)=\mathcal{J}(P)$ is locally connected by assumption, it follows that $\phi_P$ extends to a semiconjugacy between $z\mapsto\overline{z}^{d}\vert_{\mathbb{S}^1}$ and $P\vert_{\mathcal{J}(P)}$. 

\bigskip 

\subsubsection{Conformal mating of anti-polynomials}\label{poly_mating_subsec}

An anti-rational map $R:\widehat{\C}\to\widehat{\C}$ of degree $d\geq 2$ is said to be the \emph{conformal mating} of two degree $d$ monic, centered, anti-polynomials $P_1$ and $P_2$ with connected and locally connected filled Julia sets if and only if there exist continuous maps
\[\widetilde{\psi}_{1}: \mathcal{K}(P_1) \rightarrow \widehat{\mathbb{C}} \textrm{ and } \widetilde{\psi}_{2}: \mathcal{K}(P_2) \rightarrow \widehat{\C},  \] conformal on $\inter{\mathcal{K}(P_1)}$, $\inter{\mathcal{K}(P_2)}$, respectively, such that 

\begin{enumerate}\upshape
\item $\widetilde{\psi}_1(\mathcal{K}(P_1))\bigcup\widetilde{\psi}_2(\mathcal{K}(P_2))=\widehat{\C}$,

\item $\widetilde{\psi}_i\circ P_i=R\circ\widetilde{\psi}_i$,\ \textrm{for}\ $i\in\{1,2\}$, and

\item $\widetilde{\psi}_1(a)=\widetilde{\psi}_2(b)$ if and only if $a\sim_1 b$, where the equivalence relation $\sim_1$ on $\mathcal{K}(P_1) \sqcup \mathcal{K}(P_2)$ is generated by $\phi_{P_1}(s)\sim_1\phi_{P_2}(\overline{s})$ for all $s\in\mathbb{S}^1$. 
\end{enumerate}

\bigskip

\subsubsection{Conformal mating of a necklace group and an anti-polynomial}
Let $\Gamma \in \overline{\beta(\pmb{\Gamma}_{d+1})}$ be a necklace group generated by reflections in circles $C_1,\cdots,C_{d+1}$. By Proposition~\ref{group_lamination_prop}, there is a natural semiconjugacy $\phi_\Gamma: \mathbb{S}^1 \rightarrow \Lambda(\Gamma)$ between $\pmb{\rho}_d\vert_{\mathbb{S}^1}$ and $\rho_\Gamma\vert_{\Lambda(\Gamma)}$ such that $\phi_\Gamma(1)$ is the point of tangential intersection of $C_1$ and $C_{d+1}$. Recall also that $\mathcal{E}_d: \mathbb{S}^1 \rightarrow \mathbb{S}^1$ is a topological conjugacy between $\pmb{\rho}_d\vert_{\mathbb{S}^1}$ and $z\mapsto\overline{z}^{d}|_{\mathbb{S}^1}$.

\begin{definition}\label{conf_mating_equiv_reltn} 
We define the equivalence relation $\sim$ on $\mathcal{K}(\Gamma) \sqcup \mathcal{K}(P)$ generated by $\phi_\Gamma(t)\sim\phi_P(\overline{\mathcal{E}_d(t)})$ for all $t\in\mathbb{S}^1$.
\end{definition}

\begin{definition}\label{mating_def} 
Let $\Gamma \in \overline{\beta(\pmb{\Gamma}_{d+1})}$, and $P$ be a monic, centered anti-poly\-nomial such that $\mathcal{J}(P)$ is connected and locally connected. Further, let $\Omega\subsetneq\widehat{\C}$ be an open set, and $F:\overline{\Omega}\to\widehat{\C}$ be a continuous map that is anti-meromorphic on $\Omega$. We say that $F$ is a \emph{conformal mating} of $\Gamma$ with $P$ if there exist continuous maps \[ \psi_P: \mathcal{K}(P) \rightarrow \widehat{\mathbb{C}} \textrm{ and } \psi_\Gamma: \mathcal{K}(\Gamma) \rightarrow \widehat{\C},  \] conformal on $\inter{\mathcal{K}(P)}$, $\inter{\mathcal{K}(\Gamma)}$, respectively, such that 

\begin{enumerate}\upshape
\item\label{whole_sphere} $\psi_P(\mathcal{K}(P))\bigcup\psi_\Gamma(\mathcal{K}(\Gamma))=\widehat{\C}$,

\item\label{domain_omega} $\Omega=\widehat{\C}\setminus\psi_\Gamma(\overline{T^0(\Gamma)})$,

\item\label{poly_semiconj} $\psi_P\circ P=F\circ\psi_P$ on $\mathcal{K}(P)$, 

\item\label{nielsen_semiconj} $\psi_\Gamma\circ \rho_\Gamma=F\circ\psi_\Gamma$ on $\mathcal{K}(\Gamma)\setminus \inter{T^0(\Gamma)}$, and

\item\label{identifications} $\psi_\Gamma(z)=\psi_P(w)$ if and only if $z\sim w$ where $\sim$ is as in Definition~\ref{conf_mating_equiv_reltn}. 
\end{enumerate}
\end{definition} 

The following lemma connects the study of conformal matings of reflection groups and anti-polynomials to the theory of quadrature domains. 

\begin{lemma}\label{mating_is_schwarz}
If $F:\overline{\Omega}\to\widehat{\C}$ is a conformal mating of $\Gamma$ and $P$, then each component of $\Omega$ is a simply connected quadrature domain, and $F$ is the piecewise defined Schwarz reflection map associated with these quadrature domains.
\end{lemma}
\begin{proof}
As all but finitely many points of $\overline{T^0(\Gamma)}$ are contained in $\inter{\mathcal{K}(\Gamma)}$, injectivity of $\psi_\Gamma\vert_{\inter{\mathcal{K}(\Gamma)}}$ implies that $\psi_\Gamma$ can introduce at most finitely many identifications on $\overline{T^0(\Gamma)}$. In fact, these identifications may happen only at the singular points of $\partial T^0(\Gamma)$. It follows that $\Omega=\widehat{\C}\setminus\psi_\Gamma(\overline{T^0(\Gamma)})$ has at most finitely many connected components, say $\Omega_1,\cdots,\Omega_l$, and they satisfy the conditions
\begin{enumerate}\upshape
\item $\inter{\overline{\Omega_i}}=\Omega_i$, for $i\in\{1,\cdots, l\}$, and

\item $\displaystyle\overline{\Omega}=\bigcup_{i=1}^l\overline{\Omega_i}=\widehat{\C}\setminus\psi_\Gamma(\inter{T^0(\Gamma)}).$
\end{enumerate}

Since $\Gamma\in\overline{\beta(\pmb{\Gamma}_{d+1})}$, the set $\overline{T^0(\Gamma)}$ is connected by Proposition~\ref{regular_fund_pre_image_prop}. Hence, each $\Omega_i$ is simply connected. Possibly after conjugating $F$ by a M{\"o}bius map, we can also assume that $\infty\notin\partial\Omega_i$, for $i\in\{1,\cdots,l\}$. Finally, as $\rho_\Gamma$ fixes $\partial T^0(\Gamma)$ pointwise, it follows that the anti-meromorphic map $F\vert_{\Omega_i}$ continuously extends to the identity map on $\partial\Omega_i$, for each $i$. We conclude that each component $\Omega_i$ is a simply connected quadrature domain, and $F\vert_{\overline{\Omega_i}}$ is the Schwarz reflection map associated with $\Omega_i$.
\end{proof}

\bigskip

\subsection{A general mateability theorem}\label{mating_thm_subsec}

Let $\Gamma \in \overline{\beta(\pmb{\Gamma}_{d+1})}$. By \cite[Theorem~B, Remark~68]{LMM20}, there exists a unique monic, centered, critically fixed anti-polynomial $P_\Gamma$ of degree $d$ such that 
$$
\rho_\Gamma:\Lambda(\Gamma)\to\Lambda(\Gamma)\quad \textrm{and}\quad P_\Gamma:\mathcal{J}(P_\Gamma)\to\mathcal{J}(P_\Gamma)
$$ 
are topologically conjugate (see Figures~\ref{kleinian_cusp} and \ref{fat_ant}); and 
\begin{equation}
(\mathcal{E}_d\times\mathcal{E}_d)(\lambda(\Gamma))=\lambda(P_\Gamma),
\label{lami_related_eqution}
\end{equation}
where $\lambda(\Gamma)$ (respectively, $\lambda(P_\Gamma)$) is the equivalence relation on $\R/\Z$ determined by the fibers of $\phi_\Gamma:\R/\Z\to\Lambda(\Gamma)$ (respectively, of $\phi_{P_\Gamma}:\R/\Z\to\mathcal{J}(P_\Gamma)$). Moreover, the topological conjugacy $\mathfrak{H}:\Lambda(\Gamma)\to\mathcal{J}(P_\Gamma)$ satisfies 
\begin{equation}
\mathfrak{H}(\phi_\Gamma(t))=\phi_{P_\Gamma}(\mathcal{E}_d(t)),\ t\in\mathbb{S}^1.
\label{question_mark_induced_homeo}
\end{equation}

Thanks to Lemma~\ref{david_surgery_lemma}, the question of mateability of an anti-polynomial $P$ and the group $\Gamma$ can be reduced to the question of mateability of the pair of anti-polynomials $P$ and $P_\Gamma$. The main idea is to pass from the conformal mating of two anti-polynomials to that of an anti-polynomial and a necklace group by  gluing Nielsen maps of ideal polygon groups in suitable invariant Fatou components of anti-rational maps. 

\begin{prop}\label{mating_prop}
Let $P$ be a monic, postcritically finite anti-polynomial of degree $d$, and $\Gamma \in \overline{\beta(\pmb{\Gamma}_{d+1})}$; i.e., $\Gamma$ is a necklace group. Then, $\Gamma$ and $P$ are conformally mateable if $P_\Gamma$ and $P$ are conformally mateable.
\end{prop}

\begin{proof}
We suppose that $R$ is a conformal mating of $P_\Gamma$ and $P$. Then, by definition (see Subsection~\ref{poly_mating_subsec}), there exist continuous maps 
\[\widetilde{\psi}_{P_\Gamma}: \mathcal{K}(P_\Gamma) \rightarrow \widehat{\mathbb{C}} \textrm{ and } \widetilde{\psi}_P: \mathcal{K}(P) \rightarrow \widehat{\C},  \] conformal on $\inter{\mathcal{K}(P_\Gamma)}$, $\inter{\mathcal{K}(P)}$, respectively, such that 

\begin{enumerate}\upshape
\item $\widetilde{\psi}_{P_\Gamma}(\mathcal{K}(P_\Gamma))\bigcup\widetilde{\psi}_P(\mathcal{K}(P))=\widehat{\C}$,

\item $\widetilde{\psi}_{P_\Gamma}\circ P_\Gamma=R\circ\widetilde{\psi}_{P_\Gamma}$ on $\mathcal{K}(P_\Gamma)$, 

\item$\widetilde{\psi}_P\circ P=R\circ\widetilde{\psi}_P$ on $\mathcal{K}(P)$, and

\item $\widetilde{\psi}_{P_\Gamma}(a)=\widetilde{\psi}_P(b)$ if and only if $a\sim_1 b$, where the equivalence relation $\sim_1$ on $\mathcal{K}(P_\Gamma) \sqcup \mathcal{K}(P)$ is generated by $\phi_{P_\Gamma}(s)\sim_1\phi_P(\overline{s})$ for all $s\in\mathbb{S}^1$.
\end{enumerate}
In particular, $R$ is a postcritically finite (hence, subhyperbolic) anti-rational map.

Let $\mathcal{U}_i$, $i\in\{1,\cdots, k\}$, be the components of $\inter{\mathcal{K}(\Gamma)}$ containing the components of $T^0(\Gamma)$. By Proposition~\ref{regular_fund_pre_image_prop}(part 4), $\rho_\Gamma\vert_{\partial\mathcal{U}_i}$ is topologically conjugate to $\pmb{\rho}_{d_i}\vert_{\mathbb{S}^1}$ (for some $d_i\geq 2$), and hence is a degree $d_i$ expansive orientation-reversing covering of $\partial\mathcal{U}_i$. Under $\mathfrak{H}$, the boundaries of these components $\mathcal{U}_i$ are mapped to the boundaries of invariant (bounded) Fatou components $\widetilde{U}_1,\cdots, \widetilde{U}_k$ of $P_\Gamma$ such that $P_\Gamma\vert_{\partial \widetilde{U}_i}$ is topologically conjugate to $\overline{z}^{d_i}\vert_{\mathbb{S}^1}$. Thus, $$U_i:=\widetilde{\psi}_{P_\Gamma}(\widetilde{U}_i)\subset\widetilde{\psi}_{P_\Gamma}(\mathcal{K}(P_\Gamma))$$ is an invariant Fatou component of the anti-rational map $R$. Moreover, the map $R\vert_{\partial U_i}$ is topologically semiconjugate to $\overline{z}^{d_i}\vert_{\mathbb{S}^1}$. 

We will now glue the dynamics of $\rho_\Gamma\vert_{\mathcal{U}_i}$ in each $U_i$. This will be done in two steps. We first glue the Nielsen map of the regular ideal $(d_i+1)$-gon group in $U_i$, and then quasiconformally deform it to the Nielsen map $\rho_\Gamma\vert_{\mathcal{U}_i}$. To achieve the first goal, note that Lemma~\ref{david_surgery_lemma}, applied to the Fatou components $U_1,\cdots, U_k$ of $R$, provides us with a global David homeomorphism $\Psi_1$ and an anti-analytic map $F_1$ (defined on a subset of $\widehat{\C}$) such that $F_1$ is conformally conjugate to $\pmb{\rho}_{d_i}\vert_{\D}$ on $\Psi_1(U_i)$ ($i\in\{1,\cdots k\}$), and topologically conjugate to $R\vert_{\mathcal{J}(R)}$ on $\Psi_1(\mathcal{J}(R))$. Let us set
$$
\Omega_1:=\widehat{\C}\setminus\bigcup_{i=1}^k\overline{\Psi_1(T_i)},
$$
where $T_i$ is as in the proof of Lemma~\ref{david_surgery_lemma}. Note that since at most finitely many points on various $\partial \Psi_1(T_i)$ may touch, the domain of definition of $F_1$ is $\overline{\Omega_1}$. We now choose quasiconformal homeomorphisms from the topological $(d_i+1)$-gon $\Psi_1(U_i)\setminus\Omega_1$ onto the topological $(d_i+1)$-gon $\mathcal{U}_i\cap T^0(\Gamma)$ that preserve the vertices. We then pull back the standard complex structure on $\mathcal{U}_i\cap T^0(\Gamma)$ by this quasiconformal homeomorphism to $\Psi_1(U_i)\setminus\Omega_1$ ($i\in\{1,\cdots, k\}$), spread it by the dynamics to all of $\inter{\Psi_1(\widetilde{\psi}_{P_\Gamma}(\mathcal{K}(P_\Gamma)))}$, and put the standard complex structure on the rest of the sphere. The Measurable Riemann Mapping Theorem now gives us a global quasiconformal homeomorphism $\Psi_2$ that is conformal on $\inter{\Psi_1(\widetilde{\psi}_P(\mathcal{K}(P)))}$, and conjugates $F_1$ to a continuous map $F:\overline{\Psi_2(\Omega_1)}\to\widehat{\C}$ that is anti-analytic on $\Psi_2(\Omega_1)$. By construction, setting $\Psi:=\Psi_2\circ\Psi_1$ and $\Omega:=\Psi_2(\Omega_1)$, we have that $F:\overline{\Omega}\to\widehat{\C}$ is conformally conjugate to $\rho_\Gamma\vert_{\mathcal{U}_i}$ on $\Psi(U_i)$ ($i\in\{1,\cdots, k\}$), and topologically conjugate to $R\vert_{\mathcal{J}(R)}$ on $\Psi(\mathcal{J}(R))$.

We will now argue that $F$ is a conformal mating of $\Gamma$ and $P$. To this end, let us set 
$$
\psi_P:=\Psi\circ\widetilde{\psi}_P:\mathcal{K}(P)\to\widehat{\C}.
$$ 
Note that since $R\equiv\widetilde{R}$ on $\widetilde{\psi}_P(\mathcal{K}(P))$ and $\Psi$ is conformal on $\widetilde{\psi}_P(\inter{\mathcal{K}(P)})$, it follows that $\psi_P$ is conformal on $\inter{\mathcal{K}(P)}$ and $\psi_P\circ P=F\circ\psi_P$ on $\mathcal{K}(P)$.

We now set 
$$
\psi_\Gamma:=\Psi\circ\widetilde{\psi}_{P_\Gamma}\circ\mathfrak{H}:\Lambda(\Gamma)\to \Psi(\mathcal{J}(R)).
$$ 
Then, $\psi_\Gamma\circ \rho_\Gamma=F\circ\psi_\Gamma$ on $\Lambda(\Gamma)$. By our construction of $F$, we can extend $\psi_\Gamma\vert_{\Lambda(\Gamma)}$ to a conformal map 
$$
\psi_\Gamma:\inter{\mathcal{K}(\Gamma)}\to\Psi\left(\widetilde{\psi}_{P_\Gamma}(\inter{\mathcal{K}(P_\Gamma)})\right)
$$ 
such that $\psi_\Gamma(\overline{T^0(\Gamma)})=\bigcup_{i=1}^k\overline{\Psi(T_i)}$, and $\psi_\Gamma\circ \rho_\Gamma=F\circ\psi_\Gamma$ on $\mathcal{K}(\Gamma)\setminus \inter{T^0(\Gamma)}$. It also follows that 
$$
\Omega=\Psi_2(\Omega_1)=\widehat{\C}\setminus\bigcup_{i=1}^k\overline{\Psi(T_i)}=\widehat{\C}\setminus\psi_\Gamma\left(\overline{T^0(\Gamma)}\right),
$$
and
$$
\psi_P(\mathcal{K}(P))\bigcup\psi_\Gamma(\mathcal{K}(\Gamma))=\Psi\left(\widetilde{\psi}_P(\mathcal{K}(P))\bigcup\widetilde{\psi}_{P_\Gamma}(\mathcal{K}(P_\Gamma))\right)=\widehat{\C}.
$$

Thus, $F$ satisfies conditions \eqref{whole_sphere}--\eqref{nielsen_semiconj} of of Definition~\ref{mating_def}. It now remains to verify condition~\eqref{identifications}. 

To this end, let us first choose $z=\phi_\Gamma(t)\in\Lambda(\Gamma)$ and $w=\phi_P(\overline{\mathcal{E}_d(t)})\in\mathcal{J}(P)$, for some $t\in\mathbb{S}^1$. Then, 
$$
\psi_\Gamma(z) =(\Psi\circ\widetilde{\psi}_{P_\Gamma}\circ\mathfrak{H})(\phi_\Gamma(t))=\Psi(\widetilde{\psi}_{P_\Gamma}(\phi_{P_\Gamma}(\mathcal{E}_d(t))))=\Psi(\widetilde{\psi}_P(\phi_P(\overline{\mathcal{E}_d(t)})))=\psi_P(w).
$$ Thus, $\psi_\Gamma(z)=\psi_P(w)$ whenever $z\sim w$.

Conversely, let $$\psi_\Gamma(z)=\psi_P(w)$$ for some $z\in\Lambda(\Gamma)$ and $w\in\mathcal{J}(P)$. By definition of $\psi_\Gamma$ and $\psi_P$, and the fact that $\Psi$ is a homeomorphism, the above implies that $\widetilde{\psi}_{P_\Gamma}(\mathfrak{H}(z))=\widetilde{\psi}_P(w)$. Hence, there exists $s\in\mathbb{S}^1$ such that $$\mathfrak{H}(z)=\phi_{P_\Gamma}(s),\ w=\phi_P(\overline{s}).$$ Now set $t=\mathcal{E}_d^{-1}(s)$. Then, $$\mathfrak{H}(z)=\phi_{P_\Gamma}(\mathcal{E}_d(t))=\mathfrak{H}(\phi_\Gamma(t)) \implies z=\phi_\Gamma(t).$$ Therefore, $z=\phi_\Gamma(t)$ and $w=\phi_P(\overline{\mathcal{E}_d(t)})$; and hence, $z\sim w$. The proof is now complete.
\end{proof}

In order to apply the above proposition, we quote a mateability result for anti-polynomials from \cite{LLM20}.

\begin{prop}\cite[Proposition~4.23]{LLM20}\label{prop:mig}
Let $P$ and $Q$ be marked anti-polynomials of equal degree $d\geq 2$, where $P$ is postcritically finite, hyperbolic, and $Q$ is critically fixed. There is an anti-rational map that is the conformal mating of $P$ and $Q$ if and only if $\mathcal{K}(P)\sqcup\mathcal{K}(Q)/\sim_1$ is homeomorphic to $\mathbb{S}^2$, where the equivalence relation $\sim_1$ on $\mathcal{K}(P) \sqcup \mathcal{K}(Q)$ is generated by $\phi_{P}(s)\sim_1\phi_Q(\overline{s})$ for all $s\in\mathbb{S}^1$.
\end{prop}

\begin{remark}
A simple condition to check whether $\mathcal{K}(P)\sqcup\mathcal{K}(Q)/\sim_1$ is homeomorphic to $\mathbb{S}^2$  is given in \cite[Corollary~24]{LLM20} in terms of the angles of the dynamical rays (of the critically fixed anti-polynomial $Q$) landing at the repelling fixed points of $Q$.
\end{remark}

We are now ready to prove a precise version of Theorem~\ref{mating_intro_version}. The equivalence relation $\sim$ appearing in the statement below is the one introduced in Definition~\ref{conf_mating_equiv_reltn}.

\begin{theorem}[Criterion for Mateability]\label{poly_group_mating_thm}
Let $P$ be a monic, postcritically finite, hyperbolic anti-polynomial of degree $d$, and $\Gamma \in \overline{\beta(\pmb{\Gamma}_{d+1})}$; i.e., $\Gamma$ is a necklace group. Then, $P$ and $\Gamma$ are conformally mateable if and only if $\mathcal{K}(P)\sqcup\mathcal{K}(\Gamma)/\sim$ is homeomorphic to $\mathbb{S}^2$.

Moreover, if a conformal mating $F:\overline{\Omega}\to\widehat{\C}$ of $P$ and $\Gamma$ exists, then each component of $\Omega$ is a simply connected quadrature domain and $F$ is the piecewise defined Schwarz reflection map associated with these quadrature domains.
\end{theorem}
\begin{proof}
It is obvious from Definition~\ref{mating_def} that if $P$ and $\Gamma$ are mateable, then $\mathcal{K}(P)\sqcup\mathcal{K}(\Gamma)/\sim$ must be homeomorphic to $\mathbb{S}^2$. 

For the converse, let us assume that $\mathcal{K}(P)\sqcup\mathcal{K}(\Gamma)/\sim$ is homeomorphic to $\mathbb{S}^2$. It is easy to check using the definitions of $\sim$ and $\sim_1$, and Relation~\eqref{question_mark_induced_homeo} that the topological spaces $\mathcal{K}(P)\sqcup\mathcal{K}(\Gamma)/\sim$ and $\mathcal{K}(P)\sqcup\mathcal{K}(P_\Gamma)/\sim_1$ are homeomorphic. In particular, $\mathcal{K}(P)\sqcup\mathcal{K}(P_\Gamma)/\sim_1$ is homeomorphic to $\mathbb{S}^2$. Hence, Proposition~\ref{prop:mig} implies that the anti-polynomials $P$ and $P_\Gamma$ are conformally mateable. The desired conclusion that $P$ and $\Gamma$ are conformally mateable now follows from Proposition~\ref{mating_prop}.

The last statement is the content of Lemma~\ref{mating_is_schwarz}.
\end{proof}

\begin{remark}
If $F:\overline{\Omega}\to\widehat{\C}$ is a conformal mating of $P$ and $\Gamma$, then $\widehat{\C}\setminus\Omega$ is a `graph of polygons'; see Figures~\ref{mating_1_fig},~\ref{mating_2_fig}.
\end{remark}

\begin{remark}\label{why_representation}
In the statements of Proposition~\ref{mating_prop} and Theorem~\ref{poly_group_mating_thm}, being mateable depends on the representation $\xi: \pmb{\Gamma}_{d+1}\to\Gamma$ (or equivalently, the labeling of the circles $C_1,\cdots, C_{d+1}$, where $\xi(\rho_i)$ is the reflection in the circle $C_i$, for $i\in\{1,\cdots,d+1\}$). Indeed, the lamination $\lambda(P_\Gamma)$ of the anti-polynomial $P_\Gamma$ is determined by the lamination $\lambda(\Gamma)$ of the group $\Gamma$ (via relation~\eqref{lami_related_eqution}), and the lamination $\lambda(\Gamma)$ depends on the choice of the conformal map $\phi_\Gamma$, which, in turn, depends on the labeling of the circles $C_1,\cdots, C_{d+1}$ (see Proposition~\ref{group_lamination_prop}). Roughly speaking, different representations give rise to different ways of gluing the limit set of $\Gamma$ with the Julia set of an anti-polynomial, and the choice of gluing determines whether a conformal mating exists.

For example, in light of Proposition~\ref{mating_prop} and \cite[Corollary~4.24]{LLM20}, it is easy to see that the anti-polynomial $P_1(z)=\overline{z}^3-\frac{3i}{\sqrt{2}}\overline{z}$ (see Figure~\ref{double_basilica}) is conformally mateable with the group $\Gamma$ from Figure~\ref{kleinian_cusp} with the labeling of the underlying circle packing shown in Figure~\ref{kleinian_cusp}(left). However, if we consider the same group $\Gamma$ with a different labeling of the underlying circle packing such that the circles $C_2$ and $C_4$ touch at a point (note that this amounts to looking at a different element of $\overline{\beta(\pmb{\Gamma}_{d+1})}$), then it is no longer conformally mateable with the anti-polynomial $P_1$ since with this new marking of the generators of $\Gamma$, the quotient space $\mathcal{K}(P_1)\sqcup\mathcal{K}(\Gamma)/\sim$ is not homeomorphic to $\mathbb{S}^2$.
\end{remark}

\bigskip

\section{Mating reflection groups with anti-polynomials: examples}\label{mating_examples_sec}

The goal of this section is to illustrate Proposition~\ref{mating_prop} and Theorem~\ref{poly_group_mating_thm} by producing various examples of matings of anti-polynomials and necklace reflection groups. While these results guarantee the existence of conformal matings of suitable anti-polynomials and necklace reflection groups, in general, it may be hard to find explicit Schwarz reflection maps realizing the conformal matings.

However, there are two ways to achieve this in low degrees. The first one, implemented in \cite{LLMM18a,LLMM18b,LMM20}, is to study the dynamics and parameter spaces of specific families of Schwarz reflection maps, and recognize such maps as matings of anti-polynomials and reflection groups. In Subsection~\ref{known_examples_subsec}, we will indicate how the examples studied in these papers fit in our general mating framework.

In the opposite direction, to explicitly characterize the conformal mating of a given necklace group and an anti-polynomial, let us first recall that by Lemma~\ref{mating_is_schwarz}, the conformal mating is a piecewise defined Schwarz reflection map associated with a finite collection of simply connected quadrature domains. This allows one to uniformize the domains by rational maps of suitable degrees (compare Proposition~\ref{simp_conn_quad}), and use the desired dynamical properties to explicitly find these rational maps. We note that this approach requires care when both the anti-polynomial and the necklace group have non-trivial laminations; indeed, one needs to read off the contact pattern of the finite collection of simply connected quadrature domains (whose Schwarz reflection maps define the conformal mating) and the degrees of the corresponding uniformizing rational maps from the laminations of the anti-polynomial and the necklace group (equipped with a labeling of the underlying circle packing). This will be illustrated with a couple of worked out examples in Subsections~\ref{example_1_subsec} and~\ref{example_2_subsec}. The final Subsection~\ref{cauli_triangle_mating_subsec} underscores the additional analytic steps required to characterize conformal matings of parabolic anti-polynomials and necklace groups.

\bigskip 

\subsection{Some known examples}\label{known_examples_subsec} 1. For $\Gamma:=\pmb{\Gamma}_{d+1}$, the associated anti-polynomial $P_\Gamma$ is given by $P_\Gamma(z)=\overline{z}^d$. Since $\overline{z}^d$ is conformally mateable with every anti-polynomial $P$ of degree $d$, we conclude from Proposition~\ref{mating_prop} that $\pmb{\Gamma}_{d+1}$ is mateable with every postcritically finite anti-polynomial $P$ of degree $d$. In the particular case of $d=2$, the mating of $\pmb{\Gamma}_3$ and $\overline{z}^2$ is realized as the Schwarz reflection map with respect to a deltoid \cite[\S 4]{LLMM18a}, and the matings of $\pmb{\Gamma}_3$ and all other postcritically finite quadratic anti-polynomials are realized as the Schwarz reflection maps associated with a fixed cardioid and a family of circumscribed circles \cite{LLMM18b}.

2. Since $\overline{z}^d$ can be conformally mated with every critically fixed degree $d$ anti-polynomial, it follows once again from Proposition~\ref{mating_prop} and the fact that the limit set of each group in the Bers slice closure is homeomorphic to the Julia set of some critically fixed anti-polynomial (see the comments before Proposition~\ref{mating_prop}) that all necklace reflection groups can be mated with the anti-polynomial $\overline{z}^d$. By \cite{LMM20}, these conformal matings are realized as Schwarz reflection maps associated with the quadrature domains $f(\D^*)$, where $$f\in\Sigma_d^*:= \left\{ g(z)= z+\frac{a_1}{z} + \cdots +\frac{a_d}{z^d} : a_d=-\frac{1}{d}\textrm{ and } g\vert_{\mathbb{D}^*} \textrm{ is conformal}\right\}.$$

\bigskip

\subsection{Schwarz reflections in an ellipse and two inscribed disks}\label{example_1_subsec}
Consider the anti-polynomial $P_1(z)=\overline{z}^3-\frac{3i}{\sqrt{2}}\overline{z}$. Each finite critical point of $P_1$ forms a $2$-cycle (see Figure~\ref{double_basilica}).

\begin{figure}[ht!]
\begin{center}
\includegraphics[scale=0.4]{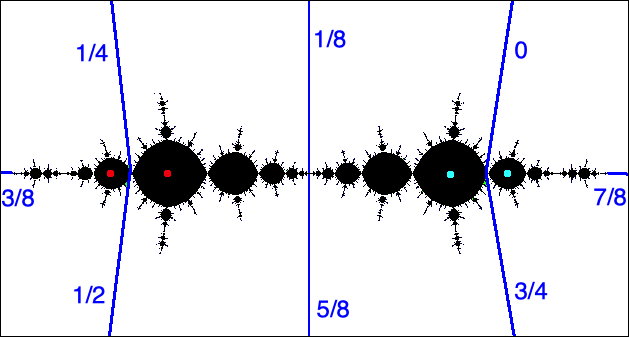}
\end{center}
\caption{The dynamical plane of $P_1(z)=\overline{z}^3-\frac{3i}{\sqrt{2}}\overline{z}$; each critical point of $P_1$ forms a $2$-cycle (the figure displayed is a $\frac{\pi}{4}$-rotate of the actual dynamical plane). The external dynamical rays of period $1$ and $2$ are marked.}
\label{double_basilica}
\end{figure}

Note that the $1/4$ and $1/2$ rays (respectively, the $0$ and $3/4$ rays) of $P_1$ land at a common fixed point on $\mathcal{J}(P_1)$. They cut $\mathcal{K}(P_1)$ into three components. We will denote the component containing the two critical points by $\mathcal{K}_1(P_1)$, the component containing the critical value in the left half-plane (respectively, the critical value in the right half-plane) by $\mathcal{K}_2(P_1)$ (respectively, $\mathcal{K}_3(P_1)$). We have the following mapping properties of the action of $P_1$ on its filled Julia set.
\begin{enumerate}\upshape
\item $P_1:\mathcal{K}_i(P_1)\to \mathcal{K}(P_1)\setminus\overline{\mathcal{K}_i(P_1)}$ has degree $1$, for $i\in\{2,3\}$, and 

\item $P_1: P_1^{-1}\left(\mathcal{K}_1(P_1)\right)\cap\mathcal{K}_1(P_1)\to\mathcal{K}_1(P_1)$ has degree $1$.
\end{enumerate}

\begin{figure}[ht!]
\includegraphics[scale=0.4]{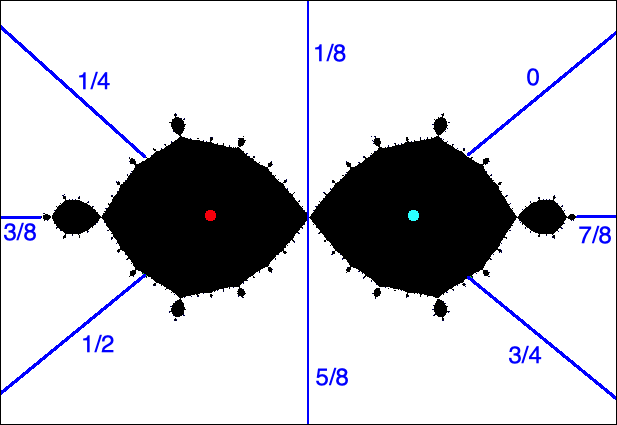}
\caption{The dynamical plane of the critically fixed cubic anti-polynomial $P_\Gamma(z)=\overline{z}^3-\frac{3i}{2}\overline{z}$ (the figure displayed is a $\frac{\pi}{4}$-rotate of the actual dynamical plane). The external dynamical rays of period $1$ and $2$ are marked.}
\label{fat_ant}
\end{figure}

Also consider the cusp reflection group $\Gamma$ shown in Figure~\ref{kleinian_cusp}, and the associated reflection map $\rho_\Gamma$. The monic, centered, critically fixed anti-polynomial $P_\Gamma$ associated to $\Gamma$ in Subsection~\ref{mating_thm_subsec} is given by $P_\Gamma(z)=\overline{z}^3-\frac{3i}{2}\overline{z}$ (see Figure~\ref{fat_ant}). 

We will now argue that $P_\Gamma$ and $P_1$ are conformally mateable. Note that the only rays landing at the separating repelling fixed point (i.e., the repelling fixed point that is a cut-point of the filled Julia set) of the critically fixed anti-polynomial $P_\Gamma$ have angles $1/8$ and $5/8$, while for the other anti-polynomial $P_1$, the rays at angles $-1/8=7/8$ and $-5/8=3/8$ land at non-cut points of $\mathcal{J}(P_1)$. Therefore, the principal ray equivalence class for $P_\Gamma$ and $P_1$ contains no cycle, and hence by \cite[Corollary~4.24, Figure~4.2]{LLM20}, the maps $P_\Gamma$ and $P_1$ are conformally mateable.

By Proposition~\ref{mating_prop}, there exists a conformal mating $F:\overline{\Omega}\to\widehat{\C}$ of $P_1$ and $\Gamma$. We set $T^0(F):=\psi_\Gamma(T^0(\Gamma))$. Since $\psi_\Gamma$ is conformal on $\inter{\mathcal{K}(\Gamma)}$, each of the two components of $T^0(F)$ is a topological triangle with its vertices removed. Moreover, by Lemma~\ref{mating_is_schwarz}, each component of $\Omega$ is a simply connected quadrature domain, and $F$ is the piecewise defined Schwarz reflection map of these quadrature domains.

According to \cite[Lemma~52 (part i)]{LMM20}, the four cusp points of $\partial T^0(\Gamma)$ (that are not cut-points of $\partial T^0(\Gamma)$) have external angles $0,1/4,1/2$, and $3/4$ (as points on $\Lambda(\Gamma)$). Since the $1/4$ and $1/2$ rays (respectively, the $0$ and $3/4$ rays) of $P_1$ land at a common fixed point, the landing points of the $1/2$ and $3/4$ rays (respectively, the landing points of the $0$ and $1/4$ rays) of $P_\Gamma$ are identified in the Julia set of the conformal mating of $P_\Gamma$ and $P_1$. There is no other identification involving the landing points of the fixed rays of $P_\Gamma$. It follows from Relation~\eqref{question_mark_induced_homeo} that each vertex of a component of $T^0(F)$ is identified with a vertex of the other component of $T^0(F)$. Hence, $\Omega$ has three connected components $\Omega_1, \Omega_2,$ and $\Omega_3$ such that 
\begin{enumerate}\upshape
\item each $\Omega_i$ is a Jordan quadrature domain,

\item $\psi_{P_1}(\mathcal{K}_i(P_1))\subset\overline{\Omega_i},\ i=1,2,3$,

\item $\partial\Omega_i\cap\partial\Omega_j$ is a singleton, for each $i\neq j\in\{1,2,3\}$.
\end{enumerate}
(See Figure~\ref{ellipse_disk_how}.)

\begin{figure}[ht!]
\includegraphics[width=0.96\textwidth]{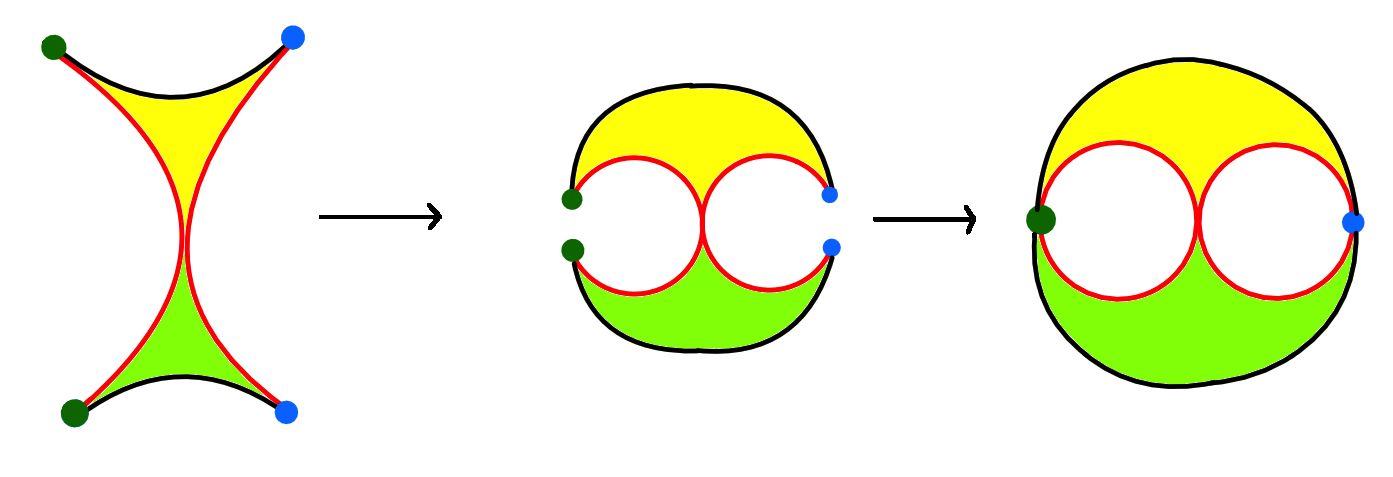}
\caption{The figure depicts the identifications on the boundary of the fundamental domain $T^0(\Gamma)$ that give rise to three pairwise touching Jordan domains.}
\label{ellipse_disk_how}
\end{figure}

Denoting the Schwarz reflection maps of $\Omega_i$ by $\sigma_i$, we have that  $$F(w)= \sigma_i(w),\ \textrm{if}\ w\in\overline{\Omega_i}.$$

It follows from the mapping degrees of $P_1$ on $\mathcal{K}_1(P_1)$, $\mathcal{K}_2(P_1)$, and $\mathcal{K}_3(P_1)$ that $\sigma_1:\sigma_1^{-1}(\inter{\Omega_1})\to\inter{\Omega_1}$ is a branched cover of degree $2$, and $\sigma_i:\sigma_i^{-1}(\inter{\Omega_i^c})\to\inter{\Omega_i^c}$ has degree $1$, for $i\in\{2,3\}$. By Proposition~\ref{simp_conn_quad}, there exist rational maps $f_1, f_2,$ and $f_3$ such that 
\begin{enumerate}\upshape
\item $f_i:\D^*\to\Omega_i$ are univalent, for $i\in\{1,2,3\}$,

\item $F\vert_{\Omega_i}\equiv f_i\circ (1/\overline{\zeta})\circ (f_i\vert_{\D^*})^{-1}$, and

\item $\Deg{f_1}=2$, $\Deg{f_i}=1$, $i\in\{2,3\}$.
\end{enumerate}

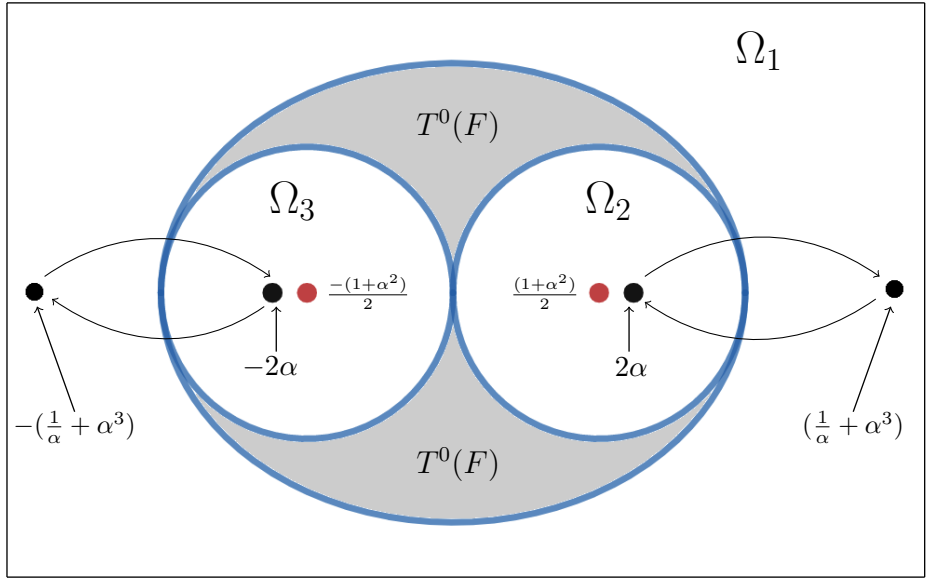
\begin{figure}[ht!]
\begin{tikzpicture}
 \node[anchor=south west,inner sep=0] at (0,0) {\includegraphics[width=0.96\textwidth]{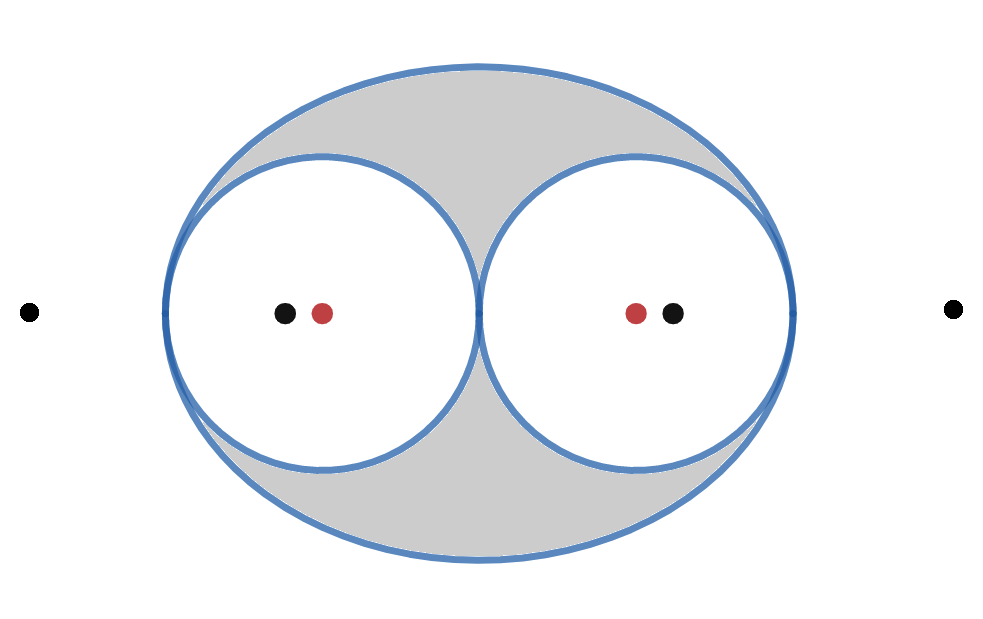}};
\node at (3.8,5) {\begin{huge}$\Omega_3$\end{huge}};
\node at (8,5) {\begin{huge}$\Omega_2$\end{huge}};
\node at (6,6) {\begin{Large}$T^0(F)$\end{Large}};
\node at (6,1.5) {\begin{Large}$T^0(F)$\end{Large}};
\node at (10,7) {\begin{huge}$\Omega_1$\end{huge}};
\node at  (3.5,2.8) {\begin{large}$-2\alpha$\end{large}};
\node at (8.3,2.8) {\begin{large}$2\alpha$\end{large}};
\draw[->,line width=0.5pt] (3.58,3)-|(3.58,3.6);
\draw[->,line width=0.5pt] (8.27,3)-|(8.27,3.6);
\draw [->] (8.48,4) to [out=35,in=145] (11.6,4);
\draw [->] (11.55,3.7) to [out=215,in=325] (8.48,3.66);
\draw [->] (0.5,4) to [out=35,in=145] (3.5,4);
\draw [->] (3.42,3.6) to [out=215,in=325] (0.6,3.7);
\node at (7.15,3.8) {\begin{small}$\frac{(1+\alpha^2)}{2}$\end{small}};
\node at (4.8,3.8) {\begin{small}$\frac{-(1+\alpha^2)}{2}$\end{small}};
\node at  (11.25,2) {$(\frac{1}{\alpha}+\alpha^3)$};
\node at (0.9,2) {$-(\frac{1}{\alpha}+\alpha^3)$};
\draw[->,line width=0.5pt] (11.25,2.2)->(11.75,3.6);
\draw[->,line width=0.5pt] (0.9,2.2)->(0.4,3.6);
\draw[-,line width=0.5pt] (0,0)->(12.2,0);
\draw[-,line width=0.5pt] (12.2,0)->(12.2,7.64);
\draw[-,line width=0.5pt] (12.2,7.64)->(0,7.64);
\draw[-,line width=0.5pt] (0,7.64)->(0,0);
\end{tikzpicture}
\caption{The conformal mating of $P_1(z)=\overline{z}^3-\frac{3i}{\sqrt{2}}\overline{z}$ and the necklace group from Figure~\ref{kleinian_cusp} is given by the piecewise Schwarz reflection map $F$ associated with the quadrature domains $\Omega_i$, $i\in\{1,2,3\}$; where $\Omega_1$ is the exterior of the ellipse $\frac{x^2}{(1+\alpha^2)^2}+\frac{y^2}{(1-\alpha^2)^2}=1$, and $\Omega_2, \Omega_3$ are the round disks $\vert z\pm \frac{1+\alpha^2}{2}\vert= \frac{1+\alpha^2}{2}$ (where $\alpha=\frac12\left((1+\sqrt{5})-\sqrt{2+2\sqrt{5}}\right)$). The two components of $T^0(F)$ are topological triangles with vertices removed. Each vertex of a component of $T^0(F)$ is identified with a vertex of the other component of $T^0(F)$. Each of the two critical points of $F$ forms a $2$-cycle.}
\label{mating_1_fig}
\end{figure}

Note that $0$ is a repelling fixed point of $P_1$. Conjugating $F$ by a M{\"o}bius map, we can assume that $\psi_{P_1}(0)=\infty\in\Omega_1$, and hence, $F(\infty)=\infty$. We can choose $f_1$ so that $f_1(\infty)=\infty$ and $f'(\infty)>0$; i.e., $$f_1(z)=\frac{az^2+bz+c}{z+d},$$ where $a>0$, $b,c,d\in\C$. Since $\infty$ is a fixed point of $F$, it follows from the formula 
$$
F\vert_{\Omega_1}\equiv f_1\circ (1/\overline{\zeta})\circ (f_1\vert_{\D^*})^{-1}
$$ 
that $f_1(0)=\infty$; i.e., $d=0$. Thus, $f_1$ reduces to the form $$f_1(z)=az+b+c/z.$$ We are now only allowed to conjugate $F$ by affine maps as conjugating by non-affine maps will, in general, destroy the normalization $f_1(0)=\infty$. Conjugating $F$ by a translation, we can now assume that $f_1(z)=az+c/z$. Finally, conjugating $F$ by a dilation and rotation, we can choose $f_1$ to be $$f_1(z)=z+\alpha^2/z,$$ for some $\alpha>0$.

\begin{figure}[ht!]
\centering
\includegraphics[width=0.48\linewidth]{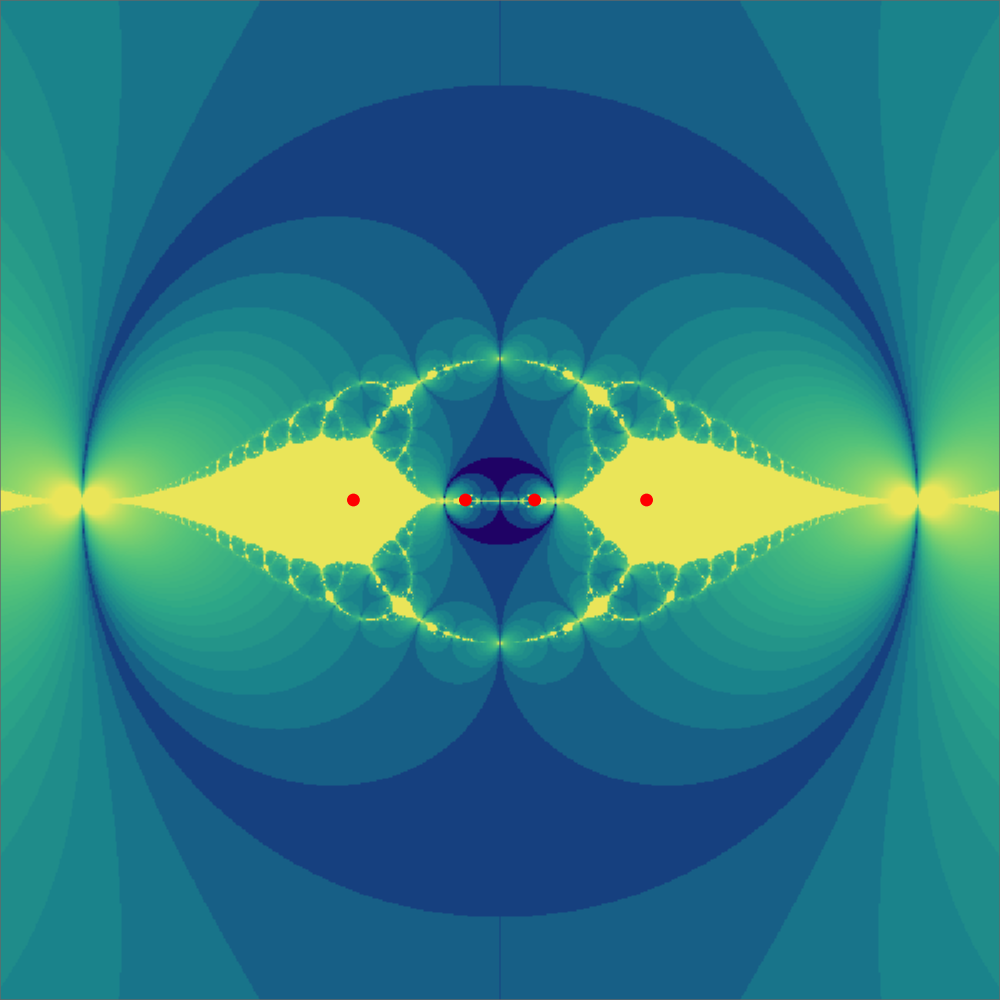}\ \includegraphics[width=0.48\linewidth]{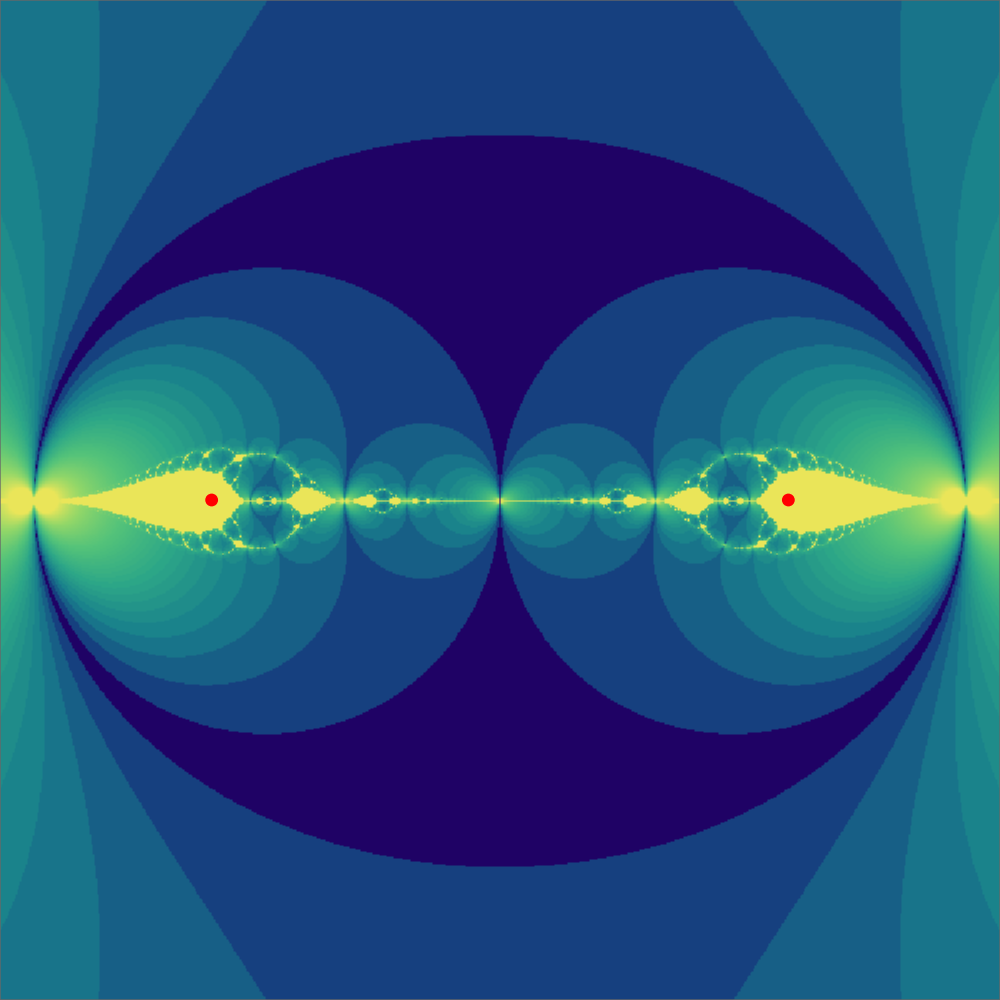}
\caption{Left: The dynamical plane of $F$ is shown. The blue/green region is the tiling set and the yellow region is the non-escaping set. The large yellow components contain the two $2$-periodic critical points (in red) of $F$. The corresponding critical values lie inside the ellipse. Right: A blow-up of the interior of the ellipse is displayed with the two critical values marked in red. The two dark blue topological triangles are the components of $T^0(F)$.}
\label{ellipse_disk_dyn_fig}
\end{figure}

A direct computation now shows that $\Omega_1$ is the exterior of the ellipse $$\frac{x^2}{(1+\alpha^2)^2}+\frac{y^2}{(1-\alpha^2)^2}=1.$$ The major and minor axes of the ellipse are along the real and imaginary axes. Moreover, the quadrature domains $\Omega_2$ and $\Omega_3$ are interiors of round circles (contained in the interior of the ellipse) with $\partial\Omega_i\cap\partial\Omega_j$ a singleton, for each $i\neq j\in\{1,2,3\}$. 

We note that since the critical points of $f_1$ are $\pm \alpha$, it is easy to see by direct computation that the critical points of $\sigma_1$ are $\pm(1/\alpha+\alpha^3)$. The corresponding critical values of $\sigma_1$ are $\pm 2\alpha$. Since the critical points of $P_1$ are $2$-periodic, the same is true for the conformal mating $F$ of $\Gamma$ and $P_1$. Therefore, the circle reflection $\sigma_2$ (respectively, $\sigma_3$) maps $2\alpha\in\R$ (respectively, $-2\alpha\in\R$) to $(1/\alpha+\alpha^3)\in\R$ (respectively, $-(1/\alpha+\alpha^3)\in\R$). Hence, the centers of the circles $\partial\Omega_i$ ($i=2,3$) must lie on the real axis. Finally, by symmetry, the radii of the circles $\partial\Omega_i$ ($i=2,3$) must be equal; i.e., the centers of the circles are at $\pm\frac{1+\alpha^2}{2}$, and the common radius is $\frac{1+\alpha^2}{2}$. A direct computation using the fact $\sigma_1(2\alpha)=1/\alpha+\alpha^3$ now yields that $$\alpha=\frac12\left((1+\sqrt{5})-\sqrt{2+2\sqrt{5}}\right).$$

\begin{prop}\label{mating_1_example_prop}
Let $P_1(z)=\overline{z}^3-\frac{3i}{\sqrt{2}}\overline{z}$, $\Gamma$ be the cusp reflection group from Figure~\ref{kleinian_cusp}, and $\alpha=\frac12\left((1+\sqrt{5})-\sqrt{2+2\sqrt{5}}\right)$. Then, the piecewise defined Schwarz reflection map $F$ in the exterior of the ellipse $$\frac{x^2}{(1+\alpha^2)^2}+\frac{y^2}{(1-\alpha^2)^2}=1,$$ and in the interiors of the circles $$\big| z\pm \frac{1+\alpha^2}{2}\big|= \frac{1+\alpha^2}{2}$$ is a conformal mating of $P_1$ and $\Gamma$.
\end{prop}

\bigskip

\subsection{Schwarz reflections in an extremal quadrature domain and a circumscribed disk}\label{example_2_subsec}

Consider the group $\Gamma$ as in Subsection~\ref{example_1_subsec} and the unicritical cubic anti-polynomial $P_2(z)=\overline{z}^3+\frac{(1+i)}{\sqrt{2}}$, which has a superattracting $2$-cycle (see Figure~\ref{cubic_unicrit}).

Note that the $0$ and $1/4$ rays of $P_2$ land at a common fixed point $\zeta$, and cut $\mathcal{K}(P_2)$ into two components. We will denote the component containing the critical value (respectively, the critical point) by $\mathcal{K}_1(P_2)$ (respectively, $\mathcal{K}_2(P_2)$). Note that the set $\mathcal{K}_1(P_2)\bigcup\{\zeta\}$ is mapped injectively onto $\mathcal{K}_2(P_2)\bigcup\{\zeta\}$ under $P_2$. On the other hand, under $P_2$, the set $\mathcal{K}_2(P_2)\bigcup\{\zeta\}$ covers itself twice, and $\mathcal{K}_1(P_2)\bigcup\{\zeta\}$ thrice. 

As in the Subsection~\ref{example_1_subsec}, one can show that $P_\Gamma$ (defined in Subsection~\ref{example_1_subsec}) and $P_2$ are conformally mateable. Indeed, the only rays landing at the separating repelling fixed point (i.e., the repelling fixed point that is a cut-point of the filled Julia set) of the critically fixed anti-polynomial $P_\Gamma$ have angles $1/8$ and $5/8$, while for the other anti-polynomial $P_2$, the rays at angles $-1/8=7/8$ and $-5/8=3/8$ land at non-cut points of $\mathcal{J}(P_2)$. Therefore, the principal ray equivalence class for $P_\Gamma$ and $P_2$ contains no cycle, and hence by \cite[Corollary~4.24]{LLM20}, the maps $P_\Gamma$ and $P_2$ are conformally mateable.

\begin{figure}[ht!]
\begin{center}
\includegraphics[scale=0.354]{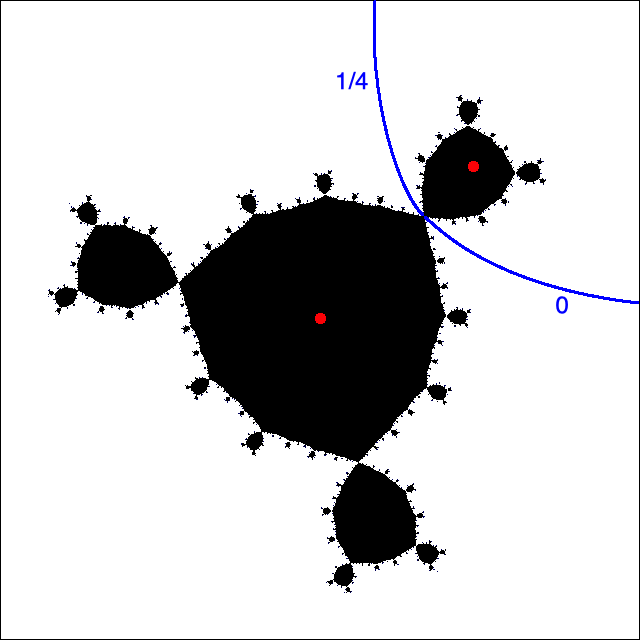}
\end{center}
\caption{The dynamical plane of the unicritical cubic anti-polynomial $P_2(z)=\overline{z}^3+\frac{(1+i)}{\sqrt{2}}$. The critical point of $P_2$ is $2$-periodic.}
\label{cubic_unicrit}
\end{figure}

By Proposition~\ref{mating_prop}, there exists a conformal mating $F:\overline{\Omega}\to\widehat{\C}$ of $P_1$ and $\Gamma$. We set $T^0(F):=\psi_\Gamma(T^0(\Gamma))$. Since $\psi_\Gamma$ is conformal on $\inter{\mathcal{K}(\Gamma)}$, each of the two components of $T^0(F)$ is a topological triangle with its vertices removed. Moreover, by Lemma~\ref{mating_is_schwarz}, each component of $\Omega$ is a simply connected quadrature domain, and $F$ is the piecewise defined Schwarz reflection map of these quadrature domains.

Note that the four cusp points of $\partial T^0(\Gamma)$ (that are not cut-points of $\partial T^0(\Gamma)$) have external angles $0,1/4,1/2$, and $3/4$ (as points on $\Lambda(\Gamma)$). Since the $0$ and $1/4$ rays of $P_2$ land at a common fixed point, the landing points of the $0$ and $3/4$ rays of $P_\Gamma$ are identified in the Julia set of the conformal mating of $P_\Gamma$ and $P_2$. There is no other identification involving the landing points of the fixed rays of $P_\Gamma$. It follows that two vertices of a component of $T^0(F)$ are identified. Hence, $\Omega$ has two connected components $\Omega_1$ and $\Omega_2$ such that 
\begin{enumerate}\upshape
\item $\Omega_1$ and $\Omega_2$ are simply connected quadrature domains,

\item $\psi_{P_2}(\mathcal{K}_1(P_2))\subset\overline{\Omega_1}$, and $\psi_{P_2}(\mathcal{K}_2(P_2))\subset\overline{\Omega_2}$,

\item $\Omega_1$ is a Jordan domain, and $\partial\Omega_2$ has a unique cut-point, 

\item $\partial\Omega_1\cap\partial\Omega_2$ is a singleton.
\end{enumerate}

Denoting the Schwarz reflection maps of $\Omega_i$ by $\sigma_i$, we have that 
$$
F(w)= \left\{\begin{array}{ll}
                    \sigma_1(w) & \mbox{if}\ w\in\overline{\Omega_1}, \\
                    \sigma_2(w) & \mbox{if}\ w\in\overline{\Omega_2}.
                                          \end{array}\right. 
$$

Conjugating $F$ by a M{\"o}bius map, we can assume that unique critical point of $F$ is $0\in\Omega_2$, the unique critical value is $\infty\in\Omega_1$, and the conformal radius of $\Omega_2$ (with conformal center at $0$) is $1$.

By mapping properties of $P_2$, we have that $\sigma_1:\sigma_1^{-1}(\inter{\Omega_1^c})\to\inter{\Omega_1^c}$ has degree $1$, and $\sigma_2:\sigma_2^{-1}(\inter{\Omega_2^c})\to\inter{\Omega_2^c}$ is a branched cover of degree $3$. By Proposition~\ref{simp_conn_quad}, there exist rational maps $f_1$ and $f_2$ such that 
\begin{enumerate}\upshape
\item $f_i:\D\to\Omega_i$ is univalent, 

\item $F\vert_{\Omega_i}\equiv f_i\circ (1/\overline{\zeta})\circ (f_i\vert_{\D})^{-1}$, and

\item $\Deg{f_1}=1$, $\Deg{f_2}=3$.
\end{enumerate}

\begin{figure}[ht!]
\begin{tikzpicture}
 \node[anchor=south west,inner sep=0] at (0,0) {\includegraphics[width=0.96\textwidth]{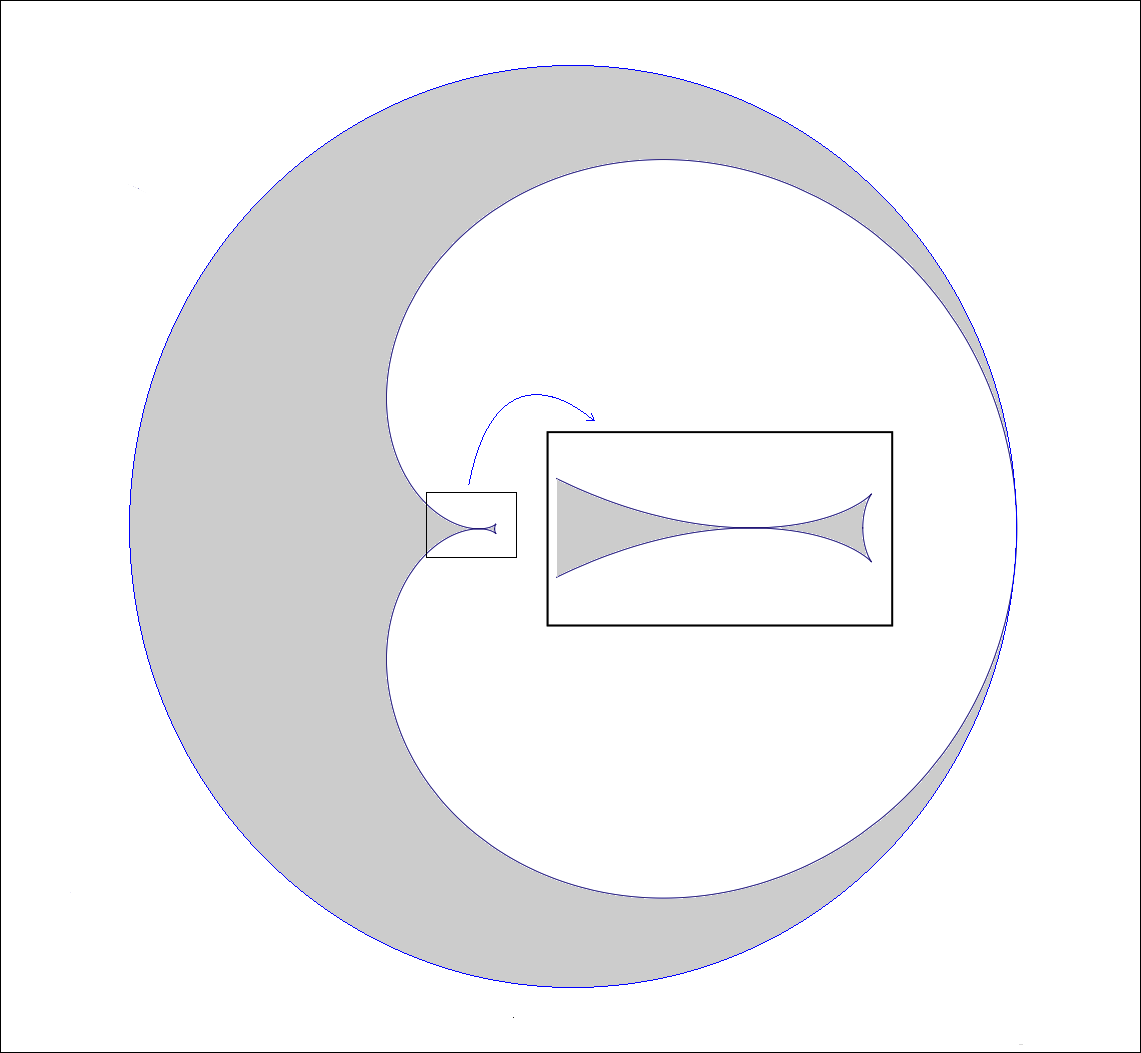}};
\node at (11,10) {\begin{huge}$\Omega_1$\end{huge}};
\node at (8,8.4) {\begin{huge}$\Omega_2$\end{huge}};
\node at (3.6,8.4) {\begin{Large}$T^0(F)$\end{Large}};
\node at (6.75,4.9) {\begin{scriptsize}$T^0(F)$\end{scriptsize}};
\draw[->,line width=0.5pt] (6.75,5.15)-|(6.75,5.6);
\node at (9,4.9) {\begin{scriptsize}$T^0(F)$\end{scriptsize}};
\draw[->,line width=0.5pt] (9,5.15)-|(9,5.6);
\end{tikzpicture}
\caption{The conformal mating of $P_2(z)=\overline{z}^3+\frac{1+i}{\sqrt{2}}$ and the necklace group from Figure~\ref{kleinian_cusp} is given by the piecewise Schwarz reflection map $F$ associated with the quadrature domains $\Omega_i$, $i\in\{1,2\}$; where $\Omega_1=\{\vert z \vert>\frac{4}{3}+\frac{2\sqrt{2}}{3}\}$, and $\Omega_2=f(\D)$, $f(z)=z+\frac{2\sqrt{2}}{3}z^2+\frac{z^3}{3}$. The two components of $T^0(F)$ are topological triangles with vertices removed. The closure of one of the components of $T^0(F)$ is a topological triangle, and two vertices of the other component of $T^0(F)$ are identified. The unique critical point $0$ of $F$ forms a $2$-cycle $0\leftrightarrow\infty$.}
\label{mating_2_fig}
\end{figure}
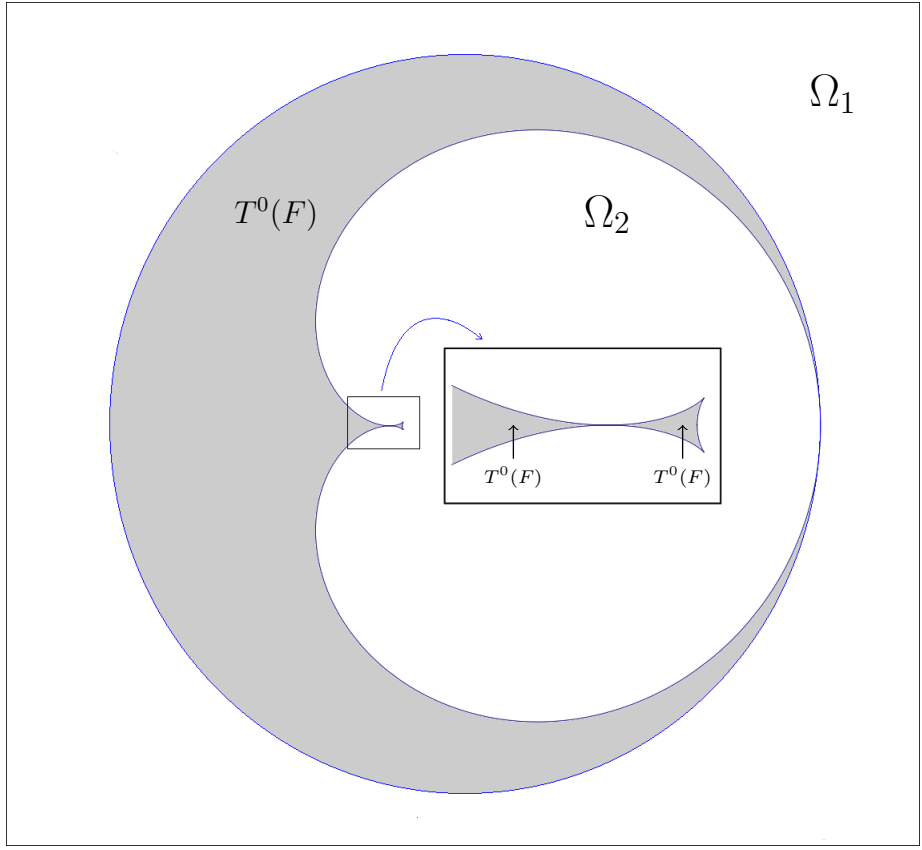

The assumption that the conformal radius of $\Omega_2$ (with conformal center at $0$) is $1$ allows us to normalize $f_2$ so that $f_2(0)=0$ and $f_2'(0)=1$. By the dynamics of $F$, it now follows that $f_2$ has a triple pole at $\infty$; so $f_2$ is a cubic polynomial. 

Let $f_2(z)=z+az^2+bz^3$. Since $F$ (in particular, $\sigma_2$) has a unique critical point, the two finite critical points of $f_2$ must lie on $\mathbb{S}^1$. A simple calculation now implies that $\vert b\vert=1/3$. Conjugating $f_2$ by a rotation, we can assume that $$f_2(z)=z+az^2+z^3/3,$$ for some $a\in\C$. We will now use the condition that $\partial\Omega_2$ has a unique cut-point (equivalently, a double point) to determine $a$. In fact, a simple numerical computation shows that the only map $f_2$ of the above form with a double point on $f_2(\mathbb{S}^1)$ is $f_2(z)=z+\frac{2\sqrt{2}}{3}z^2+\frac{z^3}{3}$ (up to conjugation by $z\mapsto -z$). Below we give a rigorous proof of this fact.

Consider the space of cubic polynomials 
$$
S_3^*:=\{f(z)=z+az^2+z^3/3: a\in\C, f\vert_{\D}\ \textrm{is\ univalent}\}.
$$ 
By \cite[Theorem~2]{Bra67}, the maps $z\pm\frac{2\sqrt{2}}{3}z^2+\frac{z^3}{3}$ maximize the absolute value of the coefficient of $z^2$ in the space $S_3^*$. Moreover, for any $f\in S_3^*$, we have $a\in\R$ (the fact that the coefficient of $z^3$ is $1/3$ implies that both critical points of $f$ lie on $\mathbb{S}^1$). Hence, the maps $z\pm\frac{2\sqrt{2}}{3}z^2+\frac{z^3}{3}$ are extremal points of $S_3^*$; i.e., they cannot be written as proper convex combinations of two distinct maps in $S_3^*$. It now follows from \cite[Theorem~3.1]{LM14} that the image of $\mathbb{S}^1$ under each of the maps $z\pm\frac{2\sqrt{2}}{3}z^2+\frac{z^3}{3}$ has a double point. Finally, by \cite[Theorem~B]{LMM19}, up to conjugation by $z\mapsto -z$, there is a unique member $f$ in $S_3^*$ with the property that $f(\mathbb{S}^1)$ has a double point. Therefore, we can choose $f_2(z)=z+\frac{2\sqrt{2}}{3}z^2+\frac{z^3}{3}$.

Since $F(\infty)=0$, it follows that $\Omega_1$ is the exterior of a round circle centered at the origin. It is easy to see that $\textrm{dist}(0,f_2(\mathbb{S}^1))=d(0,f_2(1))=\frac{4}{3}+\frac{2\sqrt{2}}{3}$. As $\partial\Omega_1$ touches $\partial\Omega_2=f_2(\mathbb{S}^1)$, it follows that $\Omega_1$ is the exterior of the circle $\{\vert z\vert=\frac{4}{3}+\frac{2\sqrt{2}}{3}\}$.

\begin{prop}\label{mating_2_example_prop}
Let $P_2(z)=\overline{z}^3+\frac{(1+i)}{\sqrt{2}}\overline{z}$, $\Gamma$ be the cusp reflection group from Figure~\ref{kleinian_cusp}, and $f(z)=z+\frac{2\sqrt{2}}{3}z^2+\frac{z^3}{3}$. Then, the piecewise defined Schwarz reflection map $F$ in the quadrature domains $f(\D)$ and $\{\vert z \vert>\frac{4}{3}+\frac{2\sqrt{2}}{3}\}$ is a conformal mating of $P_2$ and $\Gamma$.
\end{prop}

\bigskip

\subsection{Mating the cauliflower anti-polynomial with the ideal triangle group}\label{cauli_triangle_mating_subsec}

In this subsection, we will prove the existence of the conformal mating of the `cauliflower' anti-polynomial $P(z)=\overline{z}^2+1/4$ and the ideal triangle reflection group, and give an explicit description of this conformal mating. 

Note that the parabolic basin of attraction of $P$ is equal to $\inter{\mathcal{K}(P)}$, and it is a Jordan domain. Mapping $\mathcal{K}(P)$ to $\overline{\D^*}=\widehat{\C}\setminus\D$ by a Riemann map $\phi$ that sends the critical point $0$ (of $P$) to $\infty$ and the parabolic fixed point $1/2$ to $1$, we see that $\phi\circ P\circ\phi^{-1}:\overline{\D^*}\to\overline{\D^*}$ is equal to the anti-Blaschke product 
$$
B(z)=\frac{3\overline{z}^2+1}{3+\overline{z}^2}
$$ 
(compare \cite[Expos{\'e}~IX, \S II, Corollary~1]{DH07}). In other words, $P\vert_{\mathcal{K}(P)}$ is conformally conjugate to $B\vert_{\overline{\D^*}}$.

\begin{lemma}\label{parabolic_blaschke_nielsen_david_extension}
There is a homeomorphism $H:\mathbb{S}^1\to\mathbb{S}^1$ with $H(1)=1$ that conjugates $B$ to $\pmb{\rho}_2$, and extends as a David homeomorphism of $\D$.
\end{lemma}
\begin{proof}
Recall from Examples~\ref{example:expansive_blaschke} and~\ref{example:expansive_reflection} that both $B\vert_{\mathbb{S}^1}$ and $\pmb{\rho}_2\vert_{\mathbb{S}^1}$ are expansive. 

Note that $B$ has three fixed points on $\mathbb{S}^1$; namely, at $1$ and $\left(\frac{-1\pm 2\sqrt{2}i}{3}\right)$. Moreover, $1$ is a parabolic fixed point of $B$, while the other two are repelling. On the other hand, all three fixed points $1,\omega,\omega^2$ of the Nielsen map $\pmb{\rho}_2$ of the ideal triangle reflection group $\pmb{\Gamma}_2$ are parabolic. We consider the Markov partitions $\mathcal{P}\left(B, \left\{1,-\frac13+\frac{2\sqrt{2}i}{3},-\frac13-\frac{2\sqrt{2}i}{3}\right\}\right)$ and $\mathcal{P}\left(\pmb{\rho}_2, \{1,\omega,\omega^2\}\right)$. 

It was shown in Example~\ref{example:reflection_uv} that $\pmb{\rho}_2$ admits piecewise conformal extensions satisfying conditions~\eqref{condition:uv} and~\eqref{condition:holomorphic} (with respect to $\mathcal{P}\left(\pmb{\rho}_2, \{1, \omega, \omega^2\}\right)$). Thanks to Theorem~\ref{theorem:extension_generalization}, it now suffices to prove that the map $B\vert_{\mathbb{S}^1}$ also admits piecewise conformal extensions satisfying conditions~\eqref{condition:uv} and~\eqref{condition:holomorphic} (with respect to the partition $\mathcal{P}\left(B, \left\{1,-\frac13+\frac{2\sqrt{2}i}{3},-\frac13-\frac{2\sqrt{2}i}{3}\right\}\right)$). In fact, the desired extension of $B\vert_{\mathbb{S}^1}$ near the fixed points satisfying condition~\eqref{condition:holomorphic} is given by $z\mapsto 1/\overline{B(z)}$. We now proceed to find open neighborhoods $U_i, V_i$, $i\in\{1,2,3\}$, such that $U_i$s (respectively, $V_i$s) contain the interiors of the above Markov partition pieces (respectively, the $B$-images of the Markov partition pieces), the map $z\mapsto 1/\overline{B(z)}$ carries $U_i$ onto $V_i$ conformally, and the sets satisfy condition~\eqref{condition:uv}.

\begin{figure}[ht!]
\centering
\includegraphics[width=0.98\linewidth]{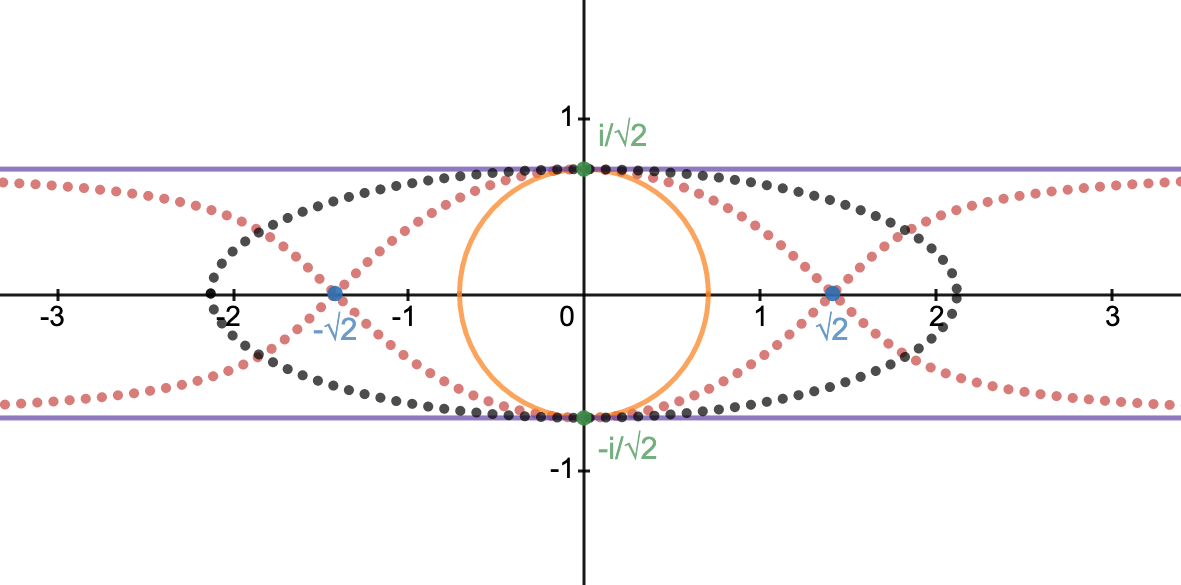}
\caption{The set $\widecheck{U}_1$ (respectively, $\widecheck{U}_2$) is the region above (respectively, below) the purple horizontal line $y=\frac{1}{\sqrt{2}}$ (respectively, $y=-\frac{1}{\sqrt{2}}$, and the set $\widecheck{U}_3$ is the disk bounded by the orange circle $\left\{x^2+y^2=\frac12\right\}$. The $\widecheck{B}$-image of $\widecheck{U}_1$ (respectively, of $\widecheck{U}_2$) is the region below (respectively, above) the dotted red curve passing through $\frac{i}{\sqrt{2}}$ and $\pm\sqrt{2}$ (respectively, through $-\frac{i}{\sqrt{2}}$ and $\pm\sqrt{2}$). Finally, the $\widecheck{B}$-image of $\widecheck{U}_3$ is the exterior of the dotted black ellipse $\frac{2x^{2}}{9}+2y^{2}=1$.}
\label{conformal_extensions_blaschke_fig}
\end{figure}

To this end, we set
$$
\widetilde{B}(z):=\frac{1}{\overline{B(z)}}=\frac{3+z^2}{3z^2+1}.
$$
It will be convenient to conjugate $\widetilde{B}$ by a M{\"o}bius map
$M(z)=\frac{z+1}{z-1},$
that preserves the real line, sends $1$ to $\infty$, the other two fixed points to $\frac{-i}{\sqrt{2}}$, and the unit circle to the imaginary axis. The conjugated map is given by
$$
\widecheck{B}(w):=M\circ\widetilde{B}\circ M^{-1}(w)=-w-\frac{1}{w}.
$$
By construction, $M$ carries the interiors of the Markov partition pieces (in the $B$-plane) to $\left(-i\infty, \frac{-i}{\sqrt{2}}\right), \left(\frac{-i}{\sqrt{2}}, \frac{i}{\sqrt{2}}\right)$, and $\left(\frac{i}{\sqrt{2}}, i\infty\right)$ (in the $\widecheck{B}$-plane). We now set 
$$
\widecheck{U}_1=\left\{{\rm Im}(w)>\frac{1}{\sqrt{2}}\right\},\quad \textrm{and}\quad \widecheck{U}_2=\left\{{\rm Im}(w)<\frac{-1}{\sqrt{2}}\right\}.
$$
The image $\widecheck{V}_1$ of $\widecheck{U}_1$ under $\widecheck{B}$ is the region below the curve 
$$
\left\{\left(1+\frac{2x^{2}}{\left(2+\sqrt{2}y\right)^{2}}\right)\left(1+\sqrt{2}y\right)=2\right\},
$$
that passes through $\frac{i}{\sqrt{2}}$, intersects the real line at $\pm\sqrt{2}$, and is contained in the horizontal strip $\left\{\frac{-1}{\sqrt{2}}\leq{\rm Im}(w)\leq\frac{1}{\sqrt{2}}\right\}$. Likewise, the image $\widecheck{V}_2$ of $\widecheck{U}_2$ under $\widecheck{B}$ is the region above the curve
$$
\left\{\left(1+\frac{2x^{2}}{\left(2-\sqrt{2}y\right)^{2}}\right)\left(1-\sqrt{2}y\right)=2\right\},
$$
that passes through $\frac{-i}{\sqrt{2}}$, intersects the real line at $\pm\sqrt{2}$, and is contained in the horizontal strip $\left\{\frac{-1}{\sqrt{2}}\leq{\rm Im}(w)\leq\frac{1}{\sqrt{2}}\right\}$. Moreover, $\widecheck{B}$ maps $\widecheck{U}_1, \widecheck{U}_2$ conformally onto $\widecheck{V}_1, \widecheck{V}_2$. Finally, we define 
$$
\widecheck{U}_3:=\{\vert z\vert<1/\sqrt{2}\}\subset\widecheck{B}(\widecheck{U}_1)\cap\widecheck{B}(\widecheck{U}_2).
$$
Then, $\widecheck{U}_3$ does not contain the critical points $\pm 1$ of $\beta$. Furthermore, 
$$
\widecheck{B}:\widecheck{U}_3\to \widecheck{V}_3
$$ 
is a conformal isomorphism, where $\widecheck{V}_3$ is the exterior of an ellipse with $\widecheck{V}_3\supset \widecheck{U}_1\cup \widecheck{U}_2$ (see Figure~\ref{conformal_extensions_blaschke_fig}). Transporting the sets $\widecheck{U}_i, \widecheck{V}_i$, $i\in\{1,2,3\}$, back to the $B$-plane via the change of coordinate $M$, we obtain our desired open sets $U_i, V_i$, $i\in\{1,2,3\}$ satisfying condition~\eqref{condition:uv}.
\end{proof}

We define an orientation-reversing continuous map on a subset of $\mathbb{S}^2$:

$$
\widetilde{B}=
\begin{cases}
H^{-1}\circ \pmb{\rho}_2\circ H,\quad {\rm in}\ \overline{\D}\setminus \inter{H^{-1}(\Pi)},\\
B,\quad {\rm in\ } \D^*,
\end{cases}
$$
where $\Pi$ is a regular ideal triangle in $\D$. We define $\mu\vert_{\D}$ to be the pullback of the standard complex structure on $\mathbb D$ by the map $H$, and set $\mu$ equal to zero everywhere else. Since $H$ is a David homeomorphism of $\D$, it follows that $\mu$ is a David coefficient on $\widehat{\C}$; i.e., it satisfies condition~\eqref{definition:david1}. 

Theorem~\ref{theorem:integrability_david} then gives us an orientation-preserving homeomorphism $\Psi$ of $\widehat{\mathbb C}$ such that the pullback of the standard complex structure under $\Psi$ is equal to $\mu$. Setting $\Omega:=\widehat{\C}\setminus\overline{\Psi(H^{-1}(\Pi))}$, we see as in the proof of Lemma~\ref{david_surgery_lemma} that the map 
$$
\sigma:=\Psi\circ \widetilde{B}\circ \Psi^{-1}:\overline{\Omega}\to\widehat{\C}
$$
is continuous and anti-analytic on $\Omega$. Moreover, $B\vert_{\overline{\D^*}}$ is conformally conjugate to $\sigma\vert_{\Psi(\overline{\D^*})}$ via $\Psi$, and $\pmb{\rho}_2\vert_{\overline{\D}}$ is conformally conjugate to $\sigma\vert_{\Psi(\overline{\D})}$ via $\Psi\circ H^{-1}$ (conformality of $\Psi\circ H^{-1}$ follows from Theorem~\ref{theorem:stoilow}).  

Since $P\vert_{\mathcal{K}(P)}$ is conformally conjugate to $B\vert_{\overline{\D^*}}$, it follows from the previous paragraph and Definition~\ref{mating_def} that $\sigma$ is a conformal mating of $P$ and the ideal triangle group $\pmb{\Gamma}_2$. Moreover, by Theorem~\ref{theorem:w11_removable}, $\Psi(\mathbb{S}^1)$ is conformally removable. We claim that $\sigma$ is the unique (up to M{\"o}bius conjugacy) conformal mating of $P$ and $\pmb{\Gamma}_2$. Indeed, if $\widetilde{\sigma}$ is another conformal mating of $P$ and $\pmb{\Gamma}_2$, then there exists a homeomorphism of $\widehat{\C}$, conformal on $\widehat{\C}\setminus\Psi(\mathbb{S}^1)$, conjugating $\sigma$ to $\widetilde{\sigma}$. But the conformal removability of $\Psi(\mathbb{S}^1)$ implies that this homeomorphism is a M{\"o}bius map; i.e., $\sigma$ and $\widetilde{\sigma}$ are M{\"o}bius conjugate. 

Note that $P$ and $\pmb{\rho}_2$ commute with $\iota:z\mapsto\overline{z}$. It follows that $\iota\circ\sigma\circ\iota$ is another conformal mating of $P$ and $\pmb{\rho}_2$. If we normalize $\Psi$ so that the unique critical point, critical value, and `parabolic' fixed point of $\sigma$ are real, then by uniqueness of the conformal mating of $P$ and $\pmb{\rho}_2$, the maps $\sigma$ and $\iota\circ\sigma\circ\iota$ must be conjugate via a M{\"o}bius map $M$ fixing these three dynamically marked points. Hence, $M=\textrm{id}$, and $\sigma=\iota\circ\sigma\circ\iota$. Here is an alternative way of seeing the real-symmetry of $\sigma$. By Remark~\ref{remark:symmetric_extension}, the David homeomorphism $H$ is real-symmetric. It follows from the construction that the map $\widetilde{B}$ and the David coefficient $\mu$ are also real-symmetric. By the uniqueness part of Theorem~\ref{theorem:integrability_david}, we conclude that the David homeomorphism $\Psi$ is real-symmetric, from which real-symmetry of $\sigma$ follows.

We now proceed to an explicit characterization of $\sigma$. By construction, $\Omega$ is a Jordan domain. Also note that by Lemma~\ref{mating_is_schwarz}, $\Omega$ is a quadrature domain, and $\sigma$ is its Schwarz reflection map. Since $\sigma$ commutes with the complex conjugation map, the domain $\Omega$ is real-symmetric. Since $\sigma:\sigma^{-1}(\inter{\Omega^c})\to\inter{\Omega^c}$ has degree $3$, Proposition~\ref{simp_conn_quad} now provides us with a rational map $R$ of degree $3$ such that $R:\overline{\D^*}\to\overline{\Omega}$ is univalent. Since $\Omega$ is real-symmetric, we can assume that $R$ has real coefficients. 

For each critical point $\xi$ of $R$ in $\D$, the point $R(1/\overline{\xi})\in\Omega$ is a critical point of $\sigma$. Since $\sigma$ has exactly one (simple) critical point, it follows that $R$ has exactly one (simple) critical point in $\D$. By the real-symmetry of $R$, this critical point must be real. By the univalence of $R\vert_{\overline{\D^*}}$, the other three critical points of $R$ lie on $\mathbb{S}^1$ and are simple. Once again, the real-symmetry of $R$ implies that one of these two critical points is real, and the other two are complex conjugates of each other.

Note that postcomposing $R$ with M{\"o}bius transformations and precomposing $R$ with M{\"o}bius maps that preserve the unit disk do not change the M{\"o}bius conjugacy class of $\sigma$. Thus, pre and postcomposing $R$ with real-symmetric M{\"o}bius maps, we can assume the following:
\begin{enumerate}\upshape
\item $R\vert_{\overline{\D^*}}$ is univalent,

\item $\overline{R(z)}=R(\overline{z})$,

\item $\textrm{Crit}(R)=\{0,1,\alpha,\overline{\alpha}\}$ for some $\alpha\in\mathbb{S}^1\setminus\{\pm 1\}$, 

\item $\sigma^{\circ n}\to R(1)$ locally uniformly on $\Psi(\overline{\D^*})$,

\item $R(\infty)=\infty,\ R'(\infty)=1$, and

\item $R(0)=2$.
\end{enumerate}

The above conditions on $R$ (combined with Vieta's formulas) can be used to obtain the following form of $R$:
\begin{equation}
R(z)=\frac{z^3+\frac{4c-1}{2c}z^2+\frac{12c^2}{1-4c}z+2c}{z^2+\frac{6c^2}{1-4c}z+c},
\label{rat_formula}
\end{equation}
for some $c\in\R\setminus\{0,1/4\}$. Moreover, elementary arguments involving the locations of the critical points of $R$ and the local geometry of the real-symmetric Jordan curve $R(\mathbb{S}^1)$ near the cusp $R(1)$ show that $c\leq\frac{1}{10}$.\footnote{We refer the reader to [\S 11.4, arXiv:2010.11256v2] for detailed proofs of the above statements.}

\begin{figure}[ht!]
\centering
\includegraphics[width=0.8\linewidth]{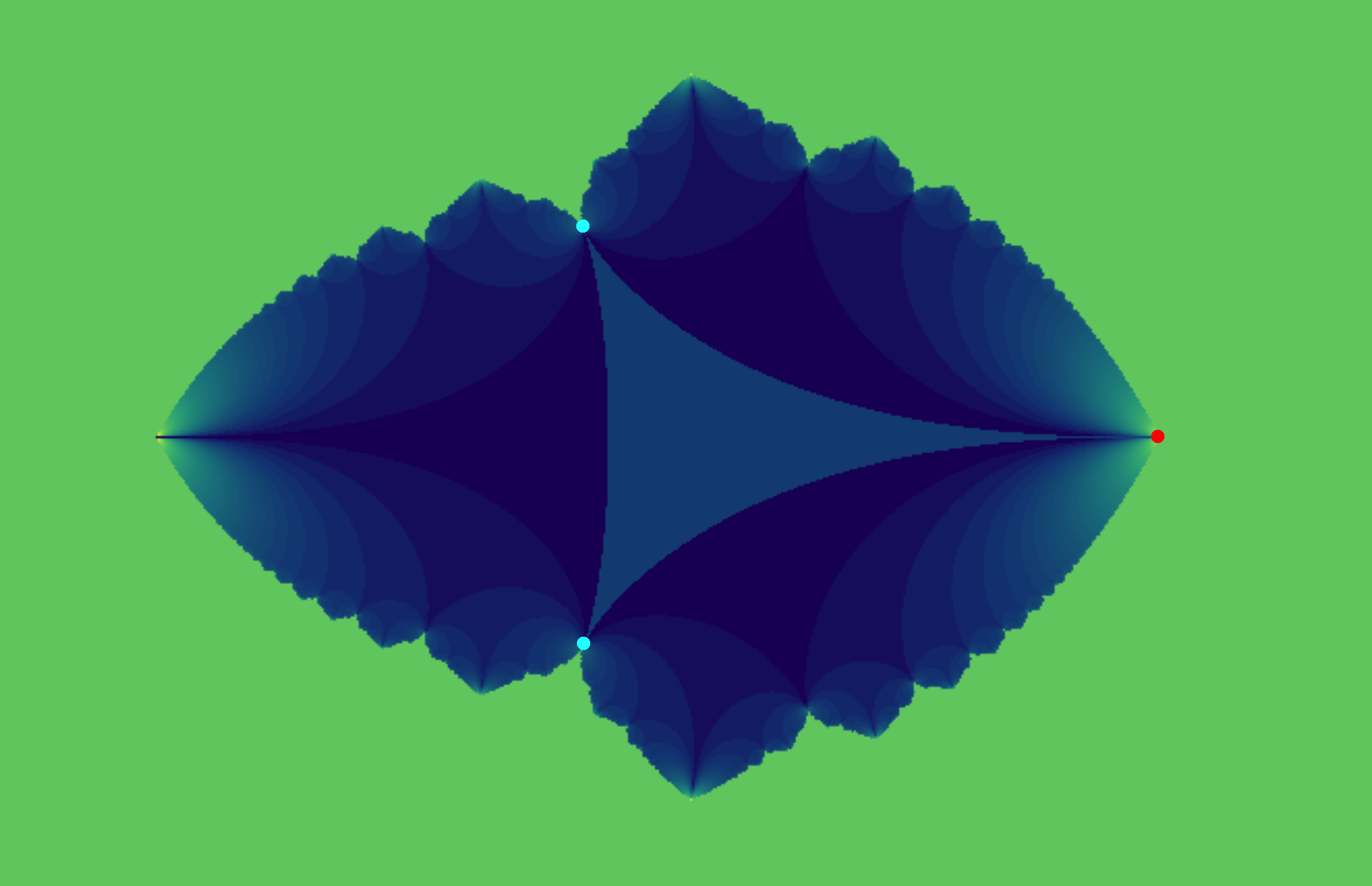}
\caption{The dynamical plane of the Schwarz reflection map $\sigma$, which is the conformal mating of $\overline{z}^2+\frac14$ and the ideal triangle reflection group $\pmb{\Gamma}_2$, is shown. The forward orbit of every point of the non-escaping set (shaded in green) converges to the point $R(1)$ on $R(\mathbb{S}^1)$ (marked in red). The non-escaping set contains an attracting petal subtending an angle $4\pi/3$ at this point. On the other hand, the two real-symmetric cusps on $R(\mathbb{S}^1)$ (marked in light blue) repel nearby points in the non-escaping set, and the non-escaping set subtends zero angles at these two cusps. (Picture courtesy: Seung-Yeop Lee.)}
\label{mating_3_fig}
\end{figure}

To find the exact value of $c$, let us first note that since $R(1)>R(-1)$ and $\displaystyle\lim_{x\to\pm\infty} R(x)=\pm\infty$, it follows from univalence of $R\vert_{\overline{\D^*}}$ that $(R(1),+\infty)\subset\Omega$. Now, the real-symmetry of $R$ and the local uniform convergence of $\sigma^{\circ n}$ to $R(1)$ on $\Psi(\overline{\D^*})$ imply that $(R(1),+\infty)$ is an attracting direction for the `parabolic' fixed point $R(1)$ of $\sigma$. This implies, in particular, that 
\begin{equation*}
R\left(\frac{1}{1+\varepsilon}\right)=\sigma(R(1+\varepsilon))<R(1+\varepsilon) \iff \frac{(4c-1)(10c-1)}{(c-1)^2(2c-1)}\varepsilon^3+O(\varepsilon^4)>0,
\end{equation*}
for $\varepsilon>0$ sufficiently small. If $c\neq \frac{1}{10}$; i.e., if $c<\frac{1}{10}$, then the above inequality implies that
$$
\frac{(4c-1)(10c-1)}{(c-1)^2(2c-1)}>0,
$$
which is impossible. Therefore, we must have $c=\frac{1}{10}$. Plugging $c=\frac{1}{10}$ in Formula~\eqref{rat_formula}, we finally have
\begin{equation}
R(z)=\frac{10z^3-30z^2+2z+2}{10z^2+z+1}.
\label{rat_formula_final}
\end{equation}
(See Figure~\ref{mating_3_fig} for the dynamical plane of the associated Schwarz reflection map $\sigma$.) 

\begin{prop}\label{mating_3_example_prop}
Let $P(z)=\overline{z}^2+\frac14$, and 
$$
R(z):=\frac{10z^3-30z^2+2z+2}{10z^2+z+1}.
$$ 
Then, $R$ is univalent on $\overline{\D^*}$, and the Schwarz reflection map $\sigma$ of the quadrature domain $R(\D^*)$ is the unique conformal mating of $P$ and the Nielsen map $\pmb{\rho}_2$.
\end{prop}

To conclude, we will show that the limit set $\Psi(\mathbb{S}^1)$ of the Schwarz reflection map $\sigma$ does not have a cusp at $-4/3$. On the other hand, the two repelling fixed points of $P$ (which are symmetrically hyperbolic) are mated with two symmetrically parabolic fixed points of the Nielsen map $\pmb{\rho}_2$. As we shall see below, these matings produce cusps; see the light blue points in Figure \ref{mating_3_fig}.

\begin{prop}\label{para_hyp_schwarz_cusp_prop}
Let $f:=\pmb{\rho}_d$, $g$ be a degree $d$ anti-Blaschke product with an attracting fixed point in $\D$ or a parabolic fixed point on $\mathbb{S}^1$, and $h$ be a topological conjugacy between $f$ and $g$. Further suppose that $\mathfrak{J}$ is the welding curve corresponding to $h$ and the point $x_0\in\mathfrak{J}$ corresponds to the fixed point $1$ of $f$.

\begin{enumerate}[\upshape(1)]
\item If $h(1)$ is a hyperbolic fixed point of $g$, then the welding curve $\mathfrak{J}$ has an untwisted cusp at $x_0$.
\item If $h(1)$ is a parabolic fixed point of $g$ with an attracting direction in $\D$, then the welding curve $\mathfrak{J}$ does not have an untwisted cusp at $x_0$.
\end{enumerate}
\end{prop}
\begin{proof}
The assumptions that $f=\pmb{\rho}_d$ and $g$ is an anti-Blaschke product imply that the mating is a Schwarz reflection map $\sigma:\overline{\Omega}\to\widehat{\C}$ in a simply connected quadrature domain $\Omega$. Moreover, the uniformizing rational map of $\Omega$ has a simple critical point on $\mathbb{S}^1$ with corresponding critical value at $x_0$. The asymptotics of $\sigma^{\circ 2}$ near $x_0$ is given by
\begin{equation}
\sigma^{\circ 2}(z)  = z + c\cdot (z-x_0)^{n/2} + O((z-x_0)^{(n+1)/2})
\end{equation}
where it is defined, for some $c\in\C^*$ and some odd integer $n\geq 3$ (see \cite[\S A.3]{LMM23}). 

(1) In this case, $\sigma^{\circ 2}$ has no attracting direction in $\Omega$ at the point $x_0$. Hence by \cite[Corollary~A.6]{LMM23}, we have that $n=3$. By \cite[\S A.4]{LMM23}, a change of coordinates of the form $z\mapsto \frac{\lambda}{(z-x_0)^{1/2}}$ brings the asymptotics of $\sigma^{\circ 2}$ near $x_0$ to the asymptotics $\zeta\mapsto \zeta+1+O(1/\zeta)$ near $\infty$. Using this fact, one can prove as in the case of standard parabolic germs that the union of iterated preimages of $\Pi$ (where $\Pi$ is the ideal polygon with vertices at the $(d+1)-$st roots of unity) under $\sigma$ can be well-approximated at $x_0$ by a sector of angle $2\pi$. Since this region lies on one side of $\mathfrak{J}$, it follows that $\mathfrak{J}$ has an untwisted cusp at $x_0$.

(2) In this case, the map $\sigma^{\circ 2}$ has a unique attracting direction in $\Omega$, and hence $n~=~5$ (see \cite[Propositions~A.4, A.5]{LMM23}). A change of coordinates of the form $z~\mapsto~\frac{\lambda}{(z-x_0)^{3/2}}$ can be used to deduce that the corresponding attracting petal can be well-approximated at $x_0$ by a sector of angle $4\pi/3$. Moreover, the union of iterated preimages of $\Pi$ under $\sigma$ can be well-approximated at $x_0$ by a sector of angle $2\pi/3$. Therefore, there are sectors of positive angle on both sides of $\mathfrak{J}$ at $x_0$. It follows that the welding curve $\mathfrak{J}$ does not have an untwisted cusp at $x_0$.
\end{proof}

\begin{remark}
It is tempting to construct the conformal mating of $P(z)=\overline{z}^2+\frac14$ and $\pmb{\rho}_2$ by starting with the dynamical plane of $P$, and replacing the dynamics of $P$ on its basin of infinity by Nielsen map $\pmb{\rho}_2$. However, our David surgery methods do not permit us to carry out this construction since the basin of infinity of $P$ is not a John domain, and hence, we cannot apply Proposition~\ref{prop:david_qc_invariance} to obtain an invariant David coefficient for the modified map (compare the proof of Lemma~\ref{david_surgery_lemma}). It would be interesting to know if one can directly pass from the dynamical plane of $P$ to the conformal mating of $P$ and $\pmb{\rho}_2$ using a more general surgery technique. 
\end{remark}

\bigskip

\section{Extremal points in spaces of schlicht functions}

The well-known De Brange's theorem (earlier known as the Bieberbach conjecture) asserts that each member $f$ of the class of \emph{schlicht} functions $$\mathcal{S} := \left\{ f(z)= z+a_2{z^2} + \cdots +a_nz^n+\cdots :\ f\vert_{\mathbb{D}} \textrm{ is univalent}\right\}$$ satisfies 
$\vert a_n\vert\leq n$, for all $n\in\N$. Moreover, the bound is sharp; the Koebe function $\sum_{n\geq 1} nz^n$ simultaneously maximizes all the coefficients.

The analogous coefficient problem for the class of \emph{external univalent} maps $$\Sigma := \left\{ f(z)= z+\frac{a_1}{z} + \cdots +\frac{a_d}{z^d}+\cdots :\ f\vert_{\D^*} \textrm{ is univalent}\right\},$$ where $\D^*=\widehat{\C}\setminus\overline{\D}$, is still open. Note that the area theorem implies the upper bound $$\vert a_n\vert\leq n^{-1/2},$$ (see \cite[Theorem~2.1]{Dur83}). But this bound is far from sharp; we refer the readers to \cite[\S 4.7]{Dur83} for a survey of known results (cf. \cite{HS05}).

The related question of establishing coefficient bounds for the truncated families
\[ \Sigma_d^* := \left\{ f(z)= z+\frac{a_1}{z} + \cdots+\frac{a_{d-1}}{z^{d-1}} -\frac{1}{d\cdot z^d} :\ f\vert_{\D^*} \textrm{ is conformal}\right\} \]
is motivated by the work of Suffridge on coefficient bounds for polynomials in class $\mathcal{S}$ \cite{Suf69,Suf72}. In fact, one can adapt the proof of \cite[Theorem~10]{Suf72} for the space $\Sigma$ to show that $$\Sigma=\overline{\bigcup_{d\geq 1}\Sigma_d^*}.$$ It is worth mentioning that the choice of $-1/d$ as the last coefficient stems from the fact that if $f(z)=z+\frac{a_1}{z} + \cdots +\frac{a_d}{z^d}\in\Sigma$, then the absolute value of the product of the non-zero critical points of $f$ is $d\vert a_d\vert$, and the univalence of $f\vert_{\D^*}$ implies that $d\vert a_d\vert\leq 1$; i.e., $\vert a_d\vert\leq \frac1d$. 

Viewing $\Sigma_d^*$ as a compact subset of a finite dimensional Euclidean space (given by the coefficients of the members of $\Sigma_d^*$), the problem of maximizing the coefficients of the members of $\Sigma_d^*$ boils down to finding extremal points of $\Sigma_d^*$. Here, a map $f\in\Sigma_d^*$ is said to be \emph{extremal} if it has no representation of the form 
$$
f=tf_1+(1-t)f_2,\ 0<t<1,
$$
as a proper convex combination of two distinct maps $f_1,f_2\in\Sigma_d^*$.

Before we proceed to study the extremal points of $\Sigma_d^*$, we need to recall some general facts on singularities on the boundary of $f(\D^*)$, for $f\in\Sigma_d^*$, $d\geq 2$. A boundary point $p\in\partial f(\D^*)$ is called \emph{regular} if there is a disc $B=B(p,\varepsilon)$ such that $f(\D^*)\cap B$ is a Jordan domain, and $\partial f(\D^*)\cap B$ is a simple non-singular real-analytic arc; otherwise $p$ is called a \emph{singular} point. Singular points on $\partial f(\D^*)$ come in two varieties. A \emph{cusp} singularity $\zeta_0\in\partial f(\D^*)$ is a critical value of $f$; it has the property that for sufficiently small $\varepsilon>0$, the intersection $B(\zeta_0,\varepsilon)\cap f(\D^*)$ is a Jordan domain. Moreover, by conformality of $f|_{\D^*}$, each cusp on $\partial f(\D^*)$ points in the inward direction towards $f(\D^*)$. On the other hand, a singular point $\zeta_0\in\partial f(\D^*)$ is said to be a \emph{double point} if for all sufficiently small $\varepsilon>0$, the intersection $B(\zeta_0,\varepsilon)\cap f(\D^*)$ is a union of two Jordan domains, and $\zeta_0$ is a non-singular boundary point of each of them. In particular, two distinct non-singular (real-analytic) local branches of $\partial f(\D^*)$ intersect tangentially at a double point $\zeta_0$.

In the rest of this section, we will assume that $d\geq 2$.

\begin{lemma}\label{no_of_cusp}
For each $f\in\Sigma_d^*$, the boundary of $f(\D^*)$ is a piecewise analytic curve with precisely $(d+1)$ cusps.
\end{lemma}
\begin{proof}
Clearly, each $f\in\Sigma_d^*$ has a critical point of multiplicity $(d-1)$ at the origin and $(d+1)$ non-zero critical points (counting multiplicity). If $f\in\Sigma_d^*$, then the absolute value of the product of all the non-zero critical points of $f$ is $1$. As all these critical points must lie in $\overline{\D}$, it follows that each non-zero critical point of $f$ must have absolute value $1$ and hence lies on $\mathbb{S}^1$. Moreover, conformality of $f\vert_{\D^*}$ implies that each of these critical points of $f$ must be simple. Thus, $f$ has $(d+1)$ distinct simple critical points on $\mathbb{S}^1$. The result follows.
\end{proof}

By Lemma~\ref{no_of_cusp} and \cite[Lemma~2.4]{LM14}, for each $f\in\Sigma_d^*$, the boundary of $f(\D^*)$ has exactly $(d+1)$ cusps and at most $(d-2)$ double points. 

\begin{definition}\label{Suffridge_polynomials} 
$f \in \Sigma_d^*$ is called a \emph{Suffridge polynomial} if $f(\mathbb{S}^1)=\partial f(\D^*)$ has $(d+1)$ cusps and $(d-2)$ double points. The curve $f(\mathbb{S}^1)$ is called a \emph{Suffridge curve}.
\end{definition}

The following result yields a connection between extremal points of $\Sigma_d^*$ and Suffridge polynomials.

\begin{theorem}\cite[Theorem~2.5]{LM14} 
Extremal points of $\Sigma_d^*$ are Suffridge polynomials.
\end{theorem}

We now describe a procedure to assign an angled tree to each Suffridge polynomial.

\begin{definition}\label{tiles_def}
Let $f\in\Sigma_d^*$ be a Suffridge polynomial, and $\Omega:=f(\D^*)$, $T:=\widehat{\C}\setminus\Omega$. We set $T^0:=T\setminus\{$Cusps and double points on $\partial T\}$. The connected components $T_1, \cdots, T_{d-1}$ of $T^0$ are called \emph{fundamental tiles} of $f$. 
\end{definition}

Note that the cusp points on $\partial T$ are not cut-points of $T$, while the double points are. Since there are exactly $(d-2)$ double points on $\partial T$ for a Suffridge polynomial, they disconnect $T$ into $(d-1)$ components. According to \cite[Proposition~4.1]{LM14}, the closure $\overline{T_i}$ of each tile is a (filled) topological triangle whose vertices are either cusps or double points of $f(\mathbb{S}^1)$ (see Figure~\ref{bi_angled_trees_fig}). Thus, the topology of the full continuum $\widehat{\C}\setminus \Omega$ is completely determined by the touching structure of these topological triangles (roughly speaking, $\widehat{\C}\setminus \Omega$  is a `tree of triangles'). This structure can be recorded by a simple combinatorial object which we now define.

\begin{definition}\label{tree_def}
\noindent\begin{enumerate}\upshape
\item A bi-angled tree $\mathcal{T}$ is a topological tree each of whose vertices are of valence at most $3$, and that is equipped with an angle function $\angle$, defined on pairs of edges incident at a common vertex, and taking values in $\{0, 2\pi/3, 4\pi/3\}$, satisfying the following conditions:
\begin{enumerate}\upshape
\item for each pair of distinct edges $e$ and $e'$ incident at a vertex $v$, we have $\angle_v(e,e')\in\{2\pi/3,4\pi/3\}$, and $\angle_v(e,e)=0$,

\item $\angle_v(e,e')=-\angle_v(e',e)$ (mod $2\pi$), and 

\item $\angle_v(e,e')+\angle_v(e',e'')=\angle_v(e,e'')$ (mod $2\pi$), where $e, e',e''$ are edges incident at a vertex $v$.
\end{enumerate}

\item Two bi-angled trees $\mathcal{T}_1$, $\mathcal{T}_2$ are said to be \emph{isomorphic} if there exists a tree isomorphism $f: \mathcal{T}_1 \rightarrow \mathcal{T}_2$ that satisfies $\angle_{f(v)}(f(e),f(e'))=\angle_v(e,e')$, for each pair of edges $e,e'$ incident at a vertex $v$ of $\mathcal{T}_1$.
\end{enumerate}
\end{definition}

We remark that the angle data of a bi-angled tree is purely combinatorial; in other words, for a planar embedding of a bi-angled tree, we neither require the edges to be straight line segments, nor require the Euclidean angle between two edges to be $\pm 2\pi/3$. However, since the function $\angle$ induces a cyclic order on the edges incident at each vertex, there is a preferred (isotopy class of) embedding of a bi-angled tree into the complex plane. 

\begin{definition}\label{associated_tree} For a Suffridge polynomial $f\in\Sigma_d^*$, we define its bi-angled tree $\mathcal{T}(f)=(V_\Omega, E_\Omega)$ in the following manner. Denote by $T_1, \cdots, T_{d-1}$ the fundamental tiles of $f$. Associate a vertex $v_i$ to the component $T_i$, and connect the vertices $v_i$ and $v_j$ by an edge if and only if $\overline{T_i}$ and $\overline{T_j}$ intersect. We now equip the tree with an \emph{angle function} $\angle$. If the valence of a vertex $v_i$ is $3$, then for each pair of consecutive edges $e, e'$ incident at $v_i$ (in the counter-clockwise circular order around $v_i$), we set $\angle_{v_i}(e, e')=2\pi/3$. On the other hand, if the valence of $v_i$ is $2$, and $e, e'$ are the edges incident at $v_i$, then $\angle_{v_i}(e,e')=2\pi/3$ (respectively, $4\pi/3$) if the double points of $\partial T$ corresponding to the edges $e$ and $e'$ are consecutive (respectively, are not consecutive) singular points of $\partial T_i$ in the counter-clockwise orientation.
\end{definition} 

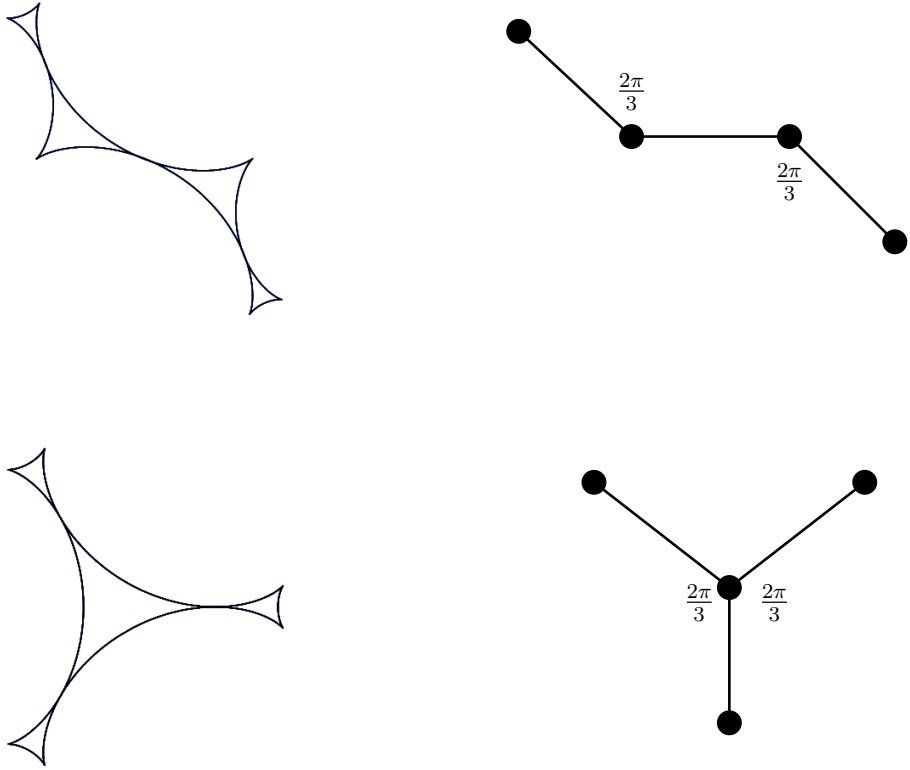
\begin{figure}[ht!]
\begin{tikzpicture}
 \node[anchor=south west,inner sep=0] at (-1,0) {\includegraphics[width=0.32\textwidth]{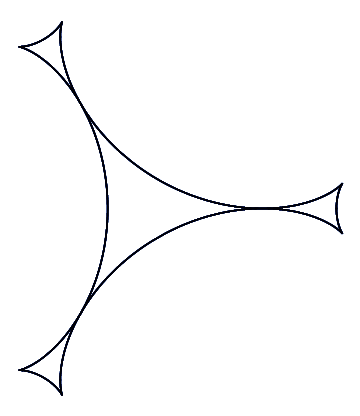}};
 \node[anchor=south west,inner sep=0] at (-1,6) {\includegraphics[width=0.32\textwidth]{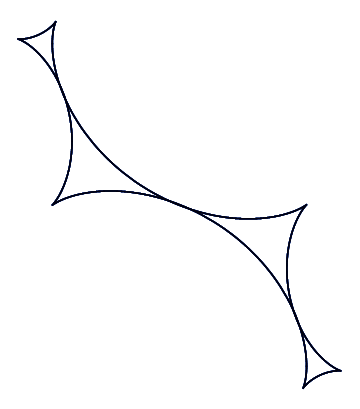}};
 \node at (6,10) [circle,fill=black] {};
\node at (7.5,8.6) [circle,fill=black] {};
\node at (9.6,8.6) [circle,fill=black] {};
\node at (11,7.2) [circle,fill=black] {};
\draw[-,line width=1pt] (6,10)->(7.5,8.6);
\draw[-,line width=1pt] (7.5,8.6)->(9.6,8.6);
\draw[-,line width=1pt] (9.6,8.6)->(11,7.2);
\node at (7.5,9.2) {\begin{Large}$\frac{2\pi}{3}$\end{Large}};
\node at (9.6,8) {\begin{Large}$\frac{2\pi}{3}$\end{Large}};
 \node at (7,4) [circle,fill=black] {};
\node at (8.8,2.6) [circle,fill=black] {};
\node at (8.8,0.8) [circle,fill=black] {};
\node at (10.6,4) [circle,fill=black] {};
\draw[-,line width=1pt] (7,4)->(8.8,2.6);
\draw[-,line width=1pt] (8.8,2.6)->(8.8,0.8);
\draw[-,line width=1pt] (8.8,2.6)->(10.6,4);
\node at (8.4,2.4) {\begin{Large}$\frac{2\pi}{3}$\end{Large}};
\node at (9.4,2.4) {\begin{Large}$\frac{2\pi}{3}$\end{Large}};
\end{tikzpicture}
\caption{The images of the unit circle under the Suffridge polynomials $f(z)\approx z-\frac{0.71i}{z}+\frac{0.71i}{3z^3}-\frac{1}{5z^5}$ and $g(z)=z+\frac{2\sqrt{2}}{5z^2}-\frac{1}{5z^5}$, along with their associated bi-angled trees, are shown.}
\label{bi_angled_trees_fig}
\end{figure}

Using David surgery techniques, we will give a new proof of the following theorem which recently appeared in \cite{LMM19}. The proof given below is essentially different from the existence proof given in \cite[Theorem~4.1]{LMM19}, which used a pinching deformation technique for Schwarz reflection maps combined with compactness of $\Sigma_d^*$.

\begin{theorem}\label{existence_suffridge_thm}
Given a bi-angled tree $\mathcal{T}$ with $(d-1)$ vertices, there exists a Suffridge polynomial $f\in\Sigma_d^*$ such that $\mathcal{T}(f)$ is isomorphic to $\mathcal{T}$. Moreover, $f$ is unique up to conjugation by multiplication by a $(d+1)$-st root of unity.
\end{theorem}

\begin{proof}[Proof of Existence]
By \cite[Proposition~59]{LMM20}, there exists a critically fixed anti-polynomial $P$ of degree $d$ whose angled Hubbard tree is isomorphic to $\mathcal{T}$ equipped with the \emph{local degree function} $\deg: V(\mathcal{T})\to\N,\ \deg(v)=2$ (see \cite{Poi13} for more general discussions on angled Hubbard trees). In particular, $P$ has $(d-1)$ distinct, simple, fixed critical points $c_1,\cdots, c_{d-1}$. These critical points correspond to the vertices $v_1,\cdots, v_{d-1}$ of $\mathcal{T}$, respectively. Let us denote the corresponding immediate basins of attraction by $U_i$, $i\in\{1,\cdots, d-1\}$. By \cite[Proposition~60]{LMM20}, we have that $$\widehat{U}:=\bigcup_{i=1}^{d-1} \overline{U_i}$$ is connected, and $\overline{U_i}$ intersects $\overline{U_j}$ if and only if there is an edge in $\mathcal{T}$ connecting $v_i$ and $v_j$. In particular, $\widehat{U}$ is a full continuum with $d-2$ cut-points.

Applying Lemma~\ref{david_surgery_lemma} on the fixed Fatou components $U_1,\cdots, U_{d-1}$ of $P$, we obtain a global David homeomorphism $\Psi$ and an anti-meromorphic map $\sigma_1$ defined on a closed, connected subset of $\widehat{\C}$ that is conformally conjugate to $\overline{z}^d\vert_{\D^*}$ on $\Psi(\mathcal{B}_\infty(P))$, and to the Nielsen map $\pmb{\rho}_2\vert_{\D}$ on each $\Psi(U_i)$. Moreover, if $\Omega_1$ is the interior of the domain of definition of $\sigma_1$, then $\sigma_1$ fixes $\partial\Omega_1$ pointwise. So $\Omega_1$ is a quadrature domain, and $\sigma_1$ is its Schwarz reflection map.

It also follows from the construction and the fullness of $\widehat{U}$ that $\Omega_1$ is simply connected. Thus, there exists a rational map $f_1$, univalent on $\D^*$, such that 
$$
f_1(\D^*)=\Omega,\ \textrm{and}\ \sigma\equiv f_1\circ(1/\overline{z})\circ \left(f_1\vert_{\D^*}\right)^{-1}\ \textrm{on}\ \Omega_1.
$$
After possibly replacing $\Omega_1$ by a M{\"o}bius image of it, we can assume that $\Omega_1$ has conformal radius $1$ with conformal center at $\infty$. Hence, we can normalize $f_1$ such that $f_1(\infty)=\infty$, $f_1'(\infty)=1$. Since $\infty$ is a critical point of multiplicity $(d-1)$ for $\sigma_1$, and $\sigma_1^{-1}(\infty)=\{\infty\}$, it follows that $\sigma_1:\sigma_1^{-1}(\Omega_1)\to\Omega_1$ is a branched covering of degree $d$. Hence, by Lemma~\ref{simp_conn_quad}, we have that $\deg{f_1}=d+1$, and $f_1$ has a pole of order $d$ at the origin. Therefore, we have $$f_1(z)=z+a_0+\frac{a_1}{z} + \cdots +\frac{a_d}{z^d}.$$
As $\sigma_1$ has no critical point other than $\infty$, it follows that $f_1$ has all of its $(d+1)$ non-zero critical points on $\mathbb{S}^1$. Moreover, the conformality of $f_1\vert_{\D^*}$ implies that these $(d+1)$ critical points of $f_1$ are simple and distinct. 

Setting $\Omega_2:=A_1(\Omega_1)$ and $f_2:=A_1\circ f_1$, where $A_1(z)=z+b$ (for some $b\in\C$), we can assume that $$f_2(z)= z+\frac{a_1}{z} + \cdots +\frac{a_d}{z^d}.$$ A simple computation (using the fact that all critical points of $f_2$ are on $\mathbb{S}^1$ and at the origin) now shows that $\vert a_d\vert=1/d$. Finally, setting $\Omega:=A_2(\Omega_2)$ and $f:=A_2\circ f_2\circ A_2^{-1}$, where $A_2(z)=\alpha z$ (for some $\alpha\in\mathbb{S}^1$), we can further assume that $f$ is of the above form with $a_d=-1/d$. Clearly, $f\in\Sigma_d^*$. Also note that the full continuum $\widehat{\C}\setminus\Omega$ is the union of $d-1$ topological triangles each of which corresponds to a fundamental domain of the ideal triangle group, and the touching points of these topological triangles correspond precisely to the $d-2$ cut-points of $\widehat{U}$. Hence, $f(\mathbb{S}^1)$ has $(d-2)$ double points; i.e. $f$ is a Suffridge polynomial in $\Sigma_d^*$. It is now easy to see from the construction of $f$ and the touching structure of the closed topological disks $\overline{U_i}$ that its bi-angled tree $\mathcal{T}(f)$ is isomorphic to $\mathcal{T}$.
\end{proof}

\begin{proof}[Proof of Uniqueness]
Assume that $g\in\Sigma_d^*$ is another Suffridge polynomial realizing the bi-angled tree $\mathcal{T}$. We set
$$
T^\infty(\sigma_f)=\bigcup_{n=0}^\infty\sigma_f^{-n}(T^0(f)),\quad T^\infty(\sigma_g)=\bigcup_{n=0}^\infty\sigma_g^{-n}(T^0(g)).
$$ 
Also note that both $\sigma_f$ and $\sigma_g$ have a superattracting fixed point at $\infty$. We denote the corresponding basins of attraction by $\mathcal{B}_\infty(\sigma_f)$ and $\mathcal{B}_\infty(\sigma_g)$. Since $\sigma_f, \sigma_g$ do not have any other critical points in $\mathcal{B}_\infty(\sigma_f)$, $\mathcal{B}_\infty(\sigma_g)$ (respectively), it follows from the proof of \cite[Theorem~9.3]{Mil06} that these basins are simply connected, and the Schwarz reflection maps restricted to these basins are conformally conjugate to $\overline{z}^d\vert_{\D}$.

Further, we denote the singular points (i.e., the cusps and double points) on $\partial T(f)$ (respectively, on $\partial T(g)$) by $S^0(f)$ (respectively, $S^0(g)$), and define
$$
S^\infty(f):=\bigcup_{n=0}^\infty\sigma_f^{-n}(S^0(f)),\quad S^\infty(g):=\bigcup_{n=0}^\infty\sigma_g^{-n}(S^0(g)).
$$

By \cite[Corollary~41]{LMM20}, we have that 
$$
\widehat{\C}=\mathcal{B}_\infty(\sigma_g)\sqcup\Lambda(\sigma_g)\sqcup T^\infty(\sigma_g),
$$
where $\Lambda(\sigma_g)$ is the common boundary of $\mathcal{B}_\infty(\sigma_g)$ and $T^\infty(\sigma_g)$,
and
$$
\widehat{\C}=\mathcal{B}_\infty(\sigma_f)\sqcup\Lambda(\sigma_f)\sqcup T^\infty(\sigma_f),
$$
where $\Lambda(\sigma_f)$ is the common boundary of $\mathcal{B}_\infty(\sigma_f)$ and $T^\infty(\sigma_f)$ (for $\sigma_f$, this can also be seen directly from the construction of $f$).

By construction, $S^\infty(g)\subset\Lambda(\sigma_g)$. As $\Lambda(\sigma_g)$ is locally connected by \cite[Proposition~33]{LMM20}, the conformal conjugacy between $\overline{z}^d\vert_{\D}$ and $\sigma_g\vert_{\mathcal{B}_\infty(\sigma_g)}$ extends as a continuous semiconjugacy between $\overline{z}^d\vert_{\mathbb{S}^1}$ and $\sigma_g\vert_{\Lambda(\sigma_g)}$. It follows that $S^\infty(g)$ is dense on $\Lambda(\sigma_g)$. The same also holds for $f$.

Since $f$ and $g$ have isomorphic bi-angled trees, there exists an orientation-preserving homeomorphism $\Phi:T(f)\to T(g)$ that carries cusps to cusps and double points to double points, and is conformal on $\inter{T(f)}$. By iterated Schwarz reflection (equivalently, by iterated lifting under the Schwarz reflection maps $\sigma_f$ and $\sigma_g$), the map $\Phi$ can be extended to a homeomorphism
$$
\Phi: T^\infty(\sigma_f)\bigcup S^\infty(f)\longrightarrow T^\infty(\sigma_g)\bigcup S^\infty(g),
$$
such that $\Phi$ is a topological conjugacy between $\sigma_f$ to $\sigma_g$, and is conformal on $T^\infty(\sigma_f)$.

\noindent\textbf{Claim.} $\Phi$ extends to a homeomorphism $\Phi:\overline{T^\infty(\sigma_f)}\longrightarrow \overline{T^\infty(\sigma_g)}$ conjugating $\sigma_f$ to $\sigma_g$.

\begin{proof}[Proof of claim]
First note that the points of $S^0(f)$ (respectively, $S^0(g)$) determine a Markov partition $\mathcal{P}(f)\equiv\mathcal{P}(\sigma_f, S^0(f))$ (respectively, $\mathcal{P}(g)\equiv\mathcal{P}(\sigma_g, S^0(g))$) for $\sigma_f\vert_{\Lambda(\sigma_f)}$ (respectively, for $\sigma_g\vert_{\Lambda(\sigma_g)}$). Since $\Lambda(\sigma_f)$ (respectively, $\Lambda(\sigma_g)$) contains no critical point of $\sigma_f$ (respectively, of $\sigma_g$), the arguments of \cite[Theorem~4]{DU91} apply to the current setting to show that $\sigma_f\vert_{\Lambda(\sigma_f)}$ and $\sigma_g\vert_{\Lambda(\sigma_g)}$ are expansive. In particular, the diameters of the iterated preimages of the above Markov partition pieces shrink to zero uniformly (alternatively, shrinking of diameters of the iterated preimages of the pieces of $\mathcal{P}(f)$ and $\mathcal{P}(g)$ can be proved using \cite[Lemma~32]{LMM20} and the parabolic dynamics of $\sigma_f, \sigma_g$ at points of $S^0(f), S^0(g)$, respectively).

We will first show that $\Phi: S^\infty(f)\to S^\infty(g)$ admits a continuous extension $\Phi:\Lambda(\sigma_f)\to\Lambda(\sigma_g)$. Since $S^\infty(f)$ (respectively, $S^\infty(g)$) is dense on $\Lambda(\sigma_f)$ (respectively, on $\Lambda(\sigma_g)$), to prove the existence of a continuous extension $\Phi:\Lambda(\sigma_f)\to\Lambda(\sigma_g)$, it suffices to show that $\Phi:S^\infty(f)\to S^\infty(g)$ is uniformly continuous. To this end, let us fix $\varepsilon >0$. Now choose $N$ so that the diameters of the $\sigma_g^{\circ N}$-preimages of the pieces of $\mathcal{P}(g)$ are less than $\varepsilon$. Next, choose $\delta>0$ so that any two non-adjacent $\sigma_f^{\circ N}$-preimages of the pieces of $\mathcal{P}(f)$ are at least $\delta$ distance away.
If $x,y\in S^\infty(f)$ are at most $\delta$ distance apart, then they lie in two adjacent $\sigma_f^{\circ N}$-preimages of the pieces of $\mathcal{P}(f)$. It follows from the construction of $\Phi$ that $\Phi(x), \Phi(y)$ lie in two adjacent $\sigma_g^{\circ N}$-preimages of the pieces of $\mathcal{P}(g)$, and hence, $d(\Phi(x),\Phi(y))<2\varepsilon$. This proves uniform continuity of $\Phi:S^\infty(f)\to S^\infty(g)$.

Applying the same argument on $\Phi^{-1}$, we get a continuous inverse $\Phi^{-1}:\Lambda(\sigma_g)\to \Lambda(\sigma_f)$. Thus, $\Phi:\Lambda(\sigma_f)\to\Lambda(\sigma_g)$ is a homeomorphism extending $\Phi: S^\infty(f)\to S^\infty(g)$.

We have now defined a bijective map $\Phi:\overline{T^\infty(\sigma_f)}\to \overline{T^\infty(\sigma_g)}$, that is continuous restricted to $T^\infty(\sigma_f)$ and $\Lambda(\sigma_f)$ separately, and conjugates $\sigma_f$ to $\sigma_g$. To complete the proof of the claim, we only need to justify that $\Phi\vert_{\Lambda(\sigma_f)}$ continuously extends $\Phi\vert_{T^\infty(\sigma_f)}$.

To this end, first observe that each component of $T^\infty(\sigma_f), T^\infty(\sigma_g)$ is a Jordan domain. Hence, $\Phi\vert_{T^\infty(\sigma_f)}$ extends homeomorphically to the boundary of each component of $T^\infty(\sigma_f)$. Moreover, this extension agrees with $\Phi\vert_{\Lambda(\sigma_f)}$ at points of $S^\infty(f)$. Since $\overline{S^\infty(f)}=\Lambda(\sigma_f)$, we conclude that the homeomorphic extension of $\Phi\vert_{T^\infty(\sigma_f)}$ to the boundary of each component of $T^\infty(\sigma_f)$ agrees with $\Phi\vert_{\Lambda(\sigma_f)}$. Finally, local connectedness of $\Lambda(\sigma_f), \Lambda(\sigma_g)$ imply that the diameters of the components of $T^\infty(\sigma_f), T^\infty(\sigma_g)$ go to zero, from which continuity of $\Phi:\overline{T^\infty(\sigma_f)}\to \overline{T^\infty(\sigma_g)}$ follows.
\end{proof}

We now note that both $\sigma_f$ and $\sigma_g$ are conformally conjugate to $\overline{z}^d\vert_\D$ on their basins of infinity via their B{\"o}ttcher coordinates $\phi_{\sigma_f}:\D\to\mathcal{B}_\infty(\sigma_f)$ and $\phi_{\sigma_g}:\D\to\mathcal{B}_\infty(\sigma_g)$. Since $\Lambda(\sigma_f)$ and $\Lambda(\sigma_g)$ are locally connected, $\phi_{\sigma_f}$ and $\phi_{\sigma_g}$ extend continuously to $\mathbb{S}^1$. Moreover, by \cite[Lemma~44]{LMM20}, the continuous extension of $\phi_{\sigma_f}$ (respectively, of $\phi_{\sigma_g}$) sends the $(d+1)$-st roots of unity to the cusp points on $\partial T(f)$ (respectively, on $\partial T(g)$). After possibly precomposing $\phi_{\sigma_g}$ with multiplication by a $(d+1)$-st root of unity, we can assume that the conformal map 
$$
\widehat{\Phi}:=\phi_{\sigma_g}\circ\phi_{\sigma_f}^{-1}:\mathcal{B}_\infty(\sigma_f)\to\mathcal{B}_\infty(\sigma_g)
$$ 
continuously extends to the cusps of $\partial T(f)$, and agrees with $\Phi$ (constructed above) at these points. Hence, $\widehat{\Phi}$ must also continuously extend to the iterated $\sigma_f$--preimages of the cusps on $\partial T(f)$, and agree with $\Phi$ at these points. As the iterated preimages of the cusps on $\partial T(f)$ are dense on $\Lambda(\sigma_f)$, it is now easy to see that $\Phi\vert_{\Lambda(\sigma_f)}$ continuously extends $\widehat{\Phi}\vert_{\mathcal{B}_\infty(\sigma_f)}$.

Thus, we have constructed a topological conjugacy between $\sigma_f$ and $\sigma_g$ that is conformal away from $\Lambda(\sigma_f)$. On the other hand, since the basin of infinity of the hyperbolic anti-polynomial $P$ (used to construct $f$ in the existence part) is a John domain, and $\Psi$ is a global David homeomorphism, Theorem~\ref{theorem:john_removable} tells us that $\Lambda(\sigma_f)$ is conformally removable. This implies that $\sigma_f$ and $\sigma_g$ are M{\"o}bius conjugate. As this conjugacy must send the superattracting fixed point of $\sigma_f$ at $\infty$ to the superattracting fixed point of $\sigma_g$ at $\infty$, it follows that the conjugacy is affine. In particular, there exists an affine map $A$ carrying $f(\D^*)$ to $g(\D^*)$. Therefore, $$M:=\left(g\vert_{\D^*}\right)^{-1}\circ A\circ f:\D^*\to\D^*$$ is a conformal map. Since $A$ is affine, it follows that $M$ fixes $\infty$, and thus is a rotation. Note that since $\left(g\vert_{\D^*}\right)^{-1}(w)$ is of the form $w+O(\frac1w)$ near $\infty$, it follows that $A(w)= \alpha w$, for some $\alpha\in\C^*$. The fact that both $f$ and $g$ have derivative $1$ at $\infty$ implies that $A\equiv M$. Finally, as the coefficient of $1/z^d$ is $-\frac1d$ for both $f$ and $g$, it follows that $\alpha^{d+1}=1$. We conclude that $g=M\circ f\circ M^{-1}$, where $M$ is rotation by a $(d+1)$-st root of unity.
\end{proof}

\end{document}